\documentclass[12pt,oneside]{amsart}

\usepackage{geometry}                
\usepackage{graphicx}
\usepackage{amssymb}
\usepackage{epstopdf}
\usepackage{color}
\usepackage{bbm, dsfont,hyperref}

\numberwithin{equation}{section}

\usepackage{verbatim}

\usepackage[dvipsnames]{xcolor}

\usepackage{ulem}
\newtheorem{theorem}{Theorem}[section]
\newtheorem{lemma}[theorem]{Lemma}
\newtheorem{corollary}[theorem]{Corollary}
\newtheorem{proposition}[theorem]{Proposition}

\theoremstyle{definition}
\newtheorem{definition}[theorem]{Definition}
\theoremstyle{remark}
\newtheorem{remark}[theorem]{Remark}

\theoremstyle{example}

\numberwithin{equation}{section}

\usepackage{amsmath}
\usepackage{mathrsfs}

\usepackage{verbatim}
\usepackage{amsmath,amssymb,amsthm}             
\usepackage{amssymb}             
\usepackage{amsfonts}            
\usepackage{mathrsfs}            
\usepackage{amsthm}
\usepackage{algorithm}  
\usepackage{algorithmicx}  
\usepackage{algpseudocode}  
\usepackage{mathtools}
\usepackage{commath}
\usepackage{physics}
\usepackage{bm}
\usepackage{graphicx}

\usepackage{float}
\usepackage{listings}
\usepackage{subfigure}
\usepackage{multirow}
\usepackage{diagbox}
\usepackage{color}
\usepackage{bbm}
\usepackage[export]{adjustbox}
\usepackage{enumerate}
\usepackage{bbm}

\usepackage{amsmath}
\usepackage{mathrsfs}

\newcommand{\RNum}[1]{\uppercase\expandafter{\romannumeral #1\relax}}

\newcommand{\specificthanks}[1]{\@fnsymbol{#1}}

\begin{document}

	\title{One-arm exponent of critical level-set for metric graph Gaussian free field in high dimensions}
	
	\author{Zhenhao Cai$^1$}
	\address[Zhenhao Cai]{School of Mathematical Sciences, Peking University}
	\email{caizhenhao@pku.edu.cn}
	\thanks{$^1$School of Mathematical Sciences, Peking University}

	\author{Jian Ding$^1$}
	\address[Jian Ding]{School of Mathematical Sciences, Peking University}
	\email{dingjian@math.pku.edu.cn}
	\maketitle

	\tableofcontents
	
	\begin{abstract}
		In this paper, we study the critical level-set of Gaussian free field (GFF) on the metric graph $\widetilde{\mathbb{Z}}^d,d>6$. We prove that the one-arm probability (i.e. the probability of the event that the origin is connected to the boundary of the box $B(N)$) is proportional to $N^{-2}$, where $B(N)$ is centered at the origin and has side length $2\lfloor N \rfloor$. Our proof is hugely inspired by Kozma and Nachmias \cite{kozma2011arm} which proves the analogous result of the critical bond percolation for $d\geq 11$, and by Werner \cite{werner2021clusters} which conjectures the similarity between the GFF level-set and the bond percolation in general and proves this connection for various geometric aspects. 
	\end{abstract}

	\section{Introduction}
	
	In this paper, we study the Gaussian free field (GFF) on the metric graph $\widetilde{\mathbb{Z}}^d$. To define it precisely, we first review the definition of discrete Gaussian free field (DGFF) on the lattice $\mathbb{Z}^d$, where we assume $d\geq 3$ in this paper. For any $x\in \mathbb{Z}^d$, let $\mathbb{P}_x$ be the law of a continuous-time simple random walk $\{S_t\}_{t\ge 0}$ on $\mathbb{Z}^d$ with starting point $x$ and transition rate $\frac{1}{2d}$ in each direction. We denote the corresponding expectation of $\mathbb{P}_x$ by $\mathbb{E}_x$. The Green's function is defined as
	\begin{equation}
			G(x,y):= \mathbb{E}_x\left( \int_{0}^{\infty} \mathbbm{1}_{S_t=y}dt  \right) ,\ \  \forall x,y\in \mathbb{Z}^d.
	\end{equation}
	The DGFF $\{\phi_x\}_{x\in \mathbb{Z}^d}$ is a mean-zero Gaussian field, whose covariance is given by \begin{equation}
		\mathbb{E}\left( \phi_{x_1}\phi_{x_2} \right) =G(x_1,x_2),\ \ \forall x_1,x_2\in \mathbb{Z}^d. 
	\end{equation}

	We denote the $\ell^1$ and $\ell^\infty$ norms by $|\cdot|_1$ and $|\cdot|$ respectively.

The level-set $E^{\ge h}:=\{x\in \mathbb{Z}^d: \phi_x\ge h\}$ ($h\in \mathbb{R}$) of DGFF has been extensively studied. It was proved that $E^{\ge h}$ exhibits a non-trivial phase transition as the level $h$ varies, and that the critical level $h_*(d)$ is positive for all $d\ge 3$ (see Bricmont, Lebowitz and Maes \cite{bricmont1987percolation}, Rodriguez and Sznitman \cite{rodriguez2013phase}, Drewitz, Pr{\'e}vost and Rodriguez \cite{drewitz2018sign}). Drewitz and Rodriguez \cite{drewitz2015high} further proved that $h_*(d)$ is asymptotic to $\sqrt{2\log(d)}$ as $d\to \infty$. In their celebrated work \cite{duminil2023equality}, Duminil-Copin, Goswami, Rodriguez and Severo established that $h_*(d)$ also serves as the critical threshold between the strongly non-percolative regime and the strongly percolative regime:

\begin{itemize}
	\item For any $h>h_*$, the probability of the existence of a cluster (i.e. connected component) of $E^{\ge h}$ that crosses the annulus $B(2N)\setminus B(N)$ (where $B(M):=\{x\in \mathbb{Z}^d:|x|\le M\}$) converges to $0$ as $N\to \infty$;

	\item For any $h<h_*$, with at least $1-e^{-cR^{c'}}$ probability there exists a cluster of $E^{\ge h}\cap  B(N)$ with diameter at least $\tfrac{N}{5}$, and moreover, any two clusters of $E^{\ge h}\cap  B(N)$ with diameter at least $\tfrac{N}{10}$ are connected by $E^{\ge h}\cap  B(2N)$.

\end{itemize}
In addition, much has been understood regarding to the percolative properties for the level-set in both the supercritical and subcritical regimes (see e.g. Drewitz, Ráth and Sapozhnikov \cite{drewitz2014chemical}, Popov and R{\'a}th \cite{popov2015decoupling}, Popov and Teixeira \cite{popov2015soft}, Goswami, Rodriguez and Severo \cite{goswami2022radius}). Despite of extensive works, our understanding in the critical regime remains limited. In this work, we focus on the critical behavior of the level-set for the GFF on the \textit{metric graph}, which is much more tractable.

 Let $\mathbb{L}^d:=\{\{x,y\}:x,y\in \mathbb{Z}^d,|x-y|_1=1\}$ be the edge set of $\mathbb{Z}^d$. For each $e=\{x,y\}\in\mathbb{L}^d $, we consider $I_e$ as a compact interval of length $d$ with two endpoints identical to $x$ and $y$ respectively. The metric graph generated by $\mathbb{Z}^d$ is defined as $\widetilde{\mathbb{Z}}^d:= \cup_{e\in \mathbb{L}^d}I_e$.

	The GFF $\{\widetilde{\phi}_v\}_{v\in \widetilde{\mathbb{Z}}^d}$ on the metric graph, as an 
	extension of the DGFF, is defined as follows. Given a DGFF $\{\phi_x\}_{x\in \mathbb{Z}^d}$, set $\widetilde{\phi}_v=\phi_v$ for all lattice points $v\in \mathbb{Z}^d$. For each interval $I_e$ with $e=\{x_1,x_2\}$, $\{\widetilde{\phi}_v\}_{v\in I_e}$ is given by an independent Brownian bridge of length $d$ with variance $2$ at time $1$, conditioned on $\widetilde{\phi}_{x_1}=\phi_{x_1}$ and $\widetilde{\phi}_{x_2}=\phi_{x_2}$. Readers may refer to \cite[Section 2]{lupu2016loop} for more details of the construction of $\{\widetilde{\phi}_v\}_{v\in \widetilde{\mathbb{Z}}^d}$. For any $h\in \mathbb{R}$, we denote the level-set of $\widetilde{\phi}$ above $h$ by 
	\begin{equation}
	\widetilde{E}^{\ge h}:=\{v\in \widetilde{\mathbb{Z}}^d:\widetilde{\phi}_v\ge h \}. 
	\end{equation}


	Lupu \cite{lupu2016loop} proved that the critical level $\widetilde{h}_*$ of $\widetilde{E}^{\ge h}$ exactly equals to $0$. More precisely, for any $h<0$, the level-set $\widetilde{E}^{\ge h}$ almost surely percolates (i.e.  contains an infinite connected component). Moreover, at the critical level $h=\widetilde{h}_*=0$, $\widetilde{E}^{\ge 0}$ does not percolate. As a corollary, the so-called \textit{one-arm probability} $\mathbb{P}[\bm{0} \xleftrightarrow[]{\widetilde{E}^{\ge 0}}\partial B(N)]$ converges to $0$ as $N\to \infty$, where $\bm{0}$ is the origin of $\mathbb{Z}^d$, $\partial A:=\{x\in A:\exists y\in \mathbb{Z}^d\setminus A\ \text{such that}\ \{x,y\}\in \mathbb{L}^d\}$, and ``$A_1\xleftrightarrow[]{\widetilde{E}^{\ge 0}} A_2$'' denotes the event that there exists a path on $\widetilde{\mathbb{Z}}^d$ in $\widetilde{E}^{\ge 0}$ joining $A_1$ and $A_2$ (see Section \ref{section_notations} for the precise definition). As for quantitative estimates, Ding and Wirth \cite{ding2020percolation} employed a martingale argument and proved various polynomial bounds for the one-arm probability in various dimensions. Precisely, they proved that for any $d\ge 3$, there exist constants $C(d), C'(d)>0$ such that for all $N>1$, 
	\begin{itemize}
		\item when $d=3$,  
		\begin{equation}\label{3d-GFF}
			\frac{C(3)}{\sqrt{N}}	\le 	\mathbb{P}\left[\bm{0} \xleftrightarrow[]{\widetilde{E}^{\ge 0}}\partial B(N)\right] \le C'(3)\sqrt{\frac{\log(N)}{N}}; 
		\end{equation}
		
		\item when $d>3$, 
		\begin{equation}\label{ineq_onearm}
			\frac{C(d)}{N^{\frac{d}{2}-1}}	\le 	\mathbb{P}\left[ \bm{0} \xleftrightarrow[]{\widetilde{E}^{\ge 0}}\partial B(N)\right] \le \frac{C'(d)}{\sqrt{N}}. 
		\end{equation}
		
	\end{itemize}
	After that, with a different approach, Drewitz, Prévost and Rodriguez \cite{drewitz2022critical} substantially improved \cite{ding2020percolation} by extending the estimates to a large class of transient graphs (note their bounds for lattices are the same as in \cite{ding2020percolation}).

	Generally, physicists and mathematicians may conjecture that there exists a constant $\rho(d)>0$ called \textit{critical exponent} such that $\mathbb{P}[ \bm{0} \xleftrightarrow[]{\widetilde{E}^{\ge 0}}\partial B(N)]= N^{-\frac{1}{\rho}+o(1)}$. As mentioned in (\ref{3d-GFF}), the $3$-dimensional critical exponent has been proved to be $2$ while the parallel problems in other dimensions remain open. In this paper, we prove that the critical exponents for $d>6$ equal to $\frac{1}{2}$. 
	\begin{theorem}\label{theorem1}
		For $d>6$, there exist constants $C_1(d),C_2(d)>0$ such that for all $N>1$, 
		\begin{equation}
			C_1N^{-2}\le 	\mathbb{P}\left[ \bm{0} \xleftrightarrow[]{\widetilde{E}^{\ge 0}}\partial B(N)\right]  \le C_2N^{-2}.
		\end{equation}
	\end{theorem}

	\begin{remark}
		In light of Lupu's coupling (\cite[Proposition 2.1]{lupu2016loop}) between the GFF and the critical loop soup (see (\ref{eq_1.7})), the analogue of Theorem \ref{theorem1} also holds for the critical loop soup cluster on $\widetilde{\mathbb{Z}}^d$ ($d>6$).
	\end{remark}

	The parallel result of Theorem \ref{theorem1} for the critical bond percolation was conjectured to be true for $d>6$ (see e.g.  Grimmett \cite[Chapter 10]{grimmett1999percolation}). Up to now, this conjecture was only proved partially. In \cite{kozma2011arm}, Kozma and Nachmias  proved it under the following assumptions: (i) $d>6$; (ii) when $p=p_c(d)$ (where $p_c$ is the critical percolation probability), the \textit{two-point function} $\mathbb{P}[x\xleftrightarrow[]{}y]$ satisfies $\mathbb{P}[x\xleftrightarrow[]{}y]\asymp |x-y|^{2-d}$ (``$f\asymp g$'' means $\exists$ constants $C'>C>0$ such that $Cg\le f\le C'g$). Extensive efforts have been made on proving the assumption (ii) (see e.g. Fitzner and van der Hofstad \cite{fitzner2017mean}, Hara and Slade \cite{hara1990mean}). Currently, the best result is given by \cite{fitzner2017mean}, which confirms Assumption (ii) for $d\ge 11$. It is worth mentioning that Hara, Slade and van der Hofstad \cite{hara2003critical} verified Assumption (ii) for the percolation on the sufficiently spread-out lattice for $d>6$, where two non-adjacent points within a bounded distance can also form an edge opened with a polynomial decaying probability (with respect to the distance). In summary, for $7\le d\le 10$, the counterpart of Theorem \ref{theorem1} for bond percolation remains open.

Let us turn back to the case of the GFF on the metric graph $\widetilde{\mathbb{Z}}^d$. In \cite[Section 5]{werner2021clusters}, Werner asserted that in high dimensions (i.e. $d>6$), the GFF on $\widetilde{\mathbb{Z}}^d$
becomes asymptotically independent and thus shares similar behavior with
the bond percolation. Our Theorem \ref{theorem1} confirms his conjecture and heuristics from the perspective of the one-arm probability, and in fact our proof of Theorem \ref{theorem1} hugely benefits from some of the many valuable ideas presented in \cite{werner2021clusters} (see e.g. Section \ref{section_proof_WW} for more detailed discussions). However, compared to the bond percolation, GFF has a considerable polynomial correlation between the values of different sites, which causes numerous technical difficulties. How to handle the correlation and build local independence is the core of proving Theorem \ref{theorem1}. In earlier works on the level-set of GFF (see e.g. \cite{drewitz2018sign,drewitz2014chemical,popov2015decoupling,rodriguez2013phase}), one usually employs the idea of decoupling to obtain the desired independence, along which usually one has to slightly change the threshold for the level-set in order to get an inequality in the desired direction. Since we work at the criticality, it is essential that we stick to the critical threshold throughout our analysis. In order to address this challenge, we will follow the Kozma--Nachmias framework as in \cite{kozma2011arm} in the overview level, and in its implementation, we combine in a novel way many ingredients such as tree expansion, exploration processes, multi-scale analysis, and so on.

An interesting direction for future research is to establish the existence of the incipient infinite cluster (IIC) for the GFF on $\widetilde{\mathbb{Z}}^d$ for $d>6$, which can be understood as the critical cluster under the conditioning of growing to infinity. Provided with the construction of the IIC, it would then be natural to study its geometric properties, including the two-point function, the dimension and the random walk on the IIC. We remark that there have been numerous studies on the IIC of the bond percolation ($d\ge 19$) and the percolation on the sufficiently spread-out lattice ($d>6$). Notably, van der Hofstad and J{\'a}rai \cite{van2004incipient} and Heydenreich, van der Hofstad and Hulshof \cite{heydenreich2014high} constructed the IIC through applications of lace expansion, with also a computation of the two-point function. Additionally, van Batenburg \cite{10.1214/ECP.v20-3570} computed the mass dimension and volume growth exponent of the IIC. Furthermore, Kozma and Nachmias \cite{kozma2009alexander} studied the random walk on the IIC, and computed its spectral dimension as well as the diameter and range for the random walk thereon. See also Heydenreich, van der Hofstad and Hulshof \cite{Heydenreich2012RandomWO} for further progress on this.

	\section{Preliminaries}\label{section_notations}

	For the convenience and preciseness of exposition, we record some necessary notations, definitions and well-known results for random walks, Brownian motions, Gaussian free fields and loop soups in this section.

	\subsection{Graph, path and set}\label{fina_section_2.1}

	We denote by $\mathbb{N}$ (resp. $\mathbb{N}^+$) the set of non-negative integers (resp. positive integers). We also denote by $\mathbb{R}^+$ the set of positive real numbers. For any $x, y\in \mathbb{Z}^d$, we write $x\sim y$ if $\{x,y\}\in \mathbb{L}^d$. Recall that $\widetilde{\mathbb{Z}}^d=\cup_{e\in \mathbb{L}^d}I_{e}$, where each $I_e=I_{\{x,y\}}$ is an interval with length $d$ and endpoints $x,y$. Note that $\mathbb{Z}^d$ is a subset of $\widetilde{\mathbb{Z}}^d$. For any $v_1,v_2\in I_e$, we denote the sub-interval of $I_e$ with endpoints $v_1$ and $v_2$ by $I_{[v_1,v_2]}$. For any $t\in [0,1]$ and any $x,y\in \mathbb{Z}^d$ with $x\sim y$, let $x+t\cdot (y-x)$ be the point in $I_{\{x,y\}}$ such that the length of the sub-interval $I_{[x,x+t\cdot (y-x)]}$ equals to $td$.

	For a subset $A\subset \widetilde{\mathbb{Z}}^d$, we write $|A|$ for the number of lattice points included in $A$. The diameter of $A$ is $\mathrm{diam}(A):=\sup_{v_1,v_2\in A\cap \mathbb{Z}^d}|v_1-v_2|$. The boundary of $A$ is defined as 
	$\partial A:= \{x\in A:\exists y\in \mathbb{Z}^d\setminus A\ \text{such that}\ x\sim y  \}$.

	 A (time-parametrized) path on $\mathbb{Z}^d$ is a function $\eta:\left[0,T \right) \to \mathbb{Z}^d$ where $T\in \mathbb{R}^+$ (or $T=\infty$) such that there exist $m\in \mathbb{N}$ (or $m=\infty$), $0=T_0<...<T_{m+1}=T$ and $\{x_i\}_{0\le i<m+1 }\subset \mathbb{Z}^d$ with $x_i\sim x_{i+1}$ for all $0\le i< m$ such that 
		$$
		\eta(t)=x_i,\  \forall 0\le i< m+1,  t\in \left[T_i,T_{i+1} \right).
		$$
	For such a path $\eta$, its length is $\mathrm{len}(\eta)=m$. Note that if $m=0$, $\eta$ is also a path (although it contains only one point) and its length is $0$. The range of $\eta$ is $\mathrm{ran}(\eta):=\{x_0,...,x_m\}$. For $0\le i< m+1$, we say that $T_i(\eta)$ is the $i$-th jumping time, $\eta^{(i)}:=x_i$ is the $i$-th position, and $H_i(\eta):=T_{i+1}(\eta)-T_i(\eta)$ is the $i$-th holding time.


	A path on $\widetilde{\mathbb{Z}}^d$ is a continuous function $\widetilde{\eta}:\left[0,T \right) \to \widetilde{\mathbb{Z}}^d$, $T\in \mathbb{R}^+\cup \{\infty\}$. When $T$ is finite, we may denote $\widetilde{\eta}(T)=\lim_{t\to T}\widetilde{\eta}(t)$. From now, we always use the notation $\eta$ for a path on $\mathbb{Z}^d$ and $\widetilde{\eta}$ for a path on $\widetilde{\mathbb{Z}}^d$. With a slight abuse of notations, let $\mathrm{ran}(\widetilde{\eta}):=\{\widetilde{\eta}(t):0\le t< T\}$ be the range of $\widetilde{\eta}$. For $0\le t_1<t_2\le T$, we define the sub-path $\widetilde{\eta}\left[t_1,t_2 \right) : \left[0,t_2-t_1 \right) \to \widetilde{\mathbb{Z}}^d$ of $\widetilde{\eta}$ as $$
	\widetilde{\eta}\left[t_1,t_2 \right) (s)= \widetilde{\eta}(s+t_1), \ \ \forall s\in \left[0,t_2-t_1 \right) . 
	$$



%
%
%


%
%



	For any subsets $A_1,A_2,F\subset \widetilde{\mathbb{Z}}^d$, we say $A_1$ and $A_2$ are connected by $F$ if either $A_1\cap A_2\neq \emptyset$, or there is a path $\widetilde{\eta}$ contained in $F$ that intersects $A_1$ and $A_2$. We write this connection relation as $A_1\xleftrightarrow[]{F}A_2$. Especially, when $A_i=\{v\}$ for some $i\in \{1,2\}$ and $v\in \widetilde{\mathbb{Z}}^d$, we may omit the braces.


	For any $x\in \mathbb{Z}^d$ and $N>0$, let $B_x(N):= \left\lbrace y\in \mathbb{Z}^d: |x-y|\le N\right\rbrace$ be the box in $\mathbb{Z}^d$ with center $x$ and side length $2\lfloor N \rfloor$. We also define the box in $\widetilde{\mathbb{Z}}^d$ as follows:
	\begin{equation}
		\widetilde{B}_x(N):=\bigcup_{y_1,y_2\in B_x(N):y_1\sim y_2,\{y_1,y_2\}\cap B(N-1)\neq\emptyset}I_{\{y_1,y_2\}}.
	\end{equation}
	Note that any interval $I_{\{y_1,y_2\}}$ with $y_1,y_2\in \partial B(N)$ is not contained in $\widetilde{B}_x(N)$. Especially, when $x$ is exactly the origin, we may omit the subscript and write $B(N):=B_{\bm{0}}(N)$, $\widetilde{B}(N):=\widetilde{B}_{\bm{0}}(N)$.


	%

	\subsection{Statements about constants} We use notations $C,C',c,c',...$ for the local constants with values changing according to the context. The numbered notations $C_1,C_2,c_1,c_2,...$ are used for global constants, which are fixed throughout the paper. We usually use the upper-case letter $C$ (maybe with some superscript or subscript) for large constants and use the lower-case $c$ for small ones. In addition, we may also use some other letters such as $K,\lambda,\delta...$ for constants. When a constant depends on some parameter or variable, we will point it out in brackets. A constant without additional specification can only depend on the dimension $d$.

	\subsection{Stretched exponential and sub-polynomial functions}
	
	We say a function $f(\cdot)$ is stretched exponentially small if there exist constants $C,c,\delta>0$ such that $f(n)\le Ce^{-cn^{\delta}}$ for all $n\ge 1$. If $f(\cdot)$ is stretched exponentially small, we may write ``$f(\cdot)=\text{s.e.}(\cdot)$''. We also say a function is super-polynomially small if there exist constants $C,c,\delta>0$ such that $f(n)\le Ce^{-c\log^{1+\delta}(n)}$ for all $n\ge 1$. Similarly, we may use the notation ``$f(\cdot)=\text{s.p.}(\cdot)$'' for such a function.

	\subsection{Random walk, bridge measure, stopping time and capacity}
	
	Recall that $\mathbb{P}_x$ is the law of the continuous-time simple random walk $\{S_t\}_{t\ge 0}$ starting from $x$. For $i\in \mathbb{N}$, let $T_i$ be the $i$-th jumping time, $S^{(i)}$ be the $i$-th position and $H_i$ be the $i$-th holding time (note that $\{S_t\}_{t\ge 0}$ is a.s. a path on $\mathbb{Z}^d$). Then $\mathbb{P}_x$ satisfies $\mathbb{P}_x(S^{(0)}=x)=1$ and $\mathbb{P}_x(S^{(n+1)}=y_2|S^{(n)}=y_1)=(2d)^{-1}\cdot \mathbbm{1}_{y_1\sim y_2}$. In addition, the holding times $\{H_i\}_{i\in \mathbb{N}}$ are independent exponential random variables with rate $1$. We denote the expectation under $\mathbb{P}_x$ by $\mathbb{E}_x$. If the starting point is exactly the origin, we may omit the subscript.

The transition probability is denoted by $p_t(x,y):= \mathbb{P}_x\left(S_t=y \right)$ for all $x,y\in \mathbb{Z}^d$ and $t\ge 0$. The (normalized) bridge measure $\mathbb{P}_{x,y}^t(\cdot)$ is the conditional distribution of $\{S_{t'}\}_{0\le t'\le t}$ (starting from $x$) given $\{S_t=y\}$.

	For $A\subset \mathbb{Z}^d$, we denote the first time that $\{S_t\}_{t\ge 0}$ intersects $A$ by $\tau_A:=\inf\{t\ge 0:S_t\in A\}$. We also denote the hitting time by $\tau_A^+:=\inf\{t\ge T_1:S_t\in A\}$. For completeness, we set $\inf\emptyset=\infty$. Especially, when $A=\{x\}$ for some $x\in \mathbb{Z}^d$, we may omit the brackets.

	For any non-empty subset $A\subset \mathbb{Z}^d$ and $x\in A$, the escape probability of $x$ with respect to $A$ is $\mathrm{esc}_A(x):=\mathbb{P}_x(\tau_A^+=\infty)$. The capacity of $A$ is defined as $\mathrm{cap}(A):=\sum_{x\in A}\mathrm{esc}_A(x)$. By Lawler and Limic \cite[Proposition 6.5.2]{lawler2010random}, one has 
	\begin{equation}\label{ineq_2.2}
		\mathrm{cap}\left( B(N)\right)    \asymp N^{d-2}. 
	\end{equation}

		\subsection{Brownian motion $\widetilde{S}_t$ on $\widetilde{\mathbb{Z}}^d$}\label{subsection_BM} $\{\widetilde{S}_t\}_{t\ge 0}$ is a continuous-time Markov process on $\widetilde{\mathbb{Z}}^d$. When $\widetilde{S}_t$ is in the interior of some interval $I_e$, it behaves as a one-dimensional standard Brownian motion. Every time when $\widetilde{S}_t$ visits a lattice point $x$, it will uniformly choose a segment from $\{I_{\{x,y\}}\}_{y\in \mathbb{Z}^d:y\sim x}$ and behave as a Brownian excursion from $x$ in this interval. Once there is an excursion hitting some $y$ with $x\sim y$, the next step continues as the same process from the new starting point $y$. The total local time of all Brownian excursions at $x$ in this single step (i.e. the part of $\widetilde{S}_t$ from $x$ to one of its neighbors $y$) is an independent exponential random variable with rate $1$. We denote the law of $\{\widetilde{S}_t\}_{t\ge 0}$ starting from $v\in \widetilde{\mathbb{Z}}^d$ by $\widetilde{\mathbb{P}}_v$. Let $\widetilde{\mathbb{E}}_v$ be the expectation under $\widetilde{\mathbb{P}}_v$. Further details about the construction of $\widetilde{S}_t$ can be found in Folz \cite{a42bc627-9189-3879-b777-a46c2fd3d133}.

		By the aforementioned construction,  given $\{S_t\}_{t\ge 0}\sim \mathbb{P}_x$, the range of the Brownian motion $\{\widetilde{S}_t\}_{t\ge 0}\sim \widetilde{\mathbb{P}}_x$ can be recovered by taking the union of all the edges traversed by $S_t$ as well as additional Brownian excursions at each $S^{(i)}$ (for $i\in \mathbb N$) where the excursions are conditioned on returning to $S^{(i)}$ before hitting one of its neighbors and the total local time at $S^{(i)}$ is $H_i$.

%

		For any $A\subset \widetilde{\mathbb{Z}}^d$, similar to $\tau_A$, we denote the first time that $\widetilde{S}_t$ intersects $A$ by $\widetilde{\tau}_A:=\inf\{t\ge 0:\widetilde{S}_t\in A \}$.


	\subsection{Loop, loop measure and loop soup}\label{section_loops}

	In this part, we introduce some basic definitions and properties about loops on both $\mathbb{Z}^d$ and $\widetilde{\mathbb{Z}}^d$. As we will discuss in Section \ref{section_main_tool}, the \textit{isomorphism theorem} for GFF in \cite{lupu2016loop} is one of the main tools in this paper. Loops are core elements of this tool. We hereby give a partial list of literatures about the isomorphism theorem (see Le Jan \cite{le2011markov}, Marcus and Rosen \cite{marcus2006markov}, Rosen \cite{rosen2014lectures} and Sznitman \cite{sznitman2012topics} for an excellent account on this topic). In its earlier form, isomorphism theorems connect the law of the Gaussian free field and local times for random walks (see e.g. Ray \cite{ray1963sojourn} and Knight \cite{knight1963random}, Dynkin \cite{dynkin1984gaussian}, Marcus and Rosen \cite{marcus1992sample}, Eisenbaum \cite{eisenbaum2006version} and Eisenbaum, Kaspi, Marcus, Rosen and Shi \cite{eisenbaum2000ray}). Extensions of isomorphism theorems were a topic of interest in the past decade, including for random interlacements by Sznitman \cite{sznitman2012isomorphism,sznitman2012topics} and for permanent processes by Fitzsimmons and Rosen \cite{fitzsimmons2014markovian}, Le Jan, Marcus and Rosen \cite{le2015permanental}. Of particular interest to our work is the isomorphism theorem discovered in Lupu \cite{lupu2016loop}, which developed the method in \cite{fitzsimmons2014markovian,le2015permanental} and presented a coupling where the sign clusters of the GFF on the metric graph are the same as the loop soup clusters at criticality (i.e., when the intensity $\alpha$ equals to $1/2$). The coupling of \cite{lupu2016loop} is very powerful and has inspired many subsequent works. It is also worth pointing out that a weaker form of this coupling was questioned in Ding \cite{ding2014asymptotics} and was proved in Zhai \cite{10.1214/18-EJP149} (which was completed independently, after the completion of \cite{lupu2016loop}).

	\subsubsection{Time-parametrized loop on $\mathbb{Z}^d$} A (time-parametrized) rooted loop $\varrho$ on $\mathbb{Z}^d$ is a path on $\mathbb{Z}^d$ whose $0$-th position and $\mathrm{len}(\varrho)$-th position are the same point. We continue to use notations such as $T(\rho)$ and $T_i(\rho)$ for paths as introduced in Section \ref{fina_section_2.1}. Two rooted loops are equivalent if they equal to each other after a time-shift. Each equivalent class $\ell$ of such rooted loops are called a loop on $\mathbb{Z}^d$. As defined in \cite{le2011markov}, the loop measure $\mu$ on the space of rooted loops is 
	\begin{equation}
		\mu(\cdot) = \sum_{x\in \mathbb{Z}^d} \int_{0}^{\infty} t^{-1} \mathbb{P}_{x,x}^{t}(\cdot)p_t(x,x)dt.
	\end{equation}
	Referring to \cite[Section 2.3]{le2011markov}, $\mu$ is invariant under the time-shift. Thus, $\mu$ induces a measure on the space of loops, which is also denoted by $\mu$. Since $\mathrm{len}(\cdot)$ and $\mathrm{ran}(\cdot)$ are invariant under the time-shift, we can define the length and range of $\ell$ as $\mathrm{len}(\ell):=\mathrm{len}(\varrho)$ and $\mathrm{ran}(\ell):=\mathrm{ran}(\varrho)$ for any $\varrho\in \ell$.

	We cite some formulas about $\mu$ in \cite[Sections 2.1 and 2.3]{le2011markov} as follows. For any integer $k\ge 2$, any $0=t_0<t_1<...<t_k<t$ and any sequence of lattice points $x_0,...,x_k$ with $x_0=x_k$ and $x_i\sim x_{i+1}$ for all $0\le i\le k-1$, one has 
	\begin{equation}\label{fina_2.4}
		\begin{split}
			&\mu\left(T(\varrho)\in dt\  \text{and}\ \forall 0\le i\le k=\mathrm{len}(\varrho), \varrho^{(i)}=x_i,T_i(\varrho)\in dt_i    \right)\\
			= &t^{-1}e^{-t} \left(2d\right)^{-k} dt_1...dt_kdt.
		\end{split}
	\end{equation}
    For each aforementioned sequence $x_0,...,x_k$, its multiplicity $J=J(x_0,...,x_k)$ is the maximal integer such that the sub-sequences $(x_{(j-1)kJ^{-1}},x_{(j-1)kJ^{-1}+1},...,x_{jkJ^{-1}})$ for $1\le j\le J$ are identical. Then we have 
    \begin{equation}\label{fina_2.5}
    	\mu\left(\{\ell : \exists \varrho\in \ell\ \text{such that}\ \forall 0\le i\le k=\mathrm{len}(\varrho), \varrho^{(i)}=x_i\} \right)= J^{-1}(2d)^{-k}. 
    \end{equation}
	In addition, for any $x\in \mathbb{Z}^d$ and $t>0$, 
	\begin{equation}\label{fina_point_loop}
		\mu\left( \mathrm{len}(\varrho)=0,\varrho^{(0)}=x, T\in dt\right)=  t^{-1}e^{-t}dt.
	\end{equation}

	For any $\alpha>0$, the \textit{loop soup} $\mathcal{L}_{\alpha}$ is defined as the Poisson point process in the space of loops on $\mathbb{Z}^d$ with intensity measure $\alpha\mu$.

		\subsubsection{Continuous loop on $\widetilde{\mathbb{Z}}^d$}\label{subsection_continuous}

	In this subsection, we review the construction of continuous loop, loop measure and loop soup introduced in \cite{lupu2016loop}. We only focus on the case of $\widetilde{\mathbb{Z}}^d$ here, and we refer to \cite[Section 2]{lupu2016loop} for more details on the construction for general graphs.

	A rooted loop on $\widetilde{\mathbb{Z}}^d$ is a path $\widetilde{\varrho}:\left[0,T \right) \to \widetilde{\mathbb{Z}}^d$ such that $\widetilde{\varrho}(0)=\widetilde{\varrho}(T)$. Similarly, a loop on $\widetilde{\mathbb{Z}}^d$ is an equivalent class of rooted loops such that each of them can be transformed into another by a time-shift. In this paper, we use the notations $\eta$, $\varrho$ and $\ell$ for a path, a rooted loop and a loop on $\mathbb{Z}^d$ respectively. We also use $\widetilde{\eta}$, $\widetilde{\varrho}$ and $\widetilde{\ell}$ for their counterparts on $\widetilde{\mathbb{Z}}^d$. Since $\mathrm{ran}(\widetilde{\varrho})$ is invariant for all $\widetilde{\varrho}\in \widetilde{\ell}$, we denote the range of $\widetilde{\ell}$ by $\mathrm{ran}(\widetilde{\ell}):=\mathrm{ran}(\widetilde{\varrho})$ for some $\widetilde{\varrho}\in \widetilde{\ell}$.


	In fact, the loops on $\mathbb{Z}^d$ and $\widetilde{\mathbb{Z}}^d$ can be divided into the following types: 
	\begin{enumerate}
		
		\item fundamental loop: a loop that visits at least two lattice points;

		\item point loop: a loop that visits exactly one lattice point;

		\item edge loop (only for loops on $\widetilde{\mathbb{Z}}^d$): a loop that is contained by a single interval $I_e$ and visits no lattice point.
		
	\end{enumerate}

By the method in  \cite{fitzsimmons2014markovian}, one can use $\{\widetilde{S}_t\}_{t\ge 0}$ to contruct a measure $\widetilde{\mu}$ on the space of the continuous loops on $\widetilde{\mathbb{Z}}^d$. For each $\alpha>0$, the loop soup on $\widetilde{\mathbb{Z}}^d$ of parameter $\alpha$, denoted by $\widetilde{\mathcal{L}}_\alpha$, is the Poisson point process with intensity measure $\alpha \widetilde{\mu}$. Actually, we always focus on $\widetilde{\mathcal{L}}_{1/2}$ in this paper. Let $\widetilde{\mathcal{L}}_{1/2}^{\mathrm{f}}$ (resp. $\widetilde{\mathcal{L}}_{1/2}^{\mathrm{p}}$, $\widetilde{\mathcal{L}}_{1/2}^{\mathrm{e}}$) be the point measure composed of fundamental loops (resp. point loops, edge loops) in $\widetilde{\mathcal{L}}_{1/2}$. We also denote by $\mathcal{L}_{1/2}^{\mathrm{f}}$ (resp. $\mathcal{L}_{1/2}^{\mathrm{p}}$) the counterpart of $\widetilde{\mathcal{L}}_{1/2}^{\mathrm{f}}$ (resp. $\widetilde{\mathcal{L}}_{1/2}^{\mathrm{p}}$) for $\mathcal{L}_{1/2}$. By the thinning property of Possion point processes, $\widetilde{\mathcal{L}}_{1/2}^{\mathrm{f}}$, $\widetilde{\mathcal{L}}_{1/2}^{\mathrm{p}}$ and $\widetilde{\mathcal{L}}_{1/2}^{\mathrm{e}}$ (resp. $\mathcal{L}_{1/2}^{\mathrm{f}}$ and $\mathcal{L}_{1/2}^{\mathrm{p}}$) are independent.

For the sake of brevity, we do not distinguish a point measure $\mathcal{L}$ from the support of $\mathcal{L}$ in notation. Hence, we may write ``$\widetilde{\ell} \in \mathcal{L}$'' for ``$\widetilde{\ell}$ is in the support of $\mathcal{L}$''.

In what follows, we review a construction of $\widetilde{\mathcal{L}}_{1/2}$ in \cite[Section 2]{lupu2016loop}, by which  $\widetilde{\mathcal{L}}_{1/2}$ can be obtained by adding Brownian excursions to the loops in $\mathcal{L}_{1/2}$.

	\textbf{For $\widetilde{\mathcal{L}}_{1/2}^{\mathrm{f}}$:} There is a coupling of $\mathcal{L}_{1/2}^{\mathrm{f}}$ and $\widetilde{\mathcal{L}}_{1/2}^{\mathrm{f}}$, which is equipped with a one-to-one mapping $\pi$ between their loops (from $\mathcal{L}_{1/2}^{\mathrm{f}}$ to $\widetilde{\mathcal{L}}_{1/2}^{\mathrm{f}}$). Moreover, given $\mathcal{L}_{1/2}^{\mathrm{f}}$, the range of each $\pi(\ell)$ for $\ell \in \mathcal{L}_{1/2}^{\mathrm{f}}$ can be recovered as follows (this is parallel to the discussions at the end of Section \ref{subsection_BM}). Arbitrarily take $\varrho\in \ell$. For each $0\le i\le \mathrm{len}(\varrho)$, let $\mathcal{B}_i$ be the union of Brownian excursions with total local time $H_i(\varrho)$, starting from $\varrho^{(i)}$ and conditioning on hitting $\varrho^{(i)}$ before its lattice neighbors. Note that $\mathcal{B}_i$ has the same distribution as $\cup_{z\in \mathbb{Z}^d:z\sim \varrho^{(i)}}I_{[\varrho^{(i)},\ \varrho^{(i)}+d^{-1}M_{i}^{z}\cdot (z-\varrho^{(i)})]}$, where $M_{i}^{z}$ is the maximum of the square of a Bessel-$0$ process with initial value $H_i(\varrho)$, conditioning on hitting $0$ before time $d$. Then the range of $\pi(\ell)$ is 
	$$
	  \bigcup_{0\le i\le \mathrm{len}(\varrho)-1 } I_{\{\varrho^{(i)},\varrho^{(i+1)}\}} \cup \bigcup_{0\le i\le \mathrm{len}(\varrho) } \mathcal{B}_i. 
	$$
 As a corollary, one has that a.s. 
		$$
	\mathrm{ran}(\ell) \subset \mathrm{ran}(\pi(\ell))\subset \cup_{x\in \mathrm{ran}(\ell) } \widetilde{B}_x(1). 
	$$

	\textbf{For $\widetilde{\mathcal{L}}_{1/2}^{\mathrm{p}}$:} Recall that the distribution of loops in $\mathcal{L}_{1/2}^{\mathrm{p}}$ is given by (\ref{fina_point_loop}). For any $x\in \mathbb{Z}^d$, let $\gamma_x^{\mathrm{p}}$ be the union of ranges of loops in $\widetilde{\mathcal{L}}_{1/2}^{\mathrm{p}}$ including $x$. Given the total holding time $H_x$ of loops in $\mathcal{L}_{1/2}^{\mathrm{p}}$ including $x$, then $\gamma_x^{\mathrm{p}}$ has the same distribution as the union of Brownian excursions with total local time $H_x$, starting from $x$ and conditioning on hitting $x$ before its lattice neighbors.

%
%
%
%
%
%
%
%
%
%
%
%
%
	
	\textbf{For $\widetilde{\mathcal{L}}_{1/2}^{\mathrm{e}}$:} For any $\{x,y\}\in \mathbb{L}^d$, we denote by $\gamma_{\{x,y\}}^{\mathrm{e}}$ the union of ranges of loops in $\widetilde{\mathcal{L}}_{1/2}^{\mathrm{e}}$ whose range is contained in $I_{\{x,y\}}$. Each $\gamma_{\{x,y\}}^{\mathrm{e}}$ has the same distribution as the non-zero points of a standard Brownian bridge in $I_{\{x,y\}}$ of length $d$, from $0$ at $x$ to $0$ at $y$.


	\subsubsection{Decomposition of loops on $\widetilde{\mathbb{Z}}^d$}\label{subsection_decomposition}

	We present an approach to decompose a loop $\widetilde{\ell}$. This decomposition is a continuous analogue of that introduced in Chang and Sapozhnikov \cite[Section 2.3]{chang2016phase} and is closely related to the \textit{spatial Markov property} of (both discrete and continuous) loop soups. Further discussions about this property can be found in Werner \cite{werner2016spatial}.


	 For two disjoint subsets $A_1,A_2\subset \mathbb{Z}^d$, consider a mapping $L(A_1,A_2)$ as follows. For a loop $\widetilde{\ell}$, define $L(A_1,A_2)(\widetilde{\ell})$ as the collection of rooted loops $\widetilde{\varrho}:\left[0,T \right) \to \widetilde{\mathbb{Z}}^d$ in the equivalence class $\widetilde{\ell}$ such that 
	\begin{itemize}
		\item $\widetilde{\varrho}(0)\in A_1$;
		
		\item $\exists t\in (0,T)$ such that $\widetilde{\varrho}(t)\in A_2$ and for all $t'\in (t,T)$, $\widetilde{\varrho}(t')\notin A_1\cup A_2$. 
	\end{itemize}
	Note that $\widetilde{\ell}$ intersects $A_1$ and $A_2$ if and only if $L(A_1,A_2)(\widetilde{\ell})\neq \emptyset$. For each $\widetilde{\varrho} \in L(A_1,A_2)(\widetilde{\ell})$, one can define a sequence of stopping time as follows:
	\begin{enumerate}
		\item $\widetilde{\tau}_0=0$; 
		
		\item $\forall k\ge 0$, $\widetilde{\tau}_{2k+1}:= \inf\{t>\widetilde{\tau}_{2k}: \widetilde{\varrho}(t)\in A_2\}$; 
		
		\item $\forall k\ge 0$, $\widetilde{\tau}_{2k+2}:= \inf\{t>\widetilde{\tau}_{2k+1}: \widetilde{\varrho}(t)\in A_1\}$.
		
	\end{enumerate}
	Let $\kappa(\widetilde{\varrho})=\kappa(\widetilde{\varrho};A_1,A_2)$ be the unique integer such that $\widetilde{\tau}_{2\kappa}=T$. Since $\kappa(\widetilde{\varrho})$ is constant for all $\widetilde{\varrho}\in L(A_1,A_2)(\widetilde{\ell})$, we also denote it by $\kappa(\widetilde{\ell})$. Note that $2\kappa(\widetilde{\ell})$ is the number of excursions in $\widetilde{\ell}$ between $A_1$ and $A_2$. For $1\le i\le \kappa(\widetilde{\varrho})$, we define the $i$-th forward crossing path as the sub-path $\widetilde{\eta}^\mathrm{F}_i:= \widetilde{\varrho}\left[ \widetilde{\tau}_{2i-2},\widetilde{\tau}_{2i-1}\right) $, and define the $i$-th backward crossing path as the sub-path  $\widetilde{\eta}^\mathrm{B}_i:= \widetilde{\varrho}\left[ \widetilde{\tau}_{2i-1},\widetilde{\tau}_{2i}\right) $. In fact, for any $\widetilde{\varrho}_1,\widetilde{\varrho}_2\in L(A_1,A_2)(\widetilde{\ell})$, the sequences of the forward crossing paths (also backward crossing paths) of $\widetilde{\varrho}_1$ and $\widetilde{\varrho}_2$, say $\{\widetilde{\eta}^\mathrm{F}_{1,i}\}_{1\le i\le \kappa(\widetilde{\ell})}$ and $\{\widetilde{\eta}^\mathrm{F}_{2,i}\}_{1\le i\le \kappa(\widetilde{\ell})}$, are identical to each other under an index translation. I.e., there is an integer $a_*\in [1,\kappa(\widetilde{\ell})-1]$ such that $\widetilde{\eta}^\mathrm{F}_{1,i}=\widetilde{\eta}^\mathrm{F}_{2,i_*}$	for all $1\le i\le \kappa(\widetilde{\ell})$, where $i_*\equiv i+a_*\  \text{mod}\ \kappa(\widetilde{\ell})$. Note that only forward crossing paths can intersect $A_1$ and no backward crossing path can. I.e., $\mathrm{ran}(\widetilde{\ell})\cap A_1= \cup_{i=1}^{\kappa(\widetilde{\ell})} \mathrm{ran}(\widetilde{\eta}^\mathrm{F}_i)\cap A_1$. See Figure \ref{fig1} for an illustration for this decomposition.

\begin{figure}[h]
	\centering
	\includegraphics[width=10cm]{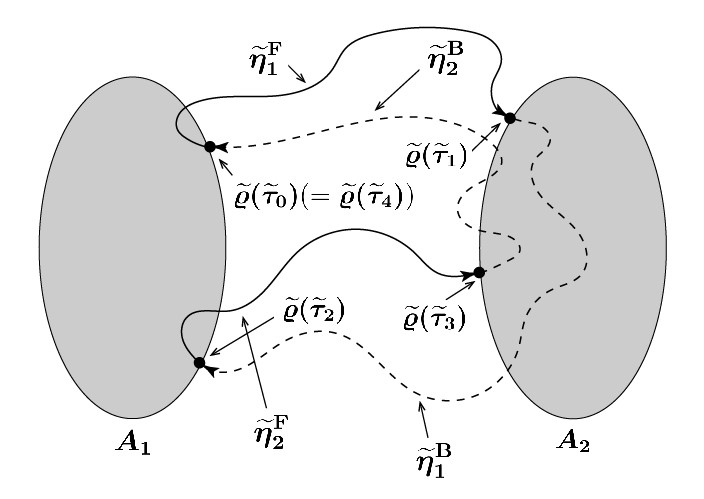}
	\caption{An illustration for the decomposition of a loop $\widetilde{\ell}$ with $\kappa(\widetilde{\ell};A_1,A_2)=2$. $\widetilde{\varrho}$ is a rooted loop in $L(A_1,A_2;\widetilde{\ell})$. }
	\label{fig1}
\end{figure}

	For any loop $\ell$, as the counterpart of $\kappa(\widetilde{\ell})$, we also define $\kappa(\ell)=\kappa(\ell;A_1,A_2)$ as the integer such that $2\kappa(\ell)$ is the number of excursions in $\ell$ between $A_1$ and $A_2$. By the relation between the loops on $\mathbb{Z}^d$ and $\widetilde{\mathbb{Z}}^d$ presented in Section \ref{subsection_continuous}, we have: for any $j\in \mathbb{N}^+$, 
	\begin{equation}\label{fina2.7}
		\widetilde{\mu}\Big(\big\{\widetilde{\ell}:\kappa(\widetilde{\ell};A_1,A_2 )=j\big\} \Big)=\mu\Big(\big\{\ell:\kappa(\ell;A_1,A_2 )=j\big\} \Big). 
	\end{equation}

	 We define a sequence of stopping times for the simple random walk as follows. For $\{S_{t}\}_{t\ge 0} \sim \mathbb{P}_x$ with $x\in \partial A_1$, we set $\hat{\tau}_0:=0$. For $i\in \mathbb{N}^+$, let $\hat{\tau}_{2i-1}:=\inf\{t>\hat{\tau}_{2i-2}:S_t\in A_2\}$ and $\hat{\tau}_{2i}:=\inf\{t>\hat{\tau}_{2i-1}:S_t\in A_1\}$. The following lemma is useful to the subsequent proof. 
\begin{lemma}[{\cite[Proposition 27 and  Exercise 31]{le2011markov}}]\label{fina_lemma2.1}
For any disjoint subsets $A_1,A_2\subset \mathbb{Z}^d$ and $j\in \mathbb{N}^+$,  
\begin{equation}
	\mu\left(\{\ell:\kappa(\ell;A_1,A_2 )=j\} \right) = j^{-1} \sum_{x\in \partial A_1}  \mathbb{P}_x\left(S_{\hat{\tau}_{2k}}=x \right).  
\end{equation}
\end{lemma}

For any $N\ge 1$, $x\in \partial B(N)$ and $A\subset B(\frac{N}{2})$, by \cite[Proposition 6.5.2]{lawler2010random} and  (\ref{ineq_2.2}),  
	 \begin{equation}\label{fina2.8}
	 		\mathbb{P}_z\left(\tau_A<\infty \right)\asymp \mathrm{cap}(A)N^{2-d}. 
	 \end{equation}
Taking $A_1= A$ and $A_2=\partial B(N)$, by (\ref{fina2.7}), Lemma \ref{fina_lemma2.1} and (\ref{fina2.8}), we get the following corollary, which is frequently-used in this paper. A similar result of this corollary can be found in \cite[Lemma 2.7]{chang2016phase}.

\begin{corollary}\label{corollary_2.2}
	For any $N\ge 1$, $A\subset B(\frac{N}{2})$ and $j\in \mathbb{N}^+$, we have 
			\begin{equation}\label{ineq_2.4}
			\bigg[ 	\widetilde{\mu}\Big(\big\{\widetilde{\ell}:\kappa(\widetilde{\ell};A,\partial B(N) )=j\big\} \Big)\bigg]^{j^{-1}}\asymp   \mathrm{cap}(A) N^{2-d}.
			\end{equation}
			As a direct consequence, one has 
			\begin{equation}\label{ineq_2.5}
				\widetilde{\mu}\left(\{\widetilde{\ell}:\mathrm{ran}(\widetilde{\ell})\cap A\neq\emptyset,\mathrm{ran}(\widetilde{\ell})\cap\partial B(N)\neq\emptyset \}\right)\asymp  \mathrm{cap}(A) N^{2-d}.
			\end{equation}
\end{corollary}

	For any disjoint subsets $A_1,A_2\subset \mathbb{Z}^d$ and a sequence of paths $\widetilde{\eta}_i:\left[ 0,T^i\right) \to \widetilde{\mathbb{Z}}^d$ for $1\le i\le k$ such that $\widetilde{\eta}_i(0)\in \partial A_2$, $\mathrm{ran}(\widetilde{\eta}_i)\cap A_1=\emptyset$ and $\widetilde{\eta}_i(T^i)\in A_1$, let $\mathfrak{L}' = \mathfrak{L}'(\widetilde{\eta}_1,... \widetilde{\eta}_k)$ be the collection of loops $\widetilde{\ell}$ with $\kappa(\widetilde{\ell})=k$ such that there exists a rooted loop $\widetilde{\varrho}\in L(A_1,A_2)(\widetilde{\ell})$, whose backward crossing paths are exactly $\widetilde{\eta}_i$ for $1\le i\le k$. Unless otherwise stated, we assume that $\widetilde{\eta}_1,... \widetilde{\eta}_k$ are different from one another to avoid the issue of periodicity. This assumption will not make any essential difference since the loop measure of all loops violating this property is $0$.

	We conclude this section by presenting the following lemma.

%

\begin{lemma}\label{lemma_new_2.3}
	We keep the notations in the last paragraph. Under the measure $\widetilde{\mu}$, conditioning on $\{\widetilde{\ell} \in \mathfrak{L}'\}$, the forward crossing paths, say $\widetilde{\eta}_i^{\mathrm{F}}$ for $1\le i\le k$, are independently distributed. Moreover, for the forward crossing path that starts from $x_{i-1}^{+}:=\widetilde{\eta}_{i-1}(T^{i-1})$ (where $x_{0}^{+}:=\widetilde{\eta}_{k}(T^{k})$) and ends at $x_{i}^{-}:=\widetilde{\eta}_{i}(0)$, its conditional distribution is given by
	$  \widetilde{\mathbb{P}}_{x_{i-1}^{+}}( \ \cdot   \ \big|\widetilde{\tau}_{A_2}=\widetilde{\tau}_{x_{i}^{-}} )$.
\end{lemma}

\cite[Proposition 3]{werner2016spatial} proved the following property of $\mathcal{L}_{1/2}$ on $\mathbb{Z}^d$. For any disjoint subsets $A_1,A_2\subset \mathbb{Z}^d$, conditioning on all excursions of loops in $\mathcal{L}_{1/2}$ that start from $A_1$, then intersect $A_2$ and finally return to $A_1$, the missing parts of the loops in $\mathcal{L}_{1/2}$ intersecting $A_1$ and $A_2$ can be sampled in two steps as follows: 
\begin{enumerate}

	\item   Suppose that the returning points and departure points of these excursions are $\{x_i\}_{i=1}^k$ and $\{y_i\}_{i=1}^k$ respectively. Sample a pairing $\{(x_i,y_{\sigma_i})\}_{i=1}^k$ (where $(\sigma_1,...,\sigma_k)$ is a permutation of $(1,...,k)$) with probability proportional to $\prod_{i=1}^kG_{A_2}\left(x_i,y_{\sigma_i}\right)$ (where $G_{A}(x,y):=\mathbb{E}_x\left(\int_{0}^{\tau_{A}}\mathbbm{1}_{S_t=y} dt \right)$ is the Green's function restricted in $\mathbb{Z}^d\setminus A$).

	\item   Given the pairing sampled above, the missing parts (i.e. the paths $\{\eta_i\}_{i=1}^k$ where $\eta_i$ starts from $x_i$, ends at $y_{\sigma_i}$ and does not intersect $A_2$) are independent, and in addition the law of each $\eta_i$ is given by $\mathbb{P}(\cdot  \mid \tau_{y_{\sigma_i}} < \tau_{A_2})$.

\end{enumerate}
In \cite{werner2016spatial}, it is also stated that the analogous proposition holds for $\widetilde{\mathcal{L}}_{1/2}$. In fact, this proposition provides even more information than Lemma \ref{lemma_new_2.3} since it not only ensures the independence between different remaining paths, but also describes the probabilities of various loop structures. Although Lemma \ref{lemma_new_2.3} and \cite[Proposition 3]{werner2016spatial} slightly differ in the definitions of excursions between disjoint subsets, their proofs are highly similar and thus we omit proof details for Lemma \ref{lemma_new_2.3}. More general statements for the spatial Markov property can be found in \cite[Section 3.3]{werner2016spatial}.

	\section{Main tools}\label{section_main_tool}

	\subsection{Isomorphism theorem: a coupling between GFF and loop soup} 
	
	In \cite{lupu2016loop}, Lupu showed a coupling between two continuous random fields on the metric graph: the GFF $\{\widetilde{\phi}_v\}_{v\in \widetilde{\mathbb{Z}}^d}$ and the occupation field $\{\widehat{\mathcal{L}}^v_{1/2}\}_{v\in \widetilde{\mathbb{Z}}^d}$ of the loop soup $\widetilde{\mathcal{L}}_{1/2}$. Precisely, for any $v\in \widetilde{\mathbb{Z}}^d$, $\widehat{\mathcal{L}}^v_{1/2}$ is the sum of local times of all loops in $\widetilde{\mathcal{L}}_{1/2}$ at $v$. In this paper, it is sufficient to note that the collection of $v\in \widetilde{\mathbb{Z}}^d$ with $\widehat{\mathcal{L}}^v_{1/2}>0$ is exactly $\cup \widetilde{\mathcal{L}}_{1/2}$ (for convenience, we denote the union of ranges of loops in a point measure $\mathcal{L}$ (resp. a collection $\mathfrak{L}$) by $\cup \mathcal{L}$ (resp. $\cup \mathfrak{L}$)).
	

	
	\begin{lemma}[{\cite[Proposition 2.1]{lupu2016loop}}]\label{lemma_iso}
		There is a coupling between the loop soup $\widetilde{\mathcal{L}}_{1/2}$ and the GFF $\{\widetilde{\phi}_v\}_{v\in \widetilde{\mathbb{Z}}^d}$ such that 
		\begin{itemize}
			\item for any $v\in \widetilde{\mathbb{Z}}^d$, $\widehat{\mathcal{L}}^v_\alpha= \frac{1}{2}\phi_v^2$; 
			
			\item the clusters composed of loops in $\widetilde{\mathcal{L}}_{1/2}$ are exactly the sign clusters of $\{\widetilde{\phi}_v\}_{v\in \widetilde{\mathbb{Z}}^d}$, where a ``sign cluster'' is a maximal connected subgraph on which $\widetilde{\phi}$ has the same sign (every $v\in \widetilde{\mathbb{Z}}^d$ with $\widetilde{\phi}_v=0$ does not belong to any sign cluster).

		\end{itemize}
	\end{lemma}
	
	
	Through Lemma \ref{lemma_iso}, a profound link between the GFF and the loop soup is unveiled, enabling a rich interplay between the GFF and the loop soup. Of particular interest to us, we get from the symmetry of GFF and Lemma \ref{lemma_iso} that
	\begin{equation}\label{eq_1.7}
		\begin{split}
			\mathbb{P}\left[ \bm{0} \xleftrightarrow[]{\widetilde{E}^{\ge 0}}\partial B(N)\right] = &\mathbb{P}\left[ \bm{0}  \xleftrightarrow[]{\widetilde{E}^{> 0}}\partial B(N)\right]\\
			=&\frac{1}{2} \mathbb{P}\left[ \bm{0}  \xleftrightarrow[]{\text{the union of all sign clusters of}\ \widetilde{\phi} }\partial B(N)\right]\\
			=& \frac{1}{2} \mathbb{P}\left[ \bm{0} \xleftrightarrow[]{\cup \widetilde{\mathcal{L}}_{1/2} }\partial B(N)\right]. 
		\end{split}
	\end{equation}
	Note that the first line of (\ref{eq_1.7}) follows from the following two facts:
	\begin{itemize}
		\item For any $x\in \mathbb{Z}^d$, $\mathbb{P}[\widetilde{\phi}_x=0]=0$; 
		
		\item For any interval $I_{\{x,y\}}$, arbitrarily given the values of endpoints $\widetilde{\phi}_x,\widetilde{\phi}_y\neq 0$, one has that $\{\widetilde{\phi}_v\}_{v\in I_{\{x,y\}}}$ (which is given by the Brownian bridge) a.s. does not have extremum $0$ in $I_{\{x,y\}}$.
		
	\end{itemize}

	\subsection{Two-point function} Using the isomorphism theorem, Lupu \cite{lupu2016loop} proved an explicit formula of the probability that two points are connected by $\cup \widetilde{\mathcal{L}}_{1/2}$. 
	
	\begin{lemma}[{\cite[Proposition 5.2]{lupu2016loop}}]\label{lemma_two_point}
		For any $x,y\in \mathbb{Z}^d$, 
		\begin{equation}
			\mathbb{P}\bigg(  x\xleftrightarrow[]{\cup \widetilde{\mathcal{L}}_{1/2}} y\bigg)  = \frac{2}{\pi}\arcsin(\frac{G(x,y)}{\sqrt{G(x,x)G(y,y)}}). 
		\end{equation}
	\end{lemma}

	It is well-known that the Green's function satisfies $G(x,y)\asymp |x-y|^{2-d}$. Thus, by Lemma \ref{lemma_two_point} we have 
	\begin{equation}\label{eq_two_points}
		\mathbb{P}\bigg(  x\xleftrightarrow[]{\cup \widetilde{\mathcal{L}}_{1/2}} y\bigg)   \asymp |x-y|^{2-d}, \ \ \forall x,y\in \mathbb{Z}^d.
	\end{equation}

	\subsection{BKR inequality}\label{section_BKR}
	
	In this subsection, we introduce another useful tool, the \textit{van den Berg-Kesten-Reimer inequality}. This inequality was conjectured in van den Berg and Kesten \cite{van1985inequalities} and then was proved by van den Berg and Fiebig \cite{van1987combinatorial} and Reimer \cite{reimer2000proof}. Borgs, Chayes and Randall \cite{borgs1999van} provided a nice exposition for this inequality.

%
%
	
	Recall the notations $\gamma_{x}^{\mathrm{p}}$ and $\gamma_{\{x,y\}}^{\mathrm{e}}$ in Section \ref{subsection_continuous}. For any connected $A\subset \mathbb{Z}^d$ with $|A|\ge 2$, let $\gamma_A^{\mathrm{f}}$ be the union of ranges of loops in $\widetilde{\mathcal{L}}_{1/2}^{\mathrm{f}}$ that visit every point in $A$ and do not visit any other lattice point. For each of $\gamma_{x}^{\mathrm{p}}$, $\gamma_{\{x,y\}}^{\mathrm{e}}$ and $\gamma_{A}^{\mathrm{f}}$, we call it a glued loop. Note that each glued loop is a random subset of $\widetilde{\mathbb{Z}}^d$, but not a loop on $\widetilde{\mathbb{Z}}^d$. We say a collection of glued loops certifies an event $\mathsf{A}$ if on the realization of this collection of glued loops, $\mathsf{A}$ happens regardless of the realization of all other glued loops. For two events $\mathsf{A}$ and $\mathsf{B}$, let $\mathsf{A} \circ \mathsf{B}$ be the event that there exist two disjoint collections of glued loops such that one collection certifies $\mathsf{A}$, and the other certifies $\mathsf{B}$. Note that in this context, ``two disjoint collections'' implies that the collections do not contain any glued loops with matching subscripts and superscripts, but it does not necessarily mean that every glued loop in one collection does not intersect any glued loop in the other collection.


	Recall in Section \ref{subsection_continuous} that each glued loop is measurable with respect to several random variables, whose distributions have been written down rigorously. Therefore, this satisfies the requirements of the framework introduced in Arratia, Garibaldi and Hales \cite{arratia2018van} for the BKR inequality on continuous spaces. Thus, we have the following lemma:
	\begin{lemma}[BKR inequality I]\label{lemma_BKR}
		If events $\mathsf{A}$ and $\mathsf{B}$ both depend on finitely many glued loops, then 
		\begin{equation}\label{ineq_BKR}
			\mathbb{P}\left( \mathsf{A} \circ \mathsf{B}\right)  \le \mathbb{P}\left( \mathsf{A}\right)   \cdot \mathbb{P}\left(   \mathsf{B}\right). 
		\end{equation}
	\end{lemma}

	However, sometimes the events that we want to study do not satisfy the condition of Lemma \ref{lemma_BKR}. Nevertheless, the BKR inequality can be applied via taking a limit. We present the following corollary as an extension of Lemma \ref{lemma_BKR}, which is adequate for this paper.

	We say an event $\mathsf{A}$ is a connecting event if there exist two finite subsets $A_1,A_2\subset \widetilde{\mathbb{Z}}^d$ such that $\mathsf{A}=\{A_1\xleftrightarrow[]{\cup \widetilde{\mathcal{L}}_{1/2}}A_2\}$. 
	
	\begin{corollary}[BKR inequality II]\label{coro_BKR}
		If events $\mathsf{A}_1,\mathsf{A}_2,...,\mathsf{A}_m$ ($m\ge 2$) are connecting events, then we have 
		\begin{equation}
			\mathbb{P}\left( \mathsf{A}_1\circ \mathsf{A}_2 \circ ... \circ \mathsf{A}_m \right) \le \prod_{i=1}^{m}  	\mathbb{P}\left( \mathsf{A}_i\right).  
		\end{equation}
	\end{corollary}
	
	\begin{proof}
		Note that we cannot apply Lemma \ref{lemma_BKR} directly since a connecting event does not only depend on finitely many glued loops. Suppose that $\mathsf{A}_i=\{A_{i,1}\xleftrightarrow[]{\cup \widetilde{\mathcal{L}}_{1/2}}A_{i,2}\}$ for $1\le i\le m$. Arbitrarily take $M\in \mathbb{N}^+$. For each $1\le i\le m$, we consider the truncated event 
		$$
		\hat{\mathsf{A}}_i=\hat{\mathsf{A}}_i(M):= \Big\{ A_{i,1} \xleftrightarrow[]{\cup \widetilde{\mathcal{L}}_{1/2}\cdot \mathbbm{1}_{\mathrm{ran}(\widetilde{\ell})\subset \widetilde{B}(M)}} A_{i,2} \Big\}. 
		$$

			If $\mathsf{A}_i\cap \hat{\mathsf{A}}_i^c$ happens, then one of $A_{i,1},A_{i,2}$ is connected to $\partial B(M)$. In addition, each $\hat{\mathsf{A}}_i$ only depends on
			$$
			\big\{\gamma_x^{\mathrm{p}}\big\}_{x\in B(M-1)}\cup \big\{\gamma_{\{x,y\}}^{\mathrm{e}}\big\}_{x,y\in \mathbb{Z}^d:I_{\{x,y\}}\in \widetilde{B}(M)}\cup \big\{\gamma_A^{\mathrm{f}}\big\}_{A\subset B(M-1)}, 
			$$
			and therefore satisfies the requirement of Lemma \ref{lemma_BKR} (since the number of these glued loops is finite). Thus, we have 
		\begin{equation*}\label{ineq_3.6}
			\begin{split}
				&\mathbb{P}\left( \mathsf{A}_1\circ \mathsf{A}_2 \circ ... \circ \mathsf{A}_m \right) \\
				\le & \mathbb{P}\left( \hat{\mathsf{A}}_1\circ \hat{\mathsf{A}}_2 \circ ... \circ \hat{\mathsf{A}}_m \right) + \sum_{i=1}^{m} \mathbb{P} \left(\mathsf{A}_i\cap \hat{\mathsf{A}}_i^c\right)       \\
				\le  &\mathbb{P}\left( \hat{\mathsf{A}}_1\circ \hat{\mathsf{A}}_2 \circ ... \circ \hat{\mathsf{A}}_m \right)+\sum_{i=1}^{m}\sum_{j=1}^{2}\sum_{z\in \mathbb{Z}^d:\widetilde{B}_z(1)\cap A_{i,j}\neq \emptyset}\mathbb{P}\left[ z \xleftrightarrow[]{\cup \widetilde{\mathcal{L}}_{1/2} } \partial B(M) \right]\\  
				\le &\prod_{i=1}^{m}  	\mathbb{P}\left( \hat{\mathsf{A}}_i\right)+\sum_{i=1}^{m}\sum_{j=1}^{2}\sum_{z\in \mathbb{Z}^d:\widetilde{B}_z(1)\cap A_{i,j}\neq \emptyset}\mathbb{P}\left[ z \xleftrightarrow[]{\cup \widetilde{\mathcal{L}}_{1/2} } \partial B(M) \right].
			\end{split}
		\end{equation*}
		The first term on the RHS is upper-bounded by $\prod_{i=1}^{m}  	\mathbb{P}\left( \mathsf{A}_i\right)$ since $\hat{\mathsf{A}}_i\subset \mathsf{A}_i$ for $1\le i\le m$. Moreover, by (\ref{3d-GFF}) and (\ref{ineq_onearm}), the second term can be arbitrarily close to $0$ if we take sufficiently large $M$. Now the proof is complete. \end{proof}

	\subsection{Tree expansion}
	
	We now review a combinatorial approach called \textit{tree expansion} introduced in Aizenman and Newman \cite{aizenman1984tree}. This approach, usually applied together with the BKR inequality, has proved to be a powerful tool in the study of percolation models (see e.g. Barsky and Aizenman \cite{barsky1991percolation} for its application in bond percolation). In this paper we only review the version mentioned in the proof of \cite[Proposition 3]{werner2021clusters}.

	\begin{lemma}[Tree expansion]\label{lemma_tree_expansion}
		For any $x\in \mathbb{Z}^d$ and $A_1,A_2\subset \widetilde{\mathbb{Z}}^d$, where $\{x\},A_1$ and $A_2$ are disjoint to one another, if $x$ is connected to both $A_1,A_2$ by $\cup \widetilde{\mathcal{L}}_{1/2}$, then there exists a glued loop $\gamma_*$ such that $\{\gamma_* \xleftrightarrow[]{\cup \widetilde{\mathcal{L}}_{1/2}} x\}\circ \{\gamma_* \xleftrightarrow[]{\cup \widetilde{\mathcal{L}}_{1/2}} A_1\} \circ \{\gamma_* \xleftrightarrow[]{\cup \widetilde{\mathcal{L}}_{1/2}} A_2\}$ happens. 
	\end{lemma}
	
	\begin{proof}
		For $j\in \{1,2\}$, since $x$ is connected to $A_j$ by $\cup \widetilde{\mathcal{L}}_{1/2}$, there must exist a finite sequence of different glued loops, say $\gamma^j_{1},...,\gamma^j_{m_j}$, such that $x\in \gamma^j_1$, $A_j\cap \gamma^j_{m_j}\neq \emptyset$, and $\gamma_i^j\cap \gamma_{i+1}^j\neq \emptyset$ for all $1\le i\le m_j-1$.

		If $\{\gamma^1_{1},...,\gamma^1_{m_1}\}$ and $\{\gamma^2_{1},...,\gamma^2_{m_2}\}$ are two disjoint collections, then we have $x\in \gamma^1_{1}$, $\gamma^1_{1}\xleftrightarrow[]{\cup_{i=2}^{m_1}\gamma^1_{i}}A_1$ and $\gamma^1_{1}\xleftrightarrow[]{\cup_{i=1}^{m_2}\gamma^2_i}A_2$. Thus, we only need to take $\gamma_*=\gamma^1_{1}$.

		Otherwise, we take $\gamma_*=\gamma^1_{m_*}$, where $m_*$ is the maximal integer in $[1,m_1]$ such that $\gamma^1_{m_*}\in \{\gamma^2_{1},...,\gamma^2_{m_2}\}$. The reasons are as follows. Let $m_{\dagger}$ be an integer in $[1,m_2]$ such that $\gamma^1_{m_*}=\gamma^2_{m_{\dagger}}$. Then we have $\gamma^1_{m_*} \xleftrightarrow[]{\cup_{i=1}^{m_{\dagger}-1}\gamma^2_i} x$, $\gamma^1_{m_*} \xleftrightarrow[]{\cup_{i=m_*+1}^{m_1}\gamma^1_i} A_1$ and $\gamma^1_{m_*}\xleftrightarrow[]{\cup_{i=m_{\dagger}+1}^{m_2}\gamma^2_i} A_2$. By the maximality of $m_*$, one has $\{\gamma^1_{m_*+1},...,\gamma^1_{m_1}\}\cap \{\gamma^2_{1},...,\gamma^2_{m_1}\}=\emptyset$. Thus, the event $\{\gamma^1_{m_*} \xleftrightarrow[]{\cup \widetilde{\mathcal{L}}_{1/2}} x\}\circ \{\gamma^1_{m_*} \xleftrightarrow[]{\cup \widetilde{\mathcal{L}}_{1/2}} A_1\} \circ \{\gamma^1_{m_*} \xleftrightarrow[]{\cup \widetilde{\mathcal{L}}_{1/2}} A_2\}$ occurs.
	\end{proof}

	\section{Proof of the lower bound}
	
	In this section, we show the proof of the lower bound in Theorem \ref{theorem1}. This proof shares the same spirit as \cite[Lemmas 2.1 and 2.2]{kozma2011arm}. Moreover, the main step (i.e. Lemma \ref{lemma_31}) was essentially sketched in \cite[Section 5.2]{werner2021clusters}.

	To simplify the formulation, we abbreviate ``$\xleftrightarrow[]{\cup\widetilde{\mathcal{L}}_{1/2}} $'' as ``$\xleftrightarrow[]{} $''.

	\begin{lemma}\label{lemma_31}
		For $d>6$, there exists $C_3(d)>0$ such that for any $N\ge 1$, 
		\begin{equation}
			\sum_{x_1,x_2\in \partial B(N)}\mathbb{P}\left( \bm{0}\xleftrightarrow[]{ }x_1,\bm{0}\xleftrightarrow[]{ }x_2 \right) \le C_3N^4.
		\end{equation}
	\end{lemma}	
	With Lemma \ref{lemma_31}, proving the lower bound in Theorem \ref{theorem1} is straightforward. 	
	\begin{proof}[Proof of the lower bound in Theorem \ref{theorem1}]
		Let $X:= \sum_{x\in \partial B(N)} \mathbbm{1}_{\bm{0}\xleftrightarrow[]{} x}$. By $|\partial B(N)|\asymp N^{d-1}$ and the two-point function estimate (\ref{eq_two_points}), we have 
		\begin{equation}\label{ineq4.2}
			\mathbb{E}X\ge cN^{2-d}\cdot N^{d-1}= cN. 
		\end{equation}
		Recall that for any non-negative random variable $Y$, $\mathbb{P}(Y>0)\ge (\mathbb{E}Y)^2/\mathbb{E}(Y^2)$. Thus, by Lemma \ref{lemma_31} and (\ref{ineq4.2}), we have 
		\begin{equation}
			\begin{split}
				\mathbb{P}\left[\bm{0} \xleftrightarrow[]{} \partial B(N) \right]\ge \frac{(\mathbb{E}X)^2}{\mathbb{E}(X^2)} \ge c^2C_3^{-1} N^{-2}. \nonumber	\qedhere
			\end{split}
		\end{equation}
	\end{proof}

	We present some inequalities that will be used for multiple times in the subsequent proof. For convenience, we set $0^{-a}=1$ for $a>0$ in this paper.

	\begin{lemma}\label{lemma_fina_4.2}
		For $d\ge 3$ and $a\in \mathbb{R}$, there exists $C(d,a)>0$ such that the following holds: 
		\begin{itemize}

			\item When $a>d$, for any $M\ge 1$, 
			\begin{equation}\label{fina_new_4.4}
				\sum_{x\in \mathbb{Z}^d\setminus B(M)} |x|^{-a} \le CM^{d-a}; 
			\end{equation}

			\item When $a\neq d-1$, for any $M\ge 1$,
				\begin{equation}\label{fina4.4}
					\max_{y\in \mathbb{Z}^d} \sum_{x\in \partial B(M)}|x-y|^{-a}\le CM^{(d-1-a)\vee 0}.
				\end{equation}

				\item When $a\neq d$, for any $M\ge 1$, 
			\begin{equation}\label{fina4.3}
				\max_{y\in \mathbb{Z}^d} \sum_{x\in B(M)} |x-y|^{-a} \le CM^{(d-a)\vee 0}; 
			\end{equation}
			
		\end{itemize}
	\end{lemma}
	\begin{proof}
		
		For (\ref{fina_new_4.4}), since $|\partial B(k)|\le Ck^{d-1}$, we have 
		\begin{equation*}
			\sum_{x\in \mathbb{Z}^d\setminus B(M)} |x|^{-a}= \sum_{k=M+1}^{\infty}\sum_{x\in \partial B(k)}|x|^{-a}\le C\sum_{k=M+1}^{\infty}k^{d-1}\cdot k^{-a}\le CM^{d-a}.
		\end{equation*}


	Now we focus on the proof of (\ref{fina4.4}). When $y\in [\mathbb{Z}^d\setminus B(1.5M)]\cup B(0.5M)$, since $|x-y|\ge 0.4M$ for all $x\in \partial B(M)$, we have 
		\begin{equation}\label{fina4.5}
			\sum_{x\in \partial B(M)} |x-y|^{-a}\le CM^{-a}\cdot |\partial B(M)| \le CM^{d-1-a}. 
		\end{equation}
		For the remaining case (i.e. $y\in B(1.5M)\setminus B(0.5M)$), we observe that $|\partial B_y(k) \cap \partial B(M)|\le Ck^{d-2}$ for all $0\le k\le 5M$, and that $\partial B_y(k) \cap \partial B(M)=\emptyset$ for all $k>5M$. Therefore, we have 
		\begin{equation}\label{fina4.6}
				\sum_{x\in \partial B(M)} |x-y|^{-a} \le C\sum_{k=1}^{5M}  k^{d-2}\cdot k^{-a} \le  CM^{(d-1-a)\vee 0}.
 		\end{equation}
		Combining (\ref{fina4.5}) and (\ref{fina4.6}), we obtain (\ref{fina4.4}).

		The proof of (\ref{fina4.3}) can be approached similarly to (\ref{fina4.4}). Specifically, we can estimate the sum on the LHS of (\ref{fina4.3}) by separately considering two cases: when $y\in \mathbb{Z}^d\setminus B(2M)$ and when $y\in B(2M)$. Further details are omitted since the calculations parallel those in (\ref{fina4.5}) and (\ref{fina4.6}).
	\end{proof}

	\begin{lemma}\label{lemma_calculation1}
		For $d>6$, there exists $C(d)>0$ such that for all $x,y\in \mathbb{Z}^d$, 
		\begin{equation}\label{ineq_4.3}
			\sum_{z\in \mathbb{Z}^d} |z-x|^{2-d}|z-y|^{2-d}\le C|x-y|^{4-d}. 
		\end{equation}
	\end{lemma}
	
	\begin{proof}
		
		When $|x-y|\le 100$, one has $|z-y|\le |z-x|+|x-y|\le 2|z-x|$ for all $z\in \mathbb{Z}^d\setminus B_x(200)$. Thus, by (\ref{fina_new_4.4}) we have 
		\begin{equation*}\label{ineq_4.4}
			\begin{split}
				\sum_{z\in \mathbb{Z}^d} |z-x|^{2-d}|z-y|^{2-d}\le &C+\sum_{z\in \mathbb{Z}^d\setminus B_x(200)} 2|z-x|^{2(2-d)}<\infty.   
			\end{split}
		\end{equation*}
	By choosing a sufficiently large constant $C$ in (\ref{ineq_4.3}), we can establish that this lemma holds for all $x,y\in \mathbb{Z}^d$ with $|x-y|\le 100$.


		For the remaining case (i.e. $|x-y|>100$), we denote $n:=\lfloor \frac{1}{2}|x-y| \rfloor$, $A_1:= B_x(n)$, $A_2:=B_y(n)$ and $A_3:=\mathbb{Z}^d\setminus (A_1\cup A_2)$. Since $|z-y|\ge |x-y|-|x-z|\ge n$ for all $z\in A_1$, by (\ref{fina4.3}) we have  
		$$
		\sum_{z\in A_1} |z-x|^{2-d}|z-y|^{2-d}\le Cn^{2-d}\sum_{z\in A_1} |z-x|^{2-d}\le Cn^{4-d}. 
		$$
		For the same reason, the sum over $z\in A_2$ is also upper-bounded by $Cn^{4-d}$. Since $\min\{|z-x|, |z-y|\}\ge n$ for $z\in A_3$, we have 
		\begin{equation*}
			A_3\subset \bigcup_{k\ge n}A_{3,k}:=\bigcup_{k\ge n} \left\lbrace z\in \mathbb{Z}^d:\min\{|z-x|, |z-y|\}=k \right\rbrace.
		\end{equation*}
		Combining this inclusion with $|A_{3,k}|\le |\partial B_x(k)|+|\partial B_y(k)|\le Ck^{d-1}$, we obtain 
		\begin{equation}
			\begin{split}
				\sum_{z\in A_3} |z-x|^{2-d}|z-y|^{2-d}\le& \sum_{k\ge n}\sum_{z\in A_{3,k}} |z-x|^{2-d}|z-y|^{2-d}\\
				\le& C\sum_{k\ge n}k^{d-1}\cdot k^{2(2-d)}
				\le Cn^{4-d}.  \nonumber
			\end{split}
		\end{equation}
		By these estimates for the sums over $A_1$, $A_2$ and $A_3$, we conclude this lemma.
	\end{proof}

	By separating the sum over $\mathbb{Z}^d$ in the same way as above, we also have the following estimates. For the sake of brevity, we will not provide further details of these proofs since they are parallel to that of Lemma \ref{lemma_calculation1}.

	\begin{lemma}\label{lemma_calculation2}
		For $d>6$, there exists $C(d)>0$ such that for all $x,y\in \mathbb{Z}^d$, 
		\begin{equation}\label{cal_4.5}
			\sum_{z\in \mathbb{Z}^d} |z-x|^{2-d}|z-y|^{6-2d}\le C|x-y|^{2-d}, 
		\end{equation} 
		\begin{equation}\label{cal_4.6}
			\sum_{z\in \mathbb{Z}^d} |z-x|^{2-d}|z-y|^{4-d}\le C|x-y|^{6-d}.
		\end{equation}
	\end{lemma}

In the following corollary of Lemmas \ref{lemma_calculation1} and \ref{lemma_calculation2}, we provide several estimates that will be used repeatedly in the subsequent proof.

	\begin{corollary}\label{fina_coro_4.5}
		For $d>6$, there exists $C(d)>0$ such that for all $x,y\in \mathbb{Z}^d$, 
			\begin{equation}\label{fin_4.6}
				\sum_{z_1,z_2\in \mathbb{Z}^d} |x-z_1|^{2-d}|z_1-z_2|^{2-d}|z_2-x|^{2-d}|z_1-y|^{2-d}\le C|x-y|^{2-d},
			\end{equation}
				\begin{equation}\label{fin_4.7}
				\sum_{z_1,z_2,z_3\in \mathbb{Z}^d} |z_1-z_2|^{2-d}|z_2-z_3|^{2-d}|z_3-z_1|^{2-d}|x-z_1|^{2-d}|y-z_2|^{2-d}\le C|x-y|^{4-d}.
			\end{equation}
	\end{corollary}
	\begin{proof}
For (\ref{fin_4.6}), by summing over $z_2$ and $z_1$ in turn, we have 
		\begin{equation*}
			\begin{split}
				&\sum_{z_1,z_2\in \mathbb{Z}^d} |x-z_1|^{2-d}|z_1-z_2|^{2-d}|z_2-x|^{2-d}|z_1-y|^{2-d}\\
				\le &\sum_{z_1\in \mathbb{Z}^d} |x-z_1|^{6-2d}|z_1-y|^{2-d} \ \ \ \ \ \ \ \ \ \ \ \ \ \ \ \ \ \ \ (\text{by Lemma}\ \ref{lemma_calculation1})\\
				\le &C|x-y|^{2-d} \ \ \ \ \  \ \ \ \ \ \ \ \ \ \ \ \ \ \  \ \ \ \ \ \ \ \  \  \  \ \ \ \ \ \  \  \ \ \  (\text{by}\ (\ref{cal_4.5})).
			\end{split}
		\end{equation*}
	For (\ref{fin_4.7}), we sum over $z_3,z_2$ and $z_1$ in turn, and then obtain 
\begin{equation*}
	\begin{split}
		&\sum_{z_1,z_2,z_3\in \mathbb{Z}^d} |z_1-z_2|^{2-d}|z_2-z_3|^{2-d}|z_3-z_1|^{2-d}|x-z_1|^{2-d}|y-z_2|^{2-d}\\
		\le &\sum_{z_1,z_2\in \mathbb{Z}^d} |z_1-z_2|^{6-2d} |x-z_1|^{2-d}|y-z_2|^{2-d}\ \ \  \  \ (\text{by Lemma}\ \ref{lemma_calculation1})\\
		\le & \sum_{z_1\in \mathbb{Z}^d}|x-z_1|^{2-d} |z_1-y|^{2-d}    \ \  \  \ \ \ \ \  \  \ \ \ \ \  \  \ \ \ \ \  \  \ \  (\text{by}\ (\ref{cal_4.5}))\\
		\le & |x-y|^{4-d}\ \  \  \ \  \  \ \ \ \ \  \  \ \ \ \ \  \  \ \ \ \ \   \ \  \  \ \ \ \ \  \  \ \ \ \ \  \  \ \ \ \ \ \ \ (\text{by Lemma}\ \ref{lemma_calculation1}).    \qedhere
	\end{split}
\end{equation*}		
	\end{proof}

	Now we are ready to prove Lemma \ref{lemma_31}.

	\begin{proof}[Proof of Lemma \ref{lemma_31}]

		When the event $\mathsf{A}_{x_1,x_2}:=\{\bm{0} \xleftrightarrow[]{}x_1,\bm{0}\xleftrightarrow[]{}x_2\}$ happens, by the tree expansion (Lemma \ref{lemma_tree_expansion}), there exists a glued loop $\gamma_*$ such that $\{\gamma_* \xleftrightarrow[]{\cup \widetilde{\mathcal{L}}_{1/2}} \bm{0}\}\circ \{\gamma_* \xleftrightarrow[]{\cup \widetilde{\mathcal{L}}_{1/2}} x_1\} \circ \{\gamma_* \xleftrightarrow[]{\cup \widetilde{\mathcal{L}}_{1/2}} x_2\}$ happens. We denote by $\mathsf{A}_{x_1,x_2}^{\mathrm{f}}$ (resp. $\mathsf{A}_{x_1,x_2}^{\mathrm{p}}$, $\mathsf{A}_{x_1,x_2}^{\mathrm{e}}$) the event that $\mathsf{A}_{x_1,x_2}$ happens and the selected glued loop $\gamma_*$ can be the one composed of fundamental loops (resp. point loops, edge loops). It follows that (note that the RHS below is not necessarily a disjoint union)
		\begin{equation}\label{eq_4.6}
			\mathsf{A}_{x_1,x_2} \subset \mathsf{A}_{x_1,x_2}^{\mathrm{f}}\cup \mathsf{A}_{x_1,x_2}^{\mathrm{p}}\cup \mathsf{A}_{x_1,x_2}^{\mathrm{e}}. 
		\end{equation}

		When the event $\mathsf{A}_{x_1,x_2}^{\mathrm{f}}$ happens, there exist $x_3,x_4,x_5\in \mathbb{Z}^d$ and a fundamental loop $\widetilde{\ell}\in \widetilde{\mathcal{L}}_{1/2}$ such that $\mathrm{ran}(\widetilde{\ell}) \cap \widetilde{B}_{x_j}(1)\neq \emptyset$ for $j\in \{3,4,5\}$, and that $\{\bm{0} \xleftrightarrow[]{}x_3\}\circ \{x_1\xleftrightarrow[]{}x_4\}\circ \{x_2\xleftrightarrow[]{}x_5\}$ happens. By the analogous result of Lemma \ref{fina_lemma2.1} for three disjoint subsets of $\mathbb{Z}^d$ (see \cite[Exercise 32]{le2011markov}) and the relation between the loops on $\mathbb{Z}^d$ and $\widetilde{\mathbb{Z}}^d$ presented in Section \ref{subsection_continuous}, the loop measure of fundamental loops that intersect $\widetilde{B}_{x_3}(1)$, $\widetilde{B}_{x_4}(1)$ and $\widetilde{B}_{x_5}(1)$ is bounded from above by $C|x_3-x_4|^{2-d}|x_4-x_5|^{2-d}|x_5-x_3|^{2-d}$. Thus, by the BKR inequality (Corollary \ref{coro_BKR}) and the two-point function estimate (\ref{eq_two_points}) , we have
		\begin{equation}\label{ineq_310}
			\begin{split}
				\mathbb{P}\left( \mathsf{A}_{x_1,x_2}^{\mathrm{f}}\right) 
				\le &C \sum_{x_3,x_4,x_5\in \mathbb{Z}^d}|x_3-x_4|^{2-d}|x_3-x_5|^{2-d}|x_4-x_5|^{2-d}\\
				&\ \ \ \ \ \ \ \ \ \ \ \ \ \ \ \cdot |x_3|^{2-d}|x_1-x_4|^{2-d}|x_2-x_5|^{2-d}\\
				:=& C \sum_{x_3,x_4,x_5\in \mathbb{Z}^d} \mathcal{T}^{x_1,x_2}_{x_3,x_4,x_5}. 
			\end{split}
		\end{equation}

		When the event $\mathsf{A}_{x_1,x_2}^{\mathrm{p}}$ (or $\mathsf{A}_{x_1,x_2}^{\mathrm{e}}$) happens, since every $\gamma^{\mathrm{p}}_\cdot$ (or $\gamma^{\mathrm{e}}_\cdot$) is contained by some $\widetilde{B}_x(1)$, $x\in \mathbb{Z}^d$, there exist $x_3,x_4,x_5\in \mathbb{Z}^d$ with $\max\{|x_3-x_4|,|x_3-x_5|\}\le 2$ such that $\{\bm{0}\xleftrightarrow[]{}x_3\}\circ \{x_1\xleftrightarrow[]{}x_4\}\circ \{x_2\xleftrightarrow[]{}x_5\}$ happens. Similar to (\ref{ineq_310}), we have
		\begin{equation}\label{ineq4.8}
			\begin{split}
				&\max\left\lbrace \mathbb{P}\left( \mathsf{A}_{x_1,x_2}^{\mathrm{p}}\right),\mathbb{P}\left( \mathsf{A}_{x_1,x_2}^{\mathrm{e}}\right)\right\rbrace\\
				\le &C\sum_{x_3\in \mathbb{Z}^d} \sum_{x_4,x_5\in B_{x_3}(2)} |x_3|^{2-d}|x_1-x_4|^{2-d}|x_2-x_5|^{2-d},  
			\end{split}
		\end{equation}
		which implies that $\mathbb{P}\left( \mathsf{A}_{x_1,x_2}^{\mathrm{p}}\right)$ and $\mathbb{P}\left( \mathsf{A}_{x_1,x_2}^{\mathrm{e}}\right)$ are also bounded from above by  $C \sum_{x_3,x_4,x_5\in \mathbb{Z}^d} \mathcal{T}^{x_1,x_2}_{x_3,x_4,x_5}$ since $\min\{|x_3-x_4|^{2-d},|x_3-x_5|^{2-d},|x_4-x_5|^{2-d}\}\ge 4^{2-d}$ for all $x_4,x_5\in B_{x_3}(2)$. Thus, by (\ref{eq_4.6}) we obtain  
		\begin{equation}\label{ineq_new_4.10}
			\sum_{x_1,x_2\in \partial B(N)}\mathbb{P}\left( \bm{0}\xleftrightarrow[]{ }x_1,\bm{0}\xleftrightarrow[]{ }x_2 \right)\le C\sum_{x_1,x_2\in \partial B(N)} \sum_{x_3,x_4,x_5\in \mathbb{Z}^d} \mathcal{T}^{x_1,x_2}_{x_3,x_4,x_5}. 
		\end{equation}


		For the case when $x_3 \in \mathbb{Z}^d\setminus B(2N)$, by Corollary \ref{fina_coro_4.5} and Lemma \ref{lemma_fina_4.2}, we have 
		\begin{equation}\label{ineq_312}
			\begin{split}
				&\sum_{x_1,x_2\in \partial B(N)}\sum_{x_3 \in \mathbb{Z}^d\setminus B(2N)}\sum_{x_4,x_5\in \mathbb{Z}^d} \mathcal{T}^{x_1,x_2}_{x_3,x_4,x_5}\\
				\le &CN^{2-d} \sum_{x_1,x_2\in \partial B(N)}\sum_{x_3 \in \mathbb{Z}^d\setminus B(2N)}\sum_{x_4,x_5\in \mathbb{Z}^d} |x_4-x_5|^{2-d}|x_1-x_4|^{2-d} \\
				&\    \ \ \ \ \ \ \ \ \ \ \cdot |x_2-x_5|^{2-d}|x_3-x_4|^{2-d}|x_3-x_5|^{2-d}\ \ \  (\text{by}\ |x_3|^{2-d}\le CN^{2-d})\\
				\le & CN^{2-d}  \sum_{x_1,x_2\in \partial B(N)}|x_1-x_2|^{4-d} \ \ \  \ \ \  \ \ \ \ \  \ \ \ \  \ \ \ \ \ \ \ \  (\text{by}\ (\ref{fin_4.7}))\\
				\le & CN^{2-d}\cdot N^3 \cdot N^{d-1}=  CN^{4} \ \   \ \ \ \ \ \ \ \ \ \ \ \ \ \ \ \ \ \ \ \ \ \  \ \ \  (\text{by}\ (\ref{fina4.4})). 
			\end{split}
		\end{equation} 
		When $x_3\in B(2N)$, by summing over $x_1$, $x_2$, $x_5$, $x_4$ and $x_3$ in turn, and using Lemmas \ref{lemma_fina_4.2} and \ref{lemma_calculation1}, we get  
		\begin{equation*}\label{ineq_311}
			\begin{split}
				&\sum_{x_1,x_2\in \partial B(N)}\sum_{x_3\in B(2N)}\sum_{x_4,x_5\in \mathbb{Z}^d} \mathcal{T}^{x_1,x_2}_{x_3,x_4,x_5}\\
				\le & CN^2 \sum_{x_3\in B(2N)}\sum_{x_4,x_5\in \mathbb{Z}^d}|x_3-x_4|^{2-d}|x_3|^{2-d}   |x_3-x_5|^{2-d}|x_4-x_5|^{2-d}\ \  (\text{by}\ (\ref{fina4.4}))\\
					\le & CN^2 \sum_{x_3\in B(2N)}\sum_{x_4\in \mathbb{Z}^d}|x_3-x_4|^{6-2d}|x_3|^{2-d}  \  \ \ \ \ \ (\text{by Lemma}\ \ref{lemma_calculation1}) \\
				\le & CN^2  \sum_{x_3\in B(2N)} |x_3|^{2-d} \ \ \ \ \ \ \ \ \ \ \ \ \ \ \ \ \ \ \ \  \ \ \ \ \ \ \ \ \   (\text{by}\ (\ref{fina4.3})) \\
				\le & CN^4 \ \ \ \ \ \ \ \ \ \ \ \ \ \ \ \ \ \ \ \  \ \ \ \ \ \ \ \ \  \ \  \ \ \  \ \ \ \ \ \  \ \ \ \ \ \ \ \ \ (\text{by}\ (\ref{fina4.3}))  .
			\end{split}
		\end{equation*}
		Combined with (\ref{ineq_new_4.10}) and (\ref{ineq_312}), it concludes Lemma \ref{lemma_31}, and thus we complete the proof of the lower bound of Theorem \ref{theorem1}.
	\end{proof}

	
	\section{The error of deleting large loops} \label{section_proof_WW}
	In \cite[Section 5.3]{werner2021clusters}, Werner presented the following heuristic: 
	
	\emph{``In fact, when $a\in (0,d)$, the $N^a$-th largest Brownian loop will have a diameter of the order of $N\times N^{-a/d+o(1)}$. This means for instance that an overwhelming fraction of the numerous large clusters will contain no loop of diameter greater than $N^b$ for $b>6/d$. In other words, if we remove all loops of diameter greater than $N^b$, one will still have at least $N^{d-6+o(1)}$ large clusters, and the estimates for the two-point function will actually remain valid.''
	}

	To sum up, Werner described a strategy to prove the following conjecture. For each fixed $b\in (\frac{6}{d},1)$ and any $x,y\in \mathbb{Z}^d$ with $|x-y|=N$, one has 
	$$
	\mathbb{P}\Big( x \xleftrightarrow[]{\cup \widetilde{\mathcal{L}}_{1/2}^{\le N^b}} y \Big) \asymp N^{2-d}, 
	$$
	where $\widetilde{\mathcal{L}}_{1/2}^{\le M}:= \widetilde{\mathcal{L}}_{1/2}\cdot \mathbbm{1}_{\mathrm{diam}(\mathrm{ran}(\widetilde{\ell}))\le M}$ is the point measure composed of loops in $\widetilde{\mathcal{L}}_{1/2}$ with diameter at most $M$.

	Inspired by the heuristics mentioned above, we prove the analogous result with respect to the one-arm probability, which is not only useful in the proof of Theorem \ref{theorem1}, but also interesting in its own right.
	\begin{proposition}\label{prop_Nb}
		For $d>6$ and any $b\in (\frac{6}{d},1)$, there exist $C_4(d),c_1(d,b)>0$ such that for all $N\ge 1$, 
		\begin{equation}
			0\le \mathbb{P}\Big[ \bm{0} \xleftrightarrow[]{\cup\widetilde{\mathcal{L}}_{1/2}} \partial B(N) \Big] - \mathbb{P}\Big[ \bm{0} \xleftrightarrow[]{\cup\widetilde{\mathcal{L}}_{1/2}^{\le N^b}}\partial B(N) \Big]  \le C_4N^{-2-c_1}. 
		\end{equation}
	\end{proposition}

	To prove Proposition \ref{prop_Nb}, we need some preparations. For any $x\in \mathbb{Z}^d$, we denote by $\mathbf{C}(x)$ the cluster of $\cup \widetilde{\mathcal{L}}_{1/2}$ containing $x$.

	\begin{lemma}[{\cite[Section 5, Equation (2)]{werner2021clusters}}]\label{lemma_2.3}
		For $d>6$, there exists $C_5(d)>0$ such that for any $x\in \mathbb{Z}^d$, $M\ge 1$ and $k\in \mathbb{N}^+$, 
		\begin{equation}
			\mathbb{E}\left( \left|\mathbf{C}(x)\cap B(M) \right|^k  \right) \le C_5^kk!M^{4k-2}.  
		\end{equation}
	\end{lemma}
	
	By Lemma \ref{lemma_2.3}, we can prove a large deviation bound for $\left|\mathbf{C}(x)\cap B(M) \right|$. The proof is parallel to \cite[Lemma 4.4]{kozma2011arm}. 
	
	\begin{lemma}\label{lemma_large_div}
		For $d>6$, there exist $C(d),c(d)>0$ such that for all $M\ge 1$ and $s>0$,
		\begin{equation}
			\mathbb{P}\Big(  \max_{x\in B(M)} \left|\mathbf{C}(x)\cap B(M) \right|>sM^4 \Big) \le C M^{d-6} e^{-cs}. 
		\end{equation}
	\end{lemma}
	
	\begin{proof}
		We enumerate the clusters of $\cup \widetilde{\mathcal{L}}_{1/2}$ that intersect $B(M)$ by $\mathbf{C}_1',...,\mathbf{C}_{m_*}'$. Note that for any $l\ge 2$, one has a.s. 
		$$
		\sum_{m=1}^{m_*} \left|\mathbf{C}_m'\cap B(M)\right|^l = \sum_{x\in B(M)} \left|\mathbf{C}(x) \cap B(M) \right|^{l-1} 
		$$
		Therefore, by applying Lemma \ref{lemma_2.3} for $k=l-1$, we have 
		\begin{equation}\label{ineq_5.4}
			\begin{split}
				\mathbb{E}\Big( \sum_{m=1}^{m_*} \left|\mathbf{C}_m'\cap B(M)\right|^l  \Big)	= &\mathbb{E}\Big( \sum_{x\in B(M)} \left|\mathbf{C}(x) \cap B(M) \right|^{l-1} \Big)\\
				\le& CC_5^{l-1}(l-1)!M^{4l+d-6}. 
			\end{split}
		\end{equation}
		Let $l_\dagger=l_\dagger(s):=\lceil \frac{s}{2C_5}+1\rceil$. By Markov's inequality, we have 
		\begin{equation*}
			\begin{split}	
				&\mathbb{P}\Big(  \max_{x\in B(M)} \left|\mathbf{C}(x)\cap B(M)  \right|>sM^4 \Big) \\
				\le&  (sM^4)^{-l_\dagger}\mathbb{E}\Big(  \max_{x\in B(M)} \left|\mathbf{C}(x)\cap B(M)  \right|^{l_\dagger} \Big)\\
				\le &(sM^4)^{-l_\dagger}\mathbb{E}\Big( \sum_{m=1}^{m_*} \left|\mathbf{C}_m'\cap B(M)\right|^{l_\dagger}  \Big)
				\le CM^{d-6}e^{-cs},
			\end{split}
		\end{equation*}
	where we applied (\ref{ineq_5.4}) and the Stirling's formula in the last inequality.
	\end{proof}

	We need the following estimate on the loop measure of all oversized loops in a large box. 
	
	\begin{lemma}\label{lemma_loop_measure}
		For any $b\in (0,1)$, $\alpha>0$ and $\epsilon>0$, we have 
		\begin{equation}\label{ineq_5.6}
		\widetilde{\mu}\left[ \left\lbrace \widetilde{\ell}:\  \mathrm{\widetilde{\ell}} \subset \widetilde{B}(N^{1+\alpha}),\mathrm{diam}(\mathrm{ran}( \widetilde{\ell}))\ge N^b \right\rbrace  \right] \le o\left( N^{(1-b+\alpha)d+\epsilon}\right).
		\end{equation}
	\end{lemma}
	\begin{proof}
		Recall the notations $L(A_1,A_2)(\cdot)$ and $\widetilde{\tau}_i$ in Section \ref{subsection_decomposition}. For each loop $\widetilde{\ell}$ involved in the LHS of (\ref{ineq_5.6}), we say it is a type I loop if there exist $x\in B(N^{1+\alpha})$ and $\widetilde{\varrho}\in L(\{x\},\partial B_x(\frac{1}{2}N^b))(\widetilde{\ell})$ such that $\big|\{\widetilde{\varrho}(t):0\le t\le \widetilde{\tau}_1\}\big|\le  N^{2b-\frac{3\epsilon}{4}}$; otherwise, we say $\widetilde{\ell}$ is a type II loop. We denote the collections of these two types of loops by $\mathfrak{L}_{\mathrm{I}}$ and $\mathfrak{L}_{\mathrm{II}}$ respectively.

		For type I loops, by (\ref{fina_2.5}) and the relation between loops on $\mathbb{Z}^d$ and $\widetilde{\mathbb{Z}}^d$ presented in Section \ref{subsection_continuous}, we know that $\widetilde{\mu}\left( \mathfrak{L}_{\mathrm{I}} \right)$ is upper-bounded by
		\begin{equation}\label{fina_5.6}
			\begin{split}
					& \sum_{x\in B(N^{1+\alpha})} \mathbb{P}_x\Big[ \big| \{S^{(i)}\}_{0\le i\le \tau_{\partial B_x(\frac{1}{2}N^b)}}\big|\le N^{2b-\frac{3\epsilon}{4}} \Big]\\
					\le & \left| B(N^{1+\alpha})\right|  \left(\mathbb{P}\left[ \tau_{\partial B(\frac{1}{2}N^b)} \le  N^{2b-\frac{\epsilon}{4}}  \right] + \mathbb{P}\left[\big|\{S^{(i)}\}_{0\le i\le N^{2b-\frac{\epsilon}{4}}} \big|\le N^{2b-\frac{3\epsilon}{4}}  \right] \right). 
			\end{split}
		\end{equation}
			By Lawler \cite[Lemma 1.5.1]{lawler2013intersections}, the first probability on the RHS of (\ref{fina_5.6}) is bounded from above by $\text{s.e.}(N)$. Moreover, for the second probability, by the Markov property and the law of large numbers for the range of a simple random walk (see e.g. Spitzer \cite[Page 35, Theorem 1]{spitzer2001principles}), we have 
		\begin{equation*}\label{ineq_5.9}
			\begin{split}
				&\mathbb{P}\left( \big|\{S^{(i)}\}_{0\le i\le N^{2b-\frac{\epsilon}{4}}} \big|\le N^{2b-\frac{3\epsilon}{4}}  \right)\\
				\le &    \mathbb{P}\bigg( \bigcap_{j=1}^{N^{\frac{\epsilon}{4}}} \Big\{ \big|\{S^{(i)}\}_{(j-1)N^{2b-\frac{\epsilon}{2}}\le i\le jN^{2b-\frac{\epsilon}{2}}} \big|\le  N^{2b-\frac{3\epsilon}{4}} \Big\} \bigg)                    \\  
				= &\left[  \mathbb{P}\left( \big|\{S^{(i)}\}_{0\le i\le N^{2b-\frac{\epsilon}{2}}} \big|\le N^{2b-\frac{3\epsilon}{4}}  \right) \right]^{N^{\frac{\epsilon}{4}}}
				\le \text{s.e.}(N). 
			\end{split}
		\end{equation*}
		Thus, the loop measure of $\mathfrak{L}_{\mathrm{I}}$ is stretched exponentially small.

		For $\mathfrak{L}_{\mathrm{II}}$, let us consider the following summation:
		\begin{equation*}
			\begin{split}
				\mathcal{T}:= \sum_{x\in B(N^{1+\alpha})} 	\widetilde{\mu} \Big[\big\{ \widetilde{\ell}:& \mathrm{ran}(\widetilde{\ell})\subset \widetilde{B}(N^{1+\alpha}), \mathrm{diam}(\mathrm{ran}(\widetilde{\ell}))\ge N^b,\\
				&|\mathrm{ran}(\widetilde{\ell})|\ge N^{2b-\frac{3\epsilon}{4}},x\in \mathrm{ran}(\widetilde{\ell})\big\}  \Big]. 
			\end{split}
		\end{equation*}
		For any type II loop $\widetilde{\ell}$, since $\mathrm{ran}(\widetilde{\ell})$ contains at least $N^{2b-\frac{3\epsilon}{4}}$ different points in $B(N^{1+\alpha})$, $\widetilde{\ell}$ must be counted by $\mathcal{T}$ for at least $N^{2b-\frac{3\epsilon}{4}}$ times. Therefore 
		\begin{equation}\label{ineq_5.11}
			\mathcal{T} \ge N^{2b-\frac{3\epsilon}{4}}\cdot \widetilde{\mu}\left( \mathfrak{L}_{\mathrm{II}} \right). 
		\end{equation}
		Moreover, by $|B(N^{1+\alpha})|\asymp N^{(1+\alpha)d}$ and (\ref{ineq_2.5}), we have 
		\begin{equation}\label{ineq_5.12}
			\mathcal{T} \le CN^{(1+\alpha)d}\cdot N^{-b(d-2)}= C N^{(1-b+\alpha)d+2b}. 
		\end{equation}
		Combining (\ref{ineq_5.11}) and (\ref{ineq_5.12}), we obtain the following estimate of $\widetilde{\mu}\left( \mathfrak{L}_{\mathrm{II}} \right)$, and thus complete the proof:
		\begin{equation*}
			\widetilde{\mu}\left( \mathfrak{L}_{\mathrm{II}} \right)\le C N^{(1-b+\alpha)d+2b}\cdot  N^{-2b+\frac{3\epsilon}{4}}=o\left( N^{(1-b+\alpha)d+\epsilon}\right).   \qedhere
		\end{equation*}
	\end{proof}

	Now we are ready to prove Proposition \ref{prop_Nb}. Recall that we abbreviate ``$\xleftrightarrow[]{\cup\widetilde{\mathcal{L}}_{1/2}} $'' as ``$\xleftrightarrow[]{} $''. In addition, we may also write ``$\xleftrightarrow[]{\cup\widetilde{\mathcal{L}}_{1/2}^{\le M}} $'' as ``$\xleftrightarrow[]{\le M} $''.

	\begin{proof}[Proof of Proposition \ref{prop_Nb}]
		For any $b\in (\frac{6}{d},1)$, we choose sufficiently small $\alpha,\epsilon>0$ such that 
		\begin{equation}\label{eq_219}
			4(1+\alpha)+(-b+\alpha)d+2\epsilon=4-bd+\alpha(d+4)+2\epsilon<-2. 
		\end{equation}
		
		Let $M=N^{1+\alpha}$. By Lemma \ref{lemma_large_div}, we have 
		\begin{equation}\label{eq_PG}
			\mathbb{P}\left[\mathsf{G} \right] :=\mathbb{P}\left[\max_{x\in B(M)} \left| \mathbf{C}(x)\cap B(M) \right|\le M^4\log^2(M)  \right]\ge 1-\text{s.p.}(N).  
		\end{equation}
		Let $V_1:=\{x\in B(N): x\xleftrightarrow[]{} \partial B_x(N) \} $ and $V_2:=\{x\in B(N): x\xleftrightarrow[]{\le N^b} \partial B_x(N) \}$. By (\ref{eq_PG}) and the translation invariance of $\widetilde{\mathcal{L}}_{1/2}$ and $\widetilde{\mathcal{L}}_{1/2}^{\le N^b}$, we have 
		\begin{equation}\label{eq_221}
			\begin{split}
				0\le 	&\mathbb{P}\left[\bm{0} \xleftrightarrow[]{} \partial B(N)\right] - \mathbb{P}\left[\bm{0} \xleftrightarrow[]{\le N^b} \partial B(N)\right]\\
				=&\left|B(N) \right|^{-1}\mathbb{E}\left|V_1\setminus V_2\right| \\
				\le &   \left|B(N) \right|^{-1} \mathbb{E}\left( \left|V_1\setminus V_2\right|\cdot \mathbbm{1}_{\mathsf{G}} \right) + \text{s.p.}(N).  
			\end{split}
		\end{equation}

		We denote by $\mathfrak{L}$ the collection of loops $\widetilde{\ell}\in\widetilde{\mathcal{L}}_{1/2}$ such that $\mathrm{ran}(\widetilde{\ell})\cap B(2N)\neq\emptyset$ and $\mathrm{diam}(\mathrm{ran}(\widetilde{\ell}))>N^b$. Like in the proof of Lemma \ref{lemma_large_div}, we enumerate the clusters of $\cup \widetilde{\mathcal{L}}_{1/2}$ intersecting $B(M)$ by $\mathbf{C}_1',...,\mathbf{C}_{m_*}'$. Note that for any $1\le m\le m_*$, if $\mathbf{C}_m'$ does not intersect any loop of $\mathfrak{L}$, then $\mathbf{C}_m'\cap (V_1\setminus V_2)=\emptyset$. Therefore,
		$$
		V_1\setminus V_2 \subset \bigcup_{m\in [1,m_*]: \exists \widetilde{\ell}\in \mathfrak{L}\ \text{with}\ \mathbf{C}_m'\cap \mathrm{ran}(\widetilde{\ell})\neq \emptyset }  \mathbf{C}_m'\cap B(N). 
		$$
		Also note that each $\widetilde{\ell}$ intersects at most one $\mathbf{C}_m'$. In addition, on the event $\mathsf{G}$, one has $|\mathbf{C}_m'\cap B(N)|\le M^4\log^2(M)$ for all $m\in [1,m_*]$. Thus, we have 
		\begin{equation}\label{eq_222}
			\left| V_1\setminus V_2\right|\cdot \mathbbm{1}_{\mathsf{G}} \le \left|\mathfrak{L} \right| \cdot M^4\log^2(M).
		\end{equation}
		Let $\mathfrak{L}_1:= \{\widetilde{\ell} \in \mathfrak{L}: \mathrm{ran}(\widetilde{\ell})\subset \widetilde{B}(M) \}$ and   $\mathfrak{L}_2:= \mathfrak{L}\setminus \mathfrak{L}_1$. Since each $\widetilde{\ell}\in \mathfrak{L}_1$ is involved in the LHS of (\ref{ineq_5.6}), by Lemma \ref{lemma_loop_measure} we have 
		\begin{equation}\label{ineq_L1}
			\mathbb{E}\left|\mathfrak{L}_1 \right| \le o\left( N^{(1-b+\alpha)d+\epsilon}\right). 
		\end{equation}
		In addition, since each $\widetilde{\ell}\in \mathfrak{L}_2$ intersects both $B(2N)$ and $\partial B(M)$, by (\ref{ineq_2.2}) and (\ref{ineq_2.5}) we have 
		\begin{equation}\label{eq_223}
			\begin{split}
				\mathbb{E}\left|\mathfrak{L}_2 \right|\le CN^{-\alpha(d-2)}. 
			\end{split}			
		\end{equation}
		Combining (\ref{eq_222}), (\ref{ineq_L1}) and (\ref{eq_223}), we get
		$$
		\mathbb{E}\left( \left|V_1\setminus V_2\right|\cdot \mathbbm{1}_{\mathsf{G}} \right)  \le M^4\log^2(M)\left[ o\left( N^{(1-b+\alpha)d+\epsilon}\right)  +CN^{-\alpha(d-2)} \right].  
		$$
		This implies that the RHS of (\ref{eq_221}) is upper-bounded by
		$$
		CN^{-d}\cdot N^{4(1+\alpha)+\epsilon}\cdot N^{(1-b+\alpha)d+\epsilon}= CN^{4(1+\alpha)+(-b+\alpha)d+2\epsilon}:= CN^{-2-c_1},
		$$
		where the existence of $c_1$ is ensured by the requirement in (\ref{eq_219}). 
	\end{proof}

	\section{Outline of the proof of the upper bound}\label{section_outline}
	
	Now we describe our strategy to prove the upper bound of Theorem \ref{theorem1}. The framework we use here is inspired by \cite{kozma2011arm}. The key novelty of our proof lies in a new exploration process, which is precisely desrcibed in Section \ref{section_exploration}.

	For $n\ge 1$, let $\theta(n):=\mathbb{P}[\bm{0}\xleftrightarrow[]{} \partial B(n) ]$. We aim to prove 
	\begin{proposition}\label{prop_recursion}
		For any $d>6$, there exist constants $c_2(d)\in (0,1),C_6(d)>0$ such that for any $\lambda\in \left(0,1 \right] $, there exists $c_3(d,\lambda)>0$ such that for all $\epsilon\in (0,c_3)$ and $N\ge 1$, 
		$$
		\theta\left( (1+\lambda)N\right) \le C_6\epsilon^{-\frac{1}{2}}N^{-2}+ 3d\epsilon^{\frac{3}{5}}N^2 \theta\left(\tfrac{\lambda N}{2} \right)\theta(N)+ (1-c_2)\theta(N)+\frac{C_4}{[(1+\lambda)N]^{2+c_1}}.   
		$$
	\end{proposition}
	
It suffices for our proof even if $c_1=0$, but we keep this stronger form in the statement in case this improvement will be useful for some future work.
	
	With Proposition 6.1 at hand, proving the desired upper bound in Theorem \ref{theorem1} is straightforward by induction.

	\begin{proof}[Proof of the upper bound in Theorem \ref{theorem1} assuming Proposition \ref{prop_recursion}]
		We choose a small enough $\lambda\in \left( 0,1 \right] $ such that 
		\begin{equation}\label{req_A1}
			(1+\lambda)^2\le 2, \ \ (1-c_2) (1+\lambda)^2 \le 1-\frac{2c_3}{3}. 
		\end{equation}
		Meanwhile, we also take a sufficiently large $M_0(d,\lambda)$ such that 
		\begin{equation}\label{req_A2}
			(2C_6+ 24d \lambda^{-2})M_0^{-\frac{1}{11}} \le \frac{c_2}{3}, \ \ M_0^{-\frac{20}{11}}\le c_3, \ \ C_4M_0^{-1}\le \frac{c_2}{3}.  
		\end{equation}
		
		Let us prove $\theta(N)\le M_0N^{-2}$ by induction. For the base, we note that the desired bound holds obviously for $N \le \sqrt{M_0}$. Assume the bound $\theta(s)\le M_0s^{-2}$ holds for all $s<(1+\lambda)N$. By Proposition \ref{prop_recursion} with $\epsilon=M_0^{-\frac{20}{11}}$ and the induction hypothesis, we have 
		\begin{equation*}\label{ineq_new_6.3}
			\begin{split}
				&\theta\left((1+\lambda)N\right)\\
				\le  &C_6\epsilon^{-\frac{1}{2}}N^{-2}+ 3d\epsilon^{\frac{3}{5}}N^2 \theta(N)\theta\left(\tfrac{\lambda N}{2} \right)+ (1-c_2)\theta(N)+C_4[(1+\lambda)N]^{-2}\\
				\le & C_6M_0^{\frac{10}{11}}N^{-2}+ 12dM_0^{\frac{10}{11}}\lambda^{-2}N^{-2}+(1-c_2)M_0N^{-2}+C_4[(1+\lambda)N]^{-2}\\
				= & \frac{M_0}{[(1+\lambda)N]^{2}}
				\left[ (1+\lambda)^2\left(C_6+12d \lambda^{-2} \right)M_0^{-\frac{1}{11}}+(1-c_2)(1+\lambda)^2+C_4M_0^{-1} \right].
			\end{split}
		\end{equation*}
		By the requirement of $\lambda$ in (\ref{req_A1}), the RHS is upper-bounded by 
		\begin{equation*}\label{ineq_new_6.4}
			M_0[(1+\lambda)N]^{-2} \left[\left(2C_6+24d \lambda^{-2} \right)M_0^{-\frac{1}{11}}+(1-\frac{2c_2}{3})+ C_4M_0^{-1} \right]. 
		\end{equation*}
		Combined with (\ref{req_A2}), this yields that 
		\begin{equation}
			\begin{split}
				\theta\left((1+\lambda)N\right)\le &M_0[(1+\lambda)N]^{-2} \left[\frac{c_2}{3}+(1-\frac{2c_2}{3})+ \frac{c_2}{3} \right]\\
				=&M_0[(1+\lambda)N]^{-2}. \nonumber
			\end{split}
		\end{equation}
		Now we finish the induction and conclude the upper bound in Theorem \ref{theorem1}.
	\end{proof}

	For $m\in \mathbb{N}^+$ and $x\in \mathbb{Z}^d$, let $\hat{B}_x(m):=\{y\in \mathbb{Z}^d:|y-x|\le m, |y-x|_1<md \}$ be the box obtained by removing all corner points of $B_x(m)$. When $x$ is the origin, we may write $\hat{B}(m):=\hat{B}_{\bm{0}}(m)$. For any $x\in \partial \hat{B}(m)$, we denote by $x^{\mathrm{in}}$ the unique point in $\partial B(m-1)$ with $x^{\mathrm{in}}\sim x$. Note that every corner point of $B(m)$ (i.e. $y\in \mathbb{Z}^d$ with $|y|_1=md$) is not adjacent to $B(m-1)$. This is why we need to restrict the definition of $x^{\mathrm{in}}$ in $\partial \hat{B}(m)$.

	The following definition is crucial for our proof.

	\begin{definition}[Tuple $\mathbf{\Psi}_{n,M}$]\label{definition_Psi}
		(1)	For $n\in \mathbb{N}^+$ and $M\in [1,\infty]$, let $\mathbf{\Psi}_{n,M}^1$ be the cluster containing $\bm{0}$ and composed of the following types of loops in $\widetilde{\mathcal{L}}_{1/2}^{\le M}$:   
		\begin{itemize}
			\item fundamental loops intersecting $\widetilde{B}(n)$; 
			
			\item point loops intersecting some $x\in B(n-1)$;
			
			\item edge loops contained in $\widetilde{B}(n)$. 
			
		\end{itemize}
		We call these loops ``involved loops". Let $\mathbf{\Psi}_{n,M}:=(\mathbf{\Psi}_{n,M}^1,\mathbf{\Psi}_{n,M}^2) $, where 
		\begin{equation}
			\mathbf{\Psi}_{n,M}^2:= \left\lbrace x\in \partial \hat{B}(n) :x\notin \mathbf{\Psi}_{n,M}^1, I_{\{x,x^{\mathrm{in}}\}}\subset \gamma_x^{\mathrm{p}}\cup \mathbf{\Psi}_{n,M}^1   \right\rbrace.  
		\end{equation}


		\noindent(2)  Let $\overline{\mathbf{\Psi}}_{n,M}:=\big[\mathbf{\Psi}_{n,M}^1 \cap \mathbb{Z}^d\setminus B(n-1)\big]\cup \mathbf{\Psi}_{n,M}^2$, $\psi_{n,M}:= \big|\overline{\mathbf{\Psi}}_{n,M}\big|$ and 
		\begin{equation}\label{eq_6.6}
			\widehat{\mathbf{\Psi}}_{n,M}:= \mathbf{\Psi}_{n,M}^1\cup \bigcup_{x\in \mathbf{\Psi}_{n,M}^2} \gamma_x^{\mathrm{p}}.
		\end{equation}
	\end{definition}

	See Figure \ref{fig2} for an illustration of this definition. Note that $\widehat{\mathbf{\Psi}}_{n,M}$ is a cluster of $\cup \widetilde{\mathcal{L}}_{1/2}^{\le M}$. In addition, for any $D\subset \mathbb{Z}^d$, $$\widehat{\mathbf{\Psi}}_{n,M}\cap D=  \big(\mathbf{\Psi}_{n,M}^1\cup \mathbf{\Psi}_{n,M}^2 \big)\cap D,$$ 
	which is measurable with respect to $\mathbf{\Psi}_{n,M}$ (but $\widehat{\mathbf{\Psi}}_{n,M}$ is not).

	\begin{figure}[h]
		\centering
		\includegraphics[width=13cm]{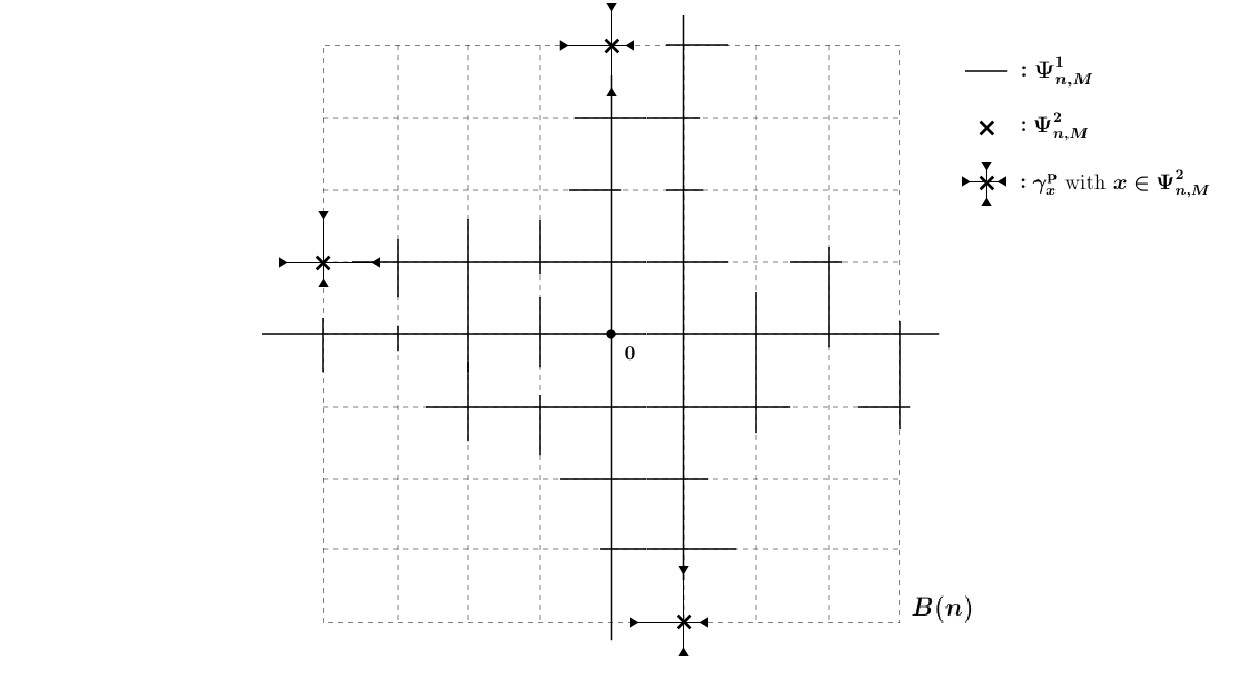}
		\caption{An illustration for $\mathbf{\Psi}_{n,M}$. }
		\label{fig2}
	\end{figure}

	\begin{remark}\label{fina_remark6.3}(1) For any $x\in\mathbf{\Psi}^{2}_{n,M}$, the glued point loop $\gamma_{x}^{\mathrm{p}}$ is only known to intersect $I_{\{x,x^{\mathrm{in}}\}}\cap \mathbf{\Psi}^{1}_{n,M}$. Moreover, for $x\in   \mathbf{\Psi}^{1}_{n,M} \cap \partial \hat{B}(n)$, $\gamma_{x}^{\mathrm{p}}$ is independent of $\mathbf{\Psi}_{n,M}$. Thus, given $\mathbf{\Psi}_{n,M}$, by the FKG inequality, the conditional distribution of $\gamma_{x}^{\mathrm{p}}\cap I_{\{x,y\}}$ (where $x\in \widehat{\mathbf{\Psi}}_{n,M}\cap \partial \hat{B}(n)$, $y\in \partial \hat{B}(n+1)$ and $x\sim y$) stochastically dominates the one without conditioning.

	 (2) At the first glance (or even the second glance), it seems more natural and also simpler to define $\mathbf{\Psi}^\diamond_{n, M}$ (as the replacement for the more complicated $\mathbf{\Psi}_{n, M}$) to be the cluster containing $\bm{0}$ and composed of all involved loops and $\gamma_x^{\mathrm{p}}\cap I_{\{x,x^{\mathrm{in}}\}}$ for $x\in \partial \hat{B}(n)$. However, $\mathbf{\Psi}_{n,M}^{\diamond}$ does not have the property as in Item (1), which is crucial in the subsequent proof. To see this, let us look at the following scenario. Arbitrarily take $x\in \partial \hat{B}(n)$ and then assume that $x\in \mathbf{\Psi}_{n,M}^{\diamond}$ and $x^{\mathrm{in}}\notin \mathbf{\Psi}_{n,M}^{\diamond}$. By the definition of $\mathbf{\Psi}_{n,M}^{\diamond}$, $\gamma_x^{\mathrm{p}}\cap I_{\{x,x^{\mathrm{in}}\}}$ is sampled and depends on the configuration of $\mathbf{\Psi}_{n,M}^{\diamond}$. In addition, for any $y\in \partial \hat{B}(n+1)$ with $x\sim y$, $\gamma_x^{\mathrm{p}}\cap I_{\{x,y\}}$ has a positive correlation with $\gamma_x^{\mathrm{p}}\cap I_{\{x,x^{\mathrm{in}}\}}$ since both of them are positively correlated to the total local time of point loops in $\widetilde{\mathcal{L}}_{1/2}^{\le M}$ intersecting $x$. Thus, arbitrarily given $\mathbf{\Psi}_{n,M}^{\diamond}$ (note that $\gamma_x^{\mathrm{p}}\cap I_{\{x,x^{\mathrm{in}}\}}$ may be  arbitrarily small), one cannot ensure the stochastic domination of $\gamma_{x}^{\mathrm{p}}\cap I_{\{x,y\}}$ as in Item (1).
\end{remark}

	We say $\mathbf{A}=(\mathbf{A}^1,\mathbf{A}^2)$ is an admissible tuple if it is a possible configuration of $\mathbf{\Psi}_{n,M}$. Parallel to (\ref{eq_6.6}), we define a random subset
	\begin{equation}\label{eq_6.7}
		\widehat{\mathbf{A}}:= \mathbf{A}^1\cup \bigcup_{x\in\mathbf{A}^2} \widehat{\gamma}_x^{\mathrm{p}}(\mathbf{A}^1),
	\end{equation}
	where $\{\widehat{\gamma}_x^{\mathrm{p}}(\mathbf{A}^1)\}_{x\in \mathbf{A}^2}$ is independent of $\widetilde{\mathcal{L}}_{1/2}$, and has the same distribution as $\{\gamma_x^{\mathrm{p}}\}_{x\in \mathbf{A}^2}$ conditioning on the event $\cap_{x\in \mathbf{A}^2} \{I_{\{x,x^{\mathrm{in}}\}}\subset \gamma_x^{\mathrm{p}}\cup \mathbf{A}^1\}$. 
	
	
	\begin{definition}[Unused loops]\label{def_unused loops}
		(1)	For any admissible $\mathbf{A}=(\mathbf{A}^1,\mathbf{A}^2)$, we denote by $\widetilde{\mathcal{L}}_{\mathbf{A},M}^{\mathrm{U}}$ the point measure composed of the following types of loops in $\widetilde{\mathcal{L}}_{1/2}^{\le M}$: 
		\begin{itemize}
			\item involved loops $\widetilde{\ell}$ with $\mathrm{ran}(\widetilde{\ell})\cap \mathbf{A}^1=\emptyset $;

			\item loops $\widetilde{\ell}$ with $\mathrm{ran}(\widetilde{\ell})\cap \widetilde{B}(n)=\emptyset$;

			\item point loops including some $x\in \mathbf{A}^1\cap \partial \hat{B}(n)$;

			\item point loops that include some $x\in \partial \hat{B}(n)\setminus (\mathbf{A}^1\cup \mathbf{A}^2)$ and do not intersect $I_{\{x,x^{\mathrm{in}}\}}\cap \mathbf{A}^1$. 
			
		\end{itemize}
		\noindent (2) We define $\widetilde{\mathcal{L}}^{\mathrm{U}}_M$ as $\widetilde{\mathcal{L}}_{\mathbf{A},M}^{\mathrm{U}}$ on the event $\{\mathbf{\Psi}_{n,M}= \mathbf{A}\}$.
	\end{definition}

	When $M=\infty$ (i.e. there is no restriction on the diameter of $\widetilde{\ell}$), we may omit the subscript $\infty$ and denote $\mathbf{\Psi}_n^1:=\mathbf{\Psi}_{n,\infty}^1$, $\mathbf{\Psi}_n^1:=\mathbf{\Psi}_{n,\infty}^2$, $\mathbf{\Psi}_n:=\mathbf{\Psi}_{n,\infty}$,  $\overline{\mathbf{\Psi}}_n:=\overline{\mathbf{\Psi}}_{n,\infty}$, $\psi_n:=\psi_{n,\infty}$, $\widehat{\mathbf{\Psi}}_n:=\widehat{\mathbf{\Psi}}_{n,\infty}$,  $\widetilde{\mathcal{L}}_{\mathbf{A}}^{\mathrm{U}}:=\widetilde{\mathcal{L}}_{\mathbf{A},\infty}^{\mathrm{U}}$ and $\widetilde{\mathcal{L}}^{\mathrm{U}}:=\widetilde{\mathcal{L}}^{\mathrm{U}}_{\infty}$.

	\begin{remark}\label{remark_51}
		We have some useful observations about $\mathbf{\Psi}_{n,M}$ as follows:

		\begin{enumerate}
			
			\item For any admissible tuple $\mathbf{A}=(\mathbf{A}^1,\mathbf{A}^2)$, when $\{\mathbf{\Psi}_{n,M}= \mathbf{A}\}$ happens, $ \widetilde{\mathcal{L}}_{1/2}^{\le M}-\widetilde{\mathcal{L}}^{\mathrm{U}}_{\mathbf{A},M}$ contains all the loops used to contruct $\widehat{\mathbf{\Psi}}_{n,M}$. In light of this, we call the loops in $\widetilde{\mathcal{L}}^{\mathrm{U}}_{\mathbf{A},M}$ unused loops.

			%
			%

			

			\item By the thinning property of Poisson point processes, given $\{\mathbf{\Psi}_{n,M}= \mathbf{A}\}$ (which is measurable with respect to $\sigma(\widetilde{\mathcal{L}}_{1/2}^{\le M}-\widetilde{\mathcal{L}}^{\mathrm{U}}_{M})$), the conditional distribution of $\widetilde{\mathcal{L}}^{\mathrm{U}}_{M}$ is the same as $\widetilde{\mathcal{L}}_{\mathbf{A},M}^{\mathrm{U}}$ without conditioning.

			
			\item Since every loop $\widetilde{\ell}$ included in $\mathbf{\Psi}_{n,M}^1$ has diameter at most $M$ and must intersect $\widetilde{B}(n)$ (by Definition \ref{definition_Psi}), we have $\mathbf{\Psi}_{n,M}^1\cap \mathbb{Z}^d\subset B(n+M)$ and thus $\mathbf{\Psi}_{n,M}^1\cap \mathbb{Z}^d\subset \mathbf{\Psi}_n^1\cap B(n+M)$. For any $x\in \mathbf{\Psi}_{n,M}^2$, we have $I_{\{x,x^{\mathrm{in}}\}}\subset \gamma_x^{\mathrm{p}}\cup \mathbf{\Psi}_{n,M}^1\subset \gamma_x^{\mathrm{p}}\cup \mathbf{\Psi}_n^1$, which implies that $x$ is either in $\mathbf{\Psi}_n^1\cap \partial \hat{B}(n)$ or in $\mathbf{\Psi}_n^2$. In conclusion,  
			\begin{equation}\label{inclusion_6.5}
				\overline{\mathbf{\Psi}}_{n,M} \subset \overline{\mathbf{\Psi}}_n \cap B(n+M). 
			\end{equation}

			\item  If the event $\{0\xleftrightarrow[]{\le M} \partial B(m)\}$ happens for some $m>n+M$, then there exists $v_{\dagger}\in  \widehat{\mathbf{\Psi}}_{n,M}$ such that $v_{\dagger}$ is connected to $\partial B(m)$ by $\cup \widetilde{\mathcal{L}}^{\mathrm{U}}_M$. Suppose that $v_{\dagger}$ is in the interval $I_{\{x_{\dagger},y_{\dagger}\}}$. We claim that either $x_{\dagger}$ or $y_{\dagger}$ is in $\overline{\mathbf{\Psi}}_{n,M}$. When $v_{\dagger}\in \gamma_{z}^{\mathrm{p}}$ for some $z\in \mathbf{\Psi}_{n,M}^2$, we know that either $x_{\dagger}$ or $y_{\dagger}$ is $z$, which is contained in $\overline{\mathbf{\Psi}}_{n,M}$. When $v_{\dagger}\in  \mathbf{\Psi}_{n,M}^1$, we verify the claim separately in the following subcases.

			\begin{enumerate}
				\item Both $x_{\dagger}$ and $y_{\dagger}$ are in $B(n-1)$: We will show that this case cannot occur by contradiction. Since $v_{\dagger}\in\mathbf{\Psi}_{n,M}^1\cap [\cup \widetilde{\mathcal{L}}^{\mathrm{U}}_M]$, there exists a loop $\widetilde{\ell}_{\dagger}\in  \widetilde{\mathcal{L}}^{\mathrm{U}}_M$ intersecting $v_{\dagger}$, which implies $\mathrm{ran}(\widetilde{\ell}_{\dagger})\cap \widetilde{B}(n) \cap \mathbf{\Psi}_{n,M}^1\neq \emptyset$. In addition, $\widetilde{\ell}_{\dagger}$ must be an involved loop since a point loop including some $x\in \partial \hat{B}(n)$ cannot intersect $B(n-1)$. These two facts cause a contradiction with $\widetilde{\ell}_{\dagger}\in  \widetilde{\mathcal{L}}^{\mathrm{U}}_M$.


				\item $x_{\dagger}\in \partial B(n-1)$ and $y_{\dagger}\in \partial \hat{B}(n)$: With the same argument as in Subcase (a), there exists $\widetilde{\ell}_{\dagger}\in  \widetilde{\mathcal{L}}^{\mathrm{U}}_M$ intersecting $v_{\dagger}$ and $\widetilde{B}(n) \cap \mathbf{\Psi}_{n,M}^1$. To avoid the same contradiction as in Subcase (a), it is necessary for $\widetilde{\ell}_{\dagger}$ to be a point loop including $y_{\dagger}$. We now prove that $y_\dagger \in \Psi^1_{n, M}$ by contradiction (this then yields the claim since $\Psi^1_{n, M}\cap \mathbb{Z}^d\setminus B(n-1)\subset \overline{\mathbf{\Psi}}_{n,M}$). Suppose that $y_{\dagger}\notin \mathbf{\Psi}_{n,M}^1$, then we have $x_{\dagger} \in \mathbf{\Psi}_{n,M}^1$ and therefore, $I_{\{x_{\dagger},y_{\dagger}\}}\subset  \mathrm{ran}(\widetilde{\ell}_{\dagger}) \cup\mathbf{\Psi}_{n,M}^1 \subset \gamma_{y_{\dagger}}^{\mathrm{p}}\cup\mathbf{\Psi}_{n,M}^1$. Thus, $\widetilde{\ell}_{\dagger}$ is a point loop containing $y_{\dagger}\in \mathbf{\Psi}_{n,M}^2$, which arrives at a contradiction with $\widetilde{\ell}_{\dagger}\in  \widetilde{\mathcal{L}}^{\mathrm{U}}_M$.


				\item $y_{\dagger}\in \partial B(n-1)$ and $x_{\dagger}\in \partial \hat{B}(n)$: For the same reason as in subcase (b), the claim is valid.

				\item $x_\dagger, y_\dagger \not\in B(n-1)$: Since $\mathbf{\Psi}_{n,M}^1$ is connected and $v_{\dagger}\in \mathbf{\Psi}^1_{n,M}$, we know that either $x_{\dagger}$ or $y_{\dagger}$ is in $\mathbf{\Psi}_{n,M}^1$, and thus is in $\overline{\mathbf{\Psi}}_{n,M}$.

			\end{enumerate}
			To sum up, we now conclude this claim (i.e. either $x_{\dagger}$ or $y_{\dagger}$ is in $\overline{\mathbf{\Psi}}_{n,M}$). Meanwhile, either $x_{\dagger}$ or $y_{\dagger}$ is connected to $\partial B(m)$ by $\cup \widetilde{\mathcal{L}}^{\mathrm{U}}_M$ since $v_{\dagger}$ does so. Putting these two results together, we have: for any $m>n+M$, 
			\begin{equation}\label{inclusion_6.6}
				\left\lbrace \bm{0}\xleftrightarrow[]{\le M} \partial B(m)\right\rbrace  \subset \bigcup_{z_1\in \overline{\mathbf{\Psi}}_{n,M}} \bigcup_{z_2\in \mathbb{Z}^d:|z_1-z_2|_1\le 1} \Big\{ z_2\xleftrightarrow[ ]{\cup \widetilde{\mathcal{L}}^{\mathrm{U}}_M}  \partial B(m)\Big\}. 
			\end{equation}


		\end{enumerate}
		
	\end{remark}

	Recall that $\mathbf{C}(x)$ is the cluster of $\cup\widetilde{\mathcal{L}}_{1/2}$ containing $x$. We take constants $b\in (\tfrac{6}{d},1)$ and $\lambda\in \left(0,1 \right] $, and fix a large integer $N$. We also take a constant $\epsilon>0$ and denote $L= \epsilon^{\frac{3}{10}}N$. Let $\overline{\mathbf{\Psi}}_n^*:= \overline{\mathbf{\Psi}}_n \cap B(n+[(1+\lambda)N]^b)$ and $\psi_{n}^*= |\overline{\mathbf{\Psi}}_n^*|$. When $\{0 \xleftrightarrow[]{} \partial B((1+\lambda)N)  \}$ happens, one of the following events occurs:
	\begin{itemize}
		\item $\mathsf{B}_0$: $\Big\{\bm{0} \xleftrightarrow[]{} \partial B\big((1+\lambda)N\big)  \Big\}\cap \Big\{\bm{0} \xleftrightarrow[]{\le [(1+\lambda)N]^b} \partial B\big((1+\lambda)N\big)  \Big\}^c$. 
		
		\item $\mathsf{B}_1$: $|\mathbf{C}(\bm{0})|\ge \epsilon N^4$.

		\item  $\mathsf{B}_2$: $\exists n\in \big[(1+\frac{\lambda}{4})N,(1+\frac{\lambda}{3})N\big]$ such that  $0<\psi_{n}^*\le L^2$ and $\bm{0} \xleftrightarrow[]{\le [(1+\lambda)N]^b} \partial B\big((1+\lambda)N\big)$.

		\item $\mathsf{B}_3$: $ \forall n\in \big[(1+\frac{\lambda}{4})N,(1+\frac{\lambda}{3})N\big]$, $\psi_{n}^*> L^2$ and $|\mathbf{C}(\bm{0})|< \epsilon N^4$.
		
	\end{itemize}
	Thus, to prove Proposition \ref{prop_recursion}, we only need to control the probabilities of these four events.

	For $\mathsf{B}_0$, by Proposition \ref{prop_Nb}, we have 
	\begin{equation}\label{ineq_B0}
		\begin{split}
			\mathbb{P}(\mathsf{B}_0)= & \mathbb{P}\Big[ \bm{0} \xleftrightarrow[]{} \partial B\big((1+\lambda)N\big) \Big] - \mathbb{P}\Big[ \bm{0} \xleftrightarrow[]{\le [(1+\lambda)N]^b}\partial B\big((1+\lambda)N\big) \Big] \\
			\le &\frac{C_4}{[(1+\lambda)N]^{2+c_1}}.
		\end{split}
	\end{equation}

	For $\mathsf{B}_1$, we use the decay rate of $|\mathbf{C}(\bm{0})|$ in the following proposition, which will be proved in Section \ref{section_volumn}. 
	\begin{proposition}\label{prop_volumn}
		For $d>6$, there exists $C_{6}(d)>0$ such that for all $M\ge 1$, 
		\begin{equation}
			\mathbb{P}\left( \left|\mathbf{C}(\bm{0}) \right|\ge M \right)  \le C_{6}M^{-\frac{1}{2}}. 
		\end{equation}
	\end{proposition} 
	By Proposition \ref{prop_volumn}, we have 
	\begin{equation}\label{ineq_B1}
		\mathbb{P}(\mathsf{B}_1) \le C_6\epsilon^{-\frac{1}{2}}N^{-2}. 
	\end{equation}

	%


	For $\mathsf{B}_2$, we denote  $\mathbf{\Psi}^{i,\blacktriangle}_n:=\mathbf{\Psi}_{n,[(1+\lambda)N]^b}^i$ for $i\in \{1,2\}$, $\mathbf{\Psi}_n^{\blacktriangle}:=\mathbf{\Psi}_{n,[(1+\lambda)N]^b}$, $\widehat{\mathbf{\Psi}}_n^{\blacktriangle}:=\widehat{\mathbf{\Psi}}_{n,[(1+\lambda)N]^b}$, $\overline{\mathbf{\Psi}}_{n}^{\blacktriangle}:=\overline{\mathbf{\Psi}}_{n,[(1+\lambda)N]^b}$ and $\psi_{n}^{\blacktriangle}:=|\overline{\mathbf{\Psi}}_{n}^{\blacktriangle}|$. Let $k_\dagger:=\min\{k\ge 0: 0<\psi_{n_k}^{\blacktriangle}\le L^2 \}$, where $n_k:=\big\lceil (1+\frac{\lambda}{4})N \big\rceil+k$.

The event $\mathsf{B}_2$ ensures the following two events:
	\begin{itemize}
		\item $\{\psi_{n_k}^{\blacktriangle}>0\}$ for any $n_k\in \big[(1+\frac{\lambda}{4})N,(1+\frac{\lambda}{3})N\big]$ (since the event $\Big\{\bm{0} \xleftrightarrow[]{\le [(1+\lambda)N]^b} \partial B\big((1+\lambda)N\big)\Big\}$ happens).
		
		\item There exists some $n_k\in \big[(1+\frac{\lambda}{4})N,(1+\frac{\lambda}{3})N\big]$ such that $0\le \psi^{\blacktriangle}_{n_k}\le L^2$ (since $\psi_{n_k}^*\ge \psi_{n_k}^{\blacktriangle}$ for all $k\ge 0$ (by (\ref{inclusion_6.5}))). 
	\end{itemize}
	To sum up, one has $\mathsf{B}_2 \subset \Big\{n_{k_\dagger}\in \big[(1+\frac{\lambda}{4})N,(1+\frac{\lambda}{3})N\big]\Big\}$. Thus, by (\ref{inclusion_6.6}), we have 
		\begin{equation}\label{fina_6.11}
		\begin{split}
			\mathbb{P}(\mathsf{B}_2)\le&\sum_{k\in \mathbb{N}:n_k\in [(1+\frac{\lambda}{4})N,(1+\frac{\lambda}{3})N]} \mathbb{P}\left(k_\dagger=k \right)  \mathbb{E}\bigg\{ \sum_{z_1\in \overline{\mathbf{\Psi}}_{n_{k}}^{\blacktriangle}} \sum_{z_2\in \mathbb{Z}^d:|z_2-z_1|_1\le 1}\\
			&\ \ \ \ \ \ \ \ \ \ \ \ \ \ \ \ \ \ \ \ \ \ \ \ \mathbb{P}\Big[ z_2 \xleftrightarrow[]{\widetilde{\mathcal{L}}^{\mathrm{U}}_{[(1+\lambda)N]^b}}  \partial B\big((1+\lambda)N\big)  \big| \mathbf{\Psi}_{n_{k}}^{\blacktriangle},k_\dagger=k \Big]\bigg\}. 
		\end{split}
		\end{equation}
	In fact, given $\mathbf{\Psi}_{n_{k}}^{\blacktriangle}$, then the unused loops $\widetilde{\mathcal{L}}^{\mathrm{U}}_{[(1+\lambda)N]^b}$ (with respect to $\mathbf{\Psi}_{n_{k}}^{\blacktriangle}$) is independent of $\mathbf{\Psi}_{n_{k'}}^{\blacktriangle}$ for all $0\le k'<k$. To see this, we only need to check the loops in $\widetilde{\mathcal{L}}^{\mathrm{U}}_{[(1+\lambda)N]^b}$ (see Definition \ref{def_unused loops}) as follows: 
	\begin{itemize}
		\item  involved loops $\widetilde{\ell}$ with $\mathrm{ran}(\widetilde{\ell})\cap \mathbf{\Psi}_{n_{k}}^{1,\blacktriangle}=\emptyset $: Since $\mathbf{\Psi}_{n_{k'}}^{1,\blacktriangle}\subset \mathbf{\Psi}_{n_{k}}^{1,\blacktriangle}$, we have $\mathrm{ran}(\widetilde{\ell})\cap \mathbf{\Psi}_{n_{k'}}^{1,\blacktriangle}=\emptyset$. Therefore, $\widetilde{\ell}$ is independent of $\mathbf{\Psi}_{n_{k'}}^{1,\blacktriangle}$.

		\item  loops $\widetilde{\ell}$ with $\mathrm{ran}(\widetilde{\ell})\cap \widetilde{B}(n_k)=\emptyset$:  Since $\widetilde{B}(n_{k'}) \subset\widetilde{B}(n_{k})$, we know that $\widetilde{\ell}$ is disjoint from $\widetilde{B}(n_{k'})$ and thus is independent of $\mathbf{\Psi}_{n_{k'}}^{1,\blacktriangle}$.

		\item  Every remaining loop $\widetilde{\ell}$ is a point loop including some $x \in \partial \hat{B}(n_k)$, which is also disjoint of $\widetilde{B}(n_{k'})$ and is independent of $\mathbf{\Psi}_{n_{k'}}^{1,\blacktriangle}$.

	\end{itemize}
	As a result, given $\mathbf{\Psi}_{n_{k}}^{\blacktriangle}$ and the occurrence of $\{k_{\dagger}=k\}$, the conditioning distribution of $\widetilde{\mathcal{L}}^{\mathrm{U}}_{[(1+\lambda)N]^b}$ (with respect to $\mathbf{\Psi}_{n_{k}}^{\blacktriangle}$) is the same as the one only given $\mathbf{\Psi}_{n_{k}}^{\blacktriangle}$. Combined with Item (2) in Remark \ref{remark_51}, this yields that for each $z_2$ involved in the RHS of (\ref{fina_6.11}), we have  
	\begin{equation}\label{fina_6.12}
		\begin{split}
			&\mathbb{P}\Big[ z_2 \xleftrightarrow[]{\widetilde{\mathcal{L}}^{\mathrm{U}}_{[(1+\lambda)N]^b}}  \partial B\big((1+\lambda)N\big)  \big| \mathbf{\Psi}_{n_{k}}^{\blacktriangle},k_\dagger=k \Big] \\
			\le &\mathbb{P}\Big[ z_2 \xleftrightarrow[]{}  \partial B\big((1+\lambda)N\big) \Big]\le \theta\left(\tfrac{\lambda N}{2} \right),
		\end{split}
	\end{equation}
	where in the last inequality we used $$\overline{\mathbf{\Psi}}_{n_{k}}^{\blacktriangle} \subset B(n_{k}+[(1+\lambda)N]^b)\subset B((1+\tfrac{\lambda}{3})N+[(1+\lambda)N]^b).$$ 
	Combining (\ref{fina_6.11}), (\ref{fina_6.12}) and $0<\psi_{n_{k_{\dagger}}}^{\blacktriangle}\le L^2$, we get 
	\begin{equation}\label{ineq_B2}
		\begin{split}
			\mathbb{P}(\mathsf{B}_2) \le &(2d+1)L^2 \theta\left(\tfrac{\lambda N}{2} \right)\mathbb{P}\left( n_{k_\dagger}\in [(1+\tfrac{\lambda}{4})N,(1+\tfrac{\lambda}{3})N]\right)\\
			\le & 3d\epsilon^{\frac{3}{5}}N^2 \theta\left(\tfrac{\lambda N}{2} \right)\theta(N). 
		\end{split}
	\end{equation}

	Finally, let us consider the event $\mathsf{B}_3$. For any $n\in \mathbb{N}^+$, let
	\begin{equation}
		\chi_{n}= \big|\{x\in B(n+L)\setminus B(n): \bm{0}\xleftrightarrow[]{} x \} \big|.
	\end{equation}
	We need the following theorem, which is the core of this paper.

	\begin{theorem}\label{theorem_regularity}
		For $d>6$,  there exist $c_4(d)>0,c_5(d)\in (0,1)$ such that for each fixed $\lambda \in \left( 0, 1 \right] $ and sufficiently small fixed $\epsilon>0$, the following holds for any large enough $N\ge 1$ and any $n\in \big[(1+\frac{\lambda}{4})N,(1+\frac{\lambda}{3})N\big]$:
		\begin{equation}
			\mathbb{P}\left(  \psi_{n}^*\ge L^2, \chi_{n}\le c_4L^4  \right)  \le (1-c_5)\theta(N). 
		\end{equation}
	\end{theorem}
	
	Now we estimate the probability of $\mathsf{B}_3$ based on Theorem \ref{theorem_regularity}. For any integer $i \in [0,\tfrac{1}{12}\lambda \epsilon^{-\frac{3}{10}}-1]$, let $n_i' := \big\lceil (1+ \tfrac{\lambda}{4}) N + iL\big\rceil$. Note that each $n_i'\in \big[ (1+ \tfrac{\lambda}{4}) N, (1+ \tfrac{\lambda}{3}) N\big]$. We also define 
	$$
	I :=\Big| \big\{  i\in [0,\tfrac{1}{12}\lambda \epsilon^{-\frac{3}{10}}-1]\cap \mathbb{N} :  \psi_{n_i'}^* \ge L^2, 
	\chi_{n_i'} \le c_4L^4  \big\}  \Big|.
	$$
	If $\mathsf{B}_3$ happens, then we have $|\mathbf{C}(\bm{0})|<\epsilon N^4$ and thus 
	$$
	\Big|\big\{i\in [0,\tfrac{1}{12}\lambda \epsilon^{-\frac{3}{10}}-1]\cap \mathbb{N}  : \chi_{n_i'} > c_4L^4  \big\}\Big|< \frac{\epsilon N^4 }{c_4 L^4}= c_4^{-1} \epsilon^{-\frac{1}{5}}.  
	$$
	Therefore, by the Markov's inequality and Theorem \ref{theorem_regularity}, we have 
	\begin{equation}\label{ineq_6.12}
		\begin{split}
			\mathbb{P}(\mathsf{B}_3)\le &\mathbb{P}\left( I \ge \tfrac{1}{12}\lambda \epsilon^{-\frac{3}{10}}-  c_4^{-1} \epsilon^{-\frac{1}{5}}-1\right)  \\
			\le & \frac{\mathbb{E}I}{\tfrac{1}{12}\lambda \epsilon^{-\frac{3}{10}}-  c_4^{-1} \epsilon^{-\frac{1}{5}}-1}\\
			\le & \frac{\frac{1}{12}\lambda \epsilon^{-\frac{3}{10}}(1-c_5)}{\frac{1}{12}\lambda \epsilon^{-\frac{3}{10}}-  c_4^{-1} \epsilon^{-\frac{1}{5}}-1}\cdot \theta(N). 
		\end{split}
	\end{equation}
	For each fixed $\lambda\in \left(0,1 \right] $, by taking a small enough $\epsilon$, we can require that 
	\begin{equation}\label{ineq_6.13}
		\frac{\frac{1}{12}\lambda \epsilon^{-\frac{3}{10}}(1-c_5)}{\frac{1}{12}\lambda \epsilon^{-\frac{3}{10}}-  c_4^{-1} \epsilon^{-\frac{1}{5}}-1}< 1-\frac{c_5}{2}:= 1-c_2. 
	\end{equation}
	By (\ref{ineq_6.12}) and (\ref{ineq_6.13}), we obtain the desired estimate for $\mathbb{P}(\mathsf{B}_3)$ as follows: 
	\begin{equation}\label{ineq_B3}
		\mathbb{P}(\mathsf{B}_3)\le (1-c_2) \theta(N). 
	\end{equation}
	
	
In conclusion, by (\ref{ineq_B0}), (\ref{ineq_B1}), (\ref{ineq_B2}) and (\ref{ineq_B3}), we conclude Proposition \ref{prop_recursion}, and thus complete the proof of Theorem \ref{theorem1} assuming Proposition \ref{prop_volumn} and Theorem \ref{theorem_regularity}. We will prove Proposition \ref{prop_volumn} in Section \ref{section_volumn}. The proof of Theorem \ref{theorem_regularity} will be established in Sections \ref{fina_section7} and \ref{section8}. Specifically, we will prove a core lemma in Section \ref{fina_section7} and then conclude Theorem \ref{theorem_regularity} in Section \ref{section8}.

	\section{Good points, locally good points and qualified points}\label{fina_section7}
	
	As in the last section, we fix $b\in (\frac{6}{d},1)$, $\lambda\in \left( 0,1\right] $ and a sufficiently small $\epsilon>0$. We also take a sufficiently large constant $K(d)>0$. For any $m\ge 1$, we denote $r_m:= K2^{m-1}$. Recall the notations $\widehat{\mathbf{A}}$ and $\widetilde{\mathcal{L}}^{\mathrm{U}}_{\mathbf{A}}$ in (\ref{eq_6.7}) and Definition \ref{def_unused loops} respectively.

	\begin{definition}[$(x,m,l)$-nice tuple]\label{def_nice_set}
		For any $x\in \mathbb{Z}^d$, $m\ge 1$ and admissible $\mathbf{A}=(\mathbf{A}^1,\mathbf{A}^2)$ (i.e. a possible configuration of $\mathbf{\Psi}_n$), we define the function
		\begin{equation}
			\Delta_{x,m}(\mathbf{A}):= 	\mathbb{E}\bigg(\sum_{y\in B_x(r_m)}\mathbbm{1}_{y\xleftrightarrow[]{\cup \widetilde{\mathcal{L}}^{\mathrm{U}}_\mathbf{A}}\widehat{\mathbf{A}}\cap B_x(r_{m}^{4d}) }  \bigg). 
		\end{equation}
		For $l\ge 1$, we say $\mathbf{A}$ is $(x,m,l)$-nice if $\Delta_{x,m}(\mathbf{A})\le r_m^4\log^{l}(r_m)$. 
	\end{definition}

Recall the notation $\overline{\mathbf{\Psi}}_n$ in Item (2) of Definition \ref{definition_Psi}, and also recall that $\overline{\mathbf{\Psi}}_n^*=\overline{\mathbf{\Psi}}_n\cap B(n+[(1+\lambda)N]^b)$.
	
	\begin{definition}[$m$-good point and regular point]\label{def_good_regular_points}\hspace*{\fill}
		\begin{enumerate}
			\item For any $m\ge 1$ and $x\in \mathbb{Z}^d$, we say $x$ is $m$-good if $\mathbf{\Psi}_n$ is $(x,m,16)$-nice. We also say $x$ is $m$-bad if it is not $m$-good.

			\item  If $x$ is $m$-good for all $m\ge 1$, then we say $x$ is regular. Otherwise, we call $x$ an irregular point.

			\item  We say $x$ is strongly regular if $y$ is regular for all $y\in B_x(K^{10d})$.

			\item   We denote the numbers of irregular, strongly regular and $m$-bad points in $\overline{\mathbf{\Psi}}_n^*$ by $\psi_n^{\mathrm{irr}}$, $\psi_n^{\mathrm{SR}}$ and $\psi_n^{m\mbox{-}\mathrm{bad}}$ respectively. 
			
		\end{enumerate}
	\end{definition}

	\begin{remark}\label{remark6.1}
		
		(1)	If $y\in B_x(r_m)\cap \widehat{\mathbf{A}}$, then $\big\{ y\xleftrightarrow[]{\cup \widetilde{\mathcal{L}}^{\mathrm{U}}_\mathbf{A}}\widehat{\mathbf{A}}\cap B_x(r_m^{4d}) \big\}$ a.s. happens. Therefore, we have 
		\begin{equation}
			\Delta_{x,m}(\mathbf{A}) \ge \big| \widehat{\mathbf{A}}\cap B_x(r_m) \big|. 
		\end{equation}
		Thus, when $\mathbf{A}$ is $(x,m,l)$-nice, one has $\big|\widehat{\mathbf{A}}\cap B_x(r_m) \big|\le r_m^4\log^{l}(r_m)$. As a result, when $x$ is $m$-good, we have
		\begin{equation}
			\big|\widehat{\mathbf{\Psi}}_n\cap B_x(r_m)\big|  \le \Delta_{x,m}(\mathbf{\Psi}_n)\le  r_m^4\log^{16}(r_m).
		\end{equation}
		(2) For any admissible tuple $\mathbf{A}$, by Item (2) in Remark \ref{remark_51}, we have 
		\begin{equation}\label{7.5}
			\mathbb{E}\bigg(\sum_{y\in B_x(r_m)}\mathbbm{1}_{y\xleftrightarrow[]{\cup \widetilde{\mathcal{L}}^{\mathrm{U}}}\widehat{\mathbf{\Psi}}_n\cap B_x(r_{m}^{4d}) } \Big|\mathbf{\Psi}_n= \mathbf{A} \bigg)= \Delta_{x,m}(\mathbf{A}). 
		\end{equation}
	\end{remark}

	The main goal of this section is to prove the following lemma, which will be a crucial ingredient in proving Theorem \ref{theorem_regularity}.

	\begin{lemma}\label{lemma_reg}
		With the same conditions as in Theorem \ref{theorem_regularity}, we have 
		\begin{equation}\label{equation_lemma_reg}
			\mathbb{P}\left( \psi_{n}^*\ge L^2, \psi_{n}^{\mathrm{irr}}\ge K^{-20d^2}\psi_{n}^* \right) \le \mathrm{s.p.}(N). 
		\end{equation}
	\end{lemma}

	The lemma above implies that when $\psi_n^*$ is at least $L^2$, with high probability, at least half of the points in $\overline{\mathbf{\Psi}}_n^*$ are strongly regular.

	\begin{corollary}\label{coro_61}
		With the same conditions as in Theorem \ref{theorem_regularity}, we have
		\begin{equation}\label{equation_coro_61}
			\mathbb{P}\left( \psi_{n}^*\ge L^2, \psi_{n}^{\mathrm{SR}}\le \tfrac{1}{2}\psi_{n}^* \right) \le \mathrm{s.p.}(N).
		\end{equation}
	\end{corollary}
	
	\begin{proof}
		For any $x\in \mathbb{Z}^d$, if $x$ is not strongly regular, then there must exist an irregular point $y$ such that $x\in B_y(K^{10d})$. Thus, we have
		$$
		\psi_{n}^*- \psi_{n}^{\mathrm{SR}}\le \big|B(K^{10d}) \big|\cdot \psi_n^{\mathrm{irr}}.
		$$
		Therefore, when $\left\lbrace \psi_{n}^*\ge L^2, \psi_{n}^{\text{SR}}\le \frac{1}{2}\psi_{n}^*\right\rbrace $ happens, one has 
		\begin{equation}\label{Xn_irr_ge}
			\psi_n^{\mathrm{irr}}\ge \big|B(K^{10d})\big|^{-1} \big(\psi_{n}^*- \psi_{n}^{\mathrm{SR}}\big) \ge cK^{-10d^2}\psi_{n}^* \ge K^{-20d^2}\psi_{n}^*. 
		\end{equation}
		By Lemma \ref{lemma_reg} and (\ref{Xn_irr_ge}), we immediately get the corollary.
	\end{proof}

	We next describe the proof of Lemma \ref{lemma_reg}. Recalling Definition \ref{def_good_regular_points}, one has the following deterministic inequality:
	\begin{equation}
		\psi_{n}^{\text{irr}} \le \sum_{m=1}^{\infty} \psi_n^{m\mbox{-}\mathrm{bad}}. 
	\end{equation}
	Combined with $\sum_{m=1}^\infty m^{-2} < 2$, it suffices to prove that for any $m\geq 1$,
	\begin{equation}\label{new_ineq_609}
		\mathbb{P}\left( \psi_{n}^*\ge L^2, \psi_n^{m\mbox{-}\text{bad}}\ge \tfrac{1}{2}m^{-2} K^{-20d^2}\psi_{n}^* \right) \le \mathrm{s.p.}(N). 
	\end{equation}

	It turns out that for large $m$, the proof of (\ref{new_ineq_609}) is fairly simple since the probability for the existence of a single $m$-bad point already decays super-polynomially, as incorporated in Lemma \ref{new_lemma6.3} below. For small $m$, however, the proof is much more delicate since this necessarily requires to control many points simultaneously, and its proof almost occupies the rest of this section.
	\begin{lemma}\label{new_lemma6.3}
		For any $d>6$, there exist constants $c(d),C(d)>0$ such that for any $x\in \mathbb{Z}^d$ and $m\ge 1$, 
		\begin{equation}
			\mathbb{P}\left( x\ \text{is}\ m\mbox{-}\mathrm{bad} \right) \le Ce^{-c\log^{16}(r_m)}.
		\end{equation}
	\end{lemma}
	\begin{proof}
		We denote the event $\mathsf{B}:= \{x\ \text{is}\ m\mbox{-}\mathrm{bad}\}$. By (\ref{7.5}) and the definition of $m$-bad points, one has 
		\begin{equation*}\label{new_ineq6.12}
			\mathbb{E}\bigg( \sum_{y\in B_x(r_m)}\mathbbm{1}_{y\xleftrightarrow[]{\cup\widetilde{\mathcal{L}}^{\mathrm{U}}}\widehat{\mathbf{\Psi}}_n\cap B_x(r_m^{4d}) }  \Big|\mathsf{B} \bigg)  \ge r_m^4\log^{16}(r_m).
		\end{equation*}
		In addition, since $\sum_{y\in B_x(r_m)}\mathbbm{1}_{y\xleftrightarrow[]{\cup\widetilde{\mathcal{L}}^{\mathrm{U}}}\widehat{\mathbf{\Psi}}_n\cap B_x(r_m^{4d}) }\le |B_x(r_m)|=(2r_m+1)^d$, we have  
		\begin{equation*}\label{new_ineq6.13}
			\begin{split}
				&	\mathbb{E}\bigg( \sum_{y\in B_x(r_m)}\mathbbm{1}_{y\xleftrightarrow[]{\cup\widetilde{\mathcal{L}}^{\mathrm{U}}}\widehat{\mathbf{\Psi}}_n\cap B_x(r_m^{4d}) }  \Big|\mathsf{B} \bigg)\\
				\le & Cr_m^d \mathbb{P}\bigg(    \sum_{y\in B_x(r_m)}\mathbbm{1}_{y\xleftrightarrow[]{\cup\widetilde{\mathcal{L}}^{\mathrm{U}}}\widehat{\mathbf{\Psi}}_n\cap B_x(r_m^{4d}) }\ge  \tfrac{1}{2}r_m^4\log^{16}(r_m)\Big|\mathsf{B} \bigg)+  \tfrac{1}{2}r_m^4\log^{16}(r_m).
			\end{split}	
		\end{equation*}
		Combining these two estimates, we get 
		\begin{equation}\label{new_ineq_615}
			\mathbb{P}\bigg(    \sum_{y\in B_x(r_m)}\mathbbm{1}_{y\xleftrightarrow[]{\cup\widetilde{\mathcal{L}}^{\mathrm{U}}}\widehat{\mathbf{\Psi}}_n\cap B_x(r_m^{4d}) }\ge  \tfrac{1}{2}r_m^4\log^{16}(r_m)\Big|\mathsf{B} \bigg)\ge cr_m^{4-d}\log^{16}(r_m).
		\end{equation}

		Since $\widehat{\mathbf{\Psi}}_n$ is connected, all points connected to $\widehat{\mathbf{\Psi}}_n$ must be connected to each other.  Thus, by Lemma \ref{lemma_large_div} we have 	
		\begin{equation*}\label{new_ineq_6.12}
			\begin{split}
				&\mathbb{P}\bigg(    \sum_{y\in B_x(r_m)}\mathbbm{1}_{y\xleftrightarrow[]{\cup\widetilde{\mathcal{L}}^{\mathrm{U}}}\widehat{\mathbf{\Psi}}_n\cap B_x(r_m^{4d}) }\ge \tfrac{1}{2}r_m^4\log^{16}(r_m)\bigg)  \\
				\le &\mathbb{P}\left(  \max_{y\in B_x(r_m)} \left|\mathbf{C}(y)\cap B(r_m) \right|\ge \tfrac{1}{2}r_m^4\log^{16}(r_m) \right)
				\le Ce^{-c\log^{16}(r_m)}. 
			\end{split}
		\end{equation*}
		Combined with (\ref{new_ineq_615}), the desired bound follows.
	\end{proof}

	By Lemma \ref{new_lemma6.3} and $\overline{\mathbf{\Psi}}_n^*\subset B(n+[(1+\lambda)N]^b) $, we have 
	\begin{equation}
		\mathbb{P}\bigg( \sum_{m\ge m_0} \psi_n^{m\mbox{-}\mathrm{bad}} \ge 1\bigg)  \le \mathrm{s.p.}(N),
	\end{equation}
	where $m_0:= \min \{m: r_{m}\ge e^{\log^{\frac{1}{4}}(N)} \}$. We now need to control the probability for small $m$ as promised. To this end, we fix an arbitrary $m\in [1, m_0-1]$.

	For $\psi_n^{m\mbox{-}\mathrm{bad}}$, we make a further decomposition as follows. Let $D_N:=\big\lfloor e^{\log^{\frac{1}{3.5}}(N)} \big\rfloor$. Note that $r_{m_0}<D_N$. For any $w\in \left[-D_N,D_N \right)^d \cap \mathbb{Z}^d$, we define 
	\begin{equation}\label{candidate set}
		F(w):=\{x\in w+2D_N\cdot \mathbb{Z}^d:x\in B(n+[(1+\lambda)N]^b)\setminus B(n-1) \}.
	\end{equation} 
	We also define
	\begin{equation}\label{7.18}
		\zeta_w=\zeta_w(n):= \big|\overline{\mathbf{\Psi}}_n^*\cap F(w)\big|,
	\end{equation}
	\begin{equation}\label{7.19}
		\zeta_w^{m\mbox{-}\mathrm{bad}}=\zeta_w^{m\mbox{-}\mathrm{bad}}(n):= 	\big|\{x\in \overline{\mathbf{\Psi}}_n^*\cap F(w):x\ \text{is}\ m\mbox{-}\text{bad}\}\big|. 
	\end{equation}
It follows from the definition that  
	\begin{equation}\label{new_eq_614}
		\psi_n^*= \sum_{w\in \left[-D_N,D_N \right)^d \cap \mathbb{Z}^d}	\zeta_w, \ \ \psi_n^{m\mbox{-}\text{bad}}= \sum_{w\in \left[-D_N,D_N \right)^d \cap \mathbb{Z}^d}	\zeta_w^{m\mbox{-}\text{bad}}.
	\end{equation}

	We claim the following inclusion relation:
	\begin{equation}\label{inclusion_633}
		\begin{split}
			&\left\lbrace \psi_{n}^*\ge L^2, \psi_n^{m\mbox{-}\mathrm{bad}}\ge \tfrac{1}{2}m^{-2} K^{-20d^2}\psi_{n}^* \right\rbrace\\
			\subset &\bigcup_{w\in \left[-D_N,D_N \right)^d \cap \mathbb{Z}^d} \Big\{ \zeta_w\ge \frac{L^2}{2^{d+2}K^{20d^2}m^2D_N^d},  \zeta_w^{m\mbox{-}\text{bad}}\ge \frac{\zeta_w}{4 K^{20d^2}m^{2}} \Big\}.
		\end{split}
	\end{equation}
We will prove a contrapositive statement of (\ref{inclusion_633}). To this end, denote
	\begin{equation}
		W_1:= \Big\{ w\in B(D_N): \zeta_w< \frac{L^2}{2^{d+2}K^{20d^2}m^2D_N^d} \Big\},
	\end{equation} 
	\begin{equation}
		W_2 := \Big\{w\in B(D_N)\setminus W_1 :  \zeta_w^{m\mbox{-}\text{bad}}< \frac{\zeta_w}{4K^{20d^2}m^{2} }\Big\}. 
	\end{equation}
	In fact, when the event on the RHS of (\ref{inclusion_633}) does not happen, one has $W_1\cup W_2=\left[-D_N,D_N \right)^d \cap \mathbb{Z}^d$. Thus, by (\ref{new_eq_614}) and $|W_1|\le \big|\left[-D_N,D_N \right)^d\big|= (2D_N)^d$, we have 
	
	\begin{equation}
		\begin{split}
			\psi_n^{m\mbox{-}\text{bad}}= &\sum_{w\in W_1} \zeta_w^{m\mbox{-}\text{bad}}+\sum_{w\in W_2} \zeta_w^{m\mbox{-}\text{bad}}\\
			< &(2D_N)^d\cdot \frac{L^2}{2^{d+2}K^{20d^2}m^2D_N^d}+ \sum_{w\in \left[-D_N,D_N \right)^d \cap \mathbb{Z}^d}\frac{\zeta_w}{4 K^{20d^2}m^{2}}\\
			=&\big[ 4K^{20d^2}m^2\big] ^{-1}\left(L^2+ \psi_{n}^* \right),
		\end{split}
	\end{equation}
	which is incompatible with the event on the LHS of (\ref{inclusion_633}), thereby completing the proof (for the contrapositive statement) of (\ref{inclusion_633}). Therefore, to get (\ref{new_ineq_609}) (which implies Lemma \ref{lemma_reg}), it is sufficient to prove the following lemma (since then (\ref{new_ineq_609}) follows via a simple union bound).
	\begin{lemma}\label{new_lemma_6.3}
		With the same conditions as in Theorem \ref{theorem_regularity}, we have 
		\begin{equation}\label{ineq_727}
			\begin{split}
				\max_{1\le m\le m_0-1,w\in \left[-D_N,D_N \right)^d \cap \mathbb{Z}^d}	\mathbb{P}\bigg[\zeta_w\ge \frac{L^2}{2^{d+2}K^{20d^2}m^2D_N^d}, \zeta_w^{m\mbox{-}\mathrm{bad}}\ge \frac{\zeta_w}{4 K^{20d^2}m^{2}} \bigg] \le \mathrm{s.p.}(N). 
			\end{split}
		\end{equation}
	\end{lemma}

	The rest of this section is devoted to the proof of Lemma \ref{new_lemma_6.3}.
	
	%
	
	\subsection{Qualified point}

	We arbitrarily fix $m\in [1,m_0-1]$ and $w\in \left[-D_N,D_N \right)^d \cap \mathbb{Z}^d$. For any $k\in \mathbb{N}$, let $s_k=s_k(m):=r_m^{(4d)^{k+1}}$. Note that $s_{k+1}=(s_k)^{4d}$. We denote 
	\begin{equation}\label{def_k0}
			k_0=k_0(m):=\min\{k\ge 1: s_{k+2}\ge \sqrt{D_N}\}.
	\end{equation}


%
	
	Recall the notation $\kappa(\widetilde{\ell};A_1,A_2)$ in Section \ref{subsection_decomposition}.

	\begin{definition}[$k$-qualified point]\label{def_DQ_point}\hspace*{\fill}
		\begin{enumerate}
			\item	For any $k\ge 1$ and $x\in F(w)$, we say $x$ is $k$-qualified  if the total number of forward crossing paths (with $A_1=B_x(s_k)$ and $A_2=\partial B_x(s_{k+1})$) of loops in $\widetilde{\mathcal{L}}_{1/2}$ is at most $\log^{6}(s_k)$. I.e., 
			\begin{equation}
				\sum_{\widetilde{\ell} \in \widetilde{\mathcal{L}}_{1/2}: \mathrm{ran}(\widetilde{\ell})\cap B_x(s_k)\neq\emptyset,\mathrm{ran}(\widetilde{\ell})\cap \partial B_x(s_{k+1})\neq\emptyset } \kappa(\widetilde{\ell};B_x(s_k),\partial B_x(s_{k+1})) \le \log^{6}(s_k).
			\end{equation}

			\item We say $x$ is $k$-unqualified if it is not $k$-qualified.

			\item We denote the number of $k$-unqualified points in $\overline{\mathbf{\Psi}}_n^*$ by $\psi_n^{k\mbox{-}\text{UQ}}$.

		\end{enumerate}
	\end{definition}

	We first show that for each lattice point, only with a small probability it is $k$-unqualified.

	\begin{lemma}\label{lemma_onepoint}
		There exist $C(d),c(d)>0$ such that for any $x\in F(w)$ and $k\ge 1$,
		\begin{equation}
			\mathbb{P}\left( x\ \text{is}\ k\text{-unqualified}  \right) \le   Ce^{-c\log^{4}(s_k)}. 
		\end{equation} 
	\end{lemma}
	\begin{proof}
		Let $N_{x,k}$ be the number of loops in $\widetilde{\mathcal{L}}_{1/2}$ that cross $B_x(s_{k+1})\setminus B_x(s_k)$. By Definition \ref{def_DQ_point} we have
		\begin{equation}\label{inequation_647}
			\begin{split}
				&\mathbb{P}\left( x\ \text{is}\ k\text{-unqualified}\right)  \\
				\le &\mathbb{P}\Big[N_{x,k} < \log^{3}(s_k), \exists\ \widetilde{\ell}\in \widetilde{\mathcal{L}}_{1/2}\ \text{with more than}\ \log^{3}(s_k)\ \text{forward crossing}\\
				&\ \ \ \ \text{paths with}\ A_1=B_x(s_k),A_2=\partial B_x(s_{k+1})\Big]+\mathbb{P}\left[N_{x,k} \ge \log^{3}(s_k) \right].
			\end{split}
		\end{equation}

		Let $\mu_{x,k}$ be the loop measure of loops with more than $\log^{3}(s_k)$ forward crossing paths with $A_1=B_x(s_k)$ and $A_2=\partial B_x(s_{k+1})$. By (\ref{ineq_2.4}), we have  
		$$
		\mu_{x,k}\le  \Big[C\cdot \text{cap}\big(B_x(s_k)\big)\cdot (s_{k+1})^{2-d}\Big]^{\log^{3}(s_k)}\le Ce^{-c\log^{4}(s_k)},
		$$
		which implies that the first term on the RHS of (\ref{inequation_647}) is bounded from above by 
		\begin{equation}\label{upper_bound_651}
			\begin{split}
				1-e^{-\frac{1}{2}\mu_{x,k}}\le \tfrac{1}{2}\mu_{x,k}\le Ce^{-c\log^{4}(s_k)}. 
			\end{split}
		\end{equation}
		For the second term, by (\ref{ineq_2.5}), the loop measure of loops crossing $B_x(s_{k+1})\setminus B_x(s_k)$ is at most $$C\cdot \text{cap}\big(B_x(s_k)\big)\cdot (s_{k+1})^{2-d}\le C's_k^{-(4d-1)(d-2)}.$$ 
		Therefore, $N_{x,k}$ is stochastically dominated by $\text{Pois}(\lambda_k)$, where $\lambda_k=\frac{1}{2}C's_k^{-(4d-1)(d-2)} $. Recall that for any $\xi,\lambda>0$ and a Poisson random variable $Y\sim \text{Pois}(\lambda)$, one has $\mathbb{E}[\exp(\xi Y)]=\exp(\lambda(e^{\xi}-1))$. Thus, by using the exponential Markov's inequality, and taking $\xi=\log(\lambda_k^{-1}+1)$, $\lambda=\lambda_k$ and $Y=N_{x,k}$, we have
		\begin{equation}\label{inequality_648}
			\begin{split}
				\mathbb{P}\left[N_{x,k} \ge \log^{3}(s_k)\right]\le e^{-\log(\lambda_k^{-1}+1)\log^{3}(s_k)}\mathbb{E}\left[e^{\log(\lambda_k^{-1}+1) N_{x,k}} \right] \le Ce^{-c\log^{4}(s_k)}.
			\end{split}
		\end{equation}

		Combining (\ref{inequation_647}), (\ref{upper_bound_651}) and (\ref{inequality_648}), we complete the proof.
	\end{proof}

	Recalling that $D_N=\big\lfloor e^{\log^{\frac{1}{3.5}}(N)} \big\rfloor$ and $k_0=\min\big\{k\ge 1: s_{k+2}\ge \sqrt{D_N}\big\}$, one has $$	\sum_{k=k_0}^{\infty}\exp(-c\log^{4}(s_{k}))\le C\exp(-c\log^{4/3.5}(N))=\text{s.p.}(N).$$ Thus, by Lemma \ref{lemma_onepoint} and $\overline{\mathbf{\Psi}}_n^*\subset B(n+[(1+\lambda)N]^b) $, we have 
	\begin{equation}\label{new_ineq_625}
	\mathbb{P}\left( \exists x\in \overline{\mathbf{\Psi}}_n^*\ \text{and} \ k\ge k_0\ \text{such that}\ x\ \text{is}\ k\text{-unqualified} \right)  \le \text{s.p.}(N). 
	\end{equation}

	Next, we will demonstrate the ``inheritability'' of qualified points. I.e., given that a lattice point $x$ is $(k+1)$-qualified, the conditional probability for $x$ to be $k$-qualified is close to $1$. Before proving that, we need a technical lemma as follows. For the sake of fluency, we leave its proof in Section \ref{section_technical_lemma_1}.
	
	\begin{lemma}\label{lemma_num_crossing}
		Let $\mathfrak{N}$ be the number of times that the Brownian motion $\widetilde{S}_t$ on $\widetilde{\mathbb{Z}}^d$ crosses $B_x(s_{k+1})\setminus B_x(s_{k})$ before hitting $\partial B_x(s_{k+2})$. Then there exists $c(d)>0$ such that for any $x\in \mathbb{Z}^d$, $y\in \partial B_x(s_{k+1})$, $z\in\partial \hat{B}_x(s_{k+2})$ and $l\in \mathbb{N}^+$, 
		\begin{equation}\label{ineq_7.31}
			\widetilde{\mathbb{P}}_y\left[\mathfrak{N}= l\big| \widetilde{\tau}_{\partial B_x(s_{k+2})}=\widetilde{\tau}_z\right] \le  s_{k+1}^{-cl}. 
		\end{equation}	 
		As a direct consequence, for any $\gamma>0$ with $e^{\gamma}<s_{k+1}^{c}$,  
		\begin{equation}\label{ineq_7.32}
			\widetilde{\mathbb{E}}_y\left[ \exp(\gamma \mathfrak{N})\big|  \widetilde{\tau}_{\partial B_x(s_{k+2})}= \widetilde{\tau}_z \right]  \le \frac{s_{k+1}^{c}}{s_{k+1}^{c}-e^{\gamma}}.
		\end{equation}
	\end{lemma}

	For any $x\in \mathbb{Z}^d$ and $j\in \mathbb{N}$, let $\Omega_{x,j}:=\cup_{l\in \mathbb{N}} \big[\partial B_x(s_{j}) \times \partial \hat{B}_x(s_{j+1})\big]^l$ be the collection of possible configurations of starting and ending points of all forward crossing paths (with $A_1=B_x(s_{j})$ and $A_2=\partial B_x(s_{j+1})$) in a collection of loops.

	For any $\omega_{x,j}\in \Omega_{x,j}$, we denote by $\mathbb{P}(\cdot \mid \omega_{x,j})$ the conditional measure given that the configuration of starting and ending points of all forward crossing paths (with $A_1=B_x(s_{j})$ and $A_2=\partial B_x(s_{j+1})$) in $\widetilde{\mathcal{L}}_{1/2}$ is equal to $\omega_{x,j}$.

		
	\begin{remark}\label{remark_7.11}
	We claim that under $\mathbb{P}(\cdot \mid \omega_{x,j})$ where $\omega_{x,j}=\big((x_1,y_1),...,(x_{l},y_l) \big)$, all the $l$ forward crossing paths are independent and their marginal distribution is given by $\widetilde{\mathbb{P}}_{x_i}( \ \cdot   \ \big|\widetilde{\tau}_{\partial B_x(s_{j+1})}=\widetilde{\tau}_{y_i} )$ for $1\leq i\leq l$. In fact, the conditioning of $\widetilde{\mathbb{P}}(\cdot \mid \omega_{x,j})$ is equivalent to ``the backward crossing paths $\widetilde{\eta}^{\mathrm{B}}_i$ for $1\le i\le l$ are compatible with $\omega_{x,j}$ (i.e. each $\widetilde{\eta}^{\mathrm{B}}_i$ starts from $y_{i-1}$ (where $y_{-1}:=y_l$) and ends at $x_i$)''. At this point, the claim follows by recalling Lemma \ref{lemma_new_2.3}.\end{remark}	
		

%
%

	The next lemma shows the inheritability of $k$-qualified points. 
	\begin{lemma}\label{new_lemma_6.6}
		For any $d>6$, there exist $C(d),c(d)>0$ such that the following holds: for every $\omega_{x,k+1} \in \Omega_{x,k+1}$ such that $x$ is $(k+1)$-qualified with respect to $\omega_{x,k+1}$, we have 
		\begin{equation}
			\mathbb{P}\left( x\ \text{is}\ k\text{-unqualified}\mid \omega_{x,k+1} \right)  \le Ce^{-c\log^{4}(s_k)}.  
		\end{equation} 	
	\end{lemma}
	\begin{proof}
		Note that the loops in $\widetilde{\mathcal{L}}_{1/2}$ crossing $B_x(s_{k+1})\setminus B_x(s_{k})$ can be divided into the following two types:
		$$\mathfrak{L}^{\mathrm{cro}}:=\left\lbrace \widetilde{\ell}\in \widetilde{\mathcal{L}}_{1/2}:\forall j\in \{0,1,2\},\mathrm{ran}(\widetilde{\ell})\cap \partial B_x(s_{k+j})\neq \emptyset \right\rbrace,$$
		$$
		\mathfrak{L}^{\mathrm{in}}:=\left\lbrace \widetilde{\ell}\in \widetilde{\mathcal{L}}_{1/2}:\mathrm{ran}(\widetilde{\ell})\subset \widetilde{B}_{x}(s_{k+2})\ \text{and}\ \forall j\in \{0,1\},\mathrm{ran}(\widetilde{\ell})\cap \partial B_x(s_{k+j})\neq \emptyset \right\rbrace.$$
%
		We denote by $\xi^{\mathrm{cro}}:=\sum_{\widetilde{\ell}\in\mathfrak{L}^{\mathrm{cro}}}\kappa(\widetilde{\ell})$ the total number of forward crossing paths (with $A_1=B_x(s_k)$, $A_2=\partial B_x(s_{k+1})$) of loops in $\mathfrak{L}^{\mathrm{cro}}$. Similarly, let  $\xi^{\mathrm{in}}:=\sum_{\widetilde{\ell}\in\mathfrak{L}^{\mathrm{in}}}\kappa(\widetilde{\ell})$.


		We enumerate the forward crossing paths (with $A_1=B_x(s_{k+1})$ and $A_2=B_x(s_{k+2})$) of loops in $\mathfrak{L}^{\mathrm{cro}}$ as $\widetilde{\eta}_i^{\mathrm{F}}$ for $1\le i\le q(\omega_{x,k+1})$, which starts from $x_i(\omega_{x,k+1})\in \partial B_x(s_{k+1})$ and ends at $y_i(\omega_{x,k+1})\in \partial \hat{B}_x(s_{k+2})$. In the rest of this proof, we write $q:=q(\omega_{x,k+1})$, $x_i:=x_i(\omega_{x,k+1})$ and $y_i:=y_i(\omega_{x,k+1})$ for short. Since $x$ is $(k+1)$-qualified with respect to $\omega_{x,k+1}$, we know that $q\le \log^{6}(s_{k+1})$. By Remark \ref{remark_7.11}, $\widetilde{\eta}_i^{\mathrm{F}}$ for $1\le i\le q$ are conditionally independent and their conditional distributions are given by $\widetilde{\mathbb{P}}_{x_i}\left(\cdot |\widetilde{\tau}_{\partial B_x(s_{k+2})}=\widetilde{\tau}_{y_i} \right)$. We denote by $\mathfrak{N}_i$ the number of times that $\widetilde{\eta}_i^{\mathrm{F}}$ crosses $B_x(s_{k+1})\setminus B_x(s_{k})$. Note that $\xi^{\mathrm{cro}}=\sum_{i=1}^{q} \mathfrak{N}_i $. By the exponential Markov's inequality, Lemma \ref{lemma_num_crossing} and $q\le \log^{6}(s_{k+1})$, we have 
		\begin{equation}\label{ineq_zeta1_629}
			\begin{split}
				&\mathbb{P}\left[\xi^{\mathrm{cro}}\ge \tfrac{1}{2}\log^{6}(s_k)\mid \omega_{x,k+1} \right]\\
				\le &e^{-\frac{1}{2}\log^{6}(s_k)} \prod_{i=1}^{q}\mathbb{E}\big( e^{\mathfrak{N}_i}\mid \omega_{x,k+1} \big) \\
				\le & e^{-\frac{1}{2}\log^{6}(s_k)} \bigg(\frac{s_{k+1}^{c}}{s_{k+1}^{c}-e}\bigg) ^{\log^{6}(s_{k+1})}
				\le  e^{-\frac{1}{4}\log^{6}(s_k)}. 
			\end{split}
		\end{equation}

		Now we consider $\xi^{\mathrm{in}}$, which is determined by $\mathfrak{L}^{\mathrm{in}}$. Since the loops in $\mathfrak{L}^{\mathrm{in}}$ all belong to $\widetilde{\mathcal{L}}_{1/2}$ and are independent of $\omega_{x,k+1}$, using the same argument in the proof of Lemma \ref{lemma_onepoint}, we have 
		\begin{equation}\label{ineq_zeta2_630}
			\mathbb{P}\left[\xi^{\mathrm{in}}\ge \tfrac{1}{2}\log^{6}(s_k)\mid \omega_{x,k+1} \right]\le Ce^{-c\log^{4}(s_k)}. 
		\end{equation}
		Combining (\ref{ineq_zeta1_629}) and (\ref{ineq_zeta2_630}), we complete the proof.
	\end{proof}

	\subsection{Locally good points}\label{section_locally_good}
	
	By Definition \ref{def_good_regular_points}, whether a lattice point $x$ is $m$-good depends on the whole configuration of $\mathbf{\Psi}_n$. This global dependence causes significant difficulty in the analysis. To this end, we approximate $m$-good points by \textit{locally good points} as we define next. Before that, we introduce some notations to simplify our presentation:

	\begin{itemize}
		\item  Let $\widetilde{\eta}^{\text{F}}_{x,i}$ for $1\le i\le q^{\text{F}}_x$ be the forward crossing paths (with $A_1=B_x(s_1)$, $A_2=\partial B_x(s_2)$) of loops in $\widetilde{\mathcal{L}}_{1/2}$ (recall that $m$ is fixed and $s_k=r_m^{(4d)^{k+1}}$). Let $\mathcal{Q}_x^{\text{F}}$ be the collection of subsets of $\{1,2,...,q^{\text{F}}_x\}$. For any $Q\in\mathcal{Q}_x^{\text{F}} $, we denote by $\mathfrak{L}^{\mathrm{F}}_x(Q)$ the collection of all forward crossing paths $\widetilde{\eta}^{\text{F}}_{x,i}$ with $i\in Q$.

		\item  We denote by $\mathfrak{L}_{x}^{\mathrm{inv}}$ the collection of involved loops $\widetilde{\ell}$ with $\mathrm{ran}(\widetilde{\ell})\subset \widetilde{B}_x(s_{2})$ (recall the definition of ``involved loops'' in Definition \ref{definition_Psi}).

		\item  For any $z\in \partial B_x(s_1)$ and $Q\in \mathcal{Q}_x^{\text{F}}$, we denote by $\mathbf{\Phi}_{x,z}^1=\mathbf{\Phi}_{x,z}^1(Q)$ the cluster of $\cup (\mathfrak{L}^{\mathrm{F}}_x(Q)\cup \mathfrak{L}_{x}^{\mathrm{inv}})$ containing $z$. Let $\mathbf{\Phi}_{x,z}^2=\mathbf{\Phi}_{x,z}^2(Q)$ be the collection of points $y\in B_x(s_0)\cap \partial \hat{B}(n)\setminus \mathbf{\Phi}_{x,z}^1$ such that $I_{\{y,y^{\mathrm{in}}\}}\subset \gamma_x^{\mathrm{p}}\cup \mathbf{\Phi}_{x,z}^1$. Then we define $\mathbf{\Phi}_{x,z}=\mathbf{\Phi}_{x,z}(Q):=(\mathbf{\Phi}_{x,z}^1,\mathbf{\Phi}_{x,z}^2)$ and
		\begin{equation}
			\widehat{\mathbf{\Phi}}_{x,z} =\widehat{\mathbf{\Phi}}_{x,z}(Q):= \mathbf{\Phi}_{x,z}^1 \cup \bigcup_{x\in \mathbf{\Phi}_{x,z}^2}  \gamma_x^{\mathrm{p}}.
		\end{equation}
		For completeness, when $z\notin \cup\left( \mathfrak{L}^{\mathrm{F}}_x(Q)\cup \mathfrak{L}_{x}^{\mathrm{inv}}\right)$, let $\mathbf{\Phi}_{x,z}^1,\mathbf{\Phi}_{x,z}^2,\mathbf{\Phi}_{x,z},\widehat{\mathbf{\Phi}}_{x,z}=\emptyset$.

		\item We define a local version of $\widetilde{\mathcal{L}}_{\mathbf{A}}^{\mathrm{U}}$ (recall Definition \ref{def_unused loops}) as follows. For any possible configuration $\mathbf{D}=(\mathbf{D}^1,\mathbf{D}^2)$ of some $\mathbf{\Phi}_{x,z}(Q)$, let $\widetilde{\mathcal{L}}_{x,\mathbf{D}}^{\mathrm{LU}}$ be the point measure composed of the following types of loops in $\widetilde{\mathcal{L}}_{1/2}$, which are contained in $\widetilde{B}_x(s_2)$:
		\begin{itemize}
			\item involved loops $\widetilde{\ell}$ with $\mathrm{ran}(\widetilde{\ell})\cap \mathbf{D}^1=\emptyset $;

			\item loops $\widetilde{\ell}$ with $\mathrm{ran}(\widetilde{\ell})\cap \widetilde{B}(n)=\emptyset$;

			\item point loops including some point $y\in  [\partial \hat{B}(n)\setminus B_x(s_{0})]\cup [\mathbf{D}^1\cap \partial \hat{B}(n)]$;


			\item point loops that include some point $y\in B_x(s_{0})\cap \partial \hat{B}(n)\setminus (\mathbf{D}^1\cup \mathbf{D}^2)$ and do not intersect $I_{\{y,y^{\mathrm{in}}\}}\cap \mathbf{D}^1$. 
			
		\end{itemize}

		\item  We define $\widetilde{\mathcal{L}}_{x,z}^{\mathrm{LU}}=\widetilde{\mathcal{L}}_{x,z}^{\mathrm{LU}}(Q)$ as $\widetilde{\mathcal{L}}_{x,\mathbf{D}}^{\mathrm{LU}}$ on the event $\{\mathbf{\Phi}_{x,z}(Q)=\mathbf{D}\}$.


		\item Parallel to (\ref{eq_6.7}), we define 
		\begin{equation}
			\widehat{\mathbf{D}}:= \mathbf{D}^1\cup \bigcup_{y\in\mathbf{D}^2} \widehat{\gamma}_y^{\mathrm{p}}(\mathbf{D}^1),
		\end{equation}
where $\{\widehat{\gamma}_x^{\mathrm{p}}(\mathbf{D}^1)\}_{x\in \mathbf{D}^2}$ is independent of $\widetilde{\mathcal{L}}_{1/2}$, and has the same distribution as $\{\gamma_x^{\mathrm{p}}\}_{x\in \mathbf{D}^2}$ given $\cap_{x\in \mathbf{D}^2} \{I_{\{x,x^{\mathrm{in}}\}}\subset \gamma_x^{\mathrm{p}}\cup \mathbf{D}^1\}$.

		\item  Let $Q^{\ddagger}$ be the collection of integers $i\in [1,q^{\text{F}}_x]$ such that $\widetilde{\eta}^{\text{F}}_{x,i}$ is contained in an involved loop. Note that $Q^{\ddagger}\in \mathcal{Q}_x^{\text{F}}$. We denote $\mathbf{\Phi}_{x,z}^{\ddagger}:=\mathbf{\Phi}_{x,z}(Q^{\ddagger})$, $\mathbf{\Phi}_{x,z}^{i,\ddagger}:=\mathbf{\Phi}_{x,z}^i(Q^{\ddagger})$ for $i\in \{1,2\}$, $\widehat{\mathbf{\Phi}}_{x,z}^{\ddagger}:=\widehat{\mathbf{\Phi}}_{x,z}(Q^{\ddagger})$ and $\widetilde{\mathcal{L}}_{x,z}^{\mathrm{LU},\ddagger}:= \widetilde{\mathcal{L}}_{x,z}^{\mathrm{LU}}(Q^{\ddagger})$. 
		
	\end{itemize}
	
%
%
%
	
	\begin{remark}\label{remark7.12} 
		 Here are some useful relations between $\mathbf{\Psi}_n$ and $\mathbf{\Phi}_{x,z}^{\ddagger}$:
		\begin{enumerate}
			\item If we delete all loops $\widetilde{\ell}$ included in $\mathbf{\Psi}_n^1$ with  $\mathrm{ran}(\widetilde{\ell})\cap \widetilde{B}_x(s_1)=\emptyset$, and all backward crossing paths with $A_1= B_x(s_1)$ and $A_2=\partial B_x(s_{2})$, then the remaining part of $\mathbf{\Psi}_n^1$ that intersects $\widetilde{B}_x(s_1)$ is composed of several clusters of the form $\mathbf{\Phi}_{x,z}^{1,\ddagger}$ for $z\in \partial B_x(s_1)$ (but not every $\mathbf{\Phi}_{x,z}^{1,\ddagger}$ necessarily intersects $\mathbf{\Psi}_n^1$). Let $U_x^{\ddagger}$ be the collection of $z\in \partial B_x(s_1)$ such that $\mathbf{\Phi}_{x,z}^{1,\ddagger}\subset \mathbf{\Psi}_n^1$. Since the deleted loops and paths are disjoint from $\widetilde{B}(s_{1})$, we have 
			\begin{equation}\label{747}
				\mathbf{\Psi}_n^1\cap \widetilde{B}(s_{1}) = \cup_{z\in U_x^{\ddagger}} \mathbf{\Phi}_{x,z}^{1,\ddagger}\cap \widetilde{B}(s_{1}). 
			\end{equation}
			Thus, if $y\in B_x(s_0)\cap \partial \hat{B}(n)\setminus \mathbf{\Psi}_n^1$ satisfies $I_{\{y,y^{\mathrm{in}}\}}\subset \gamma_y^{\mathrm{p}}\cup \mathbf{\Psi}_n^1$, then there exists some $z\in U_x^{\ddagger}$ with $I_{\{y,y^{\mathrm{in}}\}}\subset \gamma_y^{\mathrm{p}}\cup \mathbf{\Phi}_{x,z}^{1,\ddagger}$. In addition, by (\ref{747}) one has $y \notin \mathbf{\Phi}_{x,z}^{1,\ddagger}$ for all $z\in  U_x^{\ddagger}$. These two facts yield that 
			\begin{equation}\label{749}
				\mathbf{\Psi}_n^2\cap B_x(s_0)\subset  \cup_{z\in U_x^{\ddagger}}\mathbf{\Phi}_{x,z}^{2,\ddagger}.  
			\end{equation}
			By (\ref{747}) and (\ref{749}), one has 
			\begin{equation}\label{newineq_747}
				\widehat{\mathbf{\Psi}}_n\cap B_x(s_0) \subset  \cup_{z\in U_x^{\ddagger}} \widehat{\mathbf{\Phi}}_{x,z}^{\ddagger}\cap B_x(s_0).
			\end{equation}

			\item  We claim that every loop $\widetilde{\ell}\in \widetilde{\mathcal{L}}^{\mathrm{U}}$ with $\mathrm{ran}(\widetilde{\ell})\subset \widetilde{B}_x(s_2)$ is in one of the following cases:
			
			\begin{enumerate}
				\item   $\widetilde{\ell}$ is a point loop including some $y\in \mathbf{\Psi}_n^1\cap B_x(s_{0})\cap  \partial \hat{B}(n)$; 
				
				\item   $\widetilde{\ell}$ is contained in $\widetilde{\mathcal{L}}_{x,z}^{\mathrm{LU},\ddagger}$ for all $z\in U_x^{\ddagger}$.

			\end{enumerate}
		To verify the claim, it suffices to check each type of loops $\widetilde{\ell}$ with $\mathrm{ran}(\widetilde{\ell})\subset \widetilde{B}_x(s_2)$ in Definition \ref{def_unused loops} as follows.
			\begin{itemize}
				\item involved loops $\widetilde{\ell}$ with $\mathrm{ran}(\widetilde{\ell})\cap \mathbf{\Psi}_n^1=\emptyset$: For any $z\in U_x^{\ddagger}$, as mentioned in Item (1), one has $\mathbf{\Phi}_{x,z}^{1,\ddagger}\subset \mathbf{\Psi}_n^1$. Therefore, we have $\mathrm{ran}(\widetilde{\ell})\cap \mathbf{\Phi}_{x,z}^{1,\ddagger}=\emptyset$, and thus $\widetilde{\ell}\in \widetilde{\mathcal{L}}_{x,z}^{\mathrm{LU},\ddagger}$.

				\item loops $\widetilde{\ell}$ with $\mathrm{ran}(\widetilde{\ell})\cap \widetilde{B}(n)=\emptyset$: Obviously, one has $\widetilde{\ell}\in \widetilde{\mathcal{L}}_{x,z}^{\mathrm{LU},\ddagger}$ for all $z\in U_x^{\ddagger}$.

				\item point loops $\widetilde{\ell}$ including some $y\in \mathbf{\Psi}_n^1\cap \partial \hat{B}(n)$: If $y\in B_x(s_0)$, these loops are in Case (a) of the claim. Otherwise, one has $y\in \partial \hat{B}(n) \setminus B_x(s_0)$, and therefore $\widetilde{\ell}\in \widetilde{\mathcal{L}}_{x,z}^{\mathrm{LU},\ddagger}$ for all $z\in U_x^{\ddagger}$.

				\item point loops $\widetilde{\ell}$ that include some $y\in \partial \hat{B}(n)\setminus (\mathbf{\Psi}_n^1\cup \mathbf{\Psi}_n^2)$ and do not intersect $I_{\{y,y^{\mathrm{in}}\}}\cap \mathbf{\Psi}_n^1$: For any $z\in U_x^{\ddagger}$, since $y\notin \mathbf{\Psi}_n^1$, one has $y\notin \mathbf{\Phi}_{x,z}^{1,\ddagger}$. In addition, since $y\notin \mathbf{\Psi}_n^2$, we have $I_{\{y,y^{\mathrm{in}}\}}\not\subset \gamma_y^{\mathrm{p}}\cup \mathbf{\Psi}_n^{1}$, which yields $I_{\{y,y^{\mathrm{in}}\}}\not\subset \gamma_y^{\mathrm{p}}\cup \mathbf{\Phi}_{x,z}^{1,\ddagger}$, and thus $y\notin \mathbf{\Phi}_{x,z}^{2,\ddagger}$. Furthermore, since $\widetilde{\ell}$ does not intersect $I_{\{y,y^{\mathrm{in}}\}}\cap \mathbf{\Psi}_n^1$, $\widetilde{\ell}$ does not intersect $I_{\{y,y^{\mathrm{in}}\}}\cap \mathbf{\Phi}_{x,z}^{1,\ddagger}$ either. These three facts imply that $\widetilde{\ell}\in \widetilde{\mathcal{L}}_{x,z}^{\mathrm{LU},\ddagger}$.

			\end{itemize}
			To sum up, we conclude the claim.

			\item  We have the following inclusion: for any $y\in B_x(r_m)$, 
			\begin{equation}\label{750}
				\bigg\{ y\xleftrightarrow[]{\cup \widetilde{\mathcal{L}}^{\mathrm{U}}_{\mathbf{A}}\cdot \mathbbm{1}_{\widetilde{\ell}\subset \widetilde{B}_x(s_2)}}\widehat{\mathbf{\Psi}}_{n}\cap B_x(s_0) \bigg\} \subset  \bigcup_{z\in U_x^{\ddagger}} \bigg\{ y\xleftrightarrow[]{\cup \widetilde{\mathcal{L}}_{x,z}^{\mathrm{LU},\ddagger}}\widehat{\mathbf{\Phi}}_{x,z}^{\ddagger}\cap B_x(s_0) \bigg\}. 
			\end{equation}
			In fact, when the event on the LHS happens, in $ \widetilde{\mathcal{L}}^{\mathrm{U}}_{\mathbf{A}}\cdot \mathbbm{1}_{\widetilde{\ell}\subset \widetilde{B}_x(s_2)}$ there is a finite sequence of loops $\widetilde{\ell}_i$ for $1\le i\le M$ such that $\widetilde{\ell}_1$ intersects $y$, $\widetilde{\ell}_M$ intersects $\widehat{\mathbf{\Psi}}_{n}\cap B_x(s_0)$, and $\mathrm{ran}(\widetilde{\ell}_i)\cap\mathrm{ran}(\widetilde{\ell}_{i+1})\neq \emptyset$ for all $1\le i\le M-1$. Let $i_{\dagger}:=\min\{i\in [1,M]:\widetilde{\ell}_i\ \text{is a point loop including some}\ v\in \mathbf{\Psi}_n^1\cap B_x(s_0)\cap  \partial \hat{B}(n)\}$, where we set $\min \emptyset=\infty$ for completeness. There are two cases as follows. 
			\begin{enumerate}
				
				\item If $i_{\dagger}=\infty$, by Item (2), we know that $y$ is connected to $\widehat{\mathbf{\Psi}}_{n}\cap B_x(s_0)$ by $\cup \widetilde{\mathcal{L}}_{x,z}^{\mathrm{LU},\ddagger}$ for all $z\in U_x^{\ddagger}$. Combined with (\ref{747}), this yields that the event on the RHS of (\ref{750}) happens.

				\item If $i_{\dagger}\in [1,M]$, then by (\ref{747}), $\widetilde{\ell}_{i_{\dagger}}$ is a point loop including some $v_{\dagger}\in \mathbf{\Phi}_{x,z_{\dagger}}^{1,\ddagger}\cap B_x(s_0)\cap \partial \hat{B}(n)$ for some $z_\dagger\in U_x^{\ddagger}$. Note that $y$ is connected to $\widehat{\mathbf{\Phi}}_{x,z_{\dagger}}^{\ddagger}\cap B_x(s_0)$ by $\cup \{\widetilde{\ell}_1,...,\widetilde{\ell}_{i_{\dagger}}\}$. By Item (2) and the minimality of $i_{\dagger}$, we have $\widetilde{\ell}_i\in \widetilde{\mathcal{L}}_{x,z_\dagger}^{\mathrm{LU},\ddagger}$ for all $1\le i< i_{\dagger}$. Thus, since $\widetilde{\ell}_{i_{\dagger}}$ is also in $\widetilde{\mathcal{L}}_{x,z_\dagger}^{\mathrm{LU},\ddagger}$ (by $v_{\dagger}\in \mathbf{\Phi}_{x,z_{\dagger}}^{1,\ddagger}\cap \partial \hat{B}(n)$), $y$ is connected to $\widehat{\mathbf{\Phi}}_{x,z_{\dagger}}^{\ddagger}\cap B_x(s_0)$ by $\cup \widetilde{\mathcal{L}}_{x,z_\dagger}^{\mathrm{LU},\ddagger}$, which implies the occurrence of the event on the RHS of (\ref{750}). 
			
			\end{enumerate}
			To sum up, we conclude (\ref{750}). 
			
		\end{enumerate}

	\end{remark}

Recall that $s_k=r_m^{(4d)^{k+1}}$. We also introduce a local version of Definition \ref{def_nice_set}:
	
	\begin{definition}[$(x,m,l)$-locally nice tuple]
		For any $x\in \mathbb{Z}^d$, $m\ge 1$ and tuple $\mathbf{D}=(\mathbf{D}^1,\mathbf{D}^2)$ (which is a possible configuration of some $\mathbf{\Phi}_{x,z}(Q)$), we define
		\begin{equation}\label{7.51}
			\Delta_{x,m}^{\mathrm{loc}}(\mathbf{D}):= 	\mathbb{E}\bigg(\sum_{y\in B_x(r_m)}\mathbbm{1}_{y\xleftrightarrow[]{\widetilde{\mathcal{L}}_{x,\mathbf{D}}^{\mathrm{LU}}}\widehat{\mathbf{D}}\cap B_x(s_0) }  \bigg). 
		\end{equation}
		For $l\ge 1$, we say $\mathbf{D}$ is $(x,m,l)$-locally nice if $\Delta_{x,m}^{\mathrm{loc}}(\mathbf{D})\le r_m^4\log^{l}(r_m)$. 
	\end{definition}

Let $V_x=V_x(Q)$ be the collection of $z\in \partial B_x(s_1)$ such that $\mathbf{\Phi}_{x,z}^1$ intersects $B_x(s_0+1)$.

	\begin{definition}[$m$-locally good points]\label{def_local_good_regular_points}\hspace*{\fill}
		\begin{enumerate}
			\item For any $m\in [1,m_0-1]$, $w\in \left[-D_N,D_N \right)^d\cap \mathbb{Z}^d $ and $x\in F(w)$, we say $x$ is $m$-locally good if for any $Q\in \mathcal{Q}_x^{\mathrm{F}}$, the following events occur:
			\begin{enumerate}
				\item $V_x\le 2\log^6(r_m)$.

%
%
%
				
				\item Every tuple $\mathbf{\Phi}_{x,z}$ is $(x,m,9)$-locally nice. I.e., for any $z\in \partial B_x(s_1)$, $$\Delta_{x,m}^{\mathrm{loc}}(\mathbf{\Phi}_{x,z})\le r_{m}^4\log^9(r_m).$$
				
			\end{enumerate}

			\item  We say $x$ is $m$-locally bad if it is not $m$-locally good.

			\item For convenience, we also say a point $x$ is $0$-qualified (resp. $0$-unqualified) if $x$ is $m$-locally good (resp. $m$-locally bad). We remind the readers that the $k$-unqualified points defined in Definition \ref{def_DQ_point} are only valid for $k\ge 1$.

			\item  We denote the number of $m$-locally bad points in $\overline{\mathbf{\Psi}}_n^*$ by $\psi_n^{0\mbox{-}\mathrm{UQ}}$.

		\end{enumerate}
	\end{definition}

	Recall the notation $Q^{\ddagger}$ before Remark \ref{remark7.12}. We denote $V_x^{\ddagger}:=V_x(Q^{\ddagger}) $.


	\begin{lemma}\label{newlemma_locally good}
		For any $x\in \mathbb{Z}^d$, if $x$ is $m$-bad, then $x$ is $m$-locally bad. 
	\end{lemma}
	\begin{proof}
		Arbitrarily fix a configuration of $\mathbf{\Psi}_n$, say $\mathbf{A}=(\mathbf{A}^1,\mathbf{A}^2)$, such that $x$ is $m$-bad. I.e., $\Delta_{x,m}(\mathbf{A})>r_m^4\log^{16}(r_m)$. On the event $\{\mathbf{\Psi}_n=\mathbf{A}\}$, $U_x^{\ddagger}$, $V_x^{\ddagger}$ and $\mathbf{\Phi}_{x,z}^{\ddagger}$ for $z\in U_x^{\ddagger}$ are all deterministic. For any $y\in B_x(r_m)$, if $\{y\xleftrightarrow[]{\cup \widetilde{\mathcal{L}}^{\mathrm{U}}_{\mathbf{A}}}\widehat{\mathbf{\Psi}}_{n}\cap B_x(s_0)\}$ happens, then either $y$ is connected to $\widehat{\mathbf{\Psi}}_{n}\cap B_x(s_0)$ by $\cup \widetilde{\mathcal{L}}_{\mathbf{A}}^{\mathrm{U}}\cdot \mathbbm{1}_{\mathrm{ran}(\widetilde{\ell})\subset \widetilde{B}_x(s_2)}$, or $y$ is connected to $\partial B_x(s_2)$ by $\cup \widetilde{\mathcal{L}}^{\mathrm{U}}_{\mathbf{A}}$. Recall that the former event implies the one on the RHS of (\ref{750}). Thus, we have 
		\begin{equation}\label{7.52}
			\begin{split}
				\Delta_{x,m}(\mathbf{A})
				\le& \mathbb{E}\bigg(\sum_{y\in B_x(r_m)}\mathbbm{1}_{y\xleftrightarrow[]{\cup \widetilde{\mathcal{L}}^{\mathrm{U}}_\mathbf{A}\cdot \mathbbm{1}_{\mathrm{ran}(\widetilde{\ell})\subset \widetilde{B}_x(s_2)}}\widehat{\mathbf{A}}\cap B_x(s_0) }  \bigg)\\
				&+ \sum_{y\in B_x(r_m)}\mathbb{P}\Big[   y \xleftrightarrow[]{\cup \widetilde{\mathcal{L}}^{\mathrm{U}}_{\mathbf{A}}}\partial B_x(s_2)\Big] \\
				\le & \sum_{z\in U_x^{\ddagger}} \Delta_{x,m}^{\mathrm{loc}}\left( \mathbf{\Phi}_{x,z}^{\ddagger}\right) + \sum_{y\in B_x(r_m)}\mathbb{P}\Big[   y \xleftrightarrow[]{\cup \widetilde{\mathcal{L}}^{\mathrm{U}}_{\mathbf{A}}}\partial B_x(s_2)\Big].
			\end{split}
		\end{equation}
	We denote $\hat{U}_x^{\ddagger}:= U_x^{\ddagger}\cap V_x^{\ddagger}$. For any $z\in U_x^{\ddagger}$, if $\Delta_{x,m}^{\mathrm{loc}}\left( \mathbf{\Phi}_{x,z}^{\ddagger}\right)\neq 0$, then we have $\widehat{\mathbf{\Phi}}^{\ddagger}_{x,z}\cap B_x(s_0)\neq \emptyset$, which implies that $\mathbf{\Phi}_{x,z}^{1,\ddagger}\cap B_x(s_0+1)\neq \emptyset$, and thus $z\in \hat{U}_x^{\ddagger}$. Therefore, the first term on the RHS of (\ref{7.52}) is equal to $\sum_{z\in \hat{U}_x^{\ddagger}} \Delta_{x,m}^{\mathrm{loc}}\left( \mathbf{\Phi}_{x,z}^{\ddagger}\right)$. In addition, by (\ref{ineq_onearm}) and $|B_x(r_m)|\le Cr_m^d$, the second term on the RHS of (\ref{7.52}) is bounded from above by $Cr_m^ds_2^{-\frac{1}{2}}<r_m^{-30d^3}$. In conclusion, 
		\begin{equation}\label{754ineq}
			\Delta_{x,m}(\mathbf{A}) \le \sum_{z\in \hat{U}_x^{\ddagger}} \Delta_{x,m}^{\mathrm{loc}}(\mathbf{\Phi}_{x,z}^{\ddagger})+ r_m^{-30d^3}. 
		\end{equation}

	We conclude this lemma by proving its contrapositive statement as follows. Assume that $x$ is $m$-locally good. Then one has $|\hat{U}_x^{\ddagger}|\le |V_x^{\ddagger}| \le 2\log^6(r_m)$ and $\Delta_{x,m}^{\mathrm{loc}}(\mathbf{\Phi}_{x,z}^{\ddagger})\le r_m^4\log^{9}(r_m)$ for all $z\in \hat{U}_x^{\ddagger}$. Thus, by (\ref{754ineq}), we obtain that $x$ is $m$-good since
		\begin{equation*}
			\Delta_{x,m}(\mathbf{A})\le 2\log^6(r_m)\cdot r_m^4\log^{9}(r_m)< r_m^4\log^{16}(r_m).    \qedhere
		\end{equation*}
\end{proof}

	Next, we show the inheritability of $0$-qualified points. I.e., conditioned on the event that $x$ is $1$-qualified, we have $x$ is also $m$-locally good (i.e., $0$-qualified) with a uniformly high probability. Recall that the inheritability of $k$-qualified points ($k\ge 1$) has been proved in Lemma \ref{new_lemma_6.6}.

	We first record a technical lemma, where the bound is suboptimal but suffices
		for our purpose. The proof can be carried out in the same way as \cite[Lemma 1.1]{kozma2011arm}, so we just omit it. 
	

	\begin{lemma}\label{lemma_66}
		There exist $c(d),C(d)>0$ such that for any $R>1$ and any $y\in \partial B(R-1)$, 
		\begin{equation}
		\mathbb{P}\left( \tau_{y} <\tau_{\partial B(R)}\right)  \ge ce^{-C\log^2(R)}. 
		\end{equation}	
		As a direct consequence, for any $z\in \partial \hat{B}(R)$,
		\begin{equation}
			\mathbb{P}\left( \tau_{\partial B(R)}=\tau_z\right)  \ge ce^{-C\log^2(R)}. 
		\end{equation}
	\end{lemma}
	


	The following lemma presents the inheritability of $0$-qualified points. Recall the notations $\Omega_{x,j}$ and $\mathbb{P}(\cdot \mid \omega_{x,j})$ in the paragraphs before Remark \ref{remark_7.11}.

	\begin{lemma}\label{lemma_6.9}
		For any $d>6$, there exist $C(d)>0,c(d)>0$ such that for any $m\in [1,m_0-1]$, $w\in \left[-D_N,D_N \right)^d\cap \mathbb{Z}^d$, $x\in F(w)$, and any configuration $\omega_{x,1}\in \Omega_{x,1}$ such that $x$ is $1$-qualified with respect to $\omega_{x,1}$, we have 
		\begin{equation}
			\mathbb{P}\left( x\ \text{is}\ m\text{-locally bad} \mid \omega_{x,1}  \right) \le Ce^{-c\log^{7}(r_m)}. 
		\end{equation}
	\end{lemma}

	\begin{proof}

		Recall the notations $\widetilde{\eta}^{\mathrm{F}}_{x,i}$, $q_x^{\mathrm{F}}$, $\mathcal{Q}_x^{\mathrm{F}}$, $\mathfrak{L}_{x}^{\mathrm{F}}$, $\mathfrak{L}_{x}^{\mathrm{inv}}$ and $\mathbf{\Phi}_{x,z}$ below the first paragraph of Section \ref{section_locally_good}. Also recall $V_x$ in the sentence before Definition \ref{def_local_good_regular_points}.

		Since $x$ is $1$-qualified, we have $q_{x}^{\text{F}}\le \log^{6}(s_1)$, which implies $|\mathcal{Q}_x^{\mathrm{F}}|\le 2^{\log^{6}(s_1)}$. We denote the starting point and the ending point of $\widetilde{\eta}^{\mathrm{F}}_{i,x}$ by $y_i$ and $z_i$ respectively, where $y_i$ and $z_i$ are deterministic given $\omega_{x,1}$. For any $Q\in \mathcal{Q}_x^{\mathrm{F}}$, let $\mathsf{B}_{1}^Q:=\{V_x(Q)>2\log^{6}(s_1)\}$ and $\mathsf{B}_{2,z}^Q:=  \{\mathbf{\Phi}_{x,z}(Q)\ \text{is not}\ (x,m,9)\text{-nice} \}$ for $z\in \partial B_x(s_1)$. By Definition \ref{def_local_good_regular_points}, if $x$ is $m$-locally bad, then there exists $Q\in \mathcal{Q}_x^{\mathrm{F}}$ such that either $\mathsf{B}_1^Q$ happens, or $\mathsf{B}_{2,z}^Q$ happens for some $z\in \partial B_x(s_1)$.


		On the event $\mathsf{B}_1^Q$, since each $\widetilde{\eta}^{\mathrm{F}}_{x,i}$ can be contained in at most one cluster of the form $\mathbf{\Phi}_{x,z}^1$, the number of clusters $\mathbf{\Phi}_{x,z}^1$ that intersect $B_x(s_0+1)$ and do not contain any forward crossing path $\widetilde{\eta}^{\mathrm{F}}_{x,i}$ is at least $2\log^{6}(s_1)-\log^{6}(s_1)=\log^{6}(s_1)$. Since these clusters do not share a common glued loop (we excluded clusters with forward crossing paths exactly to achieve this property), their existence ensures that there are at least $\log^{6}(s_1)$ disjoint collections of glued loops certifying $\{B(s_0+1)\xleftrightarrow[]{} \partial B(s_1)\} $. Thus, by the BKR inequality and (\ref{ineq_onearm}), we have (recalling $s_0=r_m^{4d}$ and $s_1=r_m^{16d^2}$)
		\begin{equation}\label{7.66}
			\begin{split}
				\mathbb{P}\left(  \mathsf{B}_1^Q\big|\omega_{x,1} \right)  \le &\left\lbrace \mathbb{P}\left[B(s_0+1)\xleftrightarrow[]{} \partial B(s_1) \right] \right\rbrace^{\log^{6}(s_1)}\\
				\le &  \Big(  Cr_{m}^{4d(d-1)}\cdot s_1^{-\frac{1}{2}} \Big)  ^{\log^{6}(s_1)}
				\le Ce^{-c\log^{7}(r_m)}. 
			\end{split}
		\end{equation}

		Now let us focus on $\mathsf{B}_{2,z}^Q$. Similar to Item (2) in Remark \ref{remark_51}, we know that given $\{\mathbf{\Phi}_{x,z}(Q)=\mathbf{D}\}$, the conditional distribution of $\widetilde{\mathcal{L}}^{\text{LU}}_{x,z}$ is the same as $\widetilde{\mathcal{L}}_{x,\mathbf{D}}^{\mathrm{LU}}$ without conditioning. Moreover, $\widetilde{\mathcal{L}}^{\text{LU}}_{x,z}$ is independent of the conditioning $\omega_{x,1}$ since all loops in $\widetilde{\mathcal{L}}^{\text{LU}}_{x,z}$ are contained in $\widetilde{B}_x(s_{k+2})$. Thus, for any $\mathbf{D}$ such that $\{\mathbf{\Phi}_{x,z}(Q)=\mathbf{D}\}\subset \mathsf{B}_{2,z}^Q$, we have 
		\begin{equation}\label{ineq_759}
			\begin{split}
				&\mathbb{E}\bigg(  \sum_{y\in B_x(r_m)}\mathbbm{1}_{y\xleftrightarrow[]{\cup \widetilde{\mathcal{L}}^{\text{LU}}_{x,z}}\widehat{\mathbf{\Phi}}_{x,z}\cap B_x(s_0) } \Big|   \mathbf{\Phi}_{x,z}(Q)=\mathbf{D},\omega_{x,1} \bigg)\\
				=&\Delta_{x,m}^{\mathrm{loc}}(\mathbf{D}) \ge   r_{m}^4\log^{9}(r_m). 
			\end{split}
		\end{equation}	
		Recall that for any random variable $Y$ with $0\le Y\le a$ a.s. and $\mathbb{E}(Y)\ge b$, one has $\mathbb{P}(Y\ge b')\ge  (b-b')/a$ for all $b'\in (0,b)$. Thus, by taking $a=|B_x(r_m)|\le Cr_m^d$, $b=r_{m}^4\log^{9}(r_m)$ and $b'=\tfrac{1}{2}r_{m}^4\log^{9}(r_m)$, we have 
		\begin{equation}\label{new_6.69.5}
			\begin{split}
				&\mathbb{P}\bigg(  \sum_{y\in B_x(r_m)}\mathbbm{1}_{y\xleftrightarrow[]{\widetilde{\mathcal{L}}^{\text{LU}}_{x,z}}\widehat{\mathbf{\Phi}}_{x,z}\cap B_x(s_0) }\ge \tfrac{1}{2}r_{m}^4\log^{9}(r_m) \Big| \mathbf{\Phi}_{x,z}(Q)=\mathbf{D},\omega_{x,1}\bigg)\\
				\ge & cr_{m}^{4-d}\log^{9}(r_m). 
			\end{split}
		\end{equation}
		Recall that $\mathbf{C}(y)$ is the cluster of $\cup \widetilde{\mathcal{L}}_{1/2}$ containing $y$. Note that all the points $y\in B_x(r_m)$ that are connected to $\widehat{\mathbf{\Phi}}_{x,z}\cap B_x(s_0)$ are connected to each other. Therefore, by taking integral over the event $\mathsf{B}_{2,z}^{Q}$ conditioning on $\omega_{x,1}$ for both sides of (\ref{new_6.69.5}), we have 
		\begin{equation}\label{ineq_new_761}
			\begin{split}
				& \mathbb{P}\bigg( \max_{y\in B_x(r_{m})} \left|\mathbf{C}(y) \cap B_x(r_m) \right|>\tfrac{1}{2}r_{m}^4\log^{9}(r_m) \Big| \omega_{x,1} \bigg) \\
				\ge & 	 cr_{m}^{4-d}\log^{9}(r_m)\cdot \mathbb{P}\Big( \mathsf{B}_{2,z}^{Q}\big|\omega_{x,1}\Big). 
			\end{split}
		\end{equation}
		

		Now let us control the LHS of (\ref{ineq_new_761}). Recall $L(A_1,A_2)(\widetilde{\ell})$, $\kappa(\widetilde{\ell};A_1,A_2)$, and $\widetilde{\tau}_{i}$ for $1\le i\le 2\kappa(\widetilde{\ell})$ in Section \ref{subsection_decomposition}. We denote by $\mathfrak{L}(\omega_{x,1})$ the collection of loops $\widetilde{\ell}$ with $\kappa(\widetilde{\ell};B_x(s_0),\partial B_x(s_1))=q_x^{\mathrm{F}}$ such that there exists $\widetilde{\varrho}\in L(B_x(s_0),\partial B_x(s_1))(\widetilde{\ell})$ satisfying $\widetilde{\varrho}(\widetilde{\tau}_{2i})= y_i$ and $\widetilde{\varrho}(\widetilde{\tau}_{2i+1})= z_i$ for $1\le i\le q_x^{\mathrm{F}}$. In fact, for any $\widetilde{\ell}\in \mathfrak{L}(\omega_{x,1})$, the multiplicity (recalling $J$ in (\ref{fina_2.5})) of its projection on $\mathbb{Z}^d$ is upper-bounded by the number of crossings $\kappa=q_x^{\mathrm{F}}$, and therefore is at most $\log^6(s_1)$. Thus, by (\ref{fina_2.5}) and the relation between loops on $\mathbb{Z}^d$ and $\widetilde{\mathbb{Z}}^d$ mentioned in Section \ref{subsection_continuous}, we have 
		\begin{equation}\label{fin_750}
			\widetilde{\mu}\left[ \mathfrak{L}(\omega_{x,1})\right] \ge \log^{-6}(s_1) \prod_{i=1}^{q_x^{\mathrm{F}}} \mathbb{P}_{y_i}\left( \tau_{\partial B_x(s_2)}= \tau_{z_i}  \right)  \cdot \mathbb{P}_{z_i}\left( \tau_{ \partial B_x(s_1)}=  \tau_{y_i}<\infty \right). 
		\end{equation} 
			For the first part of the product on the RHS of (\ref{fin_750}), by the strong Markov property, we have 
		\begin{equation}\label{fina_7.52}
			\begin{split}
				&\mathbb{P}_{y_i}\left( \tau_{\partial B_x(s_2)}= \tau_{z_i}  \right)\\
				\ge&  \mathbb{P}_{y_i}\left( \tau_{\bm{0}}<\tau_{\partial B_x(s_2)}  \right)\cdot \mathbb{P}_{\bm{0}}\left(\tau_{\partial B_x(s_2)}= \tau_{z_i}  \right)\\
				=&\mathbb{P}_{\bm{0}}\left( \tau_{y_i}<\tau_{\partial B_x(s_2)}  \right)\cdot \mathbb{P}_{\bm{0}}\left(\tau_{\partial B_x(s_2)}= \tau_{z_i}  \right)  \ \ (\text{by reversing the random walk})       \\
				\ge &ce^{-C\log^2(s_2)} \ \ \ \ \  \ \ \ \ \  \ \ \ \ \  \ \ \ \ \  \ \ \ \ \  \ \ \ \  \ \ \ \ \ \ \ \ (\text{by Lemma}\ \ref{lemma_66}). 
			\end{split}
		\end{equation}
			For the second part, note that we can find $v_i\in \partial B_x(s_2)$ such that $y_i\in \hat{B}_{v_i}(s_2-s_1)$ for each $y_i\in B_x(s_1)$. Therefore, by the strong Markov property, we have 
		\begin{equation}\label{ineq.7.62}
			\begin{split}
				& \mathbb{P}_{z_i}\left( \tau_{ \partial B_x(s_1)}=  \tau_{y_i}<\infty \right)\\
				\ge &  \mathbb{P}_{z_i} \left( \tau_{\partial B_{v_i}(0.5s_2)}< \tau_{B_{x}(s_1)}\right) \cdot \min_{v\in \partial B_{v_i}(0.5s_2)}  \mathbb{P}_{v}\left( \tau_{v_i}<\tau_{\partial B_{v_i}(s_2-s_1)} \right) \\
				& \cdot  \mathbb{P}_{v_i}\left( \tau_{\partial B_{v_i}(s_2-s_1)}=\tau_{y_i} \right)\\
				\ge & c e^{-C\log^2(s_2)}\ \ \ \ \ \  \ \ \  \ \ \  (\text{by the invariance principle and Lemma}\ \ref{lemma_66}). 
			\end{split}
		\end{equation} 
		Combining (\ref{fin_750}), (\ref{fina_7.52}), (\ref{ineq.7.62}) and the fact that $q_x^{\mathrm{F}}\le \log^6(s_1)$, we get 
		\begin{equation}\label{new_ineq6.64}
			\begin{split}
				\widetilde{\mu}\left[ \mathfrak{L}(\omega_{x,1})\right] \ge c e^{-C\log^8(s_2)}. 
			\end{split}
		\end{equation}
	Let $\mathfrak{L}^{c}(\omega_{x,1})$ be the collection of loops in the complement of $\mathfrak{L}(\omega_{x,1})$ that cross $B_x(s_{k+2})\setminus B_x(s_{k+1})$. By (\ref{ineq_2.5}), we have 
		\begin{equation}\label{fin_7.54}
			\widetilde{\mu}\left[ \mathfrak{L}^c(\omega_{x,1})\right]\le Cs_1^{d-2}s_2^{2-d}. 
		\end{equation}
		We denote by $\mathsf{A}_{\dagger}$ the event that in $\widetilde{\mathcal{L}}_{1/2}$ there is exactly one loop in $\mathfrak{L}(\omega_{x,1})$ and there is no loop in $\mathfrak{L}^c(\omega_{x,1})$. By (\ref{new_ineq6.64}) and (\ref{fin_7.54}), 
		\begin{equation}\label{fin_7.55}
			\mathbb{P}\left(\mathsf{A}_{\dagger}\right)= \tfrac{1}{2}\widetilde{\mu}\left[ \mathfrak{L}(\omega_{x,1})\right] \cdot e^{-\frac{1}{2}\widetilde{\mu}\left[ \mathfrak{L}(\omega_{x,1})\right]}\cdot e^{-\frac{1}{2}\widetilde{\mu}\left[ \mathfrak{L}^c(\omega_{x,1})\right]}\ge c e^{-C\log^8(s_2)}. 
		\end{equation}
		Since $\mathsf{A}_{\dagger}$ implies the conditioning $\omega_{x,1}$, by Lemma \ref{lemma_large_div} and (\ref{fin_7.55}), we have 
			\begin{equation*}\label{new_ineq_6.71}
			\begin{split}
				&\mathbb{P}\Big( \max_{y\in B_x(r_{m})} \left|\mathbf{C}(y) \cap B_x(r_m) \right|>\tfrac{1}{2}r_{m}^4\log^{9}(r_m) \big| \omega_{x,1}  \Big)\\
				\le & \mathbb{P}\Big( \max_{y\in B_x(r_{m})} \left|\mathbf{C}(y) \cap B_x(r_m) \right|>\tfrac{1}{2}r_{m}^4\log^{9}(r_m) \Big) \left[	\mathbb{P}\left(\mathsf{A}_{\dagger}\right) \right]^{-1} 
				\le Ce^{-c\log^{9}(r_m)}.
			\end{split}
		\end{equation*}
	Combined with (\ref{ineq_new_761}), this gives us 
		\begin{equation}\label{newineq_6.77}
		\mathbb{P}\left(\mathsf{B}_{2,z}^{Q}\big| \omega_{x,1}\right)  \le Ce^{-c\log^{9}(r_m)}.  
	\end{equation}

			
%


%
%
	
		Finally, we conclude the desired bound as follows:  
		\begin{equation*}
			\begin{split}
				&\mathbb{P}\left( x\ \text{is}\ m\text{-locally bad} \mid \omega_{x,1}  \right)\\ \le & \sum_{Q\in \mathcal{Q}_x^{\mathrm{F}}}\mathbb{P}\left(  \mathsf{B}_1^Q\big|\omega_{x,1} \right) +  \sum_{Q\in \mathcal{Q}_x^{\mathrm{F}}} \sum_{z\in \partial B_x(s_1)} \mathbb{P}\left(\mathsf{B}_{2,z}^{Q}\big| \omega_{x,1}\right)\\
				\le &C|\mathcal{Q}_x^{\mathrm{F}}|\cdot \left(e^{-c\log^{7}(r_m)}+s_1^{d-1}e^{-c\log^{9}(r_m)} \right) \  \ \ (\text{by}\ (\ref{7.66})\ \text{and}\ (\ref{newineq_6.77}))\\
				\le &  Ce^{-c\log^{7}(r_m)}\ \ \ \ \ \ \ \ \ \ \ \ \ \ \ \ \ \ \ \ \ \ \ \ \ \ \ \ \ \ \ \ \ \ \ \ \ (\text{by}\ |\mathcal{Q}_x^{\mathrm{F}}|\le 2^{\log^{6}(s_1)}).   \qedhere
			\end{split} 
		\end{equation*}  
	\end{proof}

	\subsection{Exploration processes} \label{section_exploration}
	
	In this subsection, we introduce the exploration process $\mathcal{T}_n$ which completes a construction of $\mathbf{\Psi}_n^1$ (recall Definition \ref{definition_Psi}) upon termination. As an additional feature, during the process $\mathcal{T}_n$, we will keep track of the ordering for appearances of loops and we will record some statistics, and these will be very useful for the proof of Lemma \ref{new_lemma_6.3}.
	
		
	Recall that we already fixed $m\in [1,m_0]$ and $w\in \left[-D_N,D_N \right)^d\cap \mathbb{Z}^d$. Now we also fix an arbitrary integer $k\in [0,k_0-1]$. Recall the definitions of $F(w)$ and $k_0$ in (\ref{candidate set}) and (\ref{def_k0}) respectively. We divide $F(w)$ into $F^{\mathrm{I}}(w):= \{x\in F(w): \widetilde{B}_x(s_{k+2})\cap \widetilde{B}(n)=\emptyset\}$ and $F^{\mathrm{II}}(w):=F(w)\setminus F^{\mathrm{I}}(w)$.

%
%


	Unless otherwise specified, in the construction of $\mathcal{T}_n$, when we refer to a forward or backward crossing path, we always assume $A_1=\cup_{x\in F(w)} B_x(s_{k+1})$ and $A_2=\cup_{x\in F(w)} \partial B_x(s_{k+2})$. Note that $A_1$ and $A_2$ are disjoint since $|x_1-x_2|\ge 2D_N>2s_{s_{k+2}}$ for any distinct points $x_1,x_2\in F(w)$ and any $k\in [0,k_0-1]$. As a result, for any $x\in F(w)$, the forward and backward crossings in the annulus $B_x(s_{k+2})\setminus B_x(s_{k+1})$ are the same as with respect to $A_1 = B_x (s_{k+1})$ and $A_2 = \partial B_x(s_{k+2})$. Recall the definition of involved loops in Definition \ref{definition_Psi}. We say a crossing path is involved if it is included by some involved loop.

	%


	 The exploration process $\mathcal{T}_n$ is described as follows.  
%
%
%
%


	\textbf{Step 0:} We define the collection $\mathfrak{L}^{w}$ by  $$\mathfrak{L}^{w}:=\Big\{\widetilde{\ell}\in \widetilde{\mathcal{L}}_{1/2}:\widetilde{\ell}\ \text{intersects}\ B_x(s_{k+1})\ \text{and}\ \partial B_x(s_{k+2})\ \text{for some}\ x\in  F(w)  \Big\}.$$
	We sample every backward crossing path $\widetilde{\eta}^{\mathrm{B}}$ of every loop in $\mathfrak{L}^{w}$ except its Brownian excursions at $\cup_{x\in F(w)}\partial \hat{B}_x(s_{k+2})$ (i.e. we reserve the randomness of these Brownian excursions and only sample the remaining part of $\widetilde{\eta}^{\mathrm{B}}$). For each forward crossing path $\widetilde{\eta}^{\mathrm{F}}$ of a loop in $\mathfrak{L}^{w}$, its starting point and ending point are now fixed. Thus, now we can determine the collection of $(k+1)$-unqualified points in $F(w)$ and denote it by $\mathcal{D}$.

	We also sample all forward crossing paths $\widetilde{\eta}^{\mathrm{F}}$ contained in $\cup_{x\in F^{\mathrm{II}}(w)}\widetilde{B}_{x}(s_{k+2})$. Note that a loop $\widetilde{\ell}\in \mathfrak{L}^{w}$ is involved if and only if it contains a forward or backward crossing path that intersects $\widetilde{B}(n)$, and that the Brownian excursions of a fundamental loop $\widetilde{\ell}$ do not make any difference on whether $\widetilde{\ell}$ intersects $\widetilde{B}(n)$. In addition, every $\widetilde{\eta}^{\mathrm{F}}$ contained in $\cup_{x\in F^{\mathrm{I}}(w)}\widetilde{B}_{x}(s_{k+2})$ cannot intersect $\widetilde{B}(n)$ (since $\widetilde{B}_{x}(s_{k+2})\cap \widetilde{B}(n)=\emptyset$ for all $x\in F^{\mathrm{I}}(w)$). Thus, now we can determine which loop in $\mathfrak{L}^{w}$ is involved. Let $\mathcal{E}$ be the collection of all involved crossing paths sampled up until now.

	Since $\widetilde{B}_{x}(s_{k+2})\cap \widetilde{B}(n)=\emptyset$ for all $x\in F^{\mathrm{I}}(w)$, every loop contained in $\widetilde{B}_{x}(s_{k+2})$ is not involved. In light of this, we say a point $x\in F^{\mathrm{I}}(w)$ is inactive if there is no involved forward crossing path in $\widetilde{B}_{x}(s_{k+2})$. We also say the remaining points in $F(w)$ are active. Especially, all points in $F^{\mathrm{II}}(w)$ are active. We denote by $\mathcal{A}$ the collection of all active points. Note that $\mathcal{A}$ is already determined. Then we sample the Brownian excursions of all backward crossing paths in $\mathcal{E}$ at $\cup_{F(w)\setminus \mathcal{A}} \partial \hat{B}_x(s_{k+2})$. See Figure \ref{pic3} for an illustration of Step $0$.

		\begin{figure}[h]
		\centering
		\includegraphics[width=13cm]{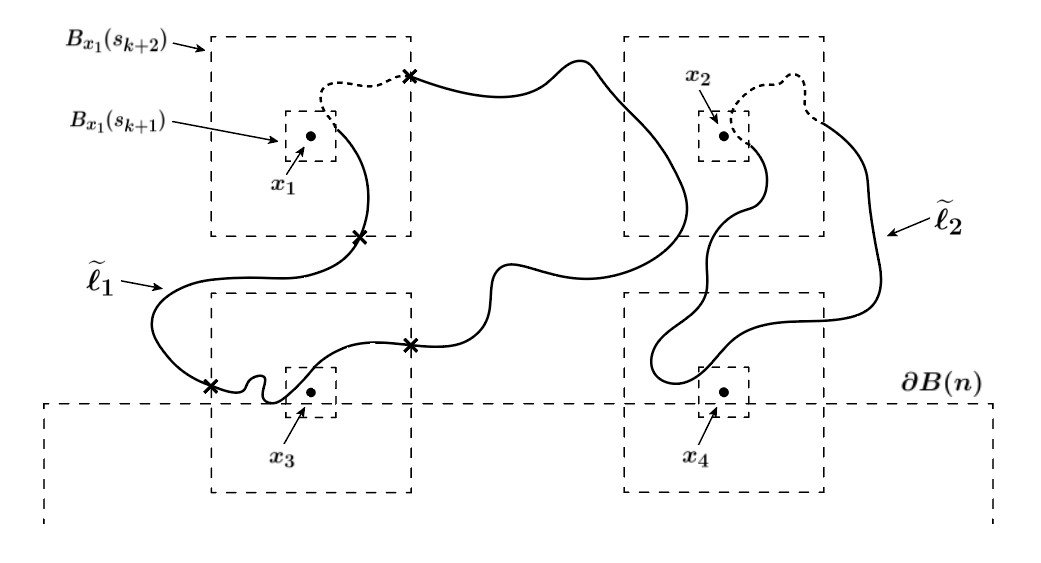}
		\caption{An illustration for Step $0$. In this example, $\widetilde{\ell}_1$ is involved but $\widetilde{\ell}_2$ is not. $x_1$, $x_3$ and $x_4$ are active, while $x_2$ is inactive. Moreover, the Brownian excursions of $\widetilde{\ell}_1$ at the positions highlighted by ``$\boldsymbol{\times}$'' are not sampled. Furthermore, the solid curves are sampled, while dashed ones are not.  }
		\label{pic3}
		\end{figure}

The following statistics for $\mathcal{T}_n$ will be recorded. We will provide their initial values, and then describe how they are updated as the construction of $\mathcal{T}_n$ proceeds:

\begin{itemize}
	
	\item $\mathbf{C}_p$ with $\mathbf{C}_0=\{\bm{0}\}\subset \widetilde{\mathbb{Z}}^d$: the existing cluster. I.e., the cluster containing $\bm{0}$ and composed of the collection of all involved loops (or their crossing paths) which have been sampled.

	\textbf{P.S.}: The subscript $p$ of $\mathbf{C}_p$ indicates that $\mathbf{C}_p$ is the existing cluster after Step $p$. The subscript $p$ in notations for other statistics also has the same meaning. As we will show later, there is an intermediate cluster $\mathbf{C}_p^{\dagger}$ in each step. We also call this intermediate cluster ``existing cluster'' although it will only be used in the construction but not in the analysis later.

	While two crossing paths of a loop may not be connected to each other by themselves, they are connected in the loop cluster (since they are from the same loop). Thus, when referring to the cluster including paths in $\mathcal{E}$, we always consider all crossing paths from the same loop as connected.


	\item $F_p$ with $F_0=\mathcal{A}$: the collection of all unvisited active points $x\in F(w)$.

	\textbf{P.S.}: We hereby introduce the definition of a visited point $x$. For any $p\in \mathbb{N}^+$ and $x\in   F_{p-1}$ (i.e. $x$ is active and is not visited up to Step $p-1$), we say $x$ is visited in Step $p$ if the aforementioned existing cluster $\mathbf{C}_p^{\dagger}$ (which grows as $\mathcal{T}_n$ progresses) intersects $\partial \hat{B}_x(s_{k+2})$. During the construction of $\mathcal{T}_n$, we say $x$ is unvisited if it is not visited yet.

	 Intuitively, ``$x$ is visited in Step $p$'' indicates that with a positive probability $x$ is connected to $\mathbf{C}_p^{\dagger}$ by the involved loops and forward crossing paths in $\widetilde{B}_x(s_{k+2})$ (which justifies our choice of the word ``visited''). See Lemma \ref{lemma_721} for a precise statement on a uniform lower bound on this probability. Note that the reason we maintain the randomness of the Brownian excursions at $\cup_{x\in \mathcal{A}}\partial \hat{B}_x(s_{k+2})$ in Step $0$ (also in some subsequent steps) is to ensure this uniform lower bound.

	\item $N_{p}^{\text{vis}}$ with $N_{0}^{\text{vis}}=0$: the number of visited points.

	\item $N_{p}^{k\text{-UQ}}$ with $N_{0}^{k\text{-UQ}}=0$: the number of visited, $k$-unqualified points.

	\item $N_{p}^{\text{vis-}\mathcal{D}}$ with $N_{0}^{\text{vis-}\mathcal{D}}=0$: the number of visited points in $\mathcal{D}$.

	\item $N_{p}^{\text{con}}$ with $N_{0}^{\text{con}}=0$: the number of visited points $x\in \mathcal{A}$ such that $x$ is connected to the existing cluster $\mathbf{C}_p^{\dagger}$ by $\cup\mathfrak{I}_{p}^x$, where $\mathfrak{J}_{p}^x$ is the collection of the involved loops and involved forward crossing paths in $\widetilde{B}_x(s_{k+2})$ and the Brownian excursions of sampled loops at $\partial \hat{B}_x(s_{k+2})$.

%

	
	%
	%

	%
	
	%
	%
	%

	\item $N_{p}^{\text{con-}\mathcal{D}}$ with $N_{0}^{\text{con-}\mathcal{D}}=0$: the number of points counted by both $N_{p}^{\text{vis-}\mathcal{D}}$ and $N_{p}^{\text{con}}$.

\end{itemize}

\textbf{Step $\boldsymbol{p}$ ($p \ge 1$)}: Suppose that we have completed the $(p-1)$-th step of $\mathcal{T}_n$ and as a result have obtained $\mathbf{C}_{p-1}$, $F_{p-1}$, $N_{p-1}^{\text{vis}}$, $N_{p-1}^{k\text{-UQ}}$, $N_{p-1}^{\text{vis-}\mathcal{D}}$, $N_{p-1}^{\text{con}}$ and $N_{p-1}^{\text{con-}\mathcal{D}}$. Now we describe the $p$-th step as follows. 

%

Firstly, we sample all unsampled fundamental loops $\widetilde{\ell}$ in the following collection except their Brownian excursions at $\cup_{x\in \mathcal{A}} \partial \hat{B}_x(s_{k+2})$: 
$$
\left\lbrace \widetilde{\ell}\notin \mathfrak{L}^w: \mathrm{ran}(\widetilde{\ell})\cap \widetilde{B}(n)\neq \emptyset, \mathrm{ran}(\widetilde{\ell})\cap \mathbf{C}_{p-1}\neq \emptyset\ \text{and}\ \forall x\in \mathcal{A}, \mathrm{ran}(\widetilde{\ell}) \not\subset \widetilde{B}_x(s_{k+2}) \right\rbrace. 
$$

\textbf{P.S.}: We say a fundamental loop is unsampled if none of its edges has been sampled.



%
%

%
%

Secondly, we sample all unsampled glued point loops $\gamma_{y}^{\mathrm{p}}$ with $\mathbf{C}_{p-1}\cap \gamma_{y}^{\mathrm{p}}\neq \emptyset$ for $y\in B(n-1)\setminus \cup_{x\in \mathcal{A}}\hat{B}_{x}(s_{k+2})$. Moreover, for each $y\in \cup_{x\in F_{p-1}}\partial \hat{B}_{x}(s_{k+2})\cap B(n-1)$, we sample whether the glued point loop $\gamma_{y}^{\mathrm{p}}$ or the Brownian excursions of all sampled loops at $y$ intersect $\mathbf{C}_{p-1}$ (but do not sample the whole configuration of $\gamma_{y}^{\mathrm{p}}$ or these Brownian excursions).

Thirdly, for each $y\in \mathbb{Z}^d$ and each of its incident edge $e$ with $I_e\subset \widetilde{B}(n)$ and $I_e\not\subset \cup_{x\in \mathcal{A}} \widetilde{B}_{x}(s_{k+2})$, if $y\in\mathbf{C}_{p-1}$ then we let $v_{e,y}$ be the furthest point in $I_e$ connected to $y$ by $\mathbf{C}_{p-1}\cap I_e$. We sample the cluster of the glued loop $\gamma_{e}^{\mathrm{e}}$ containing $v_{e,y}$.

Let $\mathbf{C}_{p}^{\dagger}$ be the cluster composed of $\mathbf{C}_{p-1}$ and all these sampled loops (or partial loops). (\textbf{P.S.}: For each $y\in \cup_{x\in F_{p-1}}\partial \hat{B}_{x}(s_{k+2})\cap B(n-1)$, if $\gamma_{y}^{\mathrm{p}}$ is sampled to intersect $\mathbf{C}_{p-1}\cap I_e$ for some edge $e$, then we include $I_e$ in $\mathbf{C}_{p}^{\dagger}$. In addition, if the Brownian excursions of all sampled loops at $y$ are sampled to intersect $\mathbf{C}_{p-1}\cap I_e$ for some edge $e$, then we add both $I_e$ and the sampled part of every loop that intersects $y$ to $\mathbf{C}_{p}^{\dagger}$.)

There are two sub-cases for the subsequent construction:

\textbf{Case $\boldsymbol{p}$.1:} If $\mathbf{C}_{p}^{\dagger}$ does not intersect $\cup_{x\in F_{p-1}} \hat{B}_x(s_{k+2})$, then we set $\mathbf{C}_{p}=\mathbf{C}_{p}^{\dagger}$, maintain all other statistics and go to the next step.


\textbf{Case $\boldsymbol{p}$.2:} Otherwise, we enumerate all points $x\in F_{p-1}$ with $\hat{B}_x(s_{k+2}) \cap \mathbf{C}_{p}^{\dagger}\neq \emptyset$ as $\{x_{l}\}_{1\le l\le l_p}$. Then we sample all forward crossing paths and loops contained in every $\widetilde{B}_{x_l}(s_{k+2})$, and sample all Brownian excursions of sampled loops at every $\partial \hat{B}_{x_l}(s_{k+2})$. Let $\mathbf{C}_p$ be the cluster composed of $\mathbf{C}_{p-1}$, and all involved loops and involved crossing paths sampled up until now. If $\mathbf{C}_{p}=\mathbf{C}_{p-1}$, then we maintain all statistics and stop the process. Otherwise, we update the values of our statistics in the following way and then go the the next step:

\begin{itemize}

\item[-] $F_{p}:=F_{p-1} \setminus \{x_1,...,x_{l_p}\}$ and  $N^{\text{vis}}_{p}:= N^{\text{vis}}_{p-1}+l_p$.

\item[-] $N_{p}^{k\text{-UQ}}:=N_{p-1}^{k\text{-UQ}}+\big|\{1\le l\le l_p:x_l\ \text{is}\ k\text{-unqualified}\}\big|$.

\item[-] $N_{p}^{\text{vis-}\mathcal{D}}:=N_{p-1}^{\text{vis-}\mathcal{D}}+\big|\{1\le l\le l_p:x_l\in \mathcal{D} \}\big|$.

\item[-]   $N_{p}^{\text{con}}:=N_{p-1}^{\text{con}}+\big|\{1\le l\le l_p:x_l\ \text{and}\ \mathbf{C}_{p}^{\dagger}\ \text{are connected by}\ \cup \mathfrak{I}_{p}^{x_l}\}\big|$.

\item[-]   $N_{p}^{\text{con-}\mathcal{D}}:=N_{p-1}^{\text{con-}\mathcal{D}}+\big|\{1\le l\le l_p:x_{l}\ \text{is counted by both}\ N_{p}^{\text{vis-}\mathcal{D}}\ \text{and}\ N_{p}^{\text{con}}\}\big|$.

\end{itemize}

This completes the construction of our exploration process. It is easy to see that each process $\mathcal{T}_n$ a.s. stops after a finite number of steps, which is denoted as $p_*$. Let $\mathbf{C}_{*}$, $F_{*}$, $N_{*}^{\text{vis}}$, $N_{*}^{k\text{-UQ}}$, $N_{*}^{\text{vis-}\mathcal{D}}$, $N_{*}^{\text{con}}$ and $N_{*}^{\text{con-}\mathcal{D}}$ be the corresponding statistics of $\mathbf{C}_p$, $F_{p}$, $N_{p}^{\text{vis}}$, $N_{p}^{k\text{-UQ}}$, $N_{p}^{\text{vis-}\mathcal{D}}$, $N_{p}^{\text{con}}$ and $N_{p}^{\text{con-}\mathcal{D}}$ when $\mathcal{T}_n$ stops.

Recall $\overline{\mathbf{\Psi}}_n^*$ in the paragraph after Remark \ref{remark_51}.


\begin{lemma}\label{lemma_718}
	(1) If an involved loop intersects $\mathbf{C}_*$, it is included in $\mathbf{C}_*$.

	\noindent(2) If $\mathbf{C}_*$ intersects an involved loop or involved forward crossing path in $\widetilde{B}_x(s_{k+2})$ for some $x\in F(w)$, then $x$ is visited.

	\noindent(3)  $\mathcal{T}_n$ constructs $\mathbf{\Psi}_n^1$ eventually. I.e., $\mathbf{C}_*=\mathbf{\Psi}_n^1$.

	\noindent(4) For any $x\in \overline{\mathbf{\Psi}}_n^*\cap F(w)$, $x$ must be visited in some step of $\mathcal{T}_n$.
\end{lemma}
\begin{proof}
We prove all these four items one by one. 

 (1) We divide the involved loops into four types and prove Item (1) separately:

\textbf{Type 1} (All involved loops $\widetilde{\ell}\in \widetilde{\mathcal{L}}_{1/2}^{\mathrm{f}}$ with $\widetilde{\ell}\notin \mathfrak{L}^w$ and $ \mathrm{ran}(\widetilde{\ell}) \not\subset \widetilde{B}_x(s_{k+2})$ for all $x\in \mathcal{A}$): If such a loop $\widetilde{\ell}$ intersects $\mathbf{C}_{*}$, then it also intersects the existing cluster in some step of $\mathcal{T}_n$, and thus is sampled and contained in $\mathbf{C}_{*}$.

\textbf{Type 2} (All involved edge loops and point loops that are not contained in $\cup_{x\in \mathcal{A}}\widetilde{B}_{x}(s_{k+2})$): For the same reason as in Type 1, these loops are included by $\mathbf{C}_{*}$.

\textbf{Type 3} (All involved loops $\widetilde{\ell}$ contained in some $\widetilde{B}_x(s_{k+2})$, $x\in \mathcal{A}$): Since $\widetilde{\ell}$ intersects $\mathbf{C}_{*}$, we know that $\widetilde{\ell}$ intersects $\mathbf{C}_{p'}^{\dagger}$ for some $p'\in [0,p_*-1]$. Since $\mathbf{C}_{p'}^{\dagger}$ is connected, we see that $\mathbf{C}_{p'}^{\dagger}$ intersects $\partial \hat{B}_x(s_{k+2})$ and as a result, $x$ is visited. Thus, $\widetilde{\ell}$ is sampled and included in $\mathbf{C}_{*}$.

\textbf{Type 4}	(All involved loops $\widetilde{\ell}\in \mathfrak{L}^w$): We assume that $\mathbf{C}_{*}$ intersects some involved loops $\widetilde{\ell}\in \mathfrak{L}^w$, and we next prove that $\widetilde{\ell}$ is included in $\mathbf{C}_{*}$. Since $\widetilde{\ell}$ is fully decomposed into backward and forward crossing paths, $\mathbf{C}_{*}$ either intersects some backward crossing path $\widetilde{\eta}^{\mathrm{B}}$ or forward crossing path $\widetilde{\eta}^{\mathrm{F}}$ of $\widetilde{\ell}$. We next consider these two (possibly overlapping) subcases separately.

\begin{enumerate}
	\item[(a)] $\mathbf{C}_{*}$ intersects $\widetilde{\eta}^{\mathrm{B}}$: By the construction of $\mathcal{T}_n$, $\widetilde{\eta}^{\mathrm{B}}$ (and also every other backward crossing path of $\widetilde{\ell}$) must be contained in $\mathbf{C}_*$. Therefore, for every $x_{\diamond}\in F(w)$ such that $\widetilde{B}_{x_{\diamond}}(s_{k+2})$ contains some forward crossing path of $\widetilde{\ell}$, $x_{\diamond}$ will be visited, which implies that $\widetilde{\ell}$ must be completely sampled, and thus is included in $\mathbf{C}_{*}$.

	\item[(b)] $\mathbf{C}_{*}$ intersects $\widetilde{\eta}^{\mathrm{F}}$: Suppose that $\widetilde{\eta}^{\mathrm{F}}$ is contained in $\widetilde{B}_{x}(s_{k+2})$ for some $x\in F(w)$. For the same reason as in Type 3, one can show that $x$ is visited, and thus $\widetilde{\eta}^{\mathrm{F}}$ is sampled and contained in $\mathbf{C}_{*}$. With the same argument as in Subcase(a), $\widetilde{\ell}$ is also included in $\mathbf{C}_{*}$.
	
\end{enumerate}
To sum up, we conclude Item (1).

(2) Since there is no involved loop in $\widetilde{B}_x(s_{k+2})$ for any $x\in F(w)\setminus \mathcal{A}$, this follows directly by combining the above analysis for Type 3 and Subcase (b) for Type 4.

(3) It is clear from the definition of $\mathcal{T}_n$ that $\mathbf{C}_*$ is composed of involved loops. Therefore, $\mathbf{C}_*\subset \mathbf{\Psi}_n^1$. If $\mathbf{C}_*\subsetneq \mathbf{\Psi}_n^1$, since $\mathbf{C}_*$ and $\mathbf{\Psi}_n^1$ are both connected subsets, in $\mathbf{\Psi}_n^1$ there exists an involved loop $\widetilde{\ell}$ that intersects $\mathbf{C}_*$ and is not included in $\mathbf{C}_*$. However, by Item (1), such $\widetilde{\ell}$ does not exist. By contradiction, we get Item (3).

(4) For any $x\in \overline{\mathbf{\Psi}}_n^*\cap F(w)$, $\widetilde{B}_x(1)$ must intersect some involved loop $\widetilde{\ell}$ or forward corssing path $\widetilde{\eta}^{\mathrm{F}}$ in $\widetilde{B}_x(s_{k+2})$, which is included in $\mathbf{\Psi}_n^1$($=\mathbf{C}_*$, by Item (3)). By Item (2), the existence of such $\widetilde{\ell}$ or $\widetilde{\eta}^{\mathrm{F}}$ implies that $x$ is visited. Now the proof is complete. \end{proof}

Recall $\zeta_w$ in (\ref{7.18}). For any $k\in \mathbb{N}$, we define
\begin{equation}
\zeta_w^{k\mbox{-}\mathrm{UQ}}:= \big|x\in \overline{\mathbf{\Psi}}_n^*\cap F(w):x\ \text{is}\ k\text{-unqualified}\big|.
\end{equation}


%

\begin{lemma}
Almost surely we have
\begin{equation}\label{equation_617}
	\zeta_{w}^{k\mbox{-}\mathrm{UQ}} \le N_{*}^{k\mbox{-}\mathrm{UQ}}  \le 	N_{*}^{\mathrm{vis}},
\end{equation}
\begin{equation}\label{equation_618}
	N_{*}^{\mathrm{con}}\le \zeta_{w} \le N_{*}^{\mathrm{vis}}, 
\end{equation}
\begin{equation}\label{ineq_7.67}
	N_{*}^{\mathrm{con}\mbox{-}\mathcal{D}}\le 	\zeta_{w}^{(k+1)\mbox{-}\mathrm{UQ}}.
\end{equation}
\end{lemma}
\begin{proof}

\textbf{Proof of (\ref{equation_617}):}  Since every point $x$ counted by $\zeta_{w}^{k\text{-UQ}}$ is in $\overline{\mathbf{\Psi}}_n^*\cap F(w)$, by Item (4) of Lemma \ref{lemma_718}, $x$ is visited in some step of $\mathcal{T}_n$. Thus, $x$ is also counted by $N_{*}^{k\text{-UQ}}$ since it is $k$-unqualified. As a result, we obtain $\zeta_{w}^{k\text{-UQ}} \le N_{*}^{k\text{-UQ}}$. The second inequality of (\ref{equation_617}) is straightforward since every point counted by $N_{*}^{k\text{-UQ}}$ is visited.

\textbf{Proof of (\ref{equation_618}):} Recall that for each $x$ counted by $N_{*}^{\text{con}}$, $x$ is connected to the existing cluster by the involved loops and involved forward crossing paths in $\widetilde{B}_x(s_{k+2})$ and the Brownian excursions of sampled loops at $\partial \hat{B}_x(s_{k+2})$. This implies that $x$ is counted by $\zeta_{w}$, and thus $N_{*}^{\text{con}} \le \zeta_{w}$. In addition, By Item (4) of Lemma \ref{lemma_718}, for every $x$ counted by $\zeta_{w}$ (i.e. $x\in \overline{\mathbf{\Psi}}_n^*\cap F(w)$), $x$ must be visited in some step of $\mathcal{T}_n$, which implies $\zeta_{w} \le N_{*}^{\text{vis}}$.

\textbf{Proof of (\ref{ineq_7.67}):} For the same reason as proving $N_{*}^{\text{con}} \le \zeta_{w}$, the points counted by $N_{*}^{\text{con}\mbox{-}\mathcal{D}}$ are in $\overline{\mathbf{\Psi}}_n^*\cap F(w)$. Thus, all these points are also counted by $\zeta_{w}^{(k+1)\mbox{-}\text{UQ}}$ since they are $(k+1)$-unqualifid. Now we also conclude (\ref{ineq_7.67}).  
\end{proof}


Recall $\mathfrak{I}_{p}^x$ in the definition of $N_{p}^{\mathrm{con}}$, and $\mathbf{C}_{p}^{\dagger}$ in the construction of Step $p$. As promised, in the following lemma we will prove a uniform lower bound for the probability that a visited point $x$ is counted by $N_*^{\mathrm{con}}$. For the sake of fluency in writing, we leave its proof in Section \ref{section_technical_lemma_2}.

Let $\mathfrak{C}_{p}^x$ be the collection of all possible configurations of $\mathcal{T}_n$ up to sampling $\mathbf{C}_{p}^{\dagger}$ such that $x$ is visited in Step $p$. For any $\omega\in \mathfrak{C}_{p}^x$, let $\mathbb{P}\left(\cdot \mid \omega \right)$ be the conditional measure given that the configuration of $\mathcal{T}_n$ up to sampling $\mathbf{C}_{p}^{\dagger}$ is exactly $\omega$.

%
%
%

\begin{lemma}\label{lemma_721}
	There exist $c(d),C(d)>0$ such that for any $x\in F(w)$, $p\in \mathbb{N}^+$ and $\omega\in \mathfrak{C}_{p}^x$, we have 
	\begin{equation}\label{new_ineq_7.89}
		\mathbb{P}\Big( x\xleftrightarrow[]{\mathfrak{I}_{p}^x} \mathbf{C}_{p}^{\dagger} \mid \omega\Big)\ge  ce^{-C\log^{2}(s_{k+2})}. 
	\end{equation}
\end{lemma}

For every $(k+1)$-qualified point $x$ which is counted by $N_{*}^{\text{vis}}$ (note that the number of such $x$ is $N_{*}^{\text{vis}}-N_{*}^{\text{vis-}\mathcal{D}}$), by the inheritability property of $k$-qualified points (see Lemmas \ref{new_lemma_6.6} and \ref{lemma_6.9}), the probability that $x$ is $k$-unqualified is at most $Ce^{-c\log^{4}(s_k)}$. As a result, $N_{*}^{k\text{-UQ}}-N_{*}^{\text{vis-}\mathcal{D}}$ is stochastically dominated by the sum of $N_{*}^{\text{vis}}-N_{*}^{\text{vis-}\mathcal{D}}$ i.i.d Bernoulli random variables with parameter $Ce^{-c\log^{4}(s_k)}$. Thus, by the Hoeffding’s inequality (see e.g. Vershynin \cite[Theorem 2.2.6]{vershynin2018high}), we have: for any $M>e^{\log^{2.9}(s_k)}$, 
\begin{equation}\label{equation_619}
\begin{split}
	&\mathbb{P}\left[N_{*}^{k\text{-UQ}}-N_{*}^{\text{vis-}\mathcal{D}}\ge e^{-\log^{2.8}(s_k)}(N_{*}^{\text{vis}}-N_{*}^{\text{vis-}\mathcal{D}}), N_{*}^{\text{vis}}-N_{*}^{\text{vis-}\mathcal{D}}\ge M \right]\\
	\le& \exp(-M^{0.8}). 
\end{split}
\end{equation}

Similarly, by Lemma \ref{lemma_721}, each time when a point $y$ is counted by $N_*^{\text{vis}}$, at least with $ce^{-C\log^{2}(s_{k+2})}$ probability it is also counted by $N_*^{\text{con}}$. Therefore, $N_*^{\text{con}}$ stochastically dominates the sum of $N_{*}^{\text{vis}}$ i.i.d Bernoulli random variables with parameter $ce^{-C\log^{2}(s_{k+2})}$. Consequently, for any $M>e^{\log^{2.2}(s_k)}$, we have 
\begin{equation}\label{ineq_N*cro}
\mathbb{P}\left( N_{*}^{\text{con}}\le  e^{-\log^{2.1}(s_k)}N_{*}^{\text{vis}},\ N_{*}^{\text{vis}}\ge M \right) \le \exp(-M^{0.8}).
\end{equation}
For the same reason, we also have: for any $M>e^{\log^{2.2}(s_k)}$, 
\begin{equation}\label{new_equation_648}
\mathbb{P}\left( N_{*}^{\text{con-}\mathcal{D}}\le  e^{-\log^{2.1}(s_k)}N_{*}^{\text{vis-}\mathcal{D}},\ N_{*}^{\text{vis-}\mathcal{D}}\ge M \right)  \le \exp(-M^{0.8}).
\end{equation}



\begin{lemma}\label{lemma_7.24}
For any $m\in [1,m_0-1]$, $w\in \left[-D_N,D_N \right)^d\cap\mathbb{Z}^d$ and $k\in [k_0-1]$, 
\begin{equation}\label{ineq_7.96}
	\mathbb{P}\left[  \zeta_{w} \ge L^{1.5},\zeta_{w}^{k\mbox{-}\mathrm{UQ}}\ge \frac{\zeta_{w}}{e^{\log^{2.2}(s_{k})}} , \zeta_{w}^{(k+1)\mbox{-}\mathrm{UQ}}<\frac{\zeta_{w}}{e^{\log^{2.2}(s_{k+1})}} \right] 
	\le e^{-N}. 
\end{equation} 
\end{lemma}
\begin{proof}
%
%
%
%
We denote the events in the LHS of (\ref{equation_619}), (\ref{ineq_N*cro}) and (\ref{new_equation_648}) by $\mathsf{A}^{k\text{-UQ}}(M)$, $\mathsf{A}^{\text{con}}(M)$ and $\mathsf{A}^{\text{con-}\mathcal{D}}(M)$ respectively. We claim some inclusion relations as follows:
\begin{equation}\label{inclusion_649}
	\begin{split}
		\mathsf{A}_1:= &\left\lbrace N_*^{\text{vis}}\ge L^{1.5}, \zeta_{w}^{k\mbox{-}\mathrm{UQ}}\ge e^{-\log^{2.4}(s_{k})}\zeta_{w}+N_*^{\text{vis-}\mathcal{D}},N_*^{\text{vis-}\mathcal{D}}\le \tfrac{1}{2}N_*^{\text{vis}}   \right\rbrace\\
		\subset &\mathsf{A}^{k\text{-UQ}}\left( \tfrac{1}{2}L^{1.5}\right)\cup \mathsf{A}^{\text{con}}(L^{1.5}),
	\end{split}
\end{equation}
\begin{equation}\label{inclusion_650}
	\begin{split}
		\mathsf{A}_2:= &  \left\lbrace N_*^{\text{vis}}\ge L^{1.5}, N_*^{\text{vis-}\mathcal{D}}> \tfrac{1}{2}N_*^{\text{vis}}, \zeta_{w}^{(k+1)\text{-UQ}}< \frac{\zeta_{w}}{e^{\log^{2.2}(s_{k+1})}}  \right\rbrace \\
		\subset & \mathsf{A}^{\text{con-}\mathcal{D}}\left( \tfrac{1}{2}L^{1.5}\right),
	\end{split}
\end{equation}
\begin{equation}\label{inclusion_651}
	\begin{split}
		\mathsf{A}_3:=& \left\lbrace N_*^{\text{vis}}\ge L^{1.5}, \zeta_{w}^{k\mbox{-}\mathrm{UQ}}\ge \frac{\zeta_{w}}{e^{\log^{2.2}(s_{k})}}, \zeta_{w}^{(k+1)\text{-UQ}}<  \frac{\zeta_{w}}{e^{\log^{2.2}(s_{k+1})}}  \right\rbrace \\
		\subset & A^{\text{con-}\mathcal{D}}\left( \tfrac{1}{2}L^{1.5}\right)\cup A^{k\text{-UQ}}\left( \tfrac{1}{2}L^{1.5}\right) \cup A^{\text{con}}(L^{1.5}).
	\end{split}
\end{equation}

We start with confirming (\ref{inclusion_649}). Since $\mathsf{A}_1$ implies $N_*^{\text{vis}}\ge L^{1.5}$ and $N_*^{\text{vis}}-N_*^{\text{vis-}\mathcal{D}}\ge \frac{1}{2}N_*^{\text{vis}}\ge \frac{1}{2}L^{1.5}$ (where $\frac{1}{2}L^{1.5}>e^{\log^{2.9}(s_k)}$ for all $k\le k_0$), on the event $\mathsf{A}_1\cap \big[\mathsf{A}^{k\mbox{-}\mathrm{UQ}}\left(\tfrac{1}{2}L^{1.5}\right)\cup \mathsf{A}^{\text{con}}(L^{1.5})\big]^c$, one has 
\begin{equation}\label{781}
	N_{*}^{k\text{-UQ}}-N_*^{\text{vis-}\mathcal{D}}< e^{-\log^{2.8}(s_k)}(N_*^{\text{vis}}-N_*^{\text{vis-}\mathcal{D}}),
\end{equation}
 \begin{equation}\label{782}
	N_*^{\text{con}}>  e^{-\log^{2.1}(s_k)}N_*^{\text{vis}}.
\end{equation}
Thus, $\mathsf{A}_1\cap \big[\mathsf{A}^{k\mbox{-}\mathrm{UQ}}\left(\frac{1}{2}L^{1.5}\right)\cup \mathsf{A}^{\text{con}}(L^{1.5})\big]^c$ implies 
\begin{equation}\label{ineq_7.74}
	\begin{split}
		\zeta_{w}^{k\mbox{-}\mathrm{UQ}} \le  &(N_{*}^{k\text{-UQ}}-N_*^{\text{vis-}\mathcal{D}})+N_*^{\text{vis-}\mathcal{D}}\ \ \ (\text{by}\ (\ref{equation_617})\ \text{and}\ (\ref{ineq_7.67})) \\
		< 	& e^{-\log^{2.8}(s_k)}(N_*^{\text{vis}}-N_*^{\text{vis-}\mathcal{D}})+N_*^{\text{vis-}\mathcal{D}}\ \ \ \ \  \ \ (\text{by}\ (\ref{781}))\\
		\le &e^{-\log^{2.8}(s_k)}N_*^{\text{vis}}+N_*^{\text{vis-}\mathcal{D}}\\
		< & e^{\log^{2.1}(s_k)}\cdot e^{-\log^{2.8}(s_k)} N_*^{\text{con}}+N_*^{\text{vis-}\mathcal{D}}\ \ \ \ \ \   \ \ (\text{by}\ (\ref{782}))\\
		\le  & e^{\log^{2.1}(s_k)}\cdot e^{-\log^{2.8}(s_k)} \zeta_{w}+N_*^{\text{vis-}\mathcal{D}}\ \ \ \ \ \  \ \ \ \ \ (\text{by}\ (\ref{equation_618})). 
	\end{split}
\end{equation}
In addition, note that $\mathsf{A}_1\subset \{\zeta_{w}^{k\mbox{-}\mathrm{UQ}}\ge e^{-\log^{2.4}(s_{k})}\zeta_{w}+N_*^{\text{vis-}\mathcal{D}} \}$. Therefore, on the event $\mathsf{A}_1\cap \big[\mathsf{A}^{k\mbox{-}\mathrm{UQ}}\left(\frac{1}{2}L^{1.5}\right)\cup \mathsf{A}^{\text{con}}(L^{1.5})\big]^c$, $\zeta_{w}^{k\mbox{-}\mathrm{UQ}} $ satisfies
$$
e^{-\log^{2.4}(s_{k})}\zeta_{w}+N_*^{\text{vis-}\mathcal{D}}	\le 	\zeta_{w}^{k\mbox{-}\mathrm{UQ}} < e^{\log^{2.1}(s_k)}\cdot e^{-\log^{2.8}(s_k)} \zeta_{w}+N_*^{\text{vis-}\mathcal{D}},
$$
which is a contradiction and in turn implies that $\mathsf{A}_1\cap \big[\mathsf{A}^{k\mbox{-}\mathrm{UQ}}\left(\frac{1}{2}L^{1.5}\right)\cup \mathsf{A}^{\text{con}}(L^{1.5})\big]^c=\emptyset$ (and thus verifies (\ref{inclusion_649})).

For (\ref{inclusion_650}), when $\mathsf{A}_2\cap \big[\mathsf{A}^{\text{con-}\mathcal{D}}\left(\frac{1}{2}L^{1.5}\right)\big]^c$ happens, since $N_*^{\text{vis-}\mathcal{D}}> \frac{1}{2}N_*^{\text{vis}}\ge \frac{1}{2}L^{1.5}>e^{\log^{2.2}(s_k)}$ for all $k\le k_0$, we have  
\begin{equation}\label{ineq_7.103}
	\begin{split}
		N_{*}^{\text{con-}\mathcal{D}}> e^{-\log^{2.1}(s_k)}N_{*}^{\text{vis-}\mathcal{D}}.  
	\end{split}
\end{equation}
Thus, the event $\mathsf{A}_2\cap \big[\mathsf{A}^{\text{con-}\mathcal{D}}\left(\frac{1}{2}L^{1.5}\right)\big]^c$ implies 
\begin{equation*}
	\begin{split}
		\zeta_{w}^{(k+1)\mbox{-}\mathrm{UQ}} \ge &N_{*}^{\text{con-}\mathcal{D}} \ \ \ \ \ \ \ \ \ \ \ \  \ \ \ \ \ \ \ \ \ \ \ \ \ \  (\text{by}\ (\ref{ineq_7.67}))\\
		>&e^{-\log^{2.1}(s_k)}N_{*}^{\text{vis-}\mathcal{D}}\ \ \ \ \ \ \ \ \ \ \ \ \ \  (\text{by}\ (\ref{ineq_7.103}))\\
		> &\tfrac{1}{2}e^{-\log^{2.1}(s_k)}N_{*}^{\text{vis}}\ \  \ (\text{by}\ N_*^{\text{vis-}\mathcal{D}}> \tfrac{1}{2}N_*^{\text{vis}})\\
		\ge & \tfrac{1}{2}e^{-\log^{2.1}(s_k)}\zeta_{w}\ \ \  \ \ \ \ \ \ \  \ \ \  \ \ \ \   (\text{by}\ (\ref{equation_618})),  
	\end{split}
\end{equation*}
which is incompatible with $\mathsf{A}_2\subset \{\zeta_{w}^{(k+1)\text{-UQ}}<  e^{-\log^{2.2}(s_{k+1})} \zeta_{w}\}$. Consequently, we obtain the inclusion in (\ref{inclusion_650}).

For (\ref{inclusion_651}), by the definition of $\mathsf{A}_2$ in (\ref{inclusion_650}) one has $$\mathsf{A}_2^c= \{N_*^{\text{vis}}< L^{1.5}\} \cup  \{N_*^{\text{vis-}\mathcal{D}}\le  \tfrac{1}{2}N_*^{\text{vis}}\}\cup \{\zeta_{w}^{(k+1)\text{-UQ}}\ge  e^{-\log^{2.2}(s_{k+1})} \zeta_{w}\}, $$
where $\{N_*^{\text{vis}}< L^{1.5}\}$ and $\{\zeta_{w}^{(k+1)\text{-UQ}}\ge  e^{-\log^{2.2}(s_{k+1})} \zeta_{w}\}$ are both incompatible with $\mathsf{A}_3$. Therefore, $\mathsf{A}_3\cap \mathsf{A}_2^c$ implies $N_*^{\text{vis-}\mathcal{D}}\le \frac{1}{2}N_*^{\text{vis}}$. Furthermore, when $\mathsf{A}_3\cap \mathsf{A}_2^c\cap \big[\mathsf{A}^{\text{con-}\mathcal{D}}\left( \frac{1}{2}L^{1.5}\right)\big]^c$ happens, we have 
\begin{equation*}
	\begin{split}
		&e^{-\log^{2.4}(s_{k})}\zeta_{w}+N_*^{\text{vis-}\mathcal{D}}\\
		< &e^{-\log^{2.4}(s_{k})}\zeta_{w}+e^{\log^{2.1}(s_{k})}N_*^{\text{con-}\mathcal{D}}\ \ \   (\text{since}\ \{N_*^{\text{vis}}\ge  L^{1.5}\}\cap \big[\mathsf{A}^{\text{con-}\mathcal{D}}( \tfrac{1}{2}L^{1.5})\big]^c\ \text{happens}) \\
		\le &e^{-\log^{2.4}(s_{k})}\cdot e^{\log^{2.2}(s_{k})} \zeta_{w}^{k\mbox{-}\mathrm{UQ}} +e^{\log^{2.1}(s_{k})}N_*^{\text{con-}\mathcal{D}}\ \ \ \ \ \ \ \ \ \  \ \ \ \ \ \ \ \ \ \ \ \ \ \  (\text{since}\ \mathsf{A}_3\ \text{happens})\\
		\le & e^{-\log^{2.4}(s_{k})}\cdot e^{\log^{2.2}(s_{k})} \zeta_{w}^{k\mbox{-}\mathrm{UQ}} +e^{\log^{2.1}(s_{k})}\zeta_{w}^{(k+1)\mbox{-}\mathrm{UQ}}\ \  \ \ \ \  \ \ \ \ \ \  \ \ \ \ \  \ \ \ \ \ \ \ \ \ \ \ \ \ \   (\text{by}\  (\ref{ineq_7.67}))\\
		< & \left[e^{-\log^{2.4}(s_{k})}\cdot e^{\log^{2.2}(s_{k})}+e^{\log^{2.1}(s_{k})}\cdot e^{\log^{2.2}(s_{k})-\log^{2.2}(s_{k+1})}  \right] \zeta_{w}^{k\mbox{-}\mathrm{UQ}} \ \  (\text{since}\ \mathsf{A}_3\ \text{happens})\\
		<  & \zeta_{w}^{k\mbox{-}\mathrm{UQ}}. 
	\end{split}
\end{equation*}
In conclusion, we have $\mathsf{A}_3\cap \mathsf{A}_2^c\cap \big[\mathsf{A}^{\text{con-}\mathcal{D}}\left( \frac{1}{2}L^{1.5}\right)\big]^c\subset\mathsf{A}_1$, which implies 
\begin{equation}\label{7.105}
	\mathsf{A}_3 \subset \mathsf{A}_1\cup \mathsf{A}_2 \cup \mathsf{A}^{\text{con-}\mathcal{D}}\left( \tfrac{1}{2}L^{1.5}\right).
\end{equation}
Combining (\ref{inclusion_649}), (\ref{inclusion_650}) and (\ref{7.105}), we obtain (\ref{inclusion_651}).

We denote the event on the LHS of (\ref{ineq_7.96}) by $\mathsf{A}_0$. Since $\zeta_{w}\le N_{*}^{\text{vis}}$ (by (\ref{equation_618})), we have $\mathsf{A}_0\subset \mathsf{A}_3$. Thus, by (\ref{equation_619}), (\ref{ineq_N*cro}), (\ref{new_equation_648}) and (\ref{inclusion_651}), we get
\begin{equation*}
	\mathbb{P}\left( \mathsf{A}_0\right)\le \mathbb{P}\left( \mathsf{A}_3\right)  \le 2e^{-(L^{1.5}/2)^{0.8}}+e^{-(L^{1.5})^{0.8}}\le e^{-N}.   \qedhere
\end{equation*}
\end{proof}

With Lemma \ref{lemma_7.24} in hand, we are ready to prove Lemma \ref{new_lemma_6.3}.
\begin{proof}[Proof of Lemma \ref{new_lemma_6.3}]
	
	By Lemma \ref{newlemma_locally good}, we know that $\zeta_w^{m\mbox{-}\mathrm{bad}}\le \zeta_w^{0\mbox{-}\mathrm{UQ}}$. Therefore, by the inequalities $\frac{L^2}{2^{d+2}K^{20d^2}m^2D_N^d}>L^{1.5}$ and $e^{-\log^{2.2}(s_{0})}< (4 K^{20d^2}m^{2})^{-1}$, we have 
\begin{equation}\label{ineq776}
	\begin{split}
		&\mathbb{P}\left[\zeta_w\ge \frac{L^2}{2^{d+2}K^{20d^2}m^2D_N^d}, \zeta_w^{m\mbox{-}\mathrm{bad}}\ge \frac{\zeta_w}{4 K^{20d^2}m^{2}} \right] \\
		\le &\sum_{k=0}^{k_0-1}\mathbb{P}\left[  \zeta_{w} \ge L^{1.5},\zeta_{w}^{k\mbox{-}\mathrm{UQ}}\ge \frac{\zeta_{w}}{e^{\log^{2.2}(s_{k})}} , \zeta_{w}^{(k+1)\mbox{-}\mathrm{UQ}}<\frac{\zeta_{w}}{e^{\log^{2.2}(s_{k+1})}} \right] \\
		&+\mathbb{P}\left[\exists x\in F(w)\ \text{and} \ k\ge k_0\ \text{such that}\ x\ \text{is}\ k\text{-unqualified} \right]. 
	\end{split}
\end{equation}
By (\ref{new_ineq_625}) and Lemma \ref{lemma_7.24}, the RHS of (\ref{ineq776}) is upper-bounded by $\mathrm{s.p.}(N)$. 	
\end{proof}

Since we have confirmed Lemma \ref{new_lemma_6.3}, the estimate in (\ref{new_ineq_609}) is now proved. As a result, Lemma \ref{lemma_reg} and Corollary \ref{coro_61} follow as well.


\subsection{Proof of technical lemmas}

This subsection includes the proofs of Lemmas \ref{lemma_num_crossing} and \ref{lemma_721}, which are related to some estimates about random walks and loop measure.

\subsubsection{Proof of Lemma \ref{lemma_num_crossing}} \label{section_technical_lemma_1}

	We first focus on the proof of (\ref{ineq_7.31}). By the relation (presented in Section \ref{subsection_BM}) between the Brownian motion $\widetilde{S}_t$ in $\widetilde{\mathbb{Z}}^d$ and the continuous-time simple random walk $S_t$ in $\mathbb{Z}^d$, it suffices to prove that
	\begin{equation}\label{fina_774}
		\mathbb{P}_y\left[\mathfrak{N}= l\big| \tau_{\partial B_x(s_{k+2})}=\tau_z\right] \le  s_{k+1}^{-cl}, 
	\end{equation}	 
where we also denote by $\mathfrak{N}$ the number of times that $S_t$ crosses $B_x(s_{k+1})\setminus B_x(s_k)$ before hitting $\partial B_x(s_{k+2})$. 

	Without loss of generality, we assume $x=\bm{0}$. Similar to $\hat{\tau}_{\cdot}$ (defined before Lemma \ref{fina_lemma2.1}), we define a sequence of stopping times as follows. Let $\bar{\tau}_{\bm{0}}=0$. For any $p\in \mathbb{N}$, we define $\bar{\tau}_{2p+1}:=\min\{\bar{\tau}_{2p}<t< \tau_{\partial B(s_{k+2})}:S_t\in \partial B(s_k) \}$ and $\bar{\tau}_{2p+2}:=\min\{t>\bar{\tau}_{2p+1}:S_t\in \partial B(s_{k+1}) \}$. Note that $\mathfrak{N}$ is the smallest integer such that $\bar{\tau}_{2\mathfrak{N}+1}=\infty$. By the law of total probability, we have 
\begin{equation}\label{ineq_7.33}
	\begin{split}
		&\mathbb{P}_y\left[\mathfrak{N}= l\big| \tau_{\partial B(s_{k+2})}=\tau_z\right]\\
		=& \frac{\sum_{y_*\in \partial B(s_{k+1})} \mathbb{P}_y\left[\mathfrak{N}= l,S_{\bar{\tau}_{2\mathfrak{N}}}=y_* ,\tau_{\partial B(s_{k+2})}=\tau_z\right]}{\sum_{y_*\in \partial B(s_{k+1})}\mathbb{P}_y\left[S_{\bar{\tau}_{2\mathfrak{N}}}=y_* ,\tau_{\partial B(s_{k+2})}=\tau_z\right]}\\
		\le & \frac{\sum_{y_*\in \partial B(s_{k+1})} \mathbb{P}_y\left[\mathfrak{N}= l,S_{\bar{\tau}_{2\mathfrak{N}}}=y_* ,\tau_{\partial B(s_{k+2})}=\tau_z\right]}{\sum_{y_*\in \partial B(s_{k+1})}\mathbb{P}_y\left[\mathfrak{N}= 0,S_{\bar{\tau}_{2\mathfrak{N}}}=y_* ,\tau_{\partial B(s_{k+2})}=\tau_z\right]}\\
		=& \frac{\sum_{y_*\in \partial B(s_{k+1})} \mathbb{P}_y\left[\mathfrak{N}= l,S_{\bar{\tau}_{2\mathfrak{N}}}=y_* ,\tau_{\partial B(s_{k+2})}=\tau_z\right]}{\mathbb{P}_y\left[\tau_{z}=\tau_{\partial B(s_{k+2})}<\tau_{\partial B(s_k)} \right]}. 
	\end{split}
\end{equation}
For each term on the numerator, by the strong Markov property, we have 
\begin{equation}\label{ineq_7.34}
	\begin{split}
		&\mathbb{P}_y\left[\mathfrak{N}= l,S_{\bar{\tau}_{2\mathfrak{N}}}=y_* ,\tau_{\partial B(s_{k+2})}=\tau_z\right]\\
		=&\sum_{y_1\in \partial B(s_{k+1})} \sum_{x_1\in \partial B(s_k)}\mathbb{P}_y\left[S_{\bar{\tau}_{2l-2}}=y_1 \right]\mathbb{P}_{y_1}\left[\tau_{x_1}=\tau_{\partial B(s_k)}<\tau_{\partial B(s_{k+2})} \right]\\
		&	\ \ \ \ \ \ \ \ \ \ \ \  	\ \ \ \ \ \ \ \ \ \ \ \  \ \cdot \mathbb{P}_{x_1}\left[\tau_{\partial B(s_{k+1})}=\tau_{y_*} \right] \mathbb{P}_{y_*}\left[\tau_{z}=\tau_{\partial B(s_{k+2})}<\tau_{\partial B(s_k)} \right]. 
	\end{split}
\end{equation}
In addition, by the Harnack's inequality (see e.g. \cite[Theorem 1.7.2]{lawler2013intersections}), one has 
\begin{equation}\label{ineq_new_737}
	\max_{x_1\in \partial B(s_{k})}\mathbb{P}_{x_1}\left[\tau_{\partial B(s_{k+1})}=\tau_{y_*} \right] \le C \mathbb{P}\left[\tau_{\partial B(s_{k+1})}=\tau_{y_*} \right],
\end{equation}
\begin{equation}\label{ineq_new_738}
	\max_{y_*\in \partial B(s_{k+1})} \mathbb{P}_{y_*}\left[\tau_{z}=\tau_{\partial B(s_{k+2})}<\tau_{\partial B(s_k)} \right]\le C \mathbb{P}_{y}\left[\tau_{z}=\tau_{\partial B(s_{k+2})}<\tau_{\partial B(s_k)} \right]. 
\end{equation}
By (\ref{ineq_7.34}), (\ref{ineq_new_737}) and (\ref{ineq_new_738}), we get
\begin{equation}\label{ineq_new_735}
	\begin{split}
		&\sum_{y_*\in \partial B(s_{k+1})} \mathbb{P}_y\left[\mathfrak{N}= l,S_{\bar{\tau}_{2\mathfrak{N}}}=y_* ,\tau_{\partial B(s_{k+2})}=\tau_z\right]\\
		\le &C\mathbb{P}_{y}\left[\tau_{z}=\tau_{\partial B(s_{k+2})}<\tau_{\partial B(s_k)} \right]\cdot \Big(\sum_{y_*\in \partial B(s_{k+1})}\mathbb{P}\left[\tau_{\partial B(s_{k+1})}=\tau_{y_*} \right]\Big) \\
		&\cdot \Big( \sum_{y_1\in \partial B(s_{k+1})}\mathbb{P}_y\left[S_{\bar{\tau}_{2l-2}}=y_1 \right]  \cdot  \sum_{x_1\in \partial B(s_k)} \mathbb{P}_{y_1}\left[\tau_{x_1}=\tau_{\partial B(s_k)}<\tau_{\partial B(s_{k+2})} \right]\Big)\\
		\le  & C\mathbb{P}_{y}\left[\tau_{z}=\tau_{\partial B(s_{k+2})}<\tau_{\partial B(s_k)} \right]\cdot \mathbb{P}_y\left[\bar{\tau}_{2l-2}<\infty \right]\\
		& \cdot \max_{y_1\in \partial B(s_{k+1})}\mathbb{P}_{y_1}\left[\tau_{\partial B(s_k)}<\tau_{\partial B(s_{k+2})} \right]. 
	\end{split}
\end{equation}
Furthermore, by \cite[Proposition 6.5.2]{lawler2010random} and (\ref{ineq_2.2}), one has 
\begin{equation}\label{ineq_7.36}
	\max_{y_1\in \partial B(s_{k+1})}\mathbb{P}_{y_1}\left[\tau_{\partial B(s_k)}<\tau_{\partial B(s_{k+2})} \right]\le Cs_{k}^{-(4d-1)(d-2)}. 
\end{equation}
Combining (\ref{ineq_7.33}), (\ref{ineq_new_735}) and (\ref{ineq_7.36}), we obtain
\begin{equation}\label{ineq_7.39}
	\mathbb{P}_y\left[\mathfrak{N}= l\big| \tau_{\partial B(s_{k+2})}=\tau_z\right]\le Cs_{k}^{-(4d-1)(d-2)}\cdot \mathbb{P}_y\left[\bar{\tau}_{2l-2}<\infty \right].  
\end{equation}

Repeating the argument in proving (\ref{ineq_7.39}) for $l-1$ times, we also have 
\begin{equation}\label{ineq_7.40}
	\mathbb{P}_y\left[\bar{\tau}_{2l-2}<\infty \right]\le \big[Cs_{k}^{-(4d-1)(d-2)}\big]^{l-1}.
\end{equation}
By (\ref{ineq_7.39}) and (\ref{ineq_7.40}), we get (\ref{fina_774}) and thus conclude (\ref{ineq_7.31}). Based on (\ref{ineq_7.31}), the proof of (\ref{ineq_7.32}) is a straightforward
calculation as follows:
\begin{equation*}
	\begin{split}
		\widetilde{\mathbb{E}}_y\left[ \exp(\gamma \mathfrak{N})\big|  \widetilde{\tau}_{\partial B_x(s_{k+2})}= \widetilde{\tau}_z \right] \le& 1+\sum_{l\ge 1}e^{\gamma l} \widetilde{\mathbb{P}}_y\left[\mathfrak{N}= l\big| \widetilde{\tau}_{\partial B_x(s_{k+2})}=\widetilde{\tau}_z\right]\\
		\le &1+ \sum_{l\ge 1}e^{\gamma l}s_{k+1}^{-cl} = \frac{s_{k+1}^{c}}{s_{k+1}^{c}-e^{\gamma}}.    \qedhere 
	\end{split}
\end{equation*}
Now we complete the proof of Lemma \ref{lemma_num_crossing}   \qed
\subsubsection{Proof of Lemma \ref{lemma_721}} \label{section_technical_lemma_2}

Before proving Lemma \ref{lemma_721}, we need the following lemma as preparation.

\begin{lemma}There exist $c(d),C(d)>0$ such that: 
\begin{enumerate}
	\item  For any $y\in \partial B(s_{k+1})$, $z\in \partial \hat{B}(s_{k+2})$ and $v\in \partial B(s_{k+2}-1)$, 
		\begin{equation}\label{ineq_750}
		\mathbb{P}_y\left[\tau_{\bm{0}}<\tau_v< \tau_{\partial B(s_{k+2})}\mid \tau_{\partial B(s_{k+2})}=\tau_z \right] \ge ce^{-C\log^2(s_{k+2})}.
		\end{equation}

	\item  For any $v_1,v_2\in B(s_{k+2}-1)$, 
 \begin{equation}\label{ineq_751}
 	\mu \left(\{\ell: \bm{0},v_1,v_2\in \mathrm{ran}(\ell)\ \text{and}\ \mathrm{ran}(\ell)\subset B(s_{k+2}-1) \} \right)\ge ce^{-C\log^2(s_{k+2})}.  
 \end{equation}

\end{enumerate}
\end{lemma}
\begin{proof}
		(1)  The inequality (\ref{ineq_750}) can be proved as follows:  
			\begin{equation}\label{fina_788}
			\begin{split}
				&\mathbb{P}_y\left[\tau_{\bm{0}}<\tau_v< \tau_{\partial B(s_{k+2})}\mid \tau_{\partial B(s_{k+2})}=\tau_z \right]\\
				\ge &\mathbb{P}_y\left[\tau_{\bm{0}}<\tau_v< \tau_{\partial B(s_{k+2})}=\tau_z \right]\\
				\ge & \mathbb{P}_y\left[\tau_{\bm{0}}<\tau_{\partial B(s_{k+2}-1)} \right]\cdot \mathbb{P}_{\bm{0}}\left[\tau_{v}<\tau_{\partial B(s_{k+2})} \right] \cdot \mathbb{P}_{v}\left[\tau_{\bm{0}}<\tau_{\partial B(s_{k+2})} \right]\\
				&  \cdot \mathbb{P}_{\bm{0}}\left[\tau_{\partial B(s_{k+2})}=\tau_{z} \right] \ \ \  \ \ \ \ \ \ \ \ \ \ (\text{by strong Markov property})\\
				=& \mathbb{P}_{\bm{0}}\left[\tau_{y}<\tau_{\partial B(s_{k+2}-1)} \right]\cdot \mathbb{P}_{\bm{0}}\left[\tau_{v}<\tau_{\partial B(s_{k+2})} \right] \cdot \mathbb{P}_{\bm{0}}\left[\tau_{v}<\tau_{\partial B(s_{k+2})} \right]\\
				&  \cdot \mathbb{P}_{\bm{0}}\left[\tau_{\partial B(s_{k+2})}=\tau_{z} \right] \ \ \ \  \ \ \ \ \ \ \ \ \  (\text{by reversing the random walk}) \\
				\ge &  ce^{-C\log^2(s_{k+2})}  \ \ \ \ \ \ \ \ \  \ \ \ \ \ \ \ \  \ \ \ \ \  \ (\text{by Lemma}\ \ref{lemma_66}).
			\end{split}
		\end{equation}

	\noindent(2) We denote by $\mathfrak{L}_{v_1,v_2}$ the collection of loops $\ell$ that satisfy the following: there exists $\varrho\in \ell$ such that before intersecting $\partial B(s_{k+2})$, $\varrho$ starts from $\bm{0}$, first hits $v_1$, then hits $v_2$ and finally return to $\bm{0}$. By (\ref{fina_2.5}), the loop measure of $\mathfrak{L}_{v_1,v_2}$ is bounded from below by 
	\begin{equation}
	\begin{split}
		&\mathbb{P}_{\bm{0}}\left[ \tau_{v_1}<\tau_{\partial B(s_{k+2})} \right] \cdot   \mathbb{P}_{v_1}\left[\tau_{v_2}<\tau_{\partial B(s_{k+2})}\right]\cdot \mathbb{P}_{v_2}\left[\tau_{\bm{0}}<\tau_{\partial B(s_{k+2})}\right]\\
		\ge &\mathbb{P}_{\bm{0}}\left[ \tau_{v_1}<\tau_{\partial B(s_{k+2})} \right]  \cdot   \mathbb{P}_{v_1}\left[\tau_{\bm{0}}<\tau_{\partial B(s_{k+2})}\right]\cdot \mathbb{P}_{\bm{0}}\left[\tau_{v_2}<\tau_{\partial B(s_{k+2})}\right]\\
		&\cdot \mathbb{P}_{v_2}\left[\tau_{\bm{0}}<\tau_{\partial B(s_{k+2})}\right]\ \ \ \ \ \ \   \ \ \ \ \ \ \ \  (\text{by strong Markov property})\\
		=& \mathbb{P}_{\bm{0}}\left[ \tau_{v_1}<\tau_{\partial B(s_{k+2})} \right]  \cdot   \mathbb{P}_{\bm{0}}\left[\tau_{v_1}<\tau_{\partial B(s_{k+2})}\right]\cdot \mathbb{P}_{\bm{0}}\left[\tau_{v_2}<\tau_{\partial B(s_{k+2})}\right]\\
		&\cdot \mathbb{P}_{\bm{0}}\left[\tau_{v_2}<\tau_{\partial B(s_{k+2})}\right]\ \ \ \ \ \  \ \  \ \ \ \ \ \ \ (\text{by reversing the random walk})\\
		\ge &ce^{-C\log^2(s_{k+2})}\  \ \ \ \ \ \ \ \ \ \ \ \ \ \ \ \ \ \ \ \ \ \  \ \ \  (\text{by Lemma}\ \ref{lemma_66}).
	\end{split}
	\end{equation}
	This implies (\ref{ineq_751}) since for every $\ell\in \mathfrak{L}_{v_1,v_2}$, one has $\bm{0},v_1,v_2\in \mathrm{ran}(\ell)$ and $\mathrm{ran}(\ell)\subset B(s_{k+2}-1)$.
\end{proof}

Recall in the construction of Step $0$ in $\mathcal{T}_n$ that $\mathcal{A}$ is the collection of active points in $F(w)$, and that $\mathcal{E}$ is the collection of involved crossing paths sampled in Step $0$. Now we are ready to prove Lemma \ref{lemma_721}. 

\begin{proof}[Proof of Lemma \ref{lemma_721}]
	Arbitrarily take $\omega\in \mathfrak{C}^x_p$. Recall that on the conditioning of $\omega$, one has that $x\in \mathcal{A}$, and that $\mathbf{C}_p^{\dagger}$ is determined and intersects $\partial \hat{B}_x(s_{k+2})$. We arbitrarily choose a point $y_{\diamond}\in \mathbf{C}_p^{\dagger}\cap \partial \hat{B}_x(s_{k+2})$ in some prefixed manner. Then one of the following happens:
	\begin{enumerate}

	\item  $\mathbf{C}_p^{\dagger}$ includes an involved fundamental loop $\widetilde{\ell}$ intersecting $y_{\diamond}$.

	\item  $y_{\diamond}\in B(n-1)$, and $\mathbf{C}_p^{\dagger}$ includes the glued point loop $\gamma_{y_{\diamond}}^{\mathrm{p}}$.

	\end{enumerate}
 We denote by $y_{\diamond}'$ the unique point in $\partial B_x(s_{k+2}-1)$ such that $y_{\diamond}'\sim y_{\diamond}$. Let $\widetilde{y}_{\diamond}:= \frac{1}{2}(y_{\diamond}+y_{\diamond}')$. We also denote by $y_{\diamond}''$ the unique point in $\partial \hat{B}_x(s_{k+2}+1)$ such that $y_{\diamond}''\sim y_{\diamond}$.

 In Case (1), recall that the Brownian excursions of $\widetilde{\ell}$ at $y_{\diamond}$ are either not sampled, or are sampled to intersect $\mathbf{C}_{p-1}\cap I_{\{y_{\diamond},y_{\diamond}'' \}}$. If these Brownian excursions are not sampled, then (recalling Section \ref{subsection_continuous}) the conditional distribution (given $\omega$) of the union of these Brownian excursions can be described as a function of an exponential random variable and a Bessel-$0$ process. Thus, conditioning on $\omega$, $\{\widetilde{y}_{\diamond}\in \mathrm{ran}(\widetilde{\ell})\}$ happens with at least probability $c(d)>0$. Otherwise (i.e. these Brownian excursions are sampled to intersect $\mathbf{C}_{p-1}\cap I_{\{y_{\diamond},y_{\diamond}'' \}}$), by the FKG inequality, $\{\widetilde{y}_{\diamond}\in \mathrm{ran}(\widetilde{\ell})\}$ also happens with at least probability $c(d)>0$. In Case (2), it also follows from the FKG inequality that $\{\widetilde{y}_{\diamond}\in \gamma_{y_{\diamond}}^{\mathrm{p}}\}$ happens with at least probability $c(d)>0$. In conclusion, to verify this lemma, it suffices to prove that 
\begin{equation}\label{fina_7.91}
	\begin{split}
			&\mathbb{P}\left(\exists\ \text{involved}\ \widetilde{\ell}\ \text{or}\ \widetilde{\eta}^{\mathrm{F}}\ \text{in}\ \widetilde{B}_x(s_{k+2})\ \text{intersecting}\ x\ \text{and}\ \widetilde{y}_{\diamond} \ \big| \  \omega \right)\\
			\ge &ce^{-C\log^2(s_{k+2})}. 
	\end{split}
\end{equation}
In what follows, We prove (\ref{fina_7.91}) separately in two different cases when $x\in F^{\mathrm{I}}(w)$ and $x\in F^{\mathrm{II}}(w)$.

\textbf{When $x\in F^{\mathrm{I}}(w)$:} Since $x\in F^{\mathrm{I}}(w)\cap \mathcal{A}$, there exists an involved forward crossing path $\widetilde{\eta}^{\mathrm{F}}$ in $\widetilde{B}_x(s_{k+2})$. Suppose that $\widetilde{\eta}^{\mathrm{F}}$ starts from $z_1\in \partial B_x(s_{k+1})$ and ends at $z_2\in \partial \hat{B}(s_{k+2}-1)$. According to the construction of $\mathcal{T}_n$, the conditional distribution (given $\omega$) of $\widetilde{\eta}^{\mathrm{F}}$ is exactly  $\widetilde{\mathbb{P}}_{z_1}\left(\cdot|\widetilde{\tau}_{\partial B_x(s_{k+2})}=\widetilde{\tau}_{z_2}  \right)$. Thus, the LHS of (\ref{fina_7.91}) is at least 
\begin{equation*}
	\widetilde{\mathbb{P}}_{z_1}\left(\widetilde{\tau}_x<\widetilde{\tau}_{\widetilde{y}_{\diamond}}< \widetilde{\tau}_{\partial B_x(s_{k+2})} \mid\widetilde{\tau}_{\partial B_x(s_{k+2})}=\widetilde{\tau}_{z_2}  \right). 
\end{equation*}
By the relation between the random walk on $\mathbb{Z}^d$ and the Brownian motion presented in Section \ref{subsection_BM}, the probability above is bounded from below by  
	\begin{equation*}
		\begin{split}
				c\cdot \mathbb{P}_{z_1}\left(\tau_x<\tau_{y_{\diamond}'}< \tau_{\partial B_x(s_{k+2})} \mid \tau_{\partial B_x(s_{k+2})}=\tau_{z_2}  \right).
		\end{split}	
	\end{equation*}
	Combined with (\ref{ineq_750}), it concludes (\ref{fina_7.91}) for $x\in F^{\mathrm{I}}(w)$.

	\textbf{When $x\in F^{\mathrm{II}}(w)$:} We arbitrarily take a lattice point $v\in B_x(s_{k+2}-1)\cap \hat{B}(n)$. 
	Since the loops in $\widetilde{B}_x(s_{k+2})$ are independent of the conditioning $\omega$, the LHS of (\ref{fina_7.91}) is at least  
		\begin{equation*}
			\frac{1}{2} \widetilde{\mu}\left(\{\widetilde{\ell}: x,\widetilde{y}_{\diamond},v\in \mathrm{ran}(\widetilde{\ell})\ \text{and}\ \mathrm{ran}(\widetilde{\ell})\subset \widetilde{B}_x(s_{k+2}) \} \right). 
		\end{equation*}
	By the relation between the loops on $\mathbb{Z}^d$ and $\widetilde{\mathbb{Z}}^d$ presented in Section \ref{subsection_continuous}, this loop measure is bounded from below by 
	\begin{equation*}
		 c\cdot \mu\left(\{\ell: x,y_{\diamond}',v\in \mathrm{ran}(\ell)\ \text{and}\ \mathrm{ran}(\ell)\subset B_x(s_{k+2}-1) \} \right). 
	\end{equation*}
	Therefore, by (\ref{ineq_751}), we also conclude (\ref{fina_7.91}) for $x\in F^{\mathrm{II}}(w)$, and thus complete the proof.\end{proof}

\section{Proof of Theorem \ref{theorem_regularity}}\label{section8}

In this section, we aim to prove Theorem \ref{theorem_regularity} by using Corollary \ref{coro_61}. This part of proof is inspired by \cite[Section 5]{kozma2011arm}. Here is an overview for this section. Our main aim is to give a lower bound for the probability of $\{\psi_n^{\mathrm{SR}}\ge \tfrac{1}{2}L^2, \chi_n\le cL^4 \}$ (see Lemma \ref{lemma_regular}), which indicates that each strongly regular point roughly generates $O(L^2)$ points in the loop cluster. To achieve this, we employ the second moment method. On the one hand, we will prove a lower bound in Lemma \ref{lemma_6.15} for the first moment of the number of points that are connected to some regular points, where a \textit{pivotal property} is employed to prevent excessive duplication for counting (see Definition \ref{fina_def_8.6}); on the other hand, we will prove an upper bound in Lemma \ref{lemma6.113} for the second moment of the aforementioned number, and thus obtain Lemma \ref{lemma_regular}. Finally, we conclude Theorem \ref{theorem_regularity} by combining Corollary \ref{coro_61} and Lemma \ref{lemma_regular}.

Recall that $K>0$ is a sufficiently large constant and $r_m=K\cdot 2^{m-1}$. Let $m_*:=\min\{m\in \mathbb{N}^+:r_{m}\ge K^{2d}\}$ and $K_*:=r_{m_*}$. Note that $K_*\in [K^{2d},2K^{2d}]$. For any $x\in \mathbb{Z}^d\setminus B(n-1)$, at least one of the $2d$ faces of $\partial B_x(K^4_*)$ is disjoint from $B(n)$. We choose one such face, denoted by $\mathcal{S}_x=\mathcal{S}_x(n,K)$, in an arbitrary and prefixed manner.

\begin{lemma}\label{lemma_610}
For $d>6$, $x\in \mathbb{Z}^d\setminus B(n-1)$ and sufficiently large $K$, if $x$ is a strongly regular point, then there exists $x'\in \partial B_x(K_*^4)$ such that $\widehat{\mathbf{\Psi}}_n\cap B_{x'}(K_*)=\emptyset$. 
\end{lemma}

\begin{proof}
 By a simple volume consideration, we can find $cK_*^{3(d-1)}$ points $x_1',...,x_{cK_*^{3(d-1)}}'$ in $\mathcal{S}_x$ such that the minimal pairwise distance is at least $3K_*$ (so in particular $B_{x'_i}(K_*) \cap B_{x'_j}(K_*) = \emptyset$ for all $i\neq j$). Since $x$ is strongly regular, by Item (1) of Remark \ref{remark6.1} we have 
\begin{equation}
	\big| B_x(2K_*^4)\cap \widehat{\mathbf{\Psi}}_n  \big| \le CK_*^{16}\log^{16}(K_*^4). 
\end{equation}
Combined with the fact that $CK_*^{16}\log^{16}(K^4_*)< cK_*^{3(d-1)}$ for all large enough $K$, it yields that there exists some $x_i'$ such that $B_{x_i'}(K_*)\cap \widehat{\mathbf{\Psi}}_n=\emptyset$. 
\end{proof}

In light of Lemma \ref{lemma_610}, for each strongly regular point $x\in \overline{\mathbf{\Psi}}_n^*$, we may define $x'=x'(\mathbf{\Psi}_n)$ to be the first point (in some arbitrary and prefixed order) such that $\widehat{\mathbf{\Psi}}_n \cap B_{x'}(K_*)=\emptyset$. Note that $x'$ is regular since $x'\in B_x(K_*^4)\subset B_x(K^{10d})$.

For any $A_1,A_2,A_3\subset \widetilde{\mathbb{Z}}^d$, we say $A_1$ and $A_2$ are connected (by $\cup \widetilde{\mathcal{L}}_{1/2}$) off $A_3$ if there exists a collection $\mathfrak{L}$ of loops in $\widetilde{\mathcal{L}}_{1/2}$ disjoint from $A_3$ such that $A_1 \xleftrightarrow[]{\cup \mathfrak{L}} A_2$. We write it as ``$A_1\xleftrightarrow[]{} A_2\ \text{off}\ A_3$''. We may omit the braces when $A_i=\{v\}$ for some $i\in \{1,2\}$ and $v\in \widetilde{\mathbb{Z}}^d$.

\begin{lemma}\label{lemma_8.2}
	For any $\widetilde{\ell}\in \widetilde{\mathcal{L}}_{1/2}$ with $\mathrm{ran}(\widetilde{\ell})\cap \widehat{\mathbf{\Psi}}_n=\emptyset$, we have $\widetilde{\ell} \in \widetilde{\mathcal{L}}^{\mathrm{U}}$. As an immediate consequence, for any $A_1, A_2 \subset \widetilde{\mathbb{Z}}^d$, the event $\{A_1\xleftrightarrow[]{} A_2\ \text{off}\ \widehat{\mathbf{\Psi}}_n\}$ is measurable with respect to $\widetilde{\mathcal{L}}^{\mathrm{U}}$. 
\end{lemma}
\begin{proof}
   We prove this lemma by contradiction. Suppose that there is $\widetilde{\ell}_{\diamond}\in \widetilde{\mathcal{L}}_{1/2}-\widetilde{\mathcal{L}}^{\mathrm{U}}$ such that $\mathrm{ran}(\widetilde{\ell}_{\diamond})\cap \widehat{\mathbf{\Psi}}_n=\emptyset$. By Definition \ref{def_unused loops}, loops in $\widetilde{\mathcal{L}}_{1/2}-\widetilde{\mathcal{L}}^{\mathrm{U}}$ can be divided into the following types:
    \begin{enumerate}
   	\item a fundamental or edge loop intersecting both $\widetilde{B}(n)$ and $\mathbf{\Psi}_n^1$;

   	\item a point loop that includes some $x\in B(n-1)$ and intersects $\mathbf{\Psi}_n^1$;

   	\item a point loop that includes some $x\in \partial \hat{B}(n)\setminus \mathbf{\Psi}_n^1$ and satisfies $I_{\{x,x^{\mathrm{in}}\}}\subset \gamma^{\mathrm{p}}_x\cup \mathbf{\Psi}_n^1$. 
   	
   \end{enumerate}
   On the one hand, $\widetilde{\ell}_{\diamond}$ does not belong to Type (1) or (2) since $\mathrm{ran}(\widetilde{\ell}_{\diamond})\cap \mathbf{\Psi}_n^1 \subset \mathrm{ran}(\widetilde{\ell}_{\diamond})\cap \widehat{\mathbf{\Psi}}_n=\emptyset$. On the other hand, if $\widetilde{\ell}_{\diamond}$ belongs to Type (3), then $\widetilde{\ell}_{\diamond}$ is a point loop including some $x\in \mathbf{\Psi}_n^2$. Since $\mathbf{\Psi}_n^2 \subset \widehat{\mathbf{\Psi}}_n$, this is contradictory with $\mathrm{ran}(\widetilde{\ell}_{\diamond})\cap \widehat{\mathbf{\Psi}}_n=\emptyset$. \end{proof}

  We recall some necessary notations before presenting the next definition: 
   
   \begin{itemize}
   	\item For any $x\in \mathbb{Z}^d$, we denote by $\mathbf{C}(x)$ the cluster of $\cup \widetilde{\mathcal{L}}_{1/2}$ containing $x$;
   	
   	\item $L= \epsilon^{\frac{3}{10}} N$ for some constant $\epsilon>0$;

   	\item $n\in [(1+\frac{\lambda}{4})N,(1+\frac{\lambda}{3})N]$ for some constant $\lambda\in \left(0,1 \right] $;

   	\item $\overline{\mathbf{\Psi}}_n^*:= \overline{\mathbf{\Psi}}_n \cap B(n_*)$ where $n_*:=n+[(1+\lambda)N]^b$, and $\psi_{n}^*= |\overline{\mathbf{\Psi}}_n^*|$.

   \end{itemize}

For any $x\in \mathbb{Z}^d\setminus B(n-1) $, we define the point $x_{\dagger}=x_{\dagger}(x,n)$ as follows: when $x\in \partial \hat{B}(n)$, let $x_{\dagger}$ be the unique point in $\partial \hat{B}(n+1)$ such that $x_{\dagger}\sim x$; otherwise, let $x_{\dagger}=x$. Then we define $\widetilde{x}_{\dagger}:=\frac{1}{2}(x+x_{\dagger})$.

\begin{definition}[potential pair]\label{potential}
For any $x\in \mathbb{Z}^d\setminus B(n-1)$, $y\in B_x(\frac{1}{2}L)\setminus B(n_*)$ and integer $M\in \mathbb{N}^+$, we say $(x,y)$ is an $M$-potential pair if the following events happen:
\begin{equation*}
	\mathsf{P}_1(x,M):= \left\lbrace x\in \overline{\mathbf{\Psi}}_n^*\right\rbrace \cap \left\lbrace x\ \text{is strongly regular} \right\rbrace\cap \left\lbrace \psi_n^{\mathrm{SR}}=M \right\rbrace \cap \left\lbrace \widetilde{x}_{\dagger}\in \widehat{\mathbf{\Psi}}_n \right\rbrace ,  
\end{equation*}
\begin{equation*}
	\mathsf{P}_2(x,y):= \left\lbrace x'\xleftrightarrow[]{}y\ \text{off}\ \widehat{\mathbf{\Psi}}_n \right\rbrace, 
\end{equation*}
\begin{equation*}
	\mathsf{P}_3(x):= \left\lbrace \mathbf{C}(x)\cap \mathbf{C}(x')=\emptyset \right\rbrace.  
\end{equation*}
\end{definition}

Note that the $M$-potential pair is not a symmetric relation since $y$ is not even necessarily strongly regular when $(x, y)$ is an $M$-potential pair.

\begin{lemma}\label{lemma_p1p2_p1}
For $d>6$, there exists $c_{6}(d)>0$ such that for all large enough $K$, any $x\in \mathbb{Z}^d\setminus B(n-1)$ and $M\in \mathbb{N}^+$,  
\begin{equation}\label{8.5}
	\sum_{y\in B_x(\frac{1}{2}L)\setminus B(n_*)}\mathbb{P}\left[	\mathsf{P}_1(x,M)\cap 	\mathsf{P}_2(x,y) \right] \ge c_{6} L^2 \mathbb{P}\left[\mathsf{P}_1(x,M)\right]. 
\end{equation}
\end{lemma}

\begin{proof}

In order to prove the lemma, it suffices to show that for any sufficiently large $K$, and an arbitrary realization $\widehat{\mathbf{A}}$ for $\widehat{\mathbf{\Psi}}_n$ on which $\mathsf{P}_1(x,M)$ occurs, we have
		\begin{equation}\label{ineq_687}
				\sum_{y\in B_x(\frac{1}{2}L)\setminus B(n_*)}	\mathbb{P}\left[\mathsf{P}_2(x,y) \mid \widehat{\mathbf{\Psi}}_n =\widehat{\mathbf{A}}  \right] \ge  cL^2 
			\end{equation} 
		(since we can obtain (\ref{8.5}) by averaging over $\widehat{\mathbf{A}}$). In what follows, we prove (\ref{ineq_687}).


By Lemma \ref{lemma_8.2} and Item (2) of Remark \ref{remark_51} we have 
\begin{equation}\label{ineq_681}
	\begin{split}
		\mathbb{P}\left[\mathsf{P}_2(x,y) \mid \widehat{\mathbf{\Psi}}_n =\widehat{\mathbf{A}}   \right]
		=&	\mathbb{P}\left( x'\xleftrightarrow[]{}y\ \text{off}\ \widehat{\mathbf{A}}\mid \widehat{\mathbf{\Psi}}_n =\widehat{\mathbf{A}}  \right)  \\
		=& 	\mathbb{P}\left( x'\xleftrightarrow[]{}y\ \text{off}\ \widehat{\mathbf{A}}\right) \\
		=&  	\mathbb{P}\left( x'\xleftrightarrow[]{}y\right) - 	\mathbb{P}\left( x'\xleftrightarrow[]{}y\ \text{only by}\ \widehat{\mathbf{A}} \right) ,
	\end{split}
\end{equation}
where ``only by $\widehat{\mathbf{A}}$'' means that in any collection of loops connecting $x'$ and $y$, there is at least one loop intersecting $\widehat{\mathbf{A}}$. On the event $\{x'\xleftrightarrow[]{}y\ \text{only by}\ \widehat{\mathbf{A}}\}$, there exists a glued loop $\gamma_{*}$ intersecting $\widehat{\mathbf{A}}$ such that $\{x'\xleftrightarrow[]{}\gamma_{*}\} \circ \{y\xleftrightarrow[]{}\gamma_{*}\}$ happens. Similar to (\ref{ineq_new_4.10}), by the BKR inequality and the two-point function estimate, we have 
\begin{equation}\label{ineq_6.82}
	\begin{split}
		&\mathbb{P}\left( x'\xleftrightarrow[]{}y\ \text{only by}\ \widehat{\mathbf{A}}\right)  \\
		\le &C\sum_{z_1\in \widehat{\mathbf{A}}\cap\mathbb{Z}^d,z_2,z_3\in \mathbb{Z}^d} \widetilde{\mu}\Big( \big\{ \widetilde{\ell}: \widetilde{\ell}\cap \widetilde{B}_{z_i}(1)\neq\emptyset,\forall i=1,2,3\big\}\Big) \\
		&\ \ \ \ \ \ \ \ \ \ \ \ \ \  \ \ \ \ \ \ \ \cdot  \mathbb{P}\left[\widetilde{B}_{z_2}(1)\xleftrightarrow[]{} x' \right]\cdot  \mathbb{P}\left[\widetilde{B}_{z_3}(1)\xleftrightarrow[]{} y \right]\\
		\le & C\sum_{z_1\in \widehat{\mathbf{A}} \cap\mathbb{Z}^d,z_2,z_3\in \mathbb{Z}^d} |z_1-z_2|^{2-d}|z_2-z_3|^{2-d}|z_3-z_1|^{2-d}|z_2-x'|^{2-d}|z_3-y|^{2-d}. 
	\end{split}
\end{equation}
Therefore, by Lemma \ref{lemma_fina_4.2} and Corollary \ref{fina_coro_4.5}, we have  
\begin{equation}\label{ineq_8.8}
	\begin{split}
		&\sum_{y\in B_x(\frac{1}{2}L)\setminus B(n_*)}\mathbb{P}\left( x'\xleftrightarrow[]{}y\ \text{only by}\ \widehat{\mathbf{A}} \right) \\
		\le &CL^2\sum_{z_1\in \widehat{\mathbf{A}} \cap\mathbb{Z}^d,z_2,z_3\in \mathbb{Z}^d} |z_1-z_2|^{2-d}|z_2-z_3|^{2-d}|z_3-z_1|^{2-d}\\
		& \ \ \ \ \ \ \ \ \ \ \  \ \ \ \ \ \ \ \  \ \ \ \ \ \cdot |z_2-x'|^{2-d} \ \ \ \ \ \ \ \ \   \ \ \ \ \ \ (\text{by}\  (\ref{fina4.3}))\\
		\le &CL^2\sum_{z_1\in \widehat{\mathbf{A}} \cap\mathbb{Z}^d} |z_1-x'|^{2-d}\ \ \ \ \ \ \ \ \ \ \ \ \ \ \ \ \ \ \ \ \ \ \ \ \  (\text{by}\  (\ref{fin_4.6}))\\
	\le &CL^2 \Big(\big|\widehat{\mathbf{A}} \cap B_{x'}(K_*)\big|+ \sum_{m=m_*+1}^{\infty}\big|\widehat{\mathbf{A}} \cap B_{x'}(r_{m})\big|\cdot r_{m-1}^{2-d}\Big).
	\end{split}
\end{equation}
It follows from the definition of $x'$ that $\big|\widehat{\mathbf{A}}\cap B_{x'}(K_*)\big|=0$. Moreover, since $x$ is strongly regular and $x'\in B_x(K^{10d})$, we know that $x'$ is regular, and thus $\big|\widehat{\mathbf{A}} \cap B_{x'}(r_{m})\big|\le r_{m}^4\log^{16}(r_{m})$. In conclusion, the RHS of (\ref{ineq_8.8}) can be upper-bounded by 
\begin{equation}\label{ineq_6.84}
	\begin{split}
		\sum_{m=m_*+1}^{\infty} r_{m}^4\log^{16}(r_{m})\cdot r_{m-1}^{2-d}\le CK^{6-d}\log^{16}(K). 
	\end{split}
\end{equation}
Combining (\ref{ineq_8.8}) and (\ref{ineq_6.84}), we obtain
\begin{equation}\label{ineq_685}
	\sum_{y\in B_x(\frac{1}{2}L)\setminus B(n_*)}\mathbb{P}\left( x'\xleftrightarrow[]{}y\ \text{only by}\ \widehat{\mathbf{A}} \right) \le  CK^{6-d}\log^{16}(K)L^2.
\end{equation}


Meanwhile, by the two-point function estimate, we have 
\begin{equation}\label{ineq_686}
	\sum_{y\in B_x(\frac{1}{2}L)\setminus B(n_*)}\mathbb{P}\left( x'\xleftrightarrow[]{}y\right) \ge c\sum_{y\in B_x(\frac{1}{2}L)\setminus B(n_*)}|x'-y|^{2-d} \ge    cL^2.
\end{equation}
Since $K^{6-d}\log^{16}(K)$ converges to $0$ as $K\to \infty$, by (\ref{ineq_681}), (\ref{ineq_685}) and (\ref{ineq_686}), there exists a constant $K_0(d)>0$ such that (\ref{ineq_687}) holds for all $K\ge K_0$.   
\end{proof}

\begin{lemma}\label{lemma_614}
For $d>6$, there exists $c_{7}(d)>0$ such that for all large enough $K$, any $x\in \mathbb{Z}^d\setminus B(n-1)$ and any $M\in \mathbb{N}^+$,  
\begin{equation}\label{new_8.14}
	\sum_{y\in B_x(\frac{1}{2}L)\setminus B(n_*)}\mathbb{P}\left[\mathsf{P}_1(x,M)\cap \mathsf{P}_2(x,y)\cap \mathsf{P}_3(x) \right] \ge c_{7} L^2 \mathbb{P}\left[\mathsf{P}_1(x,M)\right]. 
\end{equation}
\end{lemma}

\begin{proof}
To verify this lemma, it suffices to prove that for any large enough $K$ and any realization $\widehat{\mathbf{A}}$ for $\widehat{\mathbf{\Psi}}_n$ on which $\mathsf{P}_1(x,M)$ happens, one has 
	\begin{equation}\label{new_8.25}
		\begin{split}
			\sum_{y\in B_x(\frac{1}{2}L)\setminus B(n_*)} \mathbb{P}\left[ \mathsf{P}_2(x,y)\cap [\mathsf{P}_3(x)]^c \mid \widehat{\mathbf{\Psi}}_n =\widehat{\mathbf{A}}  \right] \le CL^2K^{6-d}\log^{16}(K). 
		\end{split}
	\end{equation}
	In fact, by averaging over $\widehat{\mathbf{A}}$ in (\ref{new_8.25}), we have 
	\begin{equation}\label{new_ineq_108_2}
		\begin{split}
				&\sum_{y\in B_x(\frac{1}{2}L)\setminus B(n_*)} \mathbb{P}\left[\mathsf{P}_1(x,M)\cap \mathsf{P}_2(x,y)\cap \big[\mathsf{P}_3(x)\big]^c \right] \\
			\le &CL^2K^{6-d}\log^{16}(K)\mathbb{P}\left[\mathsf{P}_1(x,M)\right].  
		\end{split}	
	\end{equation}
	For all sufficiently large $K$ with $CK^{6-d}\log^{16}(K)<\frac{1}{2}c_{6}$, (\ref{new_8.14}) follows from Lemma \ref{lemma_p1p2_p1} and (\ref{new_ineq_108_2}). We proceed to show (\ref{new_8.25}) in the remainder of this proof.

%

	 On the event $\mathsf{P}_2(x,y)\cap [\mathsf{P}_3(x)]^c$, by Lemma \ref{lemma_8.2}, $x'$ is connected to both $\widehat{\mathbf{A}}$ and $y$ by $\widetilde{\mathcal{L}}^{\mathrm{U}}_{\mathbf{A}}$. Therefore, by the tree expansion, there exists a glued loop $\gamma_*$ such that $\{\widehat{\mathbf{A}}\xleftrightarrow[]{\widetilde{\mathcal{L}}^{\text{U}}_{\mathbf{A}}} \gamma_*\}\circ \{x'\xleftrightarrow[]{} \gamma_*  \}\circ \{y\xleftrightarrow[]{} \gamma_*  \}$ happens. Thus, similar to (\ref{ineq_6.82}), we have 
	\begin{equation*}\label{new_ineq_6.100}
		\begin{split}
			& \mathbb{P}\left[ \mathsf{P}_2(x,y)\cap [\mathsf{P}_3(x)]^c \mid \widehat{\mathbf{\Psi}}_n =\widehat{\mathbf{A}}  \right] \\
			\le & C \sum_{z_1,z_2,z_3\in \mathbb{Z}^d}\mathbb{P}\Big( \widehat{\mathbf{A}} \xleftrightarrow[]{\widetilde{\mathcal{L}}^{\text{U}}_{\mathbf{A} }} z_1\Big)   |z_1-z_2|^{2-d}|z_2-z_3|^{2-d}\\
			&\ \ \ \ \ \ \ \ \ \ \ \ \ \ \  \cdot    |z_3-z_1|^{2-d}|z_2-x'|^{2-d}|z_3-y|^{2-d}. 
		\end{split}
	\end{equation*}	
Therefore, by (\ref{fina4.3}) and (\ref{fin_4.6}), we have  
	\begin{equation}\label{new_ineq_101}
		\begin{split}
				&\sum_{y\in B_x(\frac{1}{2}L)\setminus B(n_*)} \mathbb{P}\left[ \mathsf{P}_2(x,y)\cap [\mathsf{P}_3(x)]^c \mid \widehat{\mathbf{\Psi}}_n =\widehat{\mathbf{A}}  \right]\\
				\le &CL^2\sum_{z_1\in \mathbb{Z}^d}\mathbb{P}\Big( \widehat{\mathbf{A}} \xleftrightarrow[]{\widetilde{\mathcal{L}}^{\text{U}}_{\mathbf{A}}} z_1\Big) |z_1-x'|^{2-d}. 
		\end{split}
	\end{equation}	
	The sum on the RHS can be decomposed into $\mathbb{I}^{(1)}+\sum_{m=1}^{\infty}\mathbb{I}^{(2)}_{m}$, where (recall $r_1=K$)
	$$
	\mathbb{I}^{(1)}:= \sum_{z_1\in B_{x'}(K)}\mathbb{P}\Big( \widehat{\mathbf{A}} \xleftrightarrow[]{\widetilde{\mathcal{L}}^{\text{U}}_{\mathbf{A}}} z_1\Big) |z_1-x'|^{2-d}, 
	$$
	$$
	\mathbb{I}^{(2)}_{m}:= \sum_{z_1\in B_{x'}(r_{m+1})\setminus B_{x'}(r_{m}) }\mathbb{P}\Big( \widehat{\mathbf{A}} \xleftrightarrow[]{\widetilde{\mathcal{L}}^{\text{U}}_{\mathbf{A}}} z_1\Big) |z_1-x'|^{2-d}. 
	$$



For $\mathbb{I}^{(1)}$, since $|B_{x'}(K)|\asymp K^d$, $|z_1-x'|^{2-d}\le 1$ and $\widehat{\mathbf{A}}\cap B_{x'}(K_*)=\emptyset$, we have 
\begin{equation}\label{8.18}
	\begin{split}
		\mathbb{I}^{(1)}\le&  CK^d \max_{z_1\in B_{x'}(K)}\mathbb{P}\left[ z_1 \xleftrightarrow[]{} \widehat{\mathbf{A}}\setminus  B_{x'}(K_*)\right]  \\
		\le &CK^d\max_{z_1\in B_{x'}(K)}\sum_{m=m_*+1}^{\infty} \sum_{z_4\in \widehat{\mathbf{A}} \cap B_{x'}(r_{m})\setminus B_{x'}(r_{m-1})} |z_4-z_1|^{2-d}.
	\end{split}
\end{equation}
For any $z_1\in B_{x'}(K)$ and $z_4\in \widehat{\mathbf{A}}\cap B_{x'}(r_{m})\setminus B_{x'}(r_{m-1})$ for $m>m_*$, we have 
\begin{equation*}\label{8.19}
	|z_4-z_1|\ge |z_4-x'|-|x'-z_1|\ge |z_4-x'|-K\ge  cr_{m}. 
\end{equation*}
In addition, since $x$ is strongly regular (and thus $x'$ is regular), one has 
\begin{equation*}\label{8.20}
	\big|\widehat{\mathbf{A}}\cap B_{x'}(r_{m})\setminus B_{x'}(r_{m-1})\big|\le r_{m}^4\log^{16}(r_{m}). 
\end{equation*}
Thus, the RHS of (\ref{8.18}) is bounded from above by 
\begin{equation}\label{8.21}
	 CK^d\sum_{m=m_*+1}^{\infty} r_{m}^4\log^{16}(r_{m})\cdot r_{m}^{2-d}\le CK^d\cdot K_{*}^{6-d} \log(K_*)\le CK^{1-d},
\end{equation}
where we used $K_*\ge K^{2d}$ in the last inequality.

For each $\mathbb{I}^{(2)}_{m}$, since $|z_1-x'|\ge r_m$ for all $z_1\in B_{x'}(r_{m+1})\setminus B_{x'}(r_{m})$, we have (recalling Definition \ref{def_nice_set}) 
\begin{equation*}
	\begin{split}
		\mathbb{I}^{(2)}_{m}\le&  Cr_{m}^{2-d} \sum_{z_1\in B_{x'}(r_{m+1})\setminus B_{x'}(r_{m}) }\mathbb{P}\Big( \widehat{\mathbf{A}} \xleftrightarrow[]{\widetilde{\mathcal{L}}^{\text{U}}_{\mathbf{A}}} z_1\Big) \\
		\le  &Cr_{m}^{2-d}\bigg[\Delta_{x',m+1}(\mathbf{A})+\sum_{z_1\in B_{x'}(r_{m+1}) }  \mathbb{P}\Big( z_1\xleftrightarrow[]{\widetilde{\mathcal{L}}^{\text{U}}_{\mathbf{A}}} \widehat{\mathbf{A}}\setminus B_{x'}(r_{m+1}^{4d}) \Big)   \bigg].    
	\end{split}
\end{equation*}
Since $x'$ is regular, one has $\Delta_{x',m+1}(\mathbf{A})\le r_{m+1}^4\log^{16}(r_{m+1})$. In addition, since $\cup \widetilde{\mathcal{L}}^{\text{U}}_{\mathbf{A}}$ is stochastically dominated by $\cup \widetilde{\mathcal{L}}_{1/2}$, by (\ref{ineq_onearm}) we have 
\begin{equation*}
	\sum_{z_1\in B_{x'}(r_{m+1}) }  \mathbb{P}\Big(z_1  \xleftrightarrow[]{\widetilde{\mathcal{L}}^{\text{U}}_{\mathbf{A}}}\widehat{\mathbf{A}}\setminus B_{x'}(r_{m+1}^{4d}) \Big) \le Cr_{m+1}^{d}\cdot r_{m+1}^{-\frac{1}{2}\cdot 4d}<1.
\end{equation*}
Consequently, $\mathbb{I}^{(2)}_{m}$ is upper-bounded by 
\begin{equation}\label{8.24}
	 Cr_{m}^{2-d}\cdot \left[ r_{m+1}^4\log^{16}(r_{m+1})+1\right] \le Cr_{m}^{6-d}\log^{16}(r_{m}). 
\end{equation}

By (\ref{new_ineq_101}), (\ref{8.21}) and (\ref{8.24}), we obtain (\ref{new_8.25}) and finally conclude the lemma:	
\begin{equation}
	\begin{split}
		&\sum_{y\in B_x(\frac{1}{2}L)\setminus B(n_+)} \mathbb{P}\left[ \mathsf{P}_2(x,y)\cap [\mathsf{P}_3(x)]^c \mid \widehat{\mathbf{\Psi}}_n =\widehat{\mathbf{A}}  \right] \\
		\le &CL^2\bigg( \mathbb{I}^{(1)}+\sum_{m=1}^{\infty}\mathbb{I}^{(2)}_{m}\bigg) \\
		\le &CL^2 \bigg(K^{1-d}+ \sum_{m=1}^{\infty}r_{m}^{6-d}\log^{16}(r_{m})\bigg) \\
		\le &CL^2K^{6-d}\log^{16}(K).  \qedhere
	\end{split}
\end{equation}
\end{proof}


In order to introduce our aforementioned pivotal event, for any $x\in \mathbb{Z}^d\setminus B(n-1)$, we define $\mathfrak{L}_{x,K}$ as the collection of all fundamental loops that are contained in $\widetilde{B}_x(K^4_*+1)\setminus \widetilde{B}(n)$, and visit $\widetilde{x}_{\dagger}$ and every point in $\mathcal{S}_x$ (and thus also visits $x'$). Note that there exists some constant $u_0(K,d)>0$ such that $\widetilde{\mu}(\mathfrak{L}_{x,K})\ge u_0(K,d)$ for all $x\in \mathbb{Z}^d\setminus B(n-1)$. Let $\gamma_{x,K}^{\mathrm{f}}$ be the union of ranges of loops contained in both $\widetilde{\mathcal{L}}_{1/2}$ and $\mathfrak{L}_{x,K}$. It follows from the definition that every loop in $\mathfrak{L}_{x,K}$ is not involved (recalling Definition \ref{definition_Psi}), and thus $\gamma_{x,K}^{\mathrm{f}}$ is independent of $\widehat{\mathbf{\Psi}}_n$.

%
%
%
%
%
%

\begin{definition}[admissible pair]\label{fina_def_8.6}
For any $x\in \mathbb{Z}^d\setminus B(n-1)$ and $y\in B_x(\frac{1}{2}L)\setminus B(n_*)$, we say $(x,y)$ is admissible if the following events happen: 
\begin{enumerate}
	\item $x\in \overline{\mathbf{\Psi}}_n^*$ and $x$ is strongly regular.

	\item The event $\{y\xleftrightarrow[]{}\widehat{\mathbf{\Psi}}_n\}$ happens and is pivotal with respect to $\gamma_{x,K}^{\mathrm{f}}$. Precisely, ``pivotal'' means that if we delete all loops included in $\gamma_{x,K}^{\mathrm{f}}$ from $\widetilde{\mathcal{L}}_{1/2}$, then the event $\{y\xleftrightarrow[]{}\widehat{\mathbf{\Psi}}_n\}$ no longer happens.

\end{enumerate}
We denote the total number of all admissible pairs by 
\begin{equation}\label{def_W}
	W=\big|\big\{(x,y):x\in \mathbb{Z}^d\setminus B(n-1),y\in B_x(\tfrac{1}{2}L)\setminus B(n_*),(x,y)\ \text{is admissible} \big\}\big|. 
\end{equation}
\end{definition}

\begin{lemma}\label{lemma_6.15}
For all sufficiently large $K$ and any $M\in \mathbb{N}^+$, there exists $c_{8}(K,d)>0$ such that
\begin{equation}\label{new_ineq_8.22}
	\mathbb{E}\left[W\cdot \mathbbm{1}_{\psi_n^{\mathrm{SR}}=M} \right] \ge c_{8}L^2M\mathbb{P}\left( \psi_n^{\mathrm{SR}}=M\right). 
\end{equation}
\end{lemma}
\begin{proof} 
Recall the events $\mathsf{P}_1$, $\mathsf{P}_2$ and $\mathsf{P}_3$ in Definition \ref{potential}. If $(x,y)$ is an $M$-potential pair, then we have:
\begin{itemize}
	\item The event $\{\gamma_{x,K}^{\mathrm{f}}=\emptyset\}$ happens. Otherwise, since $\widetilde{x}_{\dagger} \in \widehat{\mathbf{\Psi}}_n\cap \gamma_{x,K}^{\mathrm{f}}$, we get $\mathbf{C}(x)\cap \mathbf{C}(x')\neq\emptyset$ and thus $\mathsf{P}_3(x)$ fails.

	
	\item If we add one loop $\widetilde{\ell}\in \mathfrak{L}_{x,K}$ into the configuration of $\widetilde{\mathcal{L}}_{1/2}$, then $(x,y)$ becomes an admissible pair.


\end{itemize}

We define a mapping $\pi_{x,K}^{\mathrm{f}}$ as follows, which maps a configuration of $\widetilde{\mathcal{L}}_{1/2}$ to a collection of configurations of $\widetilde{\mathcal{L}}_{1/2}$. Precisely, for any $\omega$, which is a configuration of $\widetilde{\mathcal{L}}_{1/2}$ such that $\gamma_{x,K}^{\mathrm{f}}=\emptyset$, we define  $$\pi_{x,K}^{\mathrm{f}}(\omega):= \bigg\{ \omega+ \sum_{i\in I}\mathbbm{1}_{\widetilde{\ell}_i}: \emptyset\neq I \subset \mathbb{N},\widetilde{\ell}_i\in  \mathfrak{L}_{x,K}  \bigg\}.$$
Note that $\pi_{x,K}^{\mathrm{f}}$ is an injection. By the aforementioned observations, for any $\omega$ such that $(x,y)$ is an $M$-potential pair, any configuration in $\pi_{x,K}^{\mathrm{f}}(\omega)$ satifies that $(x,y)$ is admissible and $\psi_n^{\mathrm{SR}}=M$ (recalling that $\gamma_{x,K}^{\mathrm{f}}$ is independent of $\widehat{\mathbf{\Psi}}_n$). As a result, 
\begin{equation*}
	\begin{split}
		&\mathbb{P}\left[(x,y)\ \text{is admissible},\ \psi_n^{\text{SR}}=M \right]\\
		\ge &   \mathbb{P} \left[\pi_{x,K}^{\mathrm{f}}(\omega):\omega\ \text{such that}\ \mathsf{P}_1(x,M)\cap \mathsf{P}_2(x,y)\cap \mathsf{P}_3(x)\ \text{happens} \right]                                   \\
		= &\frac{\mathbb{P}( \gamma_{x,K}^{\mathrm{f}}\neq \emptyset) }{\mathbb{P}( \gamma_{x,K}^{\mathrm{f}}=\emptyset)}\cdot \mathbb{P}\left[\mathsf{P}_1(x,M)\cap \mathsf{P}_2(x,y)\cap \mathsf{P}_3(x) \right]\\
		\ge & c(K,d)\cdot \mathbb{P}\left[\mathsf{P}_1(x,M)\cap \mathsf{P}_2(x,y)\cap \mathsf{P}_3(x) \right]\ \ \ \ \ (\text{by}\ \widetilde{\mu}(\mathfrak{L}_{x,K})\ge u_0(K,d)).   
	\end{split}
\end{equation*}
By summing over all $x\in \mathbb{Z}^d\setminus B(n-1)$ and $y\in B_x(\frac{1}{2}L)\setminus B(n_*)$, we have 
\begin{equation}\label{new_ineq_8.23}
	\begin{split}
		&\mathbb{E}\left[W\cdot \mathbbm{1}_{\psi_n^{\text{SR}}=M} \right]\\
		=& \sum_{x\in \mathbb{Z}^d\setminus B(n-1),y\in B_x(\frac{1}{2}L)\setminus B(n_*)} \mathbb{P}\left[(x,y)\ \text{is admissible},\ \psi_n^{\text{SR}}=M \right]\\
		\ge &c(K,d) \sum_{x\in \mathbb{Z}^d\setminus B(n-1),y\in B_x(\frac{1}{2}L)\setminus B(n_*)}\mathbb{P}\left[\mathsf{P}_1(x,M)\cap \mathsf{P}_2(x,y)\cap \mathsf{P}_3(x) \right]\\
		\ge & c(K,d)c_{7}L^2 \sum_{x\in \mathbb{Z}^d\setminus B(n-1)}\mathbb{P}\left[\mathsf{P}_1(x,M) \right]\ \ \ \ \ (\text{by Lemma}\  \ref{lemma_614}).
	\end{split}
\end{equation}
By Item (1) of Remark \ref{fina_remark6.3}, arbitrarily given a configuration of $\mathbf{\Psi}_n$ with $x\in \overline{\mathbf{\Psi}}_n^*$, with at least probabilty $c'(d)>0$ the event $\big\{\widetilde{x}_{\dagger}\in \widehat{\mathbf{\Psi}}_n\big\}$ occurs. Therefore, the sum on the RHS of (\ref{new_ineq_8.23}) is bounded from below by 
\begin{equation*}
	\begin{split}
		c'\sum_{x\in \mathbb{Z}^d\setminus B(n-1)} \mathbb{P}\left[ x\in \overline{\mathbf{\Psi}}_n^*,\ x\ \text{is strongly regular, and}\  \psi_n^{\mathrm{SR}}=M   \right] =c'M \mathbb{P}\left(  \psi_n^{\text{SR}}=M\right).
	\end{split}
\end{equation*}
Combined with (\ref{new_ineq_8.23}), the proof of this lemma is complete. \end{proof}

\begin{lemma}\label{lemma6.113}
For any $K>0$ and $M\ge \frac{1}{2}L^2$, there exists $C_{7}(K,d)>0$ such that \begin{equation}\label{ineq_6.113}
	\mathbb{E}\left[W^2\cdot \mathbbm{1}_{\psi_n^{\text{SR}}=M} \right] \le C_{7}L^4M^2\mathbb{P}\left( \psi_n^{\text{SR}}=M\right). 
\end{equation}
\end{lemma}
\begin{proof}
Recall that $L\asymp n$. The term on the LHS of (\ref{ineq_6.113}) can be written as 
\begin{equation*}\label{equation_829}
	\begin{split}
		\sum_{\forall i\in \{1,2\},x_i\in \mathbb{Z}^d\setminus B(n-1), y_i\in B_{x_i}(\frac{1}{2}L)\setminus B(n_*)} \mathbb{P}\left[\forall i\in \{1,2\},(x_{i},y_{i})\ \text{is admissible},\psi_n^{\text{SR}}=M\right].
	\end{split}
\end{equation*}
We divide the sum above into the following three parts, which we denote by $\mathcal{S}_1$, $\mathcal{S}_2$ and $\mathcal{S}_3$ respectively:

\begin{enumerate}
	\item[Part 1:] $x_1=x_2$, $y_1=y_2$; 
	
	\item[Part 2:] $x_1=x_2$, $y_1\neq y_2$;  
	
	\item[Part 3:] $x_1\neq x_2$.
\end{enumerate}

In what follows, we prove the upper bounds for $\mathcal{S}_1$, $\mathcal{S}_2$ and $\mathcal{S}_3$ separately. Assume that $\{\psi_n^{\text{SR}}=M\}$ occurs, and $(x_1,y_1)$ and $(x_2,y_2)$ are both admissible. We denote $\mathsf{P}(x,M):=\big\{ x\in \overline{\mathbf{\Psi}}_n^*,\ x\ \text{is strongly regular},\ \psi_n^{\mathrm{SR}}=M\big\}$.

\textbf{Part 1:} Since $\mathsf{P}(x_1,M)$ and $\{B_{x_1}(2K^4_*)\xleftrightarrow[]{} y_1\ \text{off}\ \widehat{\mathbf{\Psi}}_n\}$ both happen (note that they are certified by two disjoint collections of glued loops), by the BKR inequality and the two-point function estimate, we have 
\begin{equation}\label{ineq_117}
	\begin{split}
		\mathcal{S}_1\le &C(K,d)\sum_{x_1\in \mathbb{Z}^d\setminus B(n-1), y_1\in B_{x_1}(\frac{1}{2}L)\setminus B(n_*)} \mathbb{P}\left[ \mathsf{P}(x_1,M)\right] \cdot |x_1-y_1|^{2-d}\\
		 \le & C(K,d) L^2  \sum_{x_1\in \mathbb{Z}^d\setminus B(n-1)} \mathbb{P}\left[ \mathsf{P}(x_1,M)\right]\ \ \ \ \ \ (\text{by}\ (\ref{fina4.3}))    \\
		=& C(K,d) L^2M\cdot \mathbb{P}\left( \psi_n^{\text{SR}}=M\right). 
	\end{split}
\end{equation}

%

%



\textbf{Part 2:} Note that $\mathsf{P}(x_1,M)$, $\{B_{x_1}(2K^4_*)\xleftrightarrow[]{} y_1\ \text{off}\ \widehat{\mathbf{\Psi}}_n\}$ and $\{B_{x_1}(2K^4_*)\xleftrightarrow[]{} y_2\ \text{off}\ \widehat{\mathbf{\Psi}}_n\}$ happen. By the tree expansion, there exists a glued loop $\gamma_*$ such that 
\begin{equation*}
	\{\gamma_*\xleftrightarrow[]{} \widetilde{B}_{x_1}(2K^4_*)\ \text{off}\  \widehat{\mathbf{\Psi}}_n  \} \circ  \{\gamma_* \xleftrightarrow[]{} y_1\ \text{off}\ \widehat{\mathbf{\Psi}}_n  \} \circ  \{\gamma_* \xleftrightarrow[]{} y_2\ \text{off}\  \widehat{\mathbf{\Psi}}_n  \}
\end{equation*}
happens. Since the event $\mathsf{P}(x_1,M)$ is certified by a disjoint collection of glued loops, by the BKR inequality and the two-point function estimate, we have 
\begin{equation*}\label{new_ineq6.118}
	\begin{split}
		\mathcal{S}_2 \le & C(K,d)\sum_{x_1\in \mathbb{Z}^d\setminus B(n-1)} 	\sum_{y_1,y_2\in B_{x_1}(\frac{1}{2}L)\setminus B(n_*)} \sum_{z_1,z_2,z_3\in\mathbb{Z}^d} \mathbb{P}\left[ \mathsf{P}(x_1,M)\right]|z_1-z_2|^{2-d} \\
		&\ \ \ \ \ \ \ \ \ \ \ \ \ \ \ \ \cdot |z_2-z_3|^{2-d}|z_3-z_1|^{2-d} |z_1-x_1|^{2-d}|z_2-y_1|^{2-d}|z_3-y_2|^{2-d}.
	\end{split}
\end{equation*}
If the sum on the RHS is also over the restriction $|z_1-x_1|\le  L$ (we denote this part of sum by $\mathcal{S}_2^{(1)}$), then we sum over $y_1$ and $y_2$, and apply Lemma \ref{lemma_fina_4.2} and Corollary \ref{fina_coro_4.5} to get its upper bound as follows:
\begin{equation*}\label{ineq_6.119}
	\begin{split}
		\mathcal{S}_2^{(1)}\le &C(K,d)L^4\sum_{x_1\in \mathbb{Z}^d\setminus B(n-1)} \sum_{z_1,z_2,z_3\in\mathbb{Z}^d:|z_1-x_1|\le L} \mathbb{P}\left[ \mathsf{P}(x_1,M)\right]|z_1-z_2|^{2-d} \\
		&\ \ \ \ \ \ \ \ \ \ \ \ \ \ \ \ \ \  \ \ \ \ \ \ \ \ \ \ \ \ \ \   \cdot |z_2-z_3|^{2-d}|z_3-z_1|^{2-d} |z_1-x_1|^{2-d}\ \ \ \  (\text{by}\ (\ref{fina4.3}))\\
		\le &   C(K,d)L^{4} \sum_{x_1\in \mathbb{Z}^d\setminus B(n-1)} \sum_{z_1\in \mathbb{Z}^d: |z_1-x_1|\le  L } \mathbb{P}\left[ \mathsf{P}(x_1,M)\right]  |z_1-x_1|^{2-d} \ \ \    (\text{by}\ (\ref{fin_4.6}))     \\
		\le & C(K,d)L^{6}M \cdot \mathbb{P}\left( \psi_n^{\text{SR}}=M\right)\  \ \ \  \ \ \ \ \ \ \ \ \ \ \ \ \ \ \ \ \ \ \ \ \ \ \ \ \ \ \ \ \ \ \ \ \ \ \ \ \ \ \ \ \  (\text{by}\ (\ref{fina4.3})).
	\end{split}
\end{equation*}
In the remaining case (i.e. $|z_1-x_1|> L$; let this part of sum be $\mathcal{S}_2^{(2)}$), we sum over $y_2$, and then apply Corollary \ref{fina_coro_4.5} to obtain 
\begin{equation*}\label{ineq_6.120}
	\begin{split}
		\mathcal{S}_2^{(2)}\le 	&C(K,d) L^2 \sum_{x_1\in \mathbb{Z}^d\setminus B(n-1), y_1\in B_{x_1}(\frac{1}{2}L)\setminus B(n_*)}\sum_{z_1,z_2,z_3\in\mathbb{Z}^d: |z_1-x_1|> L}\mathbb{P}\left[ \mathsf{P}(x_1,M)\right]   \\
		&\ \  \      \ \  \cdot  |z_1-z_2|^{2-d}|z_2-z_3|^{2-d}|z_3-z_1|^{2-d} |z_1-x_1|^{2-d}|z_2-y_1|^{2-d} \ \ \  (\text{by}\ (\ref{fina4.3})) \\
		\le &C(K,d) L^2  \sum_{x_1\in \mathbb{Z}^d\setminus B(n-1), y_1\in B_{x_1}(\frac{1}{2}L)\setminus B(n_*)}\sum_{z_1\in \mathbb{Z}^d\setminus B_{x_1}(L)} \mathbb{P}\left[ \mathsf{P}(x_1,M)\right]\\
		&\ \  \     \ \ \ \ \ \  \  \ \ \  \  \  \  \ \ \ \ \  \  \ \ \  \  \  \  \ \ \ \ \ \  \  \ \ \  \  \  \  \ \ \ \ \ \  \  \  \cdot  |z_1-x_1|^{2-d}|z_1-y_1|^{2-d} \  \ \   (\text{by}\ (\ref{fin_4.6})).
	\end{split}
\end{equation*}
For any $x_1\in \mathbb{Z}^d\setminus B(n-1)$, $y_1\in B_{x_1}(\frac{1}{2}L)\setminus B(n_*)$ and $z_1\in \mathbb{Z}^d\setminus B_{x_1}(L)$, by $|z_1-y_1|\ge |z_1-x_1|-|x_1-y_1|\ge \frac{1}{2}|z_1-x_1|$ and (\ref{fina_new_4.4}), we have $$\sum_{z_1\in \mathbb{Z}^d\setminus B_{x_1}(L)}|z_1-x_1|^{2-d}|z_1-y_1|^{2-d}\le 	C\sum_{z_1\in \mathbb{Z}^d\setminus B_{x_1}(L)} |z_1-x_1|^{4-2d}\le CL^{4-d}.$$
Combined with the previous upper bound for $\mathcal{S}_2^{(2)}$, it yields that  
\begin{equation*}\label{ineq_6.122}
	\begin{split}
		\mathcal{S}_2^{(2)}\le 	& C(K,d)L^{6-d} \big|B(\tfrac{1}{2}L)\big|\cdot\sum_{x_1\in \mathbb{Z}^d\setminus B(n-1)} \mathbb{P}\left[ \mathsf{P}(x_1,M)\right] \\
		\le &C(K,d)L^{6}M\cdot \mathbb{P}\left( \psi_n^{\text{SR}}=M\right).
	\end{split}
\end{equation*}
Combining these two estimates for $\mathcal{S}_2^{(1)}$ and $\mathcal{S}_2^{(2)}$, we obtain
\begin{equation}\label{ineq_123}
	\mathcal{S}_2\le C(K,d)L^{6}M\cdot \mathbb{P}\left( \psi_n^{\text{SR}}=M\right).
\end{equation}

\textbf{Part 3:}  Since $x_1\neq x_2$, $\gamma_{x_{1},K}^{\mathrm{f}}$ is independent of $\gamma_{x_{2},K}^{\mathrm{f}}$. For each $i\in \{1,2\}$, we denote by $\mathbf{C}_{y_i,K}$ (resp. $\mathbf{C}_{y_i,K}^*$) the cluster containing $y_i$ and composed of loops in $\widetilde{\mathcal{L}}_{1/2}\cdot \mathbbm{1}_{\widetilde{\ell}\notin \mathfrak{L}_{x_1,K}\cup \mathfrak{L}_{x_2,K}}$ (resp. $\widetilde{\mathcal{L}}_{1/2}\cdot \mathbbm{1}_{\widetilde{\ell}\notin \mathfrak{L}_{x_i,K}}$).

Here are some useful observations.

\begin{enumerate}
	
	\item  For each $i\in \{1,2\}$, $\mathbf{C}_{y_i,K}=\mathbf{C}_{y_i,K}^*$. To see this, we only need to prove that $\mathbf{C}_{y_i,K}$ is disjoint from $\gamma_{x_{3-i},K}^{\mathrm{f}}$. We prove this by contradiction. Without loss of generality, assume that $\mathbf{C}_{y_1,K}$ intersects $\gamma_{x_{2},K}^{\mathrm{f}}$. Then $y_1$ can be connected to $\widehat{\mathbf{\Psi}}_n$ without $\gamma_{x_{1},K}^{\mathrm{f}}$ (since $\gamma_{x_{2},K}^{\mathrm{f}}$ intersects $\widehat{\mathbf{\Psi}}_n$), which is contradictory with the pivotality of $\gamma_{x_{1},K}^{\mathrm{f}}$.

	\item For each $i\in \{1,2\}$, $\mathbf{C}_{y_i,K}$ intersects $\gamma_{x_i,K}^{\mathrm{f}}$, and thus also intersects $B_{x_i}(2K^4_*)$. In fact, since $(x_i,y_i)$ is admissible, $\mathbf{C}_{y_i,K}^*$ must intersect $\gamma_{x_i,K}^{\mathrm{f}}$. Combined with Observation (1), it implies this observation.

	\item $\mathbf{C}_{y_1,K}$ is disjoint from $\mathbf{C}_{y_2,K}$. Otherwise, one has $\mathbf{C}_{y_1,K}=\mathbf{C}_{y_2,K}$ and therefore, $\mathbf{C}_{y_1,K}$ intersects both $\gamma_{x_1,K}^{\mathrm{f}}$ and $\gamma_{x_2,K}^{\mathrm{f}}$ (by Observation (2)). As a result, $\mathbf{C}_{y_1,K}$ and $\widehat{\mathbf{\Psi}}_n$ can be connected by either $\gamma_{x_1,K}^{\mathrm{f}}$ or $\gamma_{x_2,K}^{\mathrm{f}}$, which is contradictory with the fact that $\gamma_{x_1,K}^{\mathrm{f}}$ is pivotal.

\end{enumerate}
Observations (1) and (2) imply that
\begin{equation*}
	\{y_1\xleftrightarrow[]{} B_{x_1}(2K^4_*)\ \text{off}\  \widehat{\mathbf{\Psi}}_n    \} \circ  \{y_2\xleftrightarrow[]{} B_{x_2}(2K^4_*)\ \text{off}\  \widehat{\mathbf{\Psi}}_n    \}
\end{equation*}
happens. Since in addition the event $\mathsf{P}(x_1,M) \cap \mathsf{P}(x_2,M)$ is certified by a disjoint collection of glued loops, by the BKR inequality and the two-point function estimate, we have 
\begin{equation}\label{ineq_6.124}
	\begin{split}
		\mathcal{S}_3\le &C(K,d)\sum_{\forall i\in \{1,2\},x_i\in \mathbb{Z}^d\setminus B(n-1), y_i\in B_{x_i}(\frac{1}{2}L)\setminus B(n_*)} \mathbb{P}\left[	\mathsf{P}(x_1,M) \cap \mathsf{P}(x_2,M)\right] \\
		&\ \ \ \ \ \ \ \ \ \ \ \ \ \ \  \ \ \ \ \  \ \ \ \ \ \ \ \ \ \ \ \ \ \ \ \ \ \ \ \ \ \ \ \ \ \ \ \ \ \ \ \ \ \cdot |x_1-y_1|^{2-d}|x_2-y_2|^{2-d}\\
		\le & C(K,d)L^4\sum_{x_1,x_2\in \mathbb{Z}^d\setminus B(n-1)}\mathbb{P}\left[	\mathsf{P}(x_1,M) \cap \mathsf{P}(x_2,M)\right]\ \ \ \ \ \ \ \ (\text{by}\ (\ref{fina4.3}))\\
		\le  &C(K,d)L^4M^2 \cdot \mathbb{P}\left( \psi_n^{\text{SR}}=M\right).
	\end{split}
\end{equation}

Finally, we put (\ref{ineq_117}), (\ref{ineq_123}), (\ref{ineq_6.124}) and the requirement $M\ge \frac{1}{2}L^2$ together, and then the proof is complete.
\end{proof}

\begin{remark}\label{remark_8.9}
Recall that Observation (3) in the analysis of Part 3 above indicates that if $(x_1,y_1)$ and $(x_2,y_2)$ are both admissible, and $x_1\neq x_2$, then $\mathbf{C}_{y_1,K}$ is disjoint from $\mathbf{C}_{y_2,K}$, which implies that $y_1\neq y_2$. As a result, if $(x_1,y)$ and $(x_2,y)$ are both admissible (i.e. taking $y_1=y_2=y$), then we must have $x_1=x_2$.
\end{remark}

%
%
%
%
%
%
%
%
%
%
%
%
%
%
%
%
%
%
%
%
%
%

Recall that $\chi_{n}=\big|\{x\in B(n+L)\setminus B(n): \bm{0} \xleftrightarrow[]{} x \} \big|$.

\begin{lemma}\label{lemma_regular}
There exist $c_{9}(d)>0$ and $c_{10}(d)\in (0,1)$ such that under the same conditions as Theorem \ref{theorem_regularity}, we have  
\begin{equation}\label{equation_lemma_regular}
	\mathbb{P}\left( \psi_n^{\mathrm{SR}}\ge \tfrac{1}{2}L^2, \chi_n\le c_{9}L^4 \right) \le (1-c_{10}) \theta(N). 
\end{equation}
\end{lemma}
\begin{proof}
Recall the total number of admissible pairs $W$ in (\ref{def_W}).

We claim that $W\le \chi_n$. In fact, for each admissible pair $(x,y)$, one has $y\in B(n+L)\setminus B(n)$ and $y\xleftrightarrow[]{} \bm{0}$, and thus $y$ is counted by $\chi_n$. Moreover, by Remark \ref{remark_8.9} for each $y$ counted by $\chi_n$, it can be contained in at most one admissible pair. As a result, we obtain $W\le \chi_n$.


Recall that for any $u>0$ and random variable $Z\ge 0$ with $0<u <\mathbb{E}Z<\infty$, one has $\mathbb{P}[Z> u]\ge (\mathbb{E}Z-u)^2/\mathbb{E}[Z^2]$. Arbitrarily take $c_{9}\in (0, \frac{1}{2}c_{8})$. Note that $c_{8}L^2M>c_{9}L^4$ for all $M\ge \frac{1}{2}L^2$. Applying the general inequality above with $u=c_{9}L^4$ and $Z$ being the random variable $W$ conditioning on $\{\psi_n^{\text{SR}}=M\}$, we have: for any integer $M\ge \frac{1}{2}L^2$,   
\begin{equation}\label{ineq_6.127}
	\begin{split}
		&\mathbb{P}\left(  \chi_n > c_{9}L^4 \big| \psi_n^{\text{SR}}=M \right) \\
		\ge &\mathbb{P}\left[ W> c_{9}L^4 \big| \psi_n^{\text{SR}}=M \right]\ \ \ \ \ \ \  \ \ \ \ \ \  (\text{by}\ W\le \chi_n)\\
		\ge &  \frac{\big\{\mathbb{E}\left[W\mid \psi_n^{\mathrm{SR}}=M \right]-c_{9}L^4\big\}^2}{\mathbb{E}\left[W^2\mid \psi_n^{\text{SR}}=M \right]}   \\
		\ge &\frac{\left[c_{8}L^2M-c_{9}L^4\right]^2 }{C_{7}L^4M^2}\ge c\in (0,1)\ \ \ \ \ \ (\text{by Lemmas}\ \ref{lemma_6.15}\ \text{and}\ \ref{lemma6.113}).
	\end{split}
\end{equation}
By (\ref{ineq_6.127}), we get the desired bound as follows:
\begin{equation*}
	\begin{split}
		\mathbb{P}\Big( \psi_n^{\text{SR}}\ge \tfrac{1}{2}L^2, \chi_n\le c_{9}L^4 \Big) =&  \sum_{M\ge \frac{1}{2}L^2} \mathbb{P}\left( \psi_n^{\text{SR}}=M, \chi_n\le c_{9}L^4 \right)    \\
		\le &(1-c)\sum_{M\ge \frac{1}{2}L^2}  \mathbb{P}\left( \psi_n^{\text{SR}}=M\right)\\
		\le & (1-c)\cdot \theta(N),
		\end{split} 
\end{equation*} 
where in the last line we used the fact that $\{\psi_n^{\text{SR}}>0\}\subset \{\bm{0}\xleftrightarrow[]{} \partial B(N)\}$. 
\end{proof}

Based on Corollary \ref{coro_61} and Lemma \ref{lemma_regular}, we are ready to prove Theorem \ref{theorem_regularity}.

\begin{proof}[Proof of Theorem \ref{theorem_regularity}]
By Corollary \ref{coro_61} and Lemma \ref{lemma_regular}, we have 
\begin{equation}\label{ineq_840}
	\begin{split}
		& \mathbb{P}\Big(  \psi_{n}^*\ge L^2, \chi_{n}\le c_{9}L^4  \Big) \\
		\le & \mathbb{P}\Big( \psi_{n}^*\ge L^2, \psi_{n}^{\mathrm{SR}}\le \tfrac{1}{2}\psi_{n}^* \Big) + \mathbb{P}\Big( \psi_n^{\mathrm{SR}}\ge \tfrac{1}{2}L^2, \chi_n\le c_{9}L^4 \Big) \\
		\le & \mathrm{s.p.}(N) +  (1-c_{10}) \theta(N). 
	\end{split}
\end{equation}
For all large enough $N$, by the polynomial lower bound of $\theta(N)$ in (\ref{ineq_onearm}), the RHS of (\ref{ineq_840}) is dominated by $\frac{1}{2}c_{10}\theta(N)+ (1-c_{10}) \theta(N)= (1-\frac{1}{2}c_{10})\theta(N)$. Then Theorem \ref{theorem_regularity} follows by setting $c_4 = c_{9}$ and $c_5 = \frac{1}{2}c_{10}$.
\end{proof}

\section{Decay rate of the cluster volume}\label{section_volumn}

In this section, we will prove Proposition \ref{prop_volumn}, which then completes the proof of Theorem \ref{theorem1}. This proof is inspired by \cite{barsky1991percolation}.


Recall that for any $x\in \mathbb{Z}^d$, $\mathbf{C}(x)$ is the cluster of $\cup \widetilde{\mathcal{L}}_{1/2}$ containing $x$. Also recall that for any $A\subset \widetilde{\mathbb{Z}}^d$, $|A|$ is the number of lattice points in $A$. For any $m\in \mathbb{N}$, we denote $P_m=\mathbb{P}(|\mathbf{C}(\bm{0})|=m)$. For any $h\ge 0$, let $\mathfrak{R}(h):= \sum_{m=1}^{\infty}P_m(1-e^{-mh})$. The key is the following upper bound on $\mathfrak{R}(h)$.
\begin{lemma}\label{lemma_9.1}
For $d>6$, there exists $C_{8}(d)>0$ such that for any $h> 0$, 
\begin{equation}
	\mathfrak{R}(h) \le C_{8}h^{\frac{1}{2}}. 
\end{equation}
\end{lemma}

\begin{proof}[Proof of Proposition \ref{prop_volumn} assuming Lemma \ref{lemma_9.1}]
	For any $M\ge 1$, applying Lemma \ref{lemma_9.1} with $h=M^{-1}$, we get that
	\begin{equation}
		C_{8}M^{-\frac{1}{2}} \ge \mathfrak{R}(M^{-1})\ge \sum_{m=M}^{\infty}P_m(1-e^{-M^{-1}m})\ge (1-e^{-1}) \mathbb{P}\left(|\mathbf{C}(\bm{0})|\ge M \right),  
	\end{equation}
	thereby completing the proof of Proposition \ref{prop_volumn}.
\end{proof}



To prove Lemma \ref{lemma_9.1}, we need to consider the so-called \textit{ghost field} $\{\mathscr{G}_{x}^h\}_{x\in \mathbb{Z}^d}$ with parameter $h\ge 0$. Precisely, $\{\mathscr{G}_{x}^h\}_{x\in \mathbb{Z}^d}$ is independent of $\widetilde{\mathcal{L}}_{1/2}$ and is a collection of i.i.d. $\{0,1\}$-valued Bernoulli variables which take value $0$ with probability $e^{-h}$. Let $\mathcal{G}^h:=\{x\in \mathbb{Z}^d:\mathscr{G}_{x}^h=1\}$. With the help of the ghost field, we can write $\mathfrak{R}(h)$ as the probability of a connecting event as follows: 
\begin{equation}\label{eq_9.4}
\mathfrak{R}(h)=\sum_{m=1}^{\infty} P_m(1-e^{-mh})= \mathbb{P}\left(\bm{0} \xleftrightarrow[]{} \mathcal{G}^h\right), 
\end{equation}
where ``$\xleftrightarrow[]{}$'' is ``$\xleftrightarrow[]{\cup\widetilde{\mathcal{L}}_{1/2}}$'', and the probability space of the RHS is the product space for $\widetilde{\mathcal{L}}_{1/2}$ and $\{\mathscr{G}_{x}^h\}_{x\in \mathbb{Z}^d}$. The next lemma provides a geometric interpretation for the derivative $\mathfrak{R}'(h)$, i.e., the derivative of $\mathfrak{R}(h)$ with respect to $h$.


For any $A_1,A_2\subset \widetilde{\mathbb{Z}}^d$, we use the notation $A_1 \nleftrightarrow A_2:= \{A_1 \xleftrightarrow[]{} A_2\}^c$.

\begin{lemma}\label{new_lemma_9.2}
 For any $h\ge 0$, we have 

\begin{equation}\label{eq_9.5}
	(e^h-1) \mathfrak{R}'(h) = \mathbb{P}\left(|\mathbf{C}(\bm{0})\cap \mathcal{G}^h|=1 \right),
\end{equation}
\begin{equation}\label{eq_97}
	\mathfrak{R}'(h) = \sum_{x\in \mathbb{Z}^d} \mathbb{P}\left( \bm{0}\xleftrightarrow[]{} x,\bm{0} \nleftrightarrow \mathcal{G}^h \right).
\end{equation}
\end{lemma}

\begin{proof}
	For the LHS of (\ref{eq_9.5}), by (\ref{eq_9.4}) one has 
	\begin{equation}\label{eq_9.7}
		\begin{split}
	 \mathfrak{R}'(h) = \sum_{m=1}^{\infty} P_m \cdot m e^{-mh}.
		\end{split}	
	\end{equation}
For the RHS of (\ref{eq_9.5}), we have 
\begin{equation}\label{9.10}
	\begin{split}
		&\mathbb{P}\left(|\mathbf{C}(0)\cap \mathcal{G}^h|=1 \right)\\
		=& \sum_{m=1}^{\infty} \mathbb{P}\left(|\mathbf{C}(\bm{0})|=m \right)\cdot \mathbb{P}\left(|\mathbf{C}(0)\cap \mathcal{G}^h|=1 \Big||\mathbf{C}(\bm{0})|=m \right) \\
		=& \sum_{m=1}^{\infty}  P_m\cdot e^{-(m-1)h}\left(1-e^{-h} \right)
		=(e^h-1)\sum_{m=1}^{\infty} P_m\cdot me^{-mh}. 
	\end{split}
\end{equation}
By (\ref{eq_9.7}) and (\ref{9.10}), we get (\ref{eq_9.5}).

We now prove (\ref{eq_97}). Similar to (\ref{9.10}), we also have 
\begin{equation}\label{eq_9.11}
\begin{split}
&\sum_{x\in \mathbb{Z}^d} \mathbb{P}\left( \bm{0}\xleftrightarrow[]{} x,\bm{0} \nleftrightarrow \mathcal{G}^h \right) \\
=&\sum_{m=1}^{\infty}\sum_{x\in \mathbb{Z}^d}\mathbb{P}\left(\bm{0}\xleftrightarrow[]{} x, |\mathbf{C}(\bm{0})|=m \right)   \mathbb{P}\left( \bm{0} \nleftrightarrow \mathcal{G}^h \big | \bm{0}\xleftrightarrow[]{} x,| \mathbf{C}(\bm{0})|=m \right)\\
=& \sum_{m=1}^{\infty}e^{-mh} \sum_{x\in \mathbb{Z}^d}\mathbb{P}\left(\bm{0}\xleftrightarrow[]{} x, |\mathbf{C}(\bm{0})|=m \right) 
= \sum_{m=1}^{\infty} P_m\cdot me^{-mh}. 
\end{split}
\end{equation}
By (\ref{eq_9.7}) and (\ref{eq_9.11}), we obtain (\ref{eq_97}).
\end{proof}	



We introduce some more notations below for further analysis:
\begin{itemize}

\item Recall that for any $A_1,A_2,A_3\subset \widetilde{\mathbb{Z}}^d$, $\{A_1\xleftrightarrow[]{} A_2\ \text{off}\ A_3\}$ is the event that there exists a collection $\mathfrak{L}$ of loops in $\widetilde{\mathcal{L}}_{1/2}$ disjoint from $A_3$ such that $A_1\xleftrightarrow[]{\cup \mathfrak{L}} A_2$. For any $x\in \mathbb{Z}^d$ and $A\subset \widetilde{\mathbb{Z}}^d$, we denote
\begin{equation}\label{9.13}
	\mathbf{C}_{A}(x):=\{v\in \widetilde{\mathbb{Z}}^d: v\xleftrightarrow[]{} x\ \text{off}\ A\}.
\end{equation}

\item For three different $x,y,z \in \mathbb{Z}^d$, the event $\mathsf{E}(x,y,z)$ is defined to be the intersection of the following three events:
\begin{enumerate}
	\item $\bm{0} \xleftrightarrow[]{} x$ and $\bm{0}  \nleftrightarrow \mathcal{G}^h$;
	
	\item $y\xleftrightarrow[]{}\mathcal{G}^h$ and $z\xleftrightarrow[]{}\mathcal{G}^h$; 
	
	\item The clusters $\mathbf{C}(x)$, $\mathbf{C}(y)$ and $\mathbf{C}(z)$ are disjoint to each other. 

\end{enumerate}

See Figure \ref{pic4} for an illustration of this event.

\item For any $x\in \mathbb{Z}^d$ and $i\in \mathbb{N}^+$, let $x_i^+=x+(i,0,...,0)\in \mathbb{Z}^d$ and $x_i^-=x+(-i,0,...,0)\in \mathbb{Z}^d$. We denote the subset 
$A_{J}^x:=\cup_{i=1}^{J} \{x_i^+,x_i^- \}\cup \{x\}$.


\item 	The event $\mathsf{F}_{J}^x$ is the intersection of the following two events:
\begin{enumerate}
	\item $\gamma^{\mathrm{f}}_{A_{J}^x}\neq \emptyset$ (recalling in Section \ref{section_BKR} that $\gamma^{\mathrm{f}}_{A}$ is the glued loop composed of fundamental loops in $\widetilde{\mathcal{L}}_{1/2}$ that visit every point in $A$ and do not visit any other lattice point);
	
	\item After deleting all loops $\widetilde{\ell}$ included in $\gamma^{\mathrm{f}}_{A_{J}^x}$ (i.e. $\widetilde{\ell}$ is one of the loops that construct $\gamma^{\mathrm{f}}_{A_{J}^x}$) from $\widetilde{\mathcal{L}}_{1/2}$, the event $\mathsf{E}_{J}^x:=\mathsf{E}(x,x_J^-,x_J^+)$ occurs.
	
\end{enumerate}
Note that $\mathsf{F}_{J}^x$ implies $|\mathbf{C}(\bm{0})\cap \mathcal{G}^h|\ge 2$, which is incompatible with the event $\mathsf{E}_{J}^y$ for each $y\in \mathbb{Z}^d$.

\begin{figure}[h]
	\centering
	\includegraphics[width=10cm]{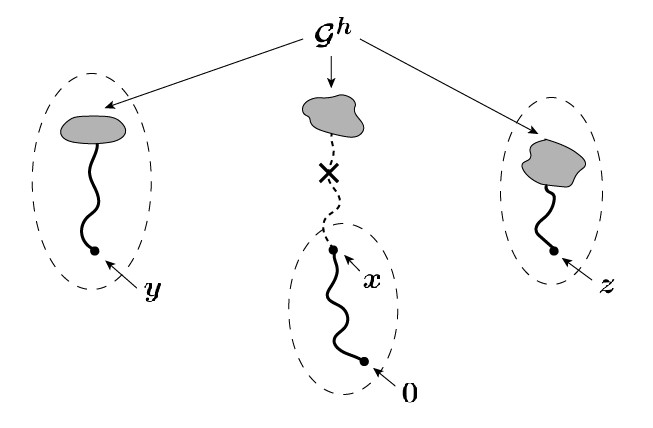}
	\caption{An illustration for the event $\mathsf{E}(x,y,z)$.  }
	\label{pic4}
\end{figure}


	%
	%
	%
	%
	%

\end{itemize}



\begin{lemma}\label{lemma_93}
For any lattice points $y_1\neq y_2$, $\mathsf{F}_{J}^{y_1}\cap \mathsf{F}_{J}^{y_2}=\emptyset$. 
\end{lemma}
\begin{proof}
	For any $w\in \mathbb{Z}^d$ and $i\in \{1,2\}$, let $\mathbf{C}_{w}$ (resp. $\mathbf{C}_{w}^{i}$) be the cluster containing $w$ and composed of loops in $\widetilde{\mathcal{L}}_{1/2}$ that are not included in $\gamma^{\mathrm{f}}_{A_{J}^{y_1}}$ or $\gamma^{\mathrm{f}}_{A_{J}^{y_2}}$ (resp. not included in $\gamma^{\mathrm{f}}_{A_{J}^{y_i}}$). We abbreviate $y_{i}^{\dagger}:=(y_i)_J^-$ and $y_{i}^{\diamond}:=(y_i)_J^+$.
	
%
%

	Assume the event $\mathsf{F}_{J}^{y_1}\cap \mathsf{F}_{J}^{y_2}$ occurs. Here are some useful observations. 
	\begin{enumerate}
	
		\item For $i\in \{1,2\}$, it follows from the definition of $\mathsf{F}_{J}^{y_i}$ that: (a) $\mathbf{C}_{y_i}^i$, $\mathbf{C}_{y_i^\dagger}^i$ and $\mathbf{C}_{y_i^\diamond}^i$ are disjoint from one another; (b) $\mathbf{C}_{y_i}^i$ contains $\bm{0}$ and is disjoint from $\mathcal{G}^h$; (c) $\mathbf{C}_{y_i^\dagger}^i$ and $\mathbf{C}_{y_i^\diamond}^i$ both intersect $\mathcal{G}^h$.

		\item For $i\in \{1,2\}$, either $\mathbf{C}_{y_i^\dagger}^i=\mathbf{C}_{y_i^\dagger}$ or $\mathbf{C}_{y_i^\diamond}^i=\mathbf{C}_{y_i^\diamond}$ since at most one of the clusters $\mathbf{C}_{y_i^\dagger}^i$ and $\mathbf{C}_{y_i^\diamond}^i$ can intersect $\gamma^{\mathrm{f}}_{A_{J}^{y_{3-i}}}$.

		\item For $i\in \{1,2\}$, $\mathbf{C}_{y_i}^i=\mathbf{C}_{y_i}$. We prove this observation by contradiction. Assume that $\mathbf{C}_{y_i}^i\neq\mathbf{C}_{y_i}$, then $\mathbf{C}_{y_i}$ must intersect $\gamma^{\mathrm{f}}_{A_{J}^{y_{3-i}}}$. Moreover, by Observation (2), without loss of generality we can also assume that $\mathbf{C}_{y_{3-i}^\dagger}^{3-i}=\mathbf{C}_{y_{3-i}^\dagger}$. Therefore, since $y_i$ can be connected to $\mathcal{G}^h$ by $\mathbf{C}_{y_i}^i\cup \gamma^{\mathrm{f}}_{A_{J}^{y_{3-i}}} \cup \mathbf{C}_{y_{3-i}^\dagger}^{3-i}$ (which equals to $\mathbf{C}_{y_i}\cup \gamma^{\mathrm{f}}_{A_{J}^{y_{3-i}}} \cup \mathbf{C}_{y_{3-i}^\dagger}$ and thus is contained in $\mathbf{C}_{y_i}^{i}$), we have that $\mathbf{C}_{y_i}^{i}$ intersects $\mathcal{G}^h$, which is contradictory with Observation (1b).

		\item For $i\in \{1,2\}$, $\mathbf{C}_{y_i^\dagger}^i$ and $\mathbf{C}_{y_i^\diamond}^i$ both intersect $\gamma^{\mathrm{f}}_{A_{J}^{y_{3-i}}}$. We prove this by contradiction. If $\mathbf{C}_{y_i^\dagger}^i\cap \gamma^{\mathrm{f}}_{A_{J}^{y_{3-i}}}=\emptyset$, then one has $\mathbf{C}_{y_i^\dagger}=\mathbf{C}_{y_i^\dagger}^i$. In addition, by Observations (1b) and (1c), $\bm{0}$ is connected to $\mathcal{G}^h$ by $\mathbf{C}_{y_i^\dagger}^{i} \cup \gamma^{\mathrm{f}}_{A_{J}^{y_{i}}} \cup \mathbf{C}_{y_i}^{i}$, which is contained in $\mathbf{C}_{\bm{0}}^{3-i}$ since $\mathbf{C}_{y_i^\dagger}^i=\mathbf{C}_{y_i^\dagger}$ and $\mathbf{C}_{y_i}^{i}=\mathbf{C}_{y_i}$ (by Observation (3)). However, this implies that  $\mathbf{C}_{\bm{0}}^{3-i}\cap \mathcal{G}^h\neq \emptyset$, and thus arrives at a contradiction  with Observation (1b).

	\end{enumerate}
We next prove the lemma. Since $\mathbf{C}_{y_1^{\dagger}}^1$ and $\mathbf{C}_{y_1^{\diamond}}^1$ both intersect $\gamma^{\mathrm{f}}_{A_{J}^{y_2}}$ (by Observation (4)), one has that $y_1^{\dagger}$ and $y_1^{\diamond}$ are connected by $\mathbf{C}_{y_1^{\dagger}}^1\cup \mathbf{C}_{y_1^{\diamond}}^1\cup \gamma^{\mathrm{f}}_{A_{J}^{y_2}}$, which is incompatible with Observation (1a). Thus, we complete the proof by contradiction.
\end{proof}

\begin{lemma}\label{lemma_9.4}
For any $d>6$ and $J\in \mathbb{N}^+$, there exists $C_{9}(d,J)>0$ such that for any $h\ge 0$, 
\begin{equation}\label{9.14}
	\sum_{x\in \mathbb{Z}^d} \mathbb{P}\left( \mathsf{E}_{J}^x\right)   \le 	C_{9}\left[\mathfrak{R}(h)-(e^h-1)\mathfrak{R}'(h)\right]. 
\end{equation}
\end{lemma}
\begin{proof}
For any $x\in \mathbb{Z}^d$, when $\mathsf{E}_{J}^x$ happens, one has $\gamma^{\mathrm{f}}_{A_{J}^x}= \emptyset$. Moreover, if we add a loop $\widetilde{\ell}$, which constructs $\gamma^{\mathrm{f}}_{A_{J}^x}$, into the configuration of $\widetilde{\mathcal{L}}_{1/2}$, then the event $\mathsf{F}_{J}^x$ occurs. Therefore, we have 
\begin{equation}\label{ineq913}
	\mathbb{P}\left( \mathsf{E}_{J}^x\right)\le \frac{\mathbb{P}(\gamma^{\mathrm{f}}_{A_{J}^x}= \emptyset)}{\mathbb{P}(\gamma^{\mathrm{f}}_{A_{J}^x}\neq \emptyset)}\cdot 	\mathbb{P}\left( \mathsf{F}_{J}^x\right)\le C(d,J)	\mathbb{P}\left( \mathsf{F}_{J}^x\right). 
\end{equation}
Recall that Lemma \ref{lemma_93} shows that the events $\mathsf{F}_{J}^x$ for $x\in \mathbb{Z}^d$ are disjoint from one another. Thus, since $\mathsf{F}_{J}^x\subset \{|\mathbf{C}(\bm{0})\cap \mathcal{G}^h|\ge 2\}$, we have 
\begin{equation}\label{ineq914}
	\sum_{x\in \mathbb{Z}^d}\mathbb{P}\left( \mathsf{F}_{J}^x\right)= \mathbb{P}\left( \cup_{x\in \mathbb{Z}^d}\mathsf{F}_{J}^x\right)\le \mathbb{P}\left(|\mathbf{C}(\bm{0})\cap \mathcal{G}^h|\ge 2 \right), 
\end{equation}
where the RHS is equal to $\mathfrak{R}(h)-(e^h-1)\mathfrak{R}'(h)$ by (\ref{eq_9.4}) and (\ref{eq_9.5}). Thus, combining (\ref{ineq913}) and (\ref{ineq914}), we conclude this lemma.
\end{proof}

Recall that ``$x\xleftrightarrow[]{} y\ \text{only by}\ A$'' means that $x\xleftrightarrow[]{} y$, and in every collection of loops in $\widetilde{\mathcal{L}}_{1/2}$ connecting $x$ and $y$ there must be some loop intersecting $A$. Next, we give three technical lemmas. 


\begin{lemma}\label{lemma_9.5}
For any $d>6$, there exists $C(d)>0$ such that for any $y\in \mathbb{Z}^d$ and $A\subset \widetilde{\mathbb{Z}}^d$,
\begin{equation}
	\mathbb{P}\left(y\xleftrightarrow[]{} \mathcal{G}^h\ \text{only by}\ A \right) \le C(d)\mathfrak{R}(h)\sum_{v\in A\cap \mathbb{Z}^d}  |y-v|^{2-d}.
\end{equation}
\end{lemma}
\begin{proof}
This proof is parallel to that of (\ref{ineq_685}). On the event $\{y\xleftrightarrow[]{} \mathcal{G}^h\ \text{only by}\ A\}$, there exists a glued loop $\gamma_*$ intersecting $A$ such that $\{\gamma_*\xleftrightarrow[]{}y\} \circ \{\gamma_*\xleftrightarrow[]{}\mathcal{G}^h\}$ happens. By the same arguments as in the proof of (\ref{ineq_6.82}) (replacing $\widehat{\mathbf{A}}$, $x'$ and $y$ in (\ref{ineq_6.82}) by $A$, $y$ and $\mathcal{G}^h$ respectively), we have 
\begin{equation*}\label{9.18}
	\begin{split}
		&	\mathbb{P}\left(y\xleftrightarrow[]{} \mathcal{G}^h\ \text{only by}\ A \right)\\
		\le &C\sum_{z_1\in A\cap\mathbb{Z}^d,z_2,z_3\in \mathbb{Z}^d} |z_1-z_2|^{2-d}|z_2-z_3|^{2-d}|z_3-z_1|^{2-d}|z_2-y|^{2-d}\mathbb{P}\left(z_3\xleftrightarrow[]{}\mathcal{G}^h \right)\\
		 =&C\mathfrak{R}(h)\sum_{z_1\in A\cap\mathbb{Z}^d,z_2,z_3\in \mathbb{Z}^d} |z_1-z_2|^{2-d}|z_2-z_3|^{2-d}|z_3-z_1|^{2-d}|z_2-y|^{2-d} (\text{by}\ (\ref{eq_9.4}))\\
		 \le &C\mathfrak{R}(h)\sum_{z_1\in A\cap\mathbb{Z}^d}|z_1-y|^{2-d}\ \ \  \ \ \ \ \ \ \ \ \ \ \ \ \ \ \ \ \  \ \ \ \ \ \ \ \ \ \ \ \ \ \ \  \ \  \ \ \  (\text{by}\ (\ref{fin_4.6})).   \qedhere
	\end{split}
\end{equation*}
\end{proof}


\begin{lemma}\label{lemma_9.6}
For any $d>6$, there exists $C(d)>0$ such that for any $J\in \mathbb{N}^+$, 
\begin{equation}\label{9.20}
	\sum_{x,y\in \mathbb{Z}^d}\mathbb{P}\left(\bm{0}\xleftrightarrow[]{}x,\bm{0}\xleftrightarrow[]{}y,\bm{0}\nleftrightarrow \mathcal{G}^h \right)|x_{J}^--y|^{2-d}\le CJ^{6-d} \mathfrak{R}'(h). 
\end{equation}
\end{lemma}
\begin{proof}
	Let $\mathcal{G}^h_*:= \{v\in \widetilde{\mathbb{Z}}^d: v\xleftrightarrow[]{}\mathcal{G}^h \}$. We claim that 
	\begin{equation}\label{new_9.16}
		\big\{\bm{0}\xleftrightarrow[]{}x,\bm{0}\xleftrightarrow[]{}y,\bm{0}\nleftrightarrow \mathcal{G}^h\big\} = \big\{\bm{0}\xleftrightarrow[]{}x\ \text{off}\ \mathcal{G}^h_* \big\}\cap \big\{\bm{0}\xleftrightarrow[]{}y\ \text{off}\ \mathcal{G}^h_* \big\}. 
	\end{equation}
	On the one hand, when $\big\{\bm{0}\xleftrightarrow[]{}x,\bm{0}\xleftrightarrow[]{}y,\bm{0}\nleftrightarrow \mathcal{G}^h\big\}$ occurs, $\mathbf{C}(\bm{0})$ does not contain any loop intersecting $\mathcal{G}^h_*$ (otherwise, $\bm{0}\xleftrightarrow[]{}\mathcal{G}^h$). Therefore, since $x,y\in \mathbf{C}(\bm{0})$ (ensured by $\bm{0}\xleftrightarrow[]{}x$ and $\bm{0}\xleftrightarrow[]{}y$), we have $\bm{0}\xleftrightarrow[]{}x\ \text{off}\ \mathcal{G}^h_* $ and $\bm{0}\xleftrightarrow[]{}y\ \text{off}\ \mathcal{G}^h_*$. On the other hand, on the event $\big\{\bm{0}\xleftrightarrow[]{}x\ \text{off}\ \mathcal{G}^h_* \big\}\cap \big\{\bm{0}\xleftrightarrow[]{}y\ \text{off}\ \mathcal{G}^h_* \big\}$, we directly have $\bm{0}\xleftrightarrow[]{}x$ and $\bm{0}\xleftrightarrow[]{}y$. In addition, we also have $\bm{0}\nleftrightarrow \mathcal{G}^h$; otherwise, one has $\bm{0}\in \mathcal{G}^h_*$, which is incompatible with both $\big\{\bm{0}\xleftrightarrow[]{}x\ \text{off}\ \mathcal{G}^h_*\big\} $ and $\big\{\bm{0}\xleftrightarrow[]{}y\ \text{off}\ \mathcal{G}^h_* \big\}$. To sum up, the event on the LHS of (\ref{new_9.16}) is contained in and contains the RHS, therefore (\ref{new_9.16}) follows.


	On the event $\big\{\bm{0}\xleftrightarrow[]{}x\ \text{off}\ \mathcal{G}^h_* \big\}\cap \big\{\bm{0}\xleftrightarrow[]{}y\ \text{off}\ \mathcal{G}^h_* \big\}$, by the tree expansion, there exists a glued loop $\gamma_*$ disjoint from $\mathcal{G}^h_*$ such that $\big\{\gamma_*\xleftrightarrow[]{}\bm{0} \ \text{off}\ \mathcal{G}^h_*\big\}\circ \big\{\gamma_*\xleftrightarrow[]{}x\ \text{off}\ \mathcal{G}^h_* \big\}\circ \big\{ \gamma_*\xleftrightarrow[]{}y\ \text{off}\ \mathcal{G}^h_*\big\}$ happens. Moreover, arbitrarily given $\{\mathcal{G}^h_*=\mathcal{K}\}$, the connection off $\mathcal{G}^h_*$ only depends on the loops in $\widetilde{\mathcal{L}}_{1/2}$ disjoint from $\mathcal{K}$, which are independent from the event $\{\mathcal{G}^h_*=\mathcal{K}\}$. This implies that for any $D_1,D_2\subset \mathbb{Z}^d$,
	\begin{equation}\label{fina_9.16}
		\mathbb{P}\left(D_1\xleftrightarrow[]{}D_2\ \text{off}\ \mathcal{G}^h_*\mid \mathcal{G}^h_* \right)\le \mathbb{P}\left(D_1\xleftrightarrow[]{}D_2\right).
	\end{equation}
Thus, with the same argument as proving (\ref{ineq_new_4.10}), we have 
\begin{equation}\label{new_9.17}
	\begin{split}
	 &\mathbb{P}\Big(\bm{0}\xleftrightarrow[]{}x\ \text{off}\ \mathcal{G}^h_* ,\bm{0}\xleftrightarrow[]{}y\ \text{off}\ \mathcal{G}^h_*\big| \mathcal{G}^h_*   \Big)    \\
		\le &C\sum_{z_1,z_2,z_3\in \mathbb{Z}^d} |z_1-z_2|^{2-d}|z_2-z_3|^{2-d}|z_3-z_1|^{2-d}|z_2-x|^{2-d} |z_3-y|^{2-d}\\
		&\ \ \ \ \ \ \ \ \ \ \ \ \ \ \cdot \mathbb{E}\Big[\mathbb{P}\Big(\bm{0}\xleftrightarrow[]{}z_1\ \text{off}\ \mathcal{G}^h_*\big| \mathcal{G}^h_*   \Big)  \Big],  
	\end{split}
\end{equation}
where we bounded $\mathbb{P}\left(z_2\xleftrightarrow[]{}x \ \text{off}\ \mathcal{G}^h_* \mid \mathcal{G}^h_* \right) $ and $\mathbb{P}\left(z_3\xleftrightarrow[]{}y \ \text{off}\ \mathcal{G}^h_* \mid \mathcal{G}^h_* \right)$ from above by $C|z_2-x|^{2-d}$ and $C|z_3-y|^{2-d}$ respectively through applying (\ref{fina_9.16}).

For the same reason as proving (\ref{new_9.16}), one has $\{\bm{0}\xleftrightarrow[]{}z_1\ \text{off}\ \mathcal{G}^h_*\}= \{\bm{0}\xleftrightarrow[]{}z_1, \bm{0}\nleftrightarrow \mathcal{G}^h\}$. Therefore, by taking integral on both sides of (\ref{new_9.17}), the LHS of (\ref{9.20}) is bounded from above by 
\begin{equation*}
	\begin{split}
		\mathbb{I}:=	C\sum_{x,y\in \mathbb{Z}^d}\sum_{z_1,z_2,z_3\in \mathbb{Z}^d} 	&|z_1-z_2|^{2-d}|z_2-z_3|^{2-d}|z_3-z_1|^{2-d}|z_2-x|^{2-d} \\
		&\cdot \mathbb{P}\left(\bm{0}\xleftrightarrow[]{}z_1, \bm{0}\nleftrightarrow \mathcal{G}^h \right)|z_3-y|^{2-d}|x_{J}^--y|^{2-d}.
	\end{split}
\end{equation*}
Since $|z_3-y|=|(z_3)_J^+-y_J^+|$ and $|x_{J}^--y|=|x-y_J^+|$, we have 
\begin{equation*}\label{9.22}
	\begin{split}
		\mathbb{I}
		=C\sum_{x,y\in \mathbb{Z}^d}\sum_{z_1,z_2,z_3\in \mathbb{Z}^d}&|z_1-z_2|^{2-d}|z_2-z_3|^{2-d}|z_3-z_1|^{2-d}|z_2-x|^{2-d} \\
		&\cdot \mathbb{P}\left(\bm{0}\xleftrightarrow[]{}z_1, \bm{0}\nleftrightarrow \mathcal{G}^h \right)|(z_3)_J^+-y_J^+|^{2-d}|x-y_J^+|^{2-d}.
	\end{split}
\end{equation*}
By calculating the sum over $y$ and $x$ in turn, we get 
\begin{equation}\label{9.23}
	\begin{split}
		\mathbb{I}\le& 	C\sum_{x\in \mathbb{Z}^d}\sum_{z_1,z_2,z_3\in \mathbb{Z}^d}|z_1-z_2|^{2-d}|z_2-z_3|^{2-d}|z_3-z_1|^{2-d}|z_2-x|^{2-d} \\
		&\ \ \ \  \ \ \ \ \  \  \ \ \ \ \ \ \ \cdot \mathbb{P}\left(\bm{0}\xleftrightarrow[]{}z_1, \bm{0}\nleftrightarrow \mathcal{G}^h \right)|(z_3)_J^+-x|^{4-d}\  \ (\text{by}\ (\ref{ineq_4.3}))\\
		\le&\mathbb{I}'(\mathbb Z^d \times \mathbb Z^d \times \mathbb Z^d)\ \ \  \ \ \ \ \ \ \ \ \ \  \ \ \ \ \ \ \  \ \ \  \ \ \ \ \ \ \ \ \  \ \ \ \ \ \ \ \ \ \  \ \ \ (\text{by}\ (\ref{cal_4.6})),  
	\end{split}	
\end{equation}
where for any $A \subset \mathbb Z^d \times \mathbb Z^d \times \mathbb Z^d$, we define 
$$\mathbb I'(A)= \sum\limits_{(z_1, z_2, z_3) \in A}|z_1-z_2|^{2-d}|z_2-z_3|^{2-d}|z_3-z_1|^{2-d}|z_2-(z_3)_J^+|^{6-d}\mathbb{P}(\bm{0}\xleftrightarrow[]{}z_1, \bm{0}\nleftrightarrow \mathcal{G}^h ).$$ We decompose $\mathbb Z^d \times \mathbb Z^d \times \mathbb Z^d  = A_1 \cup A_2 \cup A_3$ where
$$A_1:=\{(z_1,z_2,z_3):|z_2-(z_3)_J^+|\ge 0.5J\},$$ $$A_2:=\{(z_1,z_2,z_3):|z_2-(z_3)_J^+|< 0.5J,|z_1-z_2|\ge 2J\},$$ $$A_3:=\{(z_1,z_2,z_3):|z_2-(z_3)_J^+|< 0.5J,|z_1-z_2|< 2J\}.$$

We next bound $\mathbb I'(A_1), \mathbb I'(A_2)$ and $\mathbb I'(A_3)$ one after another. For $(z_1, z_2, z_3)\in A_1$, since $|z_2-(z_3)_J^+|\ge 0.5J$, $\mathbb{I}'(A_1)$ is bounded from above by 
\begin{equation}\label{9.25}
	\begin{split}
		& CJ^{6-d}\sum_{z_1,z_2,z_3\in A_i} |z_1-z_2|^{2-d}|z_2-z_3|^{2-d}|z_3-z_1|^{2-d}\mathbb{P}\left(\bm{0}\xleftrightarrow[]{}z_1, \bm{0}\nleftrightarrow \mathcal{G}^h \right)\\
		\le &CJ^{6-d} \sum_{z_1,z_2\in \mathbb{Z}^d} |z_1-z_2|^{6-2d}\mathbb{P}\left(\bm{0}\xleftrightarrow[]{}z_1, \bm{0}\nleftrightarrow \mathcal{G}^h \right) \ \ \  (\text{by}\ (\ref{ineq_4.3}))\\
		\le & CJ^{6-d}\mathfrak{R}'(h) \ \ \ \ \ \ \ \ \ \ \ \ \ \ \ \ \ \ \ \ \ \ \ \ \ \ \ \ \ \ \ \ \ \ \ \ \ \ \ \ \ \ \ \ \ \ \ \ \   (\text{by}\ (\ref{fina4.3})\ \text{and}\ (\ref{eq_97})).
	\end{split}
\end{equation} 

When $(z_1,z_2,z_3)\notin A_1$, one has $|z_2-(z_3)_J^+|<0.5J$. Therefore, by the triangle inequality,
$$|z_2-z_3|\ge |(z_3)_J^+-z_3|- |z_2-(z_3)_J^+|\ge J-0.5J=  0.5J.$$ 
Thus, for $i\in \{2,3\}$, we have 
\begin{equation}\label{9.27}
	\begin{split}
		\mathbb{I}'(A_i) \le CJ^{2-d} \sum_{z_1,z_2,z_3\in A_i}& |z_1-z_2|^{2-d}|z_3-z_1|^{2-d}|z_2-(z_3)_J^+|^{6-d}\\
		&\cdot \mathbb{P}\left(\bm{0}\xleftrightarrow[]{}z_1, \bm{0}\nleftrightarrow \mathcal{G}^h \right). 
	\end{split}
\end{equation}

For $(z_1, z_2, z_3)\in A_2$, one has $z_2\in \mathbb{Z}^d\setminus B_{z_1}(2J)$, $z_3\in B_{z_2}(2J)$ and 
\begin{equation*}
	|z_3-z_1|\ge |z_2-z_1|-|z_2-(z_3)_J^+|-|z_3-(z_3)_J^+|\ge (1-0.25-0.5)|z_2-z_1|=0.25|z_2-z_1|.
\end{equation*} 
Therefore, by (\ref{9.27}), $\mathbb{I}'(A_2)$ is upper-bounded by 
\begin{equation}\label{9.28}
	\begin{split}
	&CJ^{2-d}\sum_{z_1\in \mathbb{Z}^d,z_2\in \mathbb{Z}^d\setminus B_{z_1}(2J),z_3\in B_{z_2}(2J)}|z_1-z_2|^{4-2d} |z_2-(z_3)_J^+|^{6-d}\\
		&\ \ \ \ \ \ \ \ \ \ \ \ \ \  \ \ \ \ \ \ \ \ \ \ \ \ \ \ \ \ \ \ \ \ \ \ \ \ \ \ \ \ \ \cdot \mathbb{P}\left(\bm{0}\xleftrightarrow[]{}z_1, \bm{0}\nleftrightarrow \mathcal{G}^h \right)\\
		=& CJ^{2-d}\mathfrak{R}'(h) \sum_{z\in \mathbb{Z}^d\setminus B(2J)} |z|^{4-2d}  \sum_{z\in B(2J)} |z|^{6-d}\ \ \ \ \   \ \ \ \ \ \  \ (\text{by}\ (\ref{eq_97}))\\
		\le & CJ^{12-2d}\mathfrak{R}'(h)\le CJ^{6-d}\mathfrak{R}'(h)   \ \ \ \ \ \ (\text{by}\ (\ref{fina4.3}),\ (\ref{fina_new_4.4})\ \text{and}\ d>6). 
	\end{split}	
\end{equation}


For $(z_1,z_2,z_3)\in A_3$, one has $z_2\in B_{z_1}(2J)$ and $z_3\in B_{z_2}(1.5J)\subset B_{z_1}(4J)$. Therefore, by (\ref{9.27}) and $|z_2-(z_3)_J^+|^{6-d}\le 1$ we have 
\begin{equation*}\label{9.32}
	\begin{split}
		\mathbb{I}'(A_3)\le& CJ^{2-d} \sum_{z_1\in \mathbb{Z}^d}\sum_{z_2\in B_{z_1}(2J),z_3\in B_{z_1}(4J)} |z_1-z_2|^{2-d}|z_3-z_1|^{2-d}\mathbb{P}\left(\bm{0}\xleftrightarrow[]{}z_1, \bm{0}\nleftrightarrow \mathcal{G}^h \right)\\
		=& CJ^{2-d}\mathfrak{R}'(h) \sum_{z\in B(2J)}|z|^{2-d} \sum_{z\in B(4J)}|z|^{2-d}  \ \ \ \ \  \ \ \ \ \ \ \ \ \ (\text{by}\ (\ref{eq_97}))\\
		\le &CJ^{6-d}\mathfrak{R}'(h)\ \ \ \ \ \ \ \ \ \ \  \ \ \ \ \ \ \ \ \ \ \ \ \ \  \ \ \ \  \ \ \ \ \ \ \ \ \  \ \ \ \ \ \ \ \ \ (\text{by}\ (\ref{fina4.3})).   
	\end{split}
\end{equation*} 
Combined with (\ref{9.25}) and (\ref{9.28}), it concludes this lemma. \end{proof}




\begin{lemma}\label{lemma9.7}
For any $d>6$, there exists $C(d)>0$ such that for any $J\in \mathbb{N}^+$, 
\begin{equation}\label{9.35}
	\sum_{w\in \mathbb{Z}^d} \mathbb{P}\left(\bm{0}\xleftrightarrow[]{} w, \bm{0}\xleftrightarrow[]{} \mathcal{G}^h \right) |w_{2J}^-|^{2-d} \le CJ^{6-d} \mathfrak{R}(h). 
\end{equation}
\end{lemma}
\begin{proof}
By the same arguments as in the proof of (\ref{ineq_new_4.10}) (replacing $x_1$ and $x_2$ in (\ref{ineq_new_4.10}) by $w$ and $\mathcal{G}^h$ respectively), we have 
\begin{equation*}
	\begin{split}
		&\mathbb{P}\left(\bm{0}\xleftrightarrow[]{} w, \bm{0}\xleftrightarrow[]{} \mathcal{G}^h \right)\\
		\le&C \sum_{z_1,z_2,z_3\in \mathbb{Z}^d} |z_1-z_2|^{2-d}|z_2-z_3|^{2-d}|z_3-z_1|^{2-d}|z_1-w|^{2-d}|z_2|^{2-d}\mathbb{P}\left(z_3\xleftrightarrow[]{} \mathcal{G}^h \right)\\
		=& C\mathfrak{R}(h) \sum_{z_1,z_2,z_3\in \mathbb{Z}^d} |z_1-z_2|^{2-d}|z_2-z_3|^{2-d}|z_3-z_1|^{2-d}|z_1-w|^{2-d}|z_2|^{2-d}\ \ (\text{by}\ (\ref{eq_9.4}))\\
		\le & C\mathfrak{R}(h) |w|^{4-d} \ \  \ \ \ \ \ (\text{by}\ (\ref{fin_4.7})).      
	\end{split}
\end{equation*}
%
%
Combined with $|w_{2J}^-|=|w-\bm{0}_{2J}^+|$, this yields the desired bound: 
\begin{equation*}\label{ineq_9.38}
	\begin{split}
		\sum_{w\in \mathbb{Z}^d} \mathbb{P}\left(\bm{0}\xleftrightarrow[]{} w, \bm{0}\xleftrightarrow[]{} \mathcal{G}^h \right) |w_{2J}^-|^{2-d} \le& C\mathfrak{R}(h) \sum_{w\in \mathbb{Z}^d} |w|^{4-d}|w-\bm{0}_{2J}^+|^{2-d}\\
		\le &CJ^{6-d}\mathfrak{R}(h) \ \ \ \ \ \ \ \ \ \ (\text{by}\ (\ref{cal_4.6})).   \qedhere
	\end{split}	   
\end{equation*}
\end{proof}


Using these three technical lemmas, we can prove the following lower bound for $\sum_{x\in \mathbb{Z}^d} \mathbb{P}\left( \mathsf{E}_{L}^x\right)$.

\begin{lemma}\label{lemma_9.8}
For any $d>6$, there exist $C_{10}(d),c_{11}(d)>0$ such that for all $J\ge C_{10}$ and $h\ge 0$,
\begin{equation}\label{9.39}
	\sum_{x\in \mathbb{Z}^d} \mathbb{P}\left( \mathsf{E}_{J}^x\right)  \ge c_{11}\mathfrak{R}(h)\mathfrak{R}'(h). 
\end{equation}
\end{lemma}

\begin{proof}
For any $x\in \mathbb{Z}^d$, we define $\mathcal{A}_1,\mathcal{A}_2$ and $\mathcal{A}$ as follows:	
	\begin{itemize}
		\item $ \mathcal{A}_1= \mathcal{A}_1(x,\mathcal{G}^h):=\big\{\text{connected}\ \mathbf{C}_1\subset \widetilde{\mathbb{Z}}^d:\bm{0},x\in \mathbf{C}_1,\mathbf{C}_1\cap\mathcal{G}^h=\emptyset\big\}$;
		
		\item For any $\mathbf{C}_1\in\mathcal{A}_1$, $$\mathcal{A}_2= \mathcal{A}_2(x,\mathcal{G}^h,\mathbf{C}_1):=\big\{\text{connected}\ \mathbf{C}_2\subset \widetilde{\mathbb{Z}}^d:\mathbf{C}_2\cap \mathbf{C}_1=\emptyset,\mathbf{C}_2\cap \mathcal{G}^h\neq \emptyset\big\}.$$ 
		
		\item $\mathcal{A}=\mathcal{A}(x,\mathcal{G}^h):=\big\{(\mathbf{C}_1,\mathbf{C}_2):\mathbf{C}_1\in \mathcal{A}_1, \mathbf{C}_2\in \mathcal{A}_2\big\}$.
		
	\end{itemize}
	Recall the definition of $\mathsf{E}_{J}^x$ below (\ref{9.13}). Then it follows that $\mathsf{E}_{J}^x\cap \{ \mathbf{C}(\bm{0})=\mathbf{C}_1, \mathbf{C}(x_J^-)=\mathbf{C}_2\}\neq \emptyset$ if and only if $(\mathbf{C}_1, \mathbf{C}_2)\in \mathcal{A}$. Moreover, on the event $\{ \mathbf{C}(\bm{0})=\mathbf{C}_1, \mathbf{C}(x_J^-)=\mathbf{C}_2, (\mathbf{C}_1,\mathbf{C}_2)\in \mathcal{A} \}$, $\mathsf{E}_{J}^x$ happens if and only if $\{x_J^+\xleftrightarrow[]{}\mathcal{G}^h\ \text{off}\ \mathbf{C}_1\cup \mathbf{C}_2\}$. For any fixed $\mathbf{C}_1,\mathbf{C}_2\subset \widetilde{\mathbb{Z}}^d$, the event $\{\mathbf{C}(\bm{0})=\mathbf{C}_1,\mathbf{C}(x_J^-)=\mathbf{C}_2, (\mathbf{C}_1,\mathbf{C}_2)\in \mathcal{A}\}$ only depends on $\mathcal{G}^h\cap (\mathbf{C}_1\cup \mathbf{C}_2)$ and the loops in $\widetilde{\mathcal{L}}_{1/2}$ intersecting $\mathbf{C}_1\cup \mathbf{C}_2$, which are independent from the event $\{x_J^+\xleftrightarrow[]{}\mathcal{G}^h\ \text{off}\ \mathbf{C}_1\cup \mathbf{C}_2\}$. Therefore, we have 
\begin{equation}\label{9.36}
	\begin{split}
		&\mathbb{P}\left( \mathsf{E}_{x,J}\mid \mathbf{C}(\bm{0})=\mathbf{C}_1,\mathbf{C}(x_J^-)=\mathbf{C}_2,(\mathbf{C}_1,\mathbf{C}_2)\in \mathcal{A}  \right) \\
		=  & \mathbb{P}\left( x_J^+\xleftrightarrow[]{}\mathcal{G}^h\ \text{off}\ \mathbf{C}_1\cup \mathbf{C}_2 \mid \mathbf{C}(\bm{0})=\mathbf{C}_1,\mathbf{C}(x_J^-)=\mathbf{C}_2,(\mathbf{C}_1,\mathbf{C}_2)\in \mathcal{A}  \right)\\
		=& \mathbb{P}\left( x_J^+\xleftrightarrow[]{}\mathcal{G}^h\ \text{off}\ \mathbf{C}_1\cup \mathbf{C}_2 \right)\\
		=& \mathfrak{R}(h)-\mathbb{P}\left( x_J^+\xleftrightarrow[]{}\mathcal{G}^h\ \text{only by}\ \mathbf{C}_1\cup \mathbf{C}_2 \right) \ \ \ \ \ (\text{by}\ (\ref{eq_9.4})). 
	\end{split}
\end{equation}
By Lemma \ref{lemma_9.5}, one has 
\begin{equation}\label{9.38}
	\begin{split}
		\mathbb{P}\left( x_J^+\xleftrightarrow[]{}\mathcal{G}^h\ \text{only by}\ \mathbf{C}_1\cup \mathbf{C}_2 \right)
		\le &C\mathfrak{R}(h)\sum_{v\in(\mathbf{C}_1\cup \mathbf{C}_2)\cap \mathbb{Z}^d}  |x_J^+-v|^{2-d}\\
	\le &C\mathfrak{R}(h)\sum_{i=1}^{2}\sum_{v\in\mathbf{C}_i\cap \mathbb{Z}^d}  |x_J^+-v|^{2-d}.
	\end{split}
\end{equation}
By taking integral in (\ref{9.36}) over $\{\mathbf{C}(x_J^-)\in \mathcal{A}_2\}$ conditioning on $\{\mathbf{C}(\bm{0})=\mathbf{C}_1,\mathbf{C}_1\in \mathcal{A}_1\}$ and using (\ref{9.38}), we have 
\begin{equation*}\label{939}
	\begin{split}
		&\mathbb{P}\left( \mathsf{E}_{x,J}\mid \mathbf{C}(0)=\mathbf{C}_1,\mathbf{C}_1\in \mathcal{A}_1 \right)\\
		\ge  &\mathfrak{R}(h)\mathbb{P}\left( \mathbf{C}(x_J^-)\in \mathcal{A}_2 \mid \mathbf{C}(\bm{0})=\mathbf{C}_1,\mathbf{C}_1\in \mathcal{A}_1 \right) \\
		&- C\mathfrak{R}(h)\mathbb{P}\left( \mathbf{C}(x_J^-)\in \mathcal{A}_2 \mid \mathbf{C}(\bm{0})=\mathbf{C}_1,\mathbf{C}_1\in \mathcal{A}_1 \right)  \sum_{v\in   \mathbf{C}_1\cap \mathbb{Z}^d}  |x_J^+-v|^{2-d}     \\
		&-C\mathfrak{R}(h)\mathbb{E}\bigg[\mathbbm{1}_{\mathbf{C}(x_J^-)\in \mathcal{A}_2}\sum_{v\in   \mathbf{C}(x_J^-)\cap \mathbb{Z}^d}  |x_J^+-v|^{2-d}  \Big|  \mathbf{C}(\bm{0})=\mathbf{C}_1 ,\mathbf{C}_1\in \mathcal{A}_1\bigg]\\
		:=& \mathbb{J}_1(x,\mathbf{C}_1)-\mathbb{J}_2(x,\mathbf{C}_1)-\mathbb{J}_3(x,\mathbf{C}_1),
	\end{split}
\end{equation*}
which implies that 
\begin{equation}
	\sum_{x\in \mathbb{Z}^d}\mathbb{P}\left( \mathsf{E}_{x,J}\right)\ge \widehat{\mathbb{J}}_1-\widehat{\mathbb{J}}_2-\widehat{\mathbb{J}}_3,
\end{equation}
where $\widehat{\mathbb{J}}_i:= \sum_{x\in \mathbb{Z}^d}\mathbb{E}\left[ \mathbb{J}_i(x,\mathbf{C}(\bm{0}))\cdot \mathbbm{1}_{\mathbf{C}(\bm{0})\in \mathcal{A}_1}\right]$ for $i\in \{1,2,3\}$. In what follows, we estimate them separately.





For $\widehat{\mathbb{J}}_1$, with the same arguments as in (\ref{9.36}), we have 
\begin{equation}\label{9.41}
	\begin{split}
		\mathbb{J}_1(x,\mathbf{C}_1)=& \mathfrak{R}(h) \mathbb{P}\left( x_J^- \xleftrightarrow[]{} \mathcal{G}^{h}\ \text{off}\ \mathbf{C}_1 \mid \mathbf{C}(\bm{0})=\mathbf{C}_1,\mathbf{C}_1\in \mathcal{A}_1 \right)\\
		=& \mathfrak{R}^2(h)-\mathfrak{R}(h)\mathbb{P}\left( x_J^- \xleftrightarrow[]{} \mathcal{G}^{h}\ \text{only by}\ \mathbf{C}_1  \right). 
	\end{split}
\end{equation}
By the definition of $\mathcal{A}_1$, one has 
\begin{equation}\label{9.43}
	\begin{split}
		&\sum_{x\in \mathbb{Z}^d}\mathbb{E}\left[ \mathfrak{R}(h)\mathbb{P}\left( x_J^- \xleftrightarrow[]{} \mathcal{G}^{h}\ \text{only by}\ \mathbf{C}(\bm{0})  \right)\cdot \mathbbm{1}_{\mathbf{C}(\bm{0})\in \mathcal{A}_1} \right] \\
		\le &C\mathfrak{R}^2(h)\sum_{x\in \mathbb{Z}^d} \mathbb{E}\bigg[\sum_{v\in  \mathbf{C}(\bm{0})\cap \mathbb{Z}^d}  |x_J^--v|^{2-d}\cdot \mathbbm{1}_{\bm{0}\xleftrightarrow[]{}x,\bm{0}\nleftrightarrow \mathcal{G}^h} \bigg]\ \  (\text{by Lemma}\ \ref{lemma_9.5}) \\
		\le & C\mathfrak{R}^2(h)\sum_{x\in \mathbb{Z}^d}\sum_{v\in \mathbb{Z}^d} |x_J^--v|^{2-d} \mathbb{P}\left(\bm{0}\xleftrightarrow[]{}v,\bm{0}\xleftrightarrow[]{}x,\bm{0}\nleftrightarrow \mathcal{G}^h\right)\\
		\le & CJ^{6-d} \mathfrak{R}^2(h)\mathfrak{R}'(h)\ \ \ \ \ \ \ \ \ \ \ \ \ \ \ \ \ \ \ \ \ \ \ \ \ \ \ \ \ \ \ \ \ \ \ \ \ \ \ \  \ \ \ \ (\text{by Lemma}\ \ref{lemma_9.6}).
	\end{split}
\end{equation}
Combining (\ref{9.41}), (\ref{9.43}) and $\mathbb{P}(\mathbf{C}(\bm{0})\in \mathcal{A}_1)=\mathfrak{R}'(h)$ (by (\ref{eq_97})), we obtain 
\begin{equation}\label{9.49}
	\widehat{\mathbb{J}}_1 \ge \left(1-CJ^{6-d} \right) \mathfrak{R}^2(h)\mathfrak{R}'(h).
\end{equation}


For $\widehat{\mathbb{J}}_2$, since $\mathbb{P}\left( \mathbf{C}(x_J^-)\in \mathcal{A}_2 \mid \mathbf{C}(\bm{0})=\mathbf{C}_1,\mathbf{C}_1\in \mathcal{A}_1 \right)\le \mathbb{P}\left( x_J^-\xleftrightarrow[]{}\mathcal{G}^h \right)=\mathfrak{R}(h)$, we have 
\begin{equation*}
	\mathbb{J}_2(x,\mathbf{C}_1)\le C\mathfrak{R}^2(h)\sum_{v\in   \mathbf{C}_1\cap \mathbb{Z}^d}  |x_J^+-v|^{2-d}.
\end{equation*}	
By taking integral over the event $\{\mathbf{C}(\bm{0})\in \mathcal{A}_1\}$ (i.e. $\{\bm{0}\xleftrightarrow[]{}x,\bm{0}\nleftrightarrow \mathcal{G}^h\}$) and summing over $x\in \mathbb{Z}^d$, one has 
\begin{equation*}\label{9.47}
	\begin{split}
		\widehat{\mathbb{J}}_2 \le C\mathfrak{R}^2(h)\sum_{x\in \mathbb{Z}^d} \mathbb{E}\bigg[\sum_{v\in   \mathbf{C}(\bm{0})\cap \mathbb{Z}^d} |x_J^+-v|^{2-d}\cdot \mathbbm{1}_{\bm{0}\xleftrightarrow[]{}x,\bm{0}\nleftrightarrow \mathcal{G}^h} \bigg].
	\end{split}
\end{equation*}
For the same reason as in the third and fourth line of (\ref{9.43}), the RHS is also bounded form above by $CJ^{6-d} \mathfrak{R}^2(h)\mathfrak{R}'(h)$. To sum up, we have  
\begin{equation}\label{9.52}
	\widehat{\mathbb{J}}_2 \le CJ^{6-d} \mathfrak{R}^2(h)\mathfrak{R}'(h). 
\end{equation}


Now we consider $\widehat{\mathbb{J}}_3$. Recall in (\ref{9.13}) that for any $y\in \mathbb{Z}^d$, $\mathbf{C}_A(y):=\{v\in \widetilde{\mathbb{Z}}^d: v\xleftrightarrow[]{} y\ \text{off}\ A\}$. On the event $\{\mathbf{C}(\bm{0})=\mathbf{C}_1,\mathbf{C}_1\in \mathcal{A}_1,\mathbf{C}(x_J^-)\in \mathcal{A}_2\}$, every loop contained in $\mathbf{C}(x_J^-)$ does not intersect $\mathbf{C}_1$, and thus $\mathbf{C}(x_J^-)=\mathbf{C}_{\mathbf{C}_1}(x_J^-)$. In addition, since the event $\{\mathbf{C}_{\mathbf{C}_1}(x_J^-)\in \mathcal{A}_2\}$ only depends on $\mathcal{G}^h\setminus \mathbf{C}_1$ and the loops in $\widetilde{\mathcal{L}}_{1/2}$ disjoint from $\mathbf{C}_1$ (both of which are independent of $\{\mathbf{C}(\bm{0})=\mathbf{C}_1,\mathbf{C}_1\in \mathcal{A}_1\}$), we have 
\begin{equation}
	\begin{split}
		&\mathbb{E}\bigg[\mathbbm{1}_{\mathbf{C}(x_J^-)\in \mathcal{A}_2}\sum_{v\in \mathbf{C}(x_J^-)\cap \mathbb{Z}^d}  |x_J^+-v|^{2-d}  \Big|  \mathbf{C}(\bm{0})=\mathbf{C}_1,\mathbf{C}_1\in \mathcal{A}_1 \bigg]\\
		= & \mathbb{E}\bigg[\mathbbm{1}_{ \mathbf{C}_{\mathbf{C}_1}(x_J^-)\cap \mathcal{G}^h\neq \emptyset}\sum_{v\in \mathbf{C}_{\mathbf{C}_1}(x_J^-)\cap \mathbb{Z}^d}  |x_J^+-v|^{2-d}  \bigg],
	\end{split}
\end{equation}
which implies that
\begin{equation*}\label{9.54}
	\begin{split}
		\mathbb{J}_3(x,\mathbf{C}_1)\le C\mathfrak{R}(h)\mathbb{E}\bigg[\mathbbm{1}_{ \mathbf{C}_{\mathbf{C}_1}(x_J^-)\cap \mathcal{G}^h\neq \emptyset}\sum_{v\in \mathbf{C}_{\mathbf{C}_1}(x_J^-)\cap \mathbb{Z}^d}  |x_J^+-v|^{2-d}  \bigg].
	\end{split}
\end{equation*}
Therefore, since $\sum_{v\in \mathbf{C}_{\mathbf{C}_1}(x_J^-)\cap \mathbb{Z}^d}  |x_J^+-v|^{2-d} \le \sum_{v\in \mathbb{Z}^d}  \mathbbm{1}_{v\xleftrightarrow[]{}x_J^-}\cdot |x_J^+-v|^{2-d}$ and $\mathbbm{1}_{ \mathbf{C}_{\mathbf{C}_1}(x_J^-)\cap \mathcal{G}^h\neq \emptyset}\le \mathbbm{1}_{x_J^-\xleftrightarrow[]{} \mathcal{G}^h}$, we have 
\begin{equation}\label{9.58}
	\begin{split}
			\mathbb{J}_3(x,\mathbf{C}_1)\le&  C\mathfrak{R}(h)\sum_{v\in \mathbb{Z}^d} |v-x_J^+|^{2-d} \mathbb{P}\left(x_J^-\xleftrightarrow[]{}v,x_J^-\xleftrightarrow[]{} \mathcal{G}^h \right)\\
			=&C\mathfrak{R}(h)\sum_{v\in \mathbb{Z}^d} |(v-x_J^-)-\bm{0}_{2J}^+|^{2-d} \mathbb{P}\left(\bm{0}\xleftrightarrow[]{}v-x_J^-,\bm{0}\xleftrightarrow[]{} \mathcal{G}^h \right) \\
				\le &CJ^{d-6}\mathfrak{R}^2(h) \ \ \ \ \ \ \ \ \  \ \ \ \ \ (\text{by Lemma}\ \ref{lemma9.7}). 
	\end{split}
\end{equation}
Recalling that $\mathbb{P}(\mathbf{C}(\bm{0})\in \mathcal{A}_1)=\mathfrak{R}'(h)$, by (\ref{9.58}) we get
\begin{equation}\label{9.59}
	\begin{split}
		\widehat{\mathbb{J}}_3 
		\le CJ^{6-d} \mathfrak{R}^2(h)\mathfrak{R}'(h). 
	\end{split}
\end{equation}

Combining (\ref{9.49}), (\ref{9.52}) and (\ref{9.59}), and taking a large enough $J$, we finally complete the proof. 
\end{proof}

After getting Lemmas \ref{lemma_9.4} and \ref{lemma_9.8}, now we are ready to prove Lemma \ref{lemma_9.1}: 
\begin{proof}[Proof of Lemma \ref{lemma_9.1}]
By Lemmas \ref{lemma_9.4} and \ref{lemma_9.8}, we have: for any $h\ge 0$,  
\begin{equation*}
\frac{d \mathfrak{R}^2(h)}{d h}=2	\mathfrak{R}(h)\mathfrak{R}'(h)\le 2C_{9}c_{11}^{-1}\left[\mathfrak{R}(h)-(e^h-1)\mathfrak{R}'(h)\right],
\end{equation*}
where the RHS is upper-bounded by $2C_{9}c_{11}^{-1}$ since $\mathfrak{R}(h)$ is increasing and is at most $1$ (see (\ref{eq_9.4})). Take integral over $[0,h]$ and then we get this lemma.
\end{proof}

Recalling that Lemma \ref{lemma_9.1} is sufficient for Proposition \ref{prop_volumn}, we eventually conclude the main result Theorem \ref{theorem1}.

\section*{Acknowledgments}

J. Ding is partially supported by NSFC Key Program Project No. 12231002.

\bibliographystyle{plain}
\bibliography{ref}

\end{document}